\documentclass[a4paper, 11pt,reqno]{amsart}

\usepackage{amsmath, amsthm, amscd, amsfonts, amssymb, graphicx, color, mathrsfs}
\usepackage[all]{xy}
\usepackage[english]{babel}
\usepackage[utf8]{inputenc}
\usepackage{slashed}
\usepackage[colorlinks=true, pdfstartview=FitV, linkcolor=blue, citecolor=blue, urlcolor=blue]{hyperref}
\usepackage{comment}
\usepackage{mathtools}
\mathtoolsset{showonlyrefs}

\usepackage[a4paper,
            left=1.3in,
            right=1.3in,
            top=1.2in,
            bottom=1.2in,
            footskip=.25in]{geometry}

\usepackage{cancel}

\newtheorem{theorem}{Theorem}[section]
\newtheorem{lemma}[theorem]{Lemma}

\newtheorem{proposition}[theorem]{Proposition}

\newtheorem*{lemma*}{Lemma}

\theoremstyle{definition}
\newtheorem{definition}[theorem]{Definition}
\newtheorem{example}[theorem]{Example}

\newtheorem{notation}[theorem]{Notation}
\newtheorem{assumptions}[theorem]{Assumptions}

\theoremstyle{remark}
\newtheorem{remark}[theorem]{Remark}
\newtheorem*{questions}{Questions}

\usepackage{todonotes}

\newcommand{\Opw}{\mathrm{Op}^{\mathrm{w}}} 
\newcommand{\Optau}{\mathrm{Op}^\tau} 
\newcommand{\Op}{\mathrm{Op}} 
\newcommand{\Ker}{\mathsf{Ker}^\tau} 
\newcommand{\cv}{\mathrm{cv}^\tau} 
\newcommand{\cvi}{(\mathrm{cv}^\tau)^{-1}} 
\newcommand{\CV}{\mathrm{CV}^\tau} 
\newcommand{\CVI}{(\mathrm{CV}^\tau)^{-1}} 
\renewcommand{\k}{\kappa_{\sigma_1}}
\newcommand{\kk}{\kappa_{\sigma_2}}
\newcommand{\s}{\sigma_1}
\renewcommand{\ss}{\sigma_2}
\renewcommand{\H}{\mathbb{H}_n} 
\newcommand{\h}{\mathfrak{h}_n} 
\newcommand{\Hhat}{\widehat{\mathbb{H}}_n} 
\newcommand{\Ghat}{\widehat{G}} 
\newcommand{\Rhat}{\widehat{\mathbb{R}}} 
\newcommand{\RS}{\mathcal{H}_\pi} 
\newcommand{\R}{\mathbb{R}} 
\newcommand{\C}{\mathbb{C}} 
\newcommand{\N}{\mathbb{N}} 
\renewcommand{\Re}{\mathrm{Re}}
\newcommand{\dimG}{n} 
\newcommand{\hdim}{Q} 
\newcommand{\RO}{\mathcal{R}} 
\newcommand{\hdeg}{\nu} 
\newcommand{\Sp}{\mathrm{Sp}(2n, \mathbb{R})} 
\newcommand{\SC}{\mathscr{S}} 
\newcommand{\TD}{\mathscr{S}'} 
\newcommand{\Tr}{\mathrm{Tr}} 
\newcommand{\F}{\mathcal{F}} 

\newcommand{\Lie}[1]{\mathfrak{#1}}
\newcommand{\BL}{\mathscr{L}}
\newcommand{\esssup}{\mathop{\mathrm{ess\,sup}}}
\newcommand{\subbracket}[3]{{\left\langle #2, #3\right\rangle}_{ #1}}	 
\newcommand{\TF}{\mathscr{D}} 
\newcommand{\D}{\mathscr{D}'} 
\newcommand{\sL}{\mathcal{L}_G} 
\newcommand{\aut}{\mathcal{A}} 

\usepackage{lipsum}

\newcommand\blfootnote[1]{%
  \begingroup
  \renewcommand\thefootnote{}\footnote{#1}%
  \addtocounter{footnote}{-1}%
  \endgroup
}

\numberwithin{equation}{section}
\dedicatory{Dedicated to the memory of Marius M\u{a}ntoiu}
\title{Weyl Calculus on Graded Groups}

\author[S. Federico]{Serena Federico}

\address{
Department of Mathematics, \\
University of Bologna \\
piazza di Porta San Donato 5 \\
40126 Bologna \\
Italy}

\email{serena.federico2@unibo.it}

\author[D. Rottensteiner]{David Rottensteiner}

\address{
Department of Mathematics: Analysis,
Logic and Discrete Mathematics \\
Ghent University \\
Krijgslaan 281, S8 \\
9000 Gent \\
Belgium}

\email{david.rottensteiner@ugent.be}

\author[M. Ruzhansky]{Michael Ruzhansky}
\address{
Department of Mathematics: Analysis,
Logic and Discrete Mathematics \\
Ghent University \\
Krijgslaan 281, S8 \\
9000 Gent \\
Belgium \\
and \\
School of Mathematical Sciences \\
Queen Mary University of London \\
Mile End Road \\
London E1 4NS \\
United Kingdom}

\email{michael.ruzhansky@ugent.be}



\begin{document}

\begin{abstract}
The aim of this paper is to establish a pseudo-differential Weyl calculus on graded nilpotent Lie groups $G$ which extends the celebrated Weyl calculus on $\mathbb{R}^n$. To reach this goal, we develop a symbolic calculus for a very general class of quantization schemes, following~\cite{MR1}, using the H\"{o}rmander symbol classes $S^m_{\rho, \delta}(G)$ introduced in \cite{FR}. We particularly focus on the so-called symmetric calculi, for which quantizing and taking the adjoint commute, among them the Euclidean Weyl calculus, but we also recover the (non-symmetric) Kohn-Nirenberg calculus, on $\mathbb{R}^n$ and on general graded groups~\cite{FR}.
Several interesting applications follow directly from our calculus: expected mapping properties on Sobolev spaces, the existence of one-sided parametrices and the G\r{a}rding inequality for elliptic operators, and a generalization of the Poisson bracket for symmetric quantizations on stratified groups.

In the particular case of the Heisenberg group $\mathbb{H}_n$, we are able to answer the fundamental questions of this paper: which, among all the admissible quantizations, is the natural Weyl quantization on $\H$? 

Among other things, we discuss and investigate an analogue of the symplectic invariance property of the Weyl quantization in the setting of graded groups, as well as a notion of noncommutative Poisson bracket for symbols in the setting of stratified groups.
\end{abstract}

\maketitle

\blfootnote{Declaration of interest: none.}

\tableofcontents

\section{Introduction} \label{sec:intro}

The Weyl quantization of pseudo-differential operators on $\R^n$,
\begin{align}
	\Opw(\sigma)f(x) = \frac{1}{(2 \pi)^n} \iint\limits_{\R^n \times \R^n} e^{i (x-y) \xi} \, \sigma \Bigl (\frac{x+y}{2}, \xi \Bigr ) \, f(y) \, dy \, d\xi, \label{eq:Weyl_quant}
\end{align}
has many remarkable properties that are appreciated by mathematicians and physicists alike. Two particularly noteworthy properties are the
\begin{itemize}
    \item[-] Preservation of involution:
    \begin{align}
    \Opw(\overline{\sigma}) = \Opw(\sigma)^*; \label{eq:sym_WQ}
    \end{align}
    \item[-] Symplectic invariance:\\
    for each symplectic map $S \in \Sp$ there exists an up to a factor $\pm 1$ uniquely determined unitary operator $U_S \in \mathcal{U}\bigl ( L^2(\R^n) \bigr )$ such that
    \begin{align}
        \Opw(\sigma \circ S) = U^{-1}_S \Opw(\sigma) U_S \label{eq:symp_inv_WQ}.
    \end{align}
\end{itemize}
These properties hold true for very general types of symbols $\sigma$ on $\R^n \times \R^n$, in particular for the H\"{o}rmander symbol classes $S^m_{\rho, \delta}(\R^n)$, $m \in \R$, $0 \leq \delta \leq \rho \leq 1$, $\delta < 1$, as was shown in  H\"{o}rmander's seminal paper on the Weyl calculus of pseudo-differential operators~\cite{H}. Among the family of so-called $\tau$-quantizations
\begin{align}
    \Optau(\sigma)f(x) = \frac{1}{(2 \pi)^n} \iint\limits_{\R^n \times \R^n} e^{i (x-y) \xi} \, \sigma \Bigl (x - \tau(x-y), \xi \Bigr ) \, f(y) \, dy \, d\xi, \hspace{10pt} \tau \in [0, 1], \label{eq:tau_quant}
\end{align}
all of which give rise to a full-fledged calculus for the classes $S^m_{\rho, \delta}(\R^n)$, especially the Kohn-Nirenberg calculus for $\Op^0 \equiv \Op \equiv \Op^{\mathrm{KN}}$,
\begin{align}
    \Op(\sigma)f(x) = \frac{1}{(2 \pi)^n} \iint\limits_{\R^n \times \R^n} e^{i (x-y) \xi} \, \sigma (x, \xi) \, f(y) \, dy \, d\xi, \label{eq:KN_quant}
\end{align}
the properties \eqref{eq:sym_WQ} and \eqref{eq:symp_inv_WQ} hold true for precisely one quantization scheme: the Weyl quantization $\Opw = \Op^{\frac{1}{2}}$.

Several notable papers have developed Weyl quantizations for nilpotent groups, especially the Heisenberg group $\H$, which by construction satisfy \eqref{eq:sym_WQ} (see, e.g., \cite{D1, D2, F2, R2, HRW, C, P1, P2, B2, B1}). Dynin's early works~\cite{D1, D2} are based on the fact that for a given nilpotent group $G$ one can find another nilpotent group $N_G \geq G$ whose generic unitary irreducible representations $\pi \in \widehat{N_G}_{\mathrm{gen}}$ act on $\RS = L^2(G)$. For complex-valued symbols defined on the corresponding co-adjoint orbits $\mathcal{O}_\pi \subseteq \mathfrak{n}_G^*$, the representation then gives rise to a  Weyl-type quantization of operators a priori defined on $\SC(G) \subseteq L^2(G)$. The papers~\cite{D1, D2} establish this procedure for $G = \H$ and a group $N_{\H}$ given as a semi-direct product $\R^{2n+2} \rtimes \H$. Folland~\cite{F2} later called $\R^{2n+2} \rtimes \H$ the meta-Heisenberg group $H(\H)$ and studied it as a special case of meta-Heisenberg groups $H(G)$ of general $2$-step nilpotent groups $G$. As Folland showed, one obtains invariance of the Weyl quantization for the automorphisms of $H(\H)$ and it was conjectured that this might hold true for all meta-Heisenberg groups $H(G)$ up to degenerate examples. For Dynin's Weyl quantization on $\H$, this can in fact be considered as the adequate analog of \eqref{eq:symp_inv_WQ}. The papers \cite{D1, D2}, moreover, give an outlook on a symbolic calculus for $G = \H$, for which \cite{F2} establishes a connection with Beals and Greiner's calculus on Heisenberg manifolds~\cite{BG}. A symbolic calculus in the sense of \cite{H} for this setting is, however, only partially developed in \cite{D1, D2}. The composition of two operators with symbols in Dynin's H\"{o}rmander-type classes, for example, is only available when one of them is a differential operator. This restriction can essentially not be removed because the asymptotic series of the Dynin-Weyl product generally features infinitely many summands of growing order, as was shown in \cite[\SS~5.6.1]{R}. Other notable works of this type also employ a quantization procedure via some representation $\pi \in \widehat{N}_{\mathrm{gen}}$ of a nilpotent group $N$ for symbols $\sigma: \mathcal{O}_\pi \to \mathbb{C}$, but restrict to classes of classical (e.g., \cite{C}) or even non-differentiable symbols (e.g., \cite{B2, B1}) or do not make an attempt at a full-fledged symbolic calculus in the first place (see, e.g., \cite{P1, P2, R2, HRW}).

A very different type of calculus was first mentioned in Taylor~\cite{MT} in order to study pseudo-differential operators on $\H$ and contact manifolds. Its novel use of a generalized version of Kohn-Nirenberg quantization on Lie groups is based on the generally operator-valued group Fourier transform and appropriate operator-valued symbols. One of its main ingredients for the analysis on $\H$ is the fundamental connection between the operator-valued symbols on $\H \times \Hhat$ and Weyl-quantized operators on $\R^n$, which consequently permits the use of scalar-valued symbols on an extended phase space. The purely kernel based pseudo-differential calculus on homogeneous groups~\cite{CGGP} by Christ, Geller, G\l owacki and Polin, which is essentially equivalent to the one in \cite{MT} in the case $G = \H$, provides explicit formulas and expansions for adjoints and compositions as well as one-sided parametrices, but refrains from using symbols and essentially focuses on classical pseudo-differential operators.
Much later, however, the connection between the operator-valued symbols on $\H \times \Hhat$ and the H\"{o}rmander-Weyl calculus on $\R^n$ was further explored by Bahouri, Fermanian-Kammerer and Gallagher~ \cite{BFG} to develop a comprehensive symbolic pseudo-differential calculus on $\H$, which provides asymptotic expansions for the (sclalar-valued) adjoint and composite symbols. The use of scalar-valued symbols on an extended phase space, a feature which is very specific of $\H$, was eventually dropped by Fischer and the third author~\cite{FR} in order to develop a full-fledged Kohn-Nirenberg-type calculus on general graded nilpotent Lie groups $G$ for appropriately generalized operator-valued H\"{o}rmander classes $S^m_{\rho, \delta}(G)$. The successful use of operator-valued symbols throughout the calculus is made possible by an extensive use of harmonic analysis on graded groups and novel tools and function spaces. A profound comprehension of the properties of the associated kernels, however, remains one of the key ingredients. The work~\cite{FR} provided a parallel theory to the one developed earlier in the setting of compact Lie groups in \cite{RT1, RT2}. (For further developments see, e.g., \cite{RTW, RW, F_Int, DR, F_Loc_Glo, CR_Sub}.)

The existence of a well-functioning Kohn-Nirenberg calculus on graded groups naturally opens up the following questions.

\begin{questions}
Is it possible to find a class of $\tau$-quantizations on general graded groups $G$ which
\begin{itemize}
    \item[(Q1)] give rise to substantial pseudo-differential calculi for the symbol classes $S^m_{\rho, \delta}(G)$,
    \item[(Q2)] include the classical $\tau$-calculi \eqref{eq:tau_quant} on $G = \R^n$ as special cases,
    \item[(Q3)] include the Kohn-Nirenberg calculus on graded groups~\cite{FR} as a special case,
    \item[(Q4)] include at least one viable Weyl-type calculus that satisfies appropriate versions of \eqref{eq:sym_WQ} and \eqref{eq:symp_inv_WQ} on any graded $G$?
\end{itemize}
\end{questions}

This paper aims at positively answering all these questions. In order to treat these questions in a rigorous way, we recall one of the main tools in~\cite{FR}, the operator-valued group Fourier transform. It is determined by the generally infinite-dimensional irreducible unitary representations $\pi \in \widehat{G}$, the noncommutative analogues of the exponential functions $x \mapsto e^{i x \xi}$ on $G = \R^n$, and accordingly defined by
\begin{align}
    \widehat{f}(\pi) \equiv (\F f)(\pi) := \int_G f(x) \pi(x)^* \, dx = \int_G f(x) \pi(x^{-1}) \, dx. \label{eq:GFT}
\end{align}
Despite being an operator-valued transform, its values $\widehat{f}(\pi)$ are generally well-behaved operators on the representation space $\RS \cong L^2(\R^k)$, for some $k = k(\pi) \in \N$, e.g, compact, Hilbert-Schmidt, trace-class, etc., if the function $f$ lies in nice function spaces like $L^1(G)$, $L^2(G)$, $\SC(G)$, etc, and in these specific cases actually the whole field of operators $\widehat{f} = \{ \widehat{f}(\pi) \}_{\pi \in \widehat{G}}$ has desirable properties. In particular, for Schwartz functions $f \in \SC(G)$, one gets the Fourier inversion formula
\begin{align}
    f(x) = \int_{\Ghat} \Tr \bigl ( \pi(x) \widehat{f}(\pi) \bigr ) \, d\mu(\pi) = \int_{\Ghat} \Tr \Bigl ( \int_G \pi(y^{-1} x) f(y) \, dy \Bigr ) \, d\mu(\pi), \label{eq:IF_GFT}
\end{align}
a generalization of
\begin{align*}
    f(x) = \int_{\widehat{\R}^n} e^{ix \xi} \widehat{f}(\xi) \, \frac{d\xi}{(2 \pi)^n} =  \int_{\Rhat^n} \int_{\R^n} e^{i(x-y) \xi} f(y) \, dy \, \frac{d\xi}{(2 \pi)^n},
\end{align*}
which, in \cite{FR}, motivated the definition of the Kohn-Nirenberg-type quantization
\begin{align}
    \Op(\sigma)f(x) &= \int_{\Ghat} \Tr \bigl ( \pi(x) \sigma(x, \pi) \widehat{f}(\pi) \bigr ) \, d\mu(\pi) \\
    &= \int_{\Ghat} \Tr \Bigl ( \int_G \pi(y^{-1} x) \sigma(x, \pi) f(y) \, dy \Bigr ) \, d\mu(\pi) \label{eq:KN_quant_G}
\end{align}
for fields of operator-valued symbols $\sigma = \{ \sigma(x, \pi) \mid x \in G, \pi \in \widehat{G} \}$ with $\sigma(x, \pi): \RS^\infty \to \RS^{-\infty}$ for all $x \in G$ and almost every $\pi \in \widehat{G}$. For the special case $G = \R^n$ this quantization coincides with \eqref{eq:KN_quant}.

Starting from \eqref{eq:KN_quant_G}, M\u{a}ntoiu and the third author~\cite{MR1} investigated the question whether interesting $\tau$-quantiza\-tions would arise if one replaced the point $x \in G$ in $\sigma(x, \pi)$ by a suitable generalization of $x - \tau(x-y) \in \R^n$, say the points $x\tau(y^{-1}x)^{-1} \in G$ determined by a measurable function $\tau: G \to G$, called the \textit{quantizing function}. The resulting quantizations
\begin{align}
    \Optau(\sigma)f(x) = \int_{\Ghat} \Tr \Bigl ( \int_G \pi(y^{-1} x) \sigma(x\tau(y^{-1}x)^{-1}, \pi) f(y) \, dy \Bigr ) \, d\mu(\pi), \label{eq:tau_quant_G}
\end{align}
were subsequently studied for the much wider class of second countable unimodular type I locally compact groups by means of $C^*$-algebra theory, producing surprising noncommutative generalizations of well-known concepts like, e.g., the Wigner function and the short time Fourier transform.\footnote{For a subsequent extension including non-unimdoular type I groups we refer to M\u{a}ntoiu and Sandoval~\cite{MS}.} Although the generality of the setting necessitated a restriction of $\Optau(\sigma)$ to the class of Hilbert-Schmidt operators on $L^2(G)$, the approach pointed in the right direction: choosing the constant function $\tau : x \mapsto e_G$ for the unit element $e_G \in G$, one recovers the Kohn-Nirenberg quantization~\eqref{eq:KN_quant_G}, while for the functions $\tau(x) := \tau x$, $\tau \in [0, 1]$, on $G = \R^n$ one recovers the quantizations~\eqref{eq:tau_quant}.

A particularly valuable insight was gained from investigating those quantizing functions $\tau$ which satisfy
\begin{align}
	\Optau(\sigma)^*=\Optau(\sigma^*) \label{eq:sym_quant}
\end{align}
for all Hilbert-Schmidt operators $\Optau(\sigma)$ on $L^2(G)$. The authors of \cite{MR1} named them \textit{symmetry functions} and every quantization~\eqref{eq:tau_quant_G} arising from a symmetry function is called a \textit{symmetric quantization} on $G$. (Equivalently, we will often simply say that $\tau$ is symmetric.)
Although \eqref{eq:sym_quant} seems to be the appropriate generalization of \eqref{eq:sym_WQ}, it may in practice be hard to check for operator-valued symbols. The authors could show, however, that a quantizing function $\tau$ is symmetric if and only if it satisfies
\begin{align}
	\tau(x)= \tau(x^{-1})x \label{eq:char_sym_fun}
\end{align}
for almost all $x \in G$ (cf.~\cite[Cor.~4.2]{MR1}), a condition that can often be checked in a straight-forward manner for a whole class of groups. It is, for example, obviously satisfied by the quantizing function $\tau(x) := \frac{1}{2} x$ on $G = \R^n$, which defines the Weyl quantization, but by no other $\tau(x) := \tau x$ with $\tau \in [0, 1]$.

Clearly, the class of Hilbert-Schmidt operators $\Optau(\sigma)$ on $L^2(G)$ is a rather restricted class of pseudo-differential operators since, equivalently, its integral kernel $\mathsf{Ker}^\tau_\sigma$ must lie in $L^2(G \times G)$. Moreover, the lack of smooth structure on generic second countable unimodular type I locally compact groups makes this setting too general to extend the quantizations $\sigma \mapsto \Optau(\sigma)$ to full-fledged pseudo-differential calculi for differentiable symbol classes.
Focusing on graded nilpotent groups, however, one can study the symmetry condition \eqref{eq:sym_quant} for the H\"{o}rmander-type classes $S^m_{\rho, \delta}(G)$ introduced in \cite{FR}. The explicit description and asymptotic expansion of the symbol of the adjoint $\Op(\sigma)^*$ of a Kohn-Nirenberg-quantized symbol $\sigma \in S^m_{\rho, \delta}(G)$, $m \in \R$, $0 \leq \delta < \rho \leq 1$, provided in \cite[\SS~5.5.3]{FR}, indicates that this is also a good setting to study symmetric calculi on graded groups. It turns out that we can indeed extend the symmetry condition to the symbol classes $S^m_{\rho, \delta}(G)$, as shown in Theorem~\ref{thm:asym_exp_adjoint_tau}.

At this point the question arises whether symmetry functions are available for a given group or class of groups. For the setting of this paper, namely graded groups, this is always the case.
In fact, there are at least two canonical examples of symmetry functions available on every exponential Lie group $G$. If we denote by $\log: G \to \mathfrak{g}$ the inverse of the exponential map $\exp: \mathfrak{g} \to G$, then the first example, due to \cite[Prop.~4.3~(3)]{MR1}, is given by
\begin{align}
	\tau(x) = \int_0^1 \exp \bigl ( s \log (x) \bigr ) \, ds. \label{eq:MR_tau}
\end{align}
The second example, which seems to be new, is defined by
\begin{align}
	\tau(x) := \exp \Big( \frac{1}{2} \log(x) \Bigr ). \label{eq:onehalf_tau}
\end{align}
By the Taylor-Campbell-Hausdorff formula and the anti-symmetry of the Lie bracket, this function is easily seen to satisfy \eqref{eq:char_sym_fun}. If we express the elements $x \in G$ in exponential coordinates with respect to a basis $X_1, \ldots, X_\dimG$ of $\mathfrak{g}$, i.e., $x = \exp \bigl( x_1 X_1 + \ldots + x_\dimG X_\dimG \bigr ) = (x_1, \ldots, x_\dimG)$ for uniquely determined $x_1, \ldots, x_\dimG \in \R$, then $\tau$ takes the form
\begin{align}
	\tau(x_1, \ldots, x_\dimG) = \exp \bigl( \frac{x_1}{2} X_1 + \ldots + \frac{x_\dimG}{2} X_\dimG \bigr ) = \bigl( \frac{x_1}{2}, \ldots, \frac{x_\dimG}{2} \bigr ).
\end{align}
Since on $G = \R^n$ the exponential function coincides with the identity function, one also gets
\begin{align*}
	\int_0^1 \exp \bigl ( s \log (x) \bigr ) \, ds = \frac{x}{2},
\end{align*}
so the two symmetry functions $\tau$ defined by \eqref{eq:MR_tau} and \eqref{eq:onehalf_tau} coincide.

On $G = \H$, however, these two functions are utterly distinct. Using the standard upper triangular matrix representation of $\H$, the integral \eqref{eq:MR_tau} is easily seen to give\footnote{The factor $\frac{1}{24}$ corrects the erroneous $\frac{1}{6}$ in \cite[(A.4)]{ER}.}
\begin{align}
	\tau(x) = \tau(x_1, \ldots, x_{2n+1}) = \Bigl ( \frac{x_1}{2}, \frac{x_2}{2}, \ldots,\frac{x_{2n+1}}{2} + \sum_{j = 1}^n \frac{x_j x_{j+n}}{24} \Bigr ). \label{eq:MR_tau_Hn}
\end{align}
The symmetry function \eqref{eq:onehalf_tau}, on the other hand, is given by
\begin{align}
	\tau(x) = \tau(x_1, \ldots, x_{2n+1}) = \bigl( \frac{x_1}{2}, \ldots, \frac{x_{2n+1}}{2} \bigr ), \label{eq:onehalf_tau_Hn}
\end{align}
thus clearly differs from \eqref{eq:MR_tau_Hn}. Interestingly, both functions are members of the infinite family of symmetry functions
\begin{align}
	\tau(x) = \Bigl ( \frac{x_1}{2}, \ldots, \frac{x_{2n}}{2}, \frac{x_{2n+1}}{2}+ \sum_{j, k = 1}^{2n} c_{j, k} \, x_j x_k \Bigr ) \label{eq:fam_tau_Hn}
\end{align}
with $c_{j, k} \in \R$, $j, k = 1, \ldots, 2n$ (cf.~Example~\ref{ex:fam_sym_fun_Hn}), and in fact one can easily come up with infinitely many symmetry functions with non-linear terms on any graded group.

\medskip

In view of the abundant supply of (especially symmetric) quantizing functions, this paper sets out to develop a versatile pseudo-differential $\tau$-calculus on general graded groups $G$ that gives positive answers to (Q1), (Q2), (Q3) and (Q4). To achieve this goal, however, we have to deal with a few serious difficulties due to the generally noncommutative nature of $G$: 
\begin{itemize}
    \item[--] In contrast to $\R^n \times \widehat{\R}^n \cong T^*\R^n$, the set $G \times \Ghat$ is not a symplectic manifold in any obvious fashion. So, we cannot rely on a type of phase space analysis akin to the one $T^*\R^n$ but instead have to work with the integral kernels of $\Optau(\sigma)$. However, the convolution-based kernel identities from the Kohn-Nirenberg calculus~\cite{FR} are not available for general $\tau$.\footnote{Already the Weyl calculus on $\R^n$ has to make due without them.} This has to be compensated by a more intricate kernel analysis.
    \item[--] Moreover, the quantizing function $\tau: G \to G$ cannot be assumed to be a Lie group homomorphism since generally $\tau(xy) \neq \tau(x) \tau(y)$ for given $x, y \in G$. This complicates the whole calculus substantially.
    \item[--] In order to control all the oscillatory integrals, one has to impose reasonably restrictive conditions on $\tau$, which, however, should not rule out any of the relevant well-known or new examples of $\tau$.
    \item[--] Since there is no immediate generalization of the symplectomorphisms of $\R^n \times \widehat{\R}^n \cong T^*\R^n$ for $G \times \Ghat$, one has to find an appropriate version of condition~\eqref{eq:symp_inv_WQ}.
\end{itemize}
In what follows we develop such a $\tau$-calculus on graded groups for symbols in the H\"{o}rmander classes $S^{m}_{\rho,\delta}(G)$. Any given graded group $G$, whose topological dimension we denote by $\dimG \in \N$, is understood to be equipped with a fixed family of automorphic dilations, characterized by a set of weights $1 = v_1 \leq \ldots \leq v_\dimG \in \N$. The main results of our calculus are valid for all symbol classes $S^{m}_{\rho,\delta}(G)$, $m \in \R$, which satisfy $0\leq \delta<\frac{\rho}{v_n}\leq \rho\leq 1$. 
This peculiar restriction of the usual condition $0\leq \delta<\rho\leq 1$ appears throughout the calculus to ensure the convergence of crucial oscillatory integrals. Essentially, it arises from left and right-invariant differentiation of the quantizing function $\tau$ and the fact that the latter generally \textit{fails to be} a group homomorphism. It is therefore not surprising that the restriction disappears when $\tau = e_G$, the most trivial group homomorphism possible, and our calculus agrees with the Kohn-Nirenberg calculus~\cite{FR}.
Clearly, the restriction is also absent in the Euclidean case, where the canonical dilations are assumed to be isotropic, that is, when all weights $v_1, \ldots, v_\dimG$ equal to $1$. So the restriction, in a sense, reflects the noncommutative nature of the group and the fact that $\tau$ does not respect the group structure. We also note that the analysis of this paper covers completely the case of the symbol classes $S^{m}_{1,0}(G)$, $m \in \R$. This analysis will appear elsewhere.

However, in the case of the $G = \H$ we show in a forthcoming paper that the restriction on the gap between the optimal range $0\leq \delta<\rho\leq 1$ and $0\leq \delta<\frac{\rho}{v_n}\leq \rho\leq 1$ can be removed due to the particular structure of two-step nilpotent groups.

\medskip

We conclude this introduction by describing the organization of the paper:
Section~\ref{sec:prelim} recalls some necessary basics from harmonic analysis on graded groups and the theory of Rockland operators. Section~\ref{sec:sym_Sob} gives a quick review of the H\"{o}rmander-type symbol classes $S^m_{\rho, \delta}(G)$ and the inhomogeneous Sobolev spaces $L^p_s(G)$ developed in \cite{FR}, highlighting their well-known special cases on $\R^n$. Section~\ref{sec:sym_calc}, the main section of this paper, we give positive answers to the questions (Q1) - (Q3) by developing a symbolic $\tau$-calculus for a very wide range of quantizing functions $\tau$ and for symbols in $S^m_{\rho,\delta}(G)$ classes with $0\leq \delta<\frac{\rho}{v_n}\leq \rho\leq 1$. This includes asymptotic expansions for the symbols of adjoint and composite operators as well as a novel notion of $G$-homogeneous Poisson bracket for symmetric quantizations on stratified groups, which coincides with the classical Poisson bracket on $\R^n$. All the corresponding results for the Weyl and Kohn-Nirenberg calculi on $\R^n$ and the Kohn-Nirenberg calculus on graded groups~\cite{FR} are recovered as special cases. In Section~\ref{sec:symp_inv_G} we give a sufficiently affirmative answer to (Q4) by providing an appropriate version of the symplectic invariance~\eqref{eq:symp_inv_WQ} for any general graded group $G$ and explain why the special case $G = \H$ singles out the symmetric quantization defined by \eqref{eq:onehalf_tau_Hn} as the most viable generalization of the Weyl quantization on $\R^n$. Section~\ref{sec:applications} concludes the paper with several applications of the calculus such as the continuity of $\tau$-quantized operators on Lebesgue, Sobolev and Besov spaces, or the existence of left parametrices and the G\r{a}rding inequality for elliptic pseudo-differential operators.

\section{Preliminaries} \label{sec:prelim}

In this section we recall the notions and tools from harmonic analysis on graded groups that we consider essential for a proper definition of $\tau$-quantized operators on $G$ with symbols in the H\"{o}rmander-type classes introduced in \cite{FR}. Although we focus on the most necessary material, we want to provide enough details and references in order not to lose among our readers the interested experts on pseudo-differential theory on $\R^n$. Readers who are familiar with harmonic analysis on nilpotent Lie groups may well skip this section.

\subsection{Homogeneous groups} \label{subs:hom_gr}

A Lie algebra $\mathfrak{g} \cong \R^{\dimG}$ is called \textit{homogeneous} if it admits a family of automorphic dilations
\begin{align}
	D_r = \exp \bigl (\log (r) A \bigr ), r > 0, \label{eq:dil}
\end{align}
for some diagonalizable $A$ on $\mathfrak{g}$ with positive eigenvalues $0< v_1 \leq \ldots \leq v_\dimG$, called the dilations' weights, i.e., each $D_r$ is a vector space isomorphism which satisfies
\begin{align*}
    D_r \bigl ( [ X, Y] \bigr ) = [D_r(X), D_r(Y)].
\end{align*}
The existence of a family of dilations on a given Lie algebra $\mathfrak{g}$ implies that the Lie algebra is nilpotent. Any connected, simply connected nilpotent Lie group $G$ given by $G = \exp(\mathfrak{g})$ is called \textit{homogeneous} if its Lie algebra $\mathfrak{g}$ is homogeneous.

Given an eigenbasis $\{ X_1, \ldots, X_{\dimG} \}$ of $A = \log(D_1)$ with eigenvalues $0 < v_1 \leq \ldots \leq v_{\dimG}$, the group $G$ can be equipped with exponential coordinates with respect to this basis. Since the group is nilpotent, the group multiplication is given by
\begin{align}
	x y &= \exp(x_1 X_1 + \ldots + x_{\dimG} X_{\dimG}) \exp(y_1 X_1 + \ldots + y_{\dimG} X_{\dimG}) \nonumber \\
	&= \exp(R_1(x, y) X_1 + \ldots + R_{\dimG}(x, y) X_{\dimG}) \label{eq:exp_coord}
\end{align}
for some finite order polynomials $R_j(x, y)$, $j = 1, \ldots, \dimG$. Due to the nilpotent structure of $G$ the polynomials $R_j$ only depend on the first $j$-many variables and can be written in the simpler form
\begin{align}
    R_j(x, y) = x_j + y_j + \tilde{R}_j(x_1, \ldots, x_{j-1}, y_1, \ldots, y_{j-1}), \label{eq:pol_exp_coord1}
\end{align}
for some polynomials $\tilde{R}_j$, while the homogeneous structure of $G$ yields the additional useful identities
\begin{align} \label{eq:pol_exp_coord2}
    R_j(x, y) &= \underset{[\alpha] + [\beta] = v_j}{\sum} c_{j, \alpha, \beta} \, x^\alpha y^\beta = x_j + y_j + \underset{\substack{[\alpha] + [\beta] = v_j, \\ [\alpha] \neq 0 \neq [\beta]}}{\sum} c_{j, \alpha, \beta} \, x^\alpha y^\beta.
\end{align}
Often we will denote a given element $x \in G$ simply by the $\dimG$-tuple of coordinates $(x_1, \dots, x_\dimG) \in \R^n$ with respect the basis $\{ X_1, \ldots, X_{\dimG} \}$.
Since the coordinate expression of the unit element $e_G$ equals $(0, \ldots, 0) \in \R^\dimG$, one has $R_j(x, x^{-1}) = 0$ for all $x \in G$.

This choice of exponential coordinates also permits to express the dilations' action on group elements $x \in G$ and the bi-invariant Haar measure on $G$ very explicitly:
\begin{align*}
	D_r (x) &= \exp(r^{v_1}x_1 X_1 + \cdots + r^{v_{\dimG}} x_{\dimG} X_\dimG) = (r^{v_1}x_1, \ldots, r^{v_{\dimG}} x_{\dimG}), \\
	d(D_r(x)) &= r^\hdim dx,
\end{align*}
where $\hdim := v_1 + \ldots + v_\dimG$, the so-called homogeneous dimension of $G$. A standard convention is that the weights $v_1, \ldots, v_\dimG$ are jointly rescaled so that the lowest weight $v_1 = 1$. This also implies that $\hdim \geq \dimG$.

Note that all the Lebesgue spaces $L^p(G)$, $p \in [1, \infty]$, are defined with respect to the bi-invariant Haar measure and that the test function space $\TF(G)$ can be defined as the pullback $\log^* (\TF(\R^\dimG))$ for the global chart map $\log := \exp^{-1}: G \to \Lie{g} \cong \R^\dimG$.

Another crucial feature of homogeneous groups is the existence of so-called homogeneous quasi-norms (cf.~ \cite[\S~1.A]{FS}). 
A \textit{homogeneous quasi-norm} on $G$ is a continuous function $x \mapsto |x|:G \to [0, \infty)$ which is
\begin{itemize}
    \item[--] definite: $|x| = 0$ if and only if $x = e_G$,
    \item[--] symmetric: $|x^{-1}| = |x|$,
    \item[--] $1$-homogeneous: $|D_r x | = r|x|$.
\end{itemize}
Consequently, the open balls $B_r(x) := \{ y \in G : |y^{-1} x| < r \}$
satisfy
\begin{align*}
    B_r(x) &= D_r \bigl ( B_1(r^{-1}x) \bigr ), \\
    |B_r(x)| &= r^\hdim |B_1(x)| = r^\hdim |B_1(0)|.
\end{align*}
Moreover, every quasi-norm satisfies a quasi-triangle inequality
\begin{align} \label{eq:quasi_tri}
 |xy| \leq C (|x| + |y|) \quad \mbox{ for all } x, y \in G,
 \end{align}
for some $C \geq 1$. Note that for any homogeneous group $G$ and $p \in [1, \infty)$ the functions
\begin{align}
    |x|_p :=& \bigl ( |x_1|^{\frac{p}{v_1}} + \ldots + |x_\dimG|^{\frac{p}{v_\dimG}} \bigr )^{\frac{1}{p}}, \label{eq:p_qn} \\
    |x|_\infty :=& \max_{j = 1, \ldots, 
    \dimG} |x_j|^{\frac{1}{v_j}}, \label{eq:inf_qn}
\end{align}
for example, are homogeneous quasi-norms and that any two given homogeneous quasi-norms on $G$ are equivalent. On $G = \R^n$, equipped with the canonical isotropic dilations, clearly each of the quasi-norms \eqref{eq:p_qn} and \eqref{eq:inf_qn} is a norm, satisfying \eqref{eq:quasi_tri} for $C = 1$. Note that in fact any homogeneous group possesses a homogeneous norm (cf.~\cite{HS} and \cite[Thm.~3.1.39]{FR}).

\subsection{Graded and stratified homogeneous groups} \label{subs:gr_strat_gr}

Families of dilations naturally arise on graded and stratified Lie algebras. A Lie algebra $\mathfrak{g}$ is called \textit{graded} if it is equipped with a vector space decomposition
$\mathfrak{g}=\bigoplus_{i =1}^{\infty} \mathfrak{g}_i$, for which all but finitely many of the $\mathfrak{g}_{i}$'s are $\{0\}$
and satisfying $[\mathfrak{g}_{i}, \mathfrak{g}_{i'}]\subseteq \mathfrak{g}_{i + i'}$ for all $i, i' \in \N$. Such a decomposition is called a gradation.

Since the sequence of subspaces $\mathfrak{h}_k := \bigoplus_{i = k}^\infty\mathfrak{g}_i$ forms a (finite) nested sequence of ideals in $\mathfrak{g}$, any basis $\{X_1, \ldots, X_\dimG \}$ given as the union of bases $\{ X_{1_i}, \ldots, X_{\dimG_i}\}$ is necessarily a strong Malcev basis of $\mathfrak{g}$, passing through the ideals $\mathfrak{h}_k$ (cf.~\cite[Thm.~1.1.13]{CG} and, e.g., \cite[Lem.~4.16]{GR}).
Given such a basis, a graded Lie algebra can be equipped with a canonical family of dilations arising from the matrix $A$ defined by $A X_j := i X_j$ if $X_j \in \mathfrak{g}_i$, that is, with $D_r(X_j) = r^i X_j$. However, in general graded Lie algebras admit infinitely many possible families of dilations, even if one rescales the weights so that $v_1 = 1$.

So, every graded nilpotent Lie algebra can be equipped with a family of dilations, but let us recall that the converse is also true. That is, a Lie algebra $\mathfrak{g}$ which admits a family of dilations also admits a gradation (cf.~\cite[Prop.~1.1]{M}). Note, however, that the given dilations need not coincide with the dilations that canonically arise from the gradation.

A graded Lie algebra is called \textit{stratified} if the gradation is such that $[\mathfrak{g}_1, \mathfrak{g}_i] = \mathfrak{g}_{i + 1}$ for all $i \in \N$. Such a gradation is called a stratification of $\Lie{g}$ and the direct summand $\mathfrak{g}_i$, $i \in \N$, is referred to as the $i$-th stratum.

Any connected, simply connected nilpotent Lie group $G$ given by $G = \exp_G(\mathfrak{g})$ is called \textit{graded} or \textit{stratified} if its Lie algebra $\mathfrak{g}$ is graded or stratified, respectively.

The vector space $\Lie{g} := \R^n$ equipped with the trivial Lie algebra $[\, \cdot \,, \, \cdot \, ] = 0$ is a nilpotent Lie algebra, which admits the trivial stratification $\Lie{g} = \Lie{g}_1 = \R^n$, and $G$ can be identified with the Abelian group $(\R^n, +)$.

\subsection{Invariant derivatives and homogeneous polynomials} \label{subs:inv_der_hom_pol}

Let $G$ be a homogeneous Lie group of topological dimension $\dimG \in \N$ and fix a basis $\{ X_1,\ldots, X_\dimG \}$ of its Lie algebra $\mathfrak{g}$ consisting of eigenvectors of the matrix $A$ in \eqref{eq:dil}. Since the space of left-invariant (respectively right-invariant) vector fields can be identified with $G \times \Lie{g}$, we will, by an abuse of notation, denote by $X_j$, $j = 1, \ldots, \dimG$, the left-invariant vector field on $G$ associated to the basis vector $X_j \in \Lie{g}$,
\begin{align*}
	(X_j f)(x) := \frac{d}{dt}|_{t = 0} f \bigl (x \exp(t X_j) \bigr ) \hspace{5pt} x \in G,
\end{align*}
and by $\tilde{X}_j$ the right-invariant vector field
\begin{align*}
	(\tilde{X}_j f)(x) := \frac{d}{dt}|_{t = 0} f \bigl (\exp(t X_j) x \bigr ) \hspace{5pt} x \in G.
\end{align*}
These invariant derivatives and their higher order iterates $X^\alpha := X_1^{\alpha_1} \cdots X_n^{\alpha_\dimG}$ and $\tilde{X}^\alpha := \tilde{X_1}^{\alpha_1} \cdots \tilde{X}_n^{\alpha_\dimG}$, $\alpha \in \mathbb{N}_0^\dimG$, are the natural generalizations of the higher order partial derivatives $\partial^\alpha = \partial^{\alpha_1}_1 \cdots \partial^{\alpha_n}_n$ on the Abelian group $\R^n$ equipped with the canonical isotropic dilations.
As their Abelian counterparts, the operators $X^\alpha$ and $\tilde{X}^\alpha$ are homogeneous of degree
\begin{align*}
	[\alpha] := v_1 \alpha_1 + \ldots + v_\dimG \alpha_\dimG \in \R,
\end{align*}
in the sense that a linear operator $T: \TF(G) \to \D(G)$ is said to be homogeneous of degree $\hdeg_T \in \C$ if
\begin{align*}
	T (f \circ D_r) = r^{\hdeg_T} (Tf) \circ D_r \hspace{5pt} \mbox{ for all } \hspace{5pt} f \in \TF(G), \, r > 0.
\end{align*}
Although these higher order derivatives do not commute in general, one can express an iterate $X^\alpha X^\beta$, $\alpha,\beta\in\mathbb{N}_0^n$, as a linear combination of $X^\gamma$ with $\gamma\in\mathbb{N}_0^n, [\gamma]=[\alpha]+[\beta]$, namely
\begin{equation}\label{l.i.prod}
   X^\alpha X^\beta=\sum_{\substack{\gamma\in\mathbb{N}_0^n, |\gamma|\leq |\alpha|+|\beta|\\
   [\gamma]=[\alpha]+[\beta]}}c'_{\alpha,\beta,\gamma}X^\gamma.
\end{equation}
An analogous formula is valid for the right-invariant differential operators $\tilde{X}^\alpha \tilde{X}^\beta$.

Moreover, it is crucial to be able to rewrite left-invariant derivatives in terms of right-invariant and standard partial derivatives, and vice versa. This requires the use of homogeneous polynomials. We recall that a function $f: G \to \C$ is called a polynomial if $f \circ \exp_G$ is a polynomial on the vector space $\Lie{g} \cong \R^\dimG$ and that any distribution $f \in \D(G)$ is said to be homogeneous of degree $\hdeg_f \in \C$ if 
\begin{align*}
	f \circ D_r = r^{\hdeg_f} f \hspace{5pt} \mbox{ for all } \hspace{5pt} r > 0.
\end{align*}
Now, by \cite[Prop.~1.29]{FS}, for every $\alpha \in \mathbb{N}_0^n$ one can write
\begin{align}
    \begin{array}{rclcl}
     X^\alpha &=& \sum\limits_{\substack{|\beta| \leq |\alpha|, \\ [\beta] \geq [\alpha]}} P_{\alpha, \beta} \, \tilde{X}^\beta &=& \sum\limits_{\substack{|\beta| \leq |\alpha|, \\ [\beta] \geq [\alpha]}} R_{\alpha, \beta} \, \partial^\beta, \\
     \tilde{X}^\alpha &=& \sum\limits_{\substack{|\beta| \leq |\alpha|, \\ [\beta] \geq [\alpha]}} Q_{\alpha, \beta} \, X^\beta, && \\
     \partial^\alpha &=& \sum\limits_{\substack{|\beta| \leq |\alpha|, \\ [\beta] \geq [\alpha]}} S_{\alpha, \beta} \, X^\beta, &&
    \end{array} \label{RelationsVFs1}
\end{align}
for some homogeneous polynomials $P_{\alpha, \beta}$, $Q_{\alpha, \beta}$, $R_{\alpha, \beta}$ and $S_{\alpha, \beta}$ of homogeneous degree $[\beta] - [\alpha]$. These identities will be important tools for the symbolic calculus.

Left-invariant derivatives of right translates of functions will be a frequent inconvenience throughout the paper since the two operations do not commute. To circumvent this, one can use \eqref{RelationsVFs1} to rewrite such an expression as
\begin{align}
    X_x^\alpha f (xy) = \tilde{X}_y^\alpha f(xy) = \sum\limits_{\substack{|\beta| \leq |\alpha|, \\ [\beta] \geq [\alpha]}} Q_{\alpha, \beta}(y) \, X_y^\beta f(xy) = \sum\limits_{\substack{|\beta| \leq |\alpha|, \\ [\beta] \geq [\alpha]}} Q_{\alpha, \beta}(y) \, (X^\beta f)(xy) \label{RelationsVFs2}
\end{align}
for any sufficiently differentiable function $f:G \to \mathbb{C}$.

\subsection{Taylor expansions}

We now recall Taylor's formula on homogeneous groups~\cite[\S~1.C]{FS}, which is crucially needed for the symbolic $\tau$-calculus on graded groups.

Given a homogeneous Lie group $G$, equipped with a quasi-norm $| \, \cdot \, |$, the left Taylor polynomial of homogeneous order $M \in \N$ of a function $f: G \to \C$ at a point $x \in G$ is given by
\begin{align}
	P_{x,M}^{f}(y)=\sum_{[\alpha] \leq M}q_\alpha(y) (X^\alpha f)(x), \label{eq:Taylor_expansion}
\end{align}
where $q_\alpha$ is the homogeneous polynomial of order $[\alpha]$ uniquely determined by the condition
  \begin{align}
	\bigl ( X^\beta q_\alpha \bigr )(e_G) = \delta_{\beta, \alpha} \hspace{5pt} \text{for all} \hspace{5pt} \beta \in \N^\dimG_0. \label{eq:can_hom_pol}
\end{align}
The remainder of order $M$ is given by
\begin{align*}
	R_{x,M}^{f}(y):=f(xy) - P_{x,M}^{(f)}(y).
\end{align*}
This remainder is controlled uniformly if $f\in C^{\lceil M\rfloor+1}(G)$, with
\begin{align*}
	\lceil M\rfloor:=\max\{|\alpha|: \alpha\in \mathbb{N}_0^n\,\, \text{with}\,\, [\alpha]\leq M\}.
\end{align*}
In particular, there exists a constants $C_M>0$ and $\eta>1$ such that
\begin{align}
	|R_{x,M}^{f}(y)| \leq C_M \sum_{M< [\alpha] \leq M + v_\dimG} |y|^{[\alpha]}\sup_{|z|\leq \eta^{\lceil M\rfloor+1} |yx^{-1}|} |(X^\alpha f) (xz)|, \label{eq:est_Taylor_rem}
\end{align}
where $C_M$ is given by the mean value theorem~\cite[Th~1.33]{FS} and $\eta$ is as at the end of the proof of ~\cite[Prop~3.1.46]{FR}.\\
Moreover, for any $M\in\mathbb{N}_0$ and 
$\alpha\in\mathbb{N}_0^n$, and for $f\in C^\infty (G)$, one has
\begin{equation}
\begin{array}{rclcrcl} \label{Prop.Rem-vf}
    X^\alpha_y P_{x, M}^{f}(y) & = & P_{x,M-[\alpha]}^{X^\alpha f}(y)& \quad \text{and}\quad &
    X^\alpha_y R_{x, M}^{f}(y) & = & R_{x, M-[\alpha]}^{X^\alpha f}(y), \\
    \tilde{X}^\alpha_y P_{x,M}^{f}(y) & = & P_{x_1 = 0, M-[\alpha]}^{X^\alpha_x f(x x_1)}(y)& \quad \text{and}\quad &
    \tilde{X}^\alpha_y R_{x,M}^{f}(y) & = & R_{x_1 = 0, M-[\alpha]}^{X^\alpha_x f(x x_1)}(y).
\end{array}
\end{equation}
Note that the homogeneous polynomials $q_\alpha$, $\alpha \in \N^\dimG_0$, determined by \eqref{eq:can_hom_pol}, play an important role for the pseudo-differential calculus in \cite{FR} as well as for the one developed in this paper for several reasons: apart from their crucial role in the Taylor expansion, which will be a key tool in the symbolic calculus, the polynomials $q_\alpha$, $\alpha \in \N^\dimG_0$, form a basis of the vector space of all polynomials on $G$, and for every possible homogeneous order $M$ in 
\begin{align*}
	\mathcal{W} := \{ v_1 \alpha_1 + \ldots + v_\dimG v_\dimG \mid \alpha_1, \ldots, \alpha_\dimG \in \N_0 \},
\end{align*}
the polynomials $q_\alpha$ with $[\alpha] = M$ form a basis of the finite-dimensional subspace of polynomials of homogeneous degree equal to $M$. This property also allows us to rewrite
 \begin{align}
    q_\alpha(xy) = \sum_{[\alpha_1] + [\alpha_2] = [\alpha]} c_{\alpha_1, \alpha_2} \, q_{\alpha_1}(x) \, q_{\alpha_2}(y) \label{HomPol1}
 \end{align}
for any $\alpha \in \mathbb{N}_0^n$ and any $x, y \in G$, with coefficients $c_{\alpha_1, \alpha_2} \in \mathbb{R}$ which are independent of $x, y$.

\subsection{The Schwartz class} \label{subs:SC}

On any connected, simply connected nilpotent group $G$ the Schwartz class $\SC(G)$ is defined as the image of $\SC(\Lie{g}) \cong \SC(\R^\dimG)$, $\dimG := \mathrm{dim}(G) \in \N$, under the pullback of the global chart map $\log := \exp^{-1}: G \to \Lie{g}$ or, equivalently, of any other global polynomial chart map (cf.~\cite[\S~A.2]{CG}). Now if $G$ is homogeneous, one can show (cf.~\cite[Cor.~A.2.3]{CG}, \cite[\S~1.D]{FS} and \cite[Lem.~3.1.56]{FR}) that for any quasi-norm $| \, \cdot \, |$ on $G$ and arbitrary but fixed $p_1, p_2, p_3 \in [1, \infty]$ each of the families of seminorms
\begin{itemize}
	\item[(1)] $\bigl \{ f \mapsto \| x^\alpha X^\beta f \|_{L^{p_1}(G)} \mid \alpha, \beta \in \N^\dimG_0 \bigr \}$, \\
	\item[(2)] $\bigl \{ f \mapsto \| q_\alpha X^\beta f \|_{L^{p_2}(G)} \mid \alpha, \beta \in \N^\dimG_0 \bigr \}$, \\
	\item[(3)] $\bigl \{ f \mapsto \max_{[\beta] \leq N} \|(1 + |x|)^N X^\beta f \|_{L^{p_3}(G)} \mid N \in \N_0 \bigr \}$,
\end{itemize}
and each of the corresponding families for which $X^\beta$ is replaced by $\tilde{X}^\beta$, generates the same topology on $\SC(G)$ as the standard induced seminorms
\begin{itemize}
	\item[(0)] $\bigl \{ f \mapsto \| x^\alpha \partial^\beta (f \circ \exp) \|_{L^\infty(\Lie{g})} \mid \alpha, \beta \in \N^\dimG_0 \bigr \}$.
\end{itemize}
Without loss of generality we will from now on use the family of seminorms (2) with $p_2 = \infty$ and the abbreviation
\begin{align}
	\| f \|_{\SC, N} := \max_{[\alpha], [\beta] \leq N} \| q_\alpha X^\beta f \|_{L^\infty(G)}.\nonumber
\end{align}

The space of tempered distributions $\TD(G)$ is the topological dual of $\SC(G)$.

\subsection{Unitary group representations} \label{subs: unirrep}

In this subsection we recall some fundamental facts about representations of Lie groups necessary to develop our calculus on graded Lie groups. For more details we refer to the monograph~\cite{CG}.

A unitary irreducible representation is a group homomorphism
$\pi: G \to \mathcal{U}(\RS)$, for some complex Hilbert space $\RS$, which is
\begin{itemize}
	\item[--] strongly continuous: $\pi(x_j)v \to \pi(x)v$ for each $v \in \RS$ whenever $x_j \to x$ in $G$;
	\item[--] irreducible: $\pi(G)V \subseteq V$ for any subspace $V \subseteq \RS \Rightarrow$ $V = \RS$ or $V = \{0\}$.
\end{itemize}
The unitary dual $\Ghat$ of a Lie group $G$ is the set of equivalence classes $[\pi]$ of unitary irreducible representations of $G$, where any two representations $\pi: G \to \mathcal{U}(\RS)$ and $\pi': G \to \mathcal{U}(\RS')$ are called unitarily equivalent if there exists a $U \in \mathcal{U}(\RS, \RS')$ such that $U^{-1} \pi'(x) U = \pi(x)$ holds for all $x \in G$.
By a common abuse of notation, we will denote the equivalence classes by $\pi$ instead of $[\pi]$ when there is no risk of confusion.

A vector $v \in \RS$ is called smooth if the Hilbert space-valued function $x \mapsto \pi(x)v: G \to \RS$ is infinitely Fr\'{e}chet-differentiable on $G$\footnote{Equivalently, if $\pi \bigl (\log(\, \cdot \,) \bigr )v$ is infinitely Fr\'{e}chet-differentiable on $\Lie{g} \cong \R^\dimG$.}, and we denote the space of smooth vectors by $\RS^\infty$. We recall that if $G$ is nilpotent and $\pi \in \Ghat$ is realized to act on the Hilbert space $\RS = L^2(\R^d)$, then $\RS^\infty$ can be identified with $\SC(\R^d)$.

For a given $\pi \in \Ghat$, its infinitesimal representation $d\pi$ is the skew-adjoint representation of the Lie algebra $\Lie{g}$ defined by
\begin{align}
    d\pi(X)v := \frac{d}{dt}|_{t = 0} \Bigl ( \pi(\exp(tX))v - v \Bigr ) \label{eq:inf_rep}
\end{align}
for any $X \in \Lie{g}$ and $v \in \mathrm{dom}(d\pi(X)) = \{v \in \RS \mid \eqref{eq:inf_rep} \text{ converges in } \RS \} \supseteq \RS^\infty$.

This representation can naturally be extended to $\Lie{U}(\Lie{g})$, the universal enveloping algebra of the complexified Lie algebra $\Lie{g}_\C$, which by the Poincar\'{e}-Birkhoff-Witt theorem can be identified with the algebra $\mathrm{Diff}(G)$ of left-invariant (respectively right-invariant) differential operators on $G$ with complex coefficients. This identification allows us to consider the symbol
\begin{align}
    \pi(T) := d\pi(T)
\end{align}
 in the representation $\pi \in \Ghat$ of any given $T \in \mathrm{Diff}(G)$, e.g., the higher order left and right derivatives $X^\alpha$ and $\tilde{X}^\alpha \in \mathrm{Diff}(G)$, $\alpha \in \N_0^\dimG$, from Subsection~\ref{subs:inv_der_hom_pol}.
 
In the special case of the Abelian nilpotent Lie group $G = \R^n$ the unitary dual is given by the Abelian group
 \begin{align*}
    \widehat{\R}^n = \{x \mapsto e^{i x \xi} \mid \xi \in \R^n \},
 \end{align*}
which is isomorphic (as a Lie group) to $\R^n$. For each $\pi_\xi \in \widehat{\R}^n$ the natural representation space is $\RS = \C$ and $\pi_\xi(x) v = e^{i x \xi} v$ for all $x \in \R^n$, $v \in \C$. The corresponding infinitesimal represention is given by $d\pi_\xi(X)v = i (\sum_{j = 1}^n x_j \xi_j) v$ for all $X = x_1 e_1 + \ldots + x_n e_n \in \R^n \cong \mathrm{Lie}(\R^n)$ and $v \in \C = \RS^\infty$. Consequently, the symbol in $\pi_\xi$ of any $k$-th order partial differential operator with complex coefficients $T = \sum_{|\alpha| \leq k} c_\alpha \partial^\alpha \in \mathrm{Diff}(\R^n)$ is given by
\begin{align*}
    \pi_\xi(T) = \sum_{|\alpha| \leq k} c_\alpha i^{|\alpha|} \xi^\alpha = \sigma_T(x, \xi),
\end{align*}
which is the symbol of $T$ in any of the $\tau$-quantizations \eqref{eq:tau_quant} on $\R^n$.

\subsection{Rockland operators} \label{subs:RO}

In the setting of homogeneous Lie groups the notion of a Rockland operator is the natural generalization of the Laplace operator on $\R^n$ and of the canonical sub-Laplacians on stratified groups. The name goes back to Rockland's conjecture~\cite{RO} that for any given positive, homogeneous, left-invariant operator $\RO \in \mathrm{Diff}(G)$ on a homogeneous group $G$ the condition
\begin{align}
    \text{ (RO) } \forall \pi \in \Ghat \setminus \{ 1 \}, \forall v \in \RS^\infty: \pi(\RO)v = 0 \Rightarrow v = 0,
\end{align}
nowadays known as the Rockland condition, is equivalent to
\begin{align}
    \RO f \in C^\infty(G) \Rightarrow f \in C^\infty(G),
\end{align}
the hypoellipticity of $\RO$, a fact that was proved in \cite{RO} for $G = \H$. After a proof that hypoellipticity implies (RO) in \cite{B}, the conjecture was finally settled in \cite{HN}.

Throughout this paper any operator $\RO \in \mathrm{Diff}(G)$ with the above properties will be called a \textit{Rockland operator}. We will denote its homogeneous degree by $\hdeg \in \N$.

\medskip

We recall that a family of homogeneous dilations on a Lie group admits a Rockland operator if and only if there exists some $r_o \in \R$ such that $r_o v_1, \ldots, r_o v_\dimG \in \N$, which in turn holds if and only if the dilations are the canonical dilations of some gradation on $G$ (cf.~\cite{M,tERo}). So, although every homogeneous group $G$ can be equipped with a gradation (cf.~Subsection~\ref{subs:gr_strat_gr}), there are many families of homogeneous dilations that do not admit Rockland operators.

However, note that any graded group $G$ in fact possesses an infinite family of positive Rockland operators, e.g., the operators given by
\begin{align}
		\RO = \sum_{j=1}^\dimG (-1)^{\frac{\hdeg_0}{v_j}} c_j \hspace{1pt} X_j^{\frac{2 \hdeg_0}{v_j}} \hspace{5pt} \text{ with } \hspace{5pt} c_1, \ldots, c_\dimG > 0, \label{eq:gen_RO}
	\end{align}
for any strong Malcev basis $\{ X_1, \ldots, X_\dimG \}$ passing through the gradation (cf.~Subsection~\ref{subs:gr_strat_gr}) and any common multiple $\hdeg_o$ of $v_1, \ldots, v_\dimG$, are positive Rockland operators of homogeneous degree $\hdeg = 2 \hdeg_o$.

If $G$ is stratified, equipped with the canonical dilation structure, and $\{ X_1, \ldots, X_{\dimG_1} \}$ forms a basis of the first stratum $\Lie{g}_1$, then the sub-Laplacian
\begin{align}
    \sL = X_1^2 + \ldots + X_\dimG^2 \label{eq:gen_sL}
\end{align}
is a (negative) Rockland operator of homogeneous degree $\hdeg = 2$.

For the trivial stratification of $\Lie{g} = \R^n$, given by $\Lie{g} = \Lie{g}_1$, the Laplace operator $\partial_1^2 + \ldots + \partial_n^2$ on the corresponding stratified group $G = \exp(\Lie{g}) = (\R^n, +)$ is a particular case of sub-Laplacian if we choose the canonical homogeneous structure, whose dilations are given by 
$D_r(x) = r x$, $r > 0$. For any other family of isotropic or non-isotropic family of dilations, whose weights $(v_1, \ldots, v_\dimG) \in \R^\dimG \setminus \{ 1, \ldots, 1 \}$ may be assumed to be determined by the Euclidean standard basis $\{e_1, \ldots, e_n \}$ of $\Lie{g}$, any differential operator of the form
\eqref{eq:gen_RO} is a hypoelliptic operator on $G = \R^n$ if and only $r_o v_1, \ldots, r_o v_\dimG \in \N$ for some $r_o \in \R$.

\begin{assumptions}
Since the Rockland operators $\RO$ and $-\sL$ from \eqref{eq:gen_RO} and \eqref{eq:gen_sL}, respectively, are positive and essentially self-adjoint on $L^2(G)$, we will henceforth suppose that any given Rockland operator $\RO$ satisfies these properties.
\end{assumptions}

\subsection{The group Fourier transform} \label{subs:GFT}

In the following we recall a few essential properties of the Fourier transform as well as some crucially related notions. Among them are two noncommutative $L^p$-spaces on the unitary dual $\Ghat$, which are of particular importance for the Kohn-Nirenberg calculus~\cite{FR} and, a fortiori, for the $\tau$-calculus developed in this paper: the Hilbert space $L^2(\Ghat)$ and the von Neumann algebra $L^\infty(\Ghat)$.

To discuss these spaces, however, we need to equip $\Ghat$ with a suitable topology, namely one that generates the $\sigma$-algebra of Borel sets which is used to prove the Fourier inversion and Plancherel theorems, the so-called Mackey Borel structure. Following the mainstream convention, we always assume that $\Ghat$ is equipped with the so-called Fell topology. The resulting topological space $\Ghat$ is generally $(T_0)$ but not Hausdorff. (More more details we refer to, e.g., \cite[Ch.~18]{Di77} and \cite[Ch.~7]{FollAHA}.)

Thus, viewing $\Ghat$ as a topological space, the space $L^2(\Ghat)$ can be defined as a direct integral Hilbert space which arises from the Fourier inversion formula for Schwartz functions: for every $f \in \SC(G)$ the Fourier transform $\widehat{f}$ is a field of Hilbert-Schmidt operators $\widehat{f}(\pi)$ on $\RS$ which is measurable with respect to the Mackey-Borel structure. Moreover, there exists a measure $\mu$ on $\Ghat$, the so-called Plancherel measure, such that the inversion formula~\eqref{eq:IF_GFT} holds for all $x \in G$ and such that, for $f \in \SC(G)$, the Plancherel formula
\begin{align}
	\int_G |f(x)|^2 \, dx &= \int_{\Ghat} \| \widehat{f}(\pi) \|_{\mathtt{HS}(\RS)}^2 \, d\mu(\pi), \label{eq:Plancherel}
\end{align}
holds.
This identity is crucial for a rigorous version of the quantization formula~\eqref{eq:tau_quant_G} once we know that the integral converges for certain symbol classes. In addition, the restriction of $\F$ to $\SC(G) \subseteq L^1(G) \cap L^2(G)$ extends uniquely to a unitary isomorphism from $L^2(G)$ onto the direct integral Hilbert space of measurable fields of Hilbert-Schmidt operators
\begin{align*}
	L^2(\Ghat) := \int_{\Ghat}^\oplus \mathtt{HS}(\RS) \, d\mu(\pi) \hspace{5pt} \text{with norm} \hspace{5pt} \| \sigma \|_{L^2(\Ghat)} = \Bigl ( \int_{\Ghat} \| \sigma_\pi \|_{\mathtt{HS}(\RS)}^2 \, d\mu(\pi) \Bigr )^{1/2}.
\end{align*}
Note that its elements are the equivalence classes of such fields determined up to measure zero and unitary equivalence of $\pi \in [\pi] \in \Ghat$. The unitary isomorphism from $L^2(G)$ to $L^2(\Ghat)$ is called the Plancherel transform and its existence implies that the measure $\mu$, called the Plancherel measure on $\Ghat$, is uniquely determined.

The space $L^\infty(\Ghat)$ is called the group von Neumann algebra (see, e.g. \cite{FR}) and is defined as the space of all measurable fields $\sigma = \{ \sigma_\pi \}_{\pi \in \Ghat}$ of bounded operators $\sigma_\pi: \RS \to \RS$ which are essentially uniformly bounded with respect to the Plancherel meausre $\mu$, i.e.,
\begin{align*}
	L^\infty(\Ghat) := \bigl \{ \sigma = \{ \sigma_\pi \}_{\pi \in \Ghat} \, \mbox{ meas.} \mid \sigma_\pi \in \BL(\RS), \, \esssup_{\pi \in \Ghat} \| \sigma_\pi \|_{\BL(\RS)} < \infty \bigr \}.
\end{align*}
As in the case of $L^2(\Ghat)$, the elements of $L^\infty(\Ghat)$ are equivalence classes determined up to Plancherel measure zero and unitary equivalence of $\pi \in [\pi] \in \Ghat$, and pointwise multiplication (in this sense) of any two fields $\sigma_1$ and $\sigma_2$ is easily seen to turn $L^\infty(\Ghat)$ into a $C^*$-algebra. For general graded groups $G$ it plays the same important role for defining symbol classes on $G \times \Ghat$ that the $C^*$-algebra $L^\infty(\widehat{\R}^n) \cong L^\infty(\R^n)$ plays for symbols defined on $T^*\R^n \cong \R^n \times \widehat{\R}^n$ in the classical case.

Note that because of the bound $\| \pi \|_{\BL(\RS)} = 1$ for all $\pi \in \Ghat$, the Fourier integral~\eqref{eq:GFT} converges for all $f \in L^1(G)$ and the Fourier transform $\F$ maps the space continuously into $L^\infty(\Ghat)$. Since the elements of $L^1(G)$ act continuously on $L^2(G)$ by convolution from the right (cf.~Young's convolution inequality), the Fourier transform defined on $L^1(G)$ extends, as a bounded linear map, to the space $\BL_{\mathrm{L}}(L^2(G))$ of all bounded left-invariant operators on $L^2(G)$. By a special version of the Schwartz kernel theorem (see, e.g., \cite[Cor.~3.2.1]{FR}), these operators can be viewed as convolution operators with kernels $\kappa \in \TD(G)$. If $\mathscr{K}(G)$ denotes the space of all $\kappa \in \mathscr{D}'(G)$ such that the linear map
\begin{align*}
	f \mapsto f * \kappa: \mathscr{D}(G) \to L^2(G)
\end{align*}
extends to a bounded operator on $L^2(G)$, then $\mathscr{K}(G)$ equipped with the convolution product and the usual $\BL(L^2(G))$-operator norm is a von Neumann algebra, which is isomorphic to $\BL_{\mathrm{L}}(L^2(G))$. As a consequence, the Fourier transform on $L^1(G)$ extends to a von Neumann algebra isomorphism from $\mathscr{K}(G) \cong \BL_{\mathrm{L}}(L^2(G))$ to $L^\infty(\Ghat)$, which justifies its name. The Plancherel formula helps express these isomorphisms in a very clear fashion: given $T \in \BL_{\mathrm{L}}(L^2(G))$, its uniquely determined convolution kernel $\kappa \in \mathscr{K}(G)$ and the corresponding Fourier multiplier $\sigma = \widehat{\kappa} \in L^\infty(\Ghat)$, one has
\begin{align*}
	\subbracket{L^2(G)}{T \phi}{\psi} = \subbracket{L^2(G)}{\phi * \kappa}{\psi} = \int_{\Ghat} \Tr \bigl ( \sigma_\pi \hspace{2pt} \widehat{\phi}(\pi) \hspace{2pt} \widehat{\psi}(\pi)^* \bigr ) \,d\mu
\end{align*}
 for all $\phi, \psi \in L^2(G)$

In our case we need not distinguish the Fourier transforms on $L^1(G)$, $L^2(G)$ or $\mathscr{K}(G)$ if the context is unambiguous, in which case we will simply speak of the (group) Fourier transform $\F$ and denote the images of functions/distributions by $\F (f)$ or $\widehat{f}$.

\subsection{The Heisenberg group}

The simplest non-Abelian nilpotent Lie algebras are the \textit{Heisenberg Lie algebras} $\h$ with $n \in \N$. For a given $n \in \N$, it is defined as the vector space $\R^{2n+1}$ equipped with the Lie bracket whose only non-trivial relation is
\begin{align*}
	[X_j, X_{n+j}] = X_{2n+1}, \hspace{5pt} j = 1, \ldots, n,
\end{align*}
for the Euclidean standard basis $\{ X_1, \ldots, X_{2n+1} \}$. This Lie algebra admits the stratification
\begin{align*}
	\h = \R\mathrm{-span}\{ X_1, \ldots, X_{2n} \} \oplus \R X_{2n+1},
\end{align*}
and hence a canonical family of homogeneous dilations, determined by
\begin{align}
\begin{split} 
\label{eq:hom_dil_Hn}
    D_r(X_j) &= r X_j, \hspace{5pt} j = 1, \ldots, 2n, \\
    D_r(X_{2n+1}) &= r^2 X_{2n+1}.
\end{split}
\end{align}

The corresponding stratified Lie group $\H := \exp(\h)$ is called the \textit{Heisenberg group} and its group law is given by
\begin{align*}
	&x y = \exp \bigl ( x_1 X_1 + \ldots + x_{2n+1} X_{2n+1} \bigr ) \exp \bigl ( y_1 X_1 + \ldots + y_{2n+1} X_{2n+1} \bigr ) \\
	&= \exp \Bigl ( (x_1 + y_1) X_1 + \ldots + (x_{2n} + y_{2n}) X_{2n} \\
	&\hspace{120pt} + \bigl ( x_{2n+1} + y_{2n+1} + \sum_{j=1}^n\frac{x_j y_{n+j} - x_{n+j} y_j}{2} \bigr )X_{2n+1} \Bigr ).
\end{align*}
The left-invariant basis vector fields on $\H$,
\begin{align}
\begin{split}
	X_j &= \partial_j - \frac{1}{2} x_{n+j} \partial_{2n+1}, \quad j = 1, \dots, n,\\ X_{n+j} &= \partial_{n+j} + \frac{1}{2} x_j \partial_{2n+1}, \quad j = 1, \dots, n, \\  X_{2n+1}& = \partial_{2n+1},  \label{eq:vf_Hn}
    \end{split}
\end{align}
are of homogeneous degrees $1$, $1$ and $2$, respectively, according to \eqref{eq:hom_dil_Hn}.

Apart from the quasi-norms~\ref{eq:p_qn} and \ref{eq:inf_qn}, which are defined for any family of dilations on $\H$, the canonical dilations admit a homogeneous norm of the form
\begin{align*}
	|x| := \Bigl ( (|x_1|^2+ \ldots + |x_{2n}|^2)^2 +
	\frac{1}{16} |x_{2n+1}|^2 \Bigr )^{\frac{1}{4}},
\end{align*}
called the Kor\'{a}nyi-Cygan-norm.

The unitary dual $\Hhat$ is exhausted by, up to Plancherel measure zero and unitary equivalence, the so-called Schr\"{o}dinger representations on $L^2(\R^n)$, since finite dimensional representations have Plancherel measure zero. For a fixed $\lambda \in \R \setminus \{ 0 \}$, the Schr\"{o}dinger representation $\pi_\lambda \in \Hhat$ can be realized on $\mathcal{H}_{\pi_\lambda} := L^2(\R^n)$ as
\begin{align}
    \bigl ( \pi_\lambda(x) h \bigr )(u) &= \bigl ( \pi_\lambda(x', x'', x_{2n+1}) h \bigr )(u) \\ &= e^{i \lambda (x_{2n+1} + \frac{1}{2}x'x'')} e^{i \sqrt{\lambda} x''u} h(u + \sqrt{|\lambda|} x') \label{eq:Schr_rep}
\end{align}
for $u \mapsto h(u) \in L^2(\R^n)$ and $x' := (x_1, \ldots, x_n)$, $x'' := (x_{n+1}, \ldots, x_{2n}) \in \R^n$, and the Plancherel identity~\eqref{eq:Plancherel} takes the simple form
\begin{align*}
	\int_{\H} |f(x)|^2 \, dx &= \int_{\R \setminus \{ 0 \}} \| \widehat{f}(\pi_\lambda) \|_{\mathtt{HS}(\mathcal{H}_{\pi_\lambda})}^2 \, | \lambda |^n \, d\lambda.
\end{align*}
Since the infinitesimal Schr\"{o}dinger representations of the vector fields~\eqref{eq:vf_Hn} are given by
\begin{equation*}
\begin{array}{rcl}
	\bigl ( \pi_\lambda(X_j) h \bigr )(u) &=& \sqrt{| \lambda |} \partial_{u_j} h(u), \\
	\bigl ( \pi_\lambda(X_{n+j}) h \bigr )(u) &=& i \sqrt{\lambda} u_j \hspace{2pt} h(u), \\
	\bigl ( \pi_\lambda(X_{2n+1}) h \bigr )(u) &=& i \lambda \hspace{2pt} h(u),
\end{array}
\end{equation*}
the symbol in $\pi_\lambda \in \Hhat$ of the canonical sub-Laplacian $\mathcal{L}_{\H} = X_1^2 + \ldots + X_{2n}^2$ is given by
\begin{align}
    \pi_\lambda \bigl ( \mathcal{L}_{\H} \bigr ) = | \lambda | \sum_{j=1}^n (\partial_j^2 - u_j^2), \label{eq:harmonic_oscillator}
\end{align}
that is, $|\lambda|$ times the harmonic oscillator on $L^2(\R^n)$.

\section{Symbol classes and Sobolev spaces} \label{sec:sym_Sob}

To set the stage for the symbolic calculus for the $\tau$-quantizations \eqref{eq:tau_quant_G}, we recall some crucial facts about the H\"{o}rmander-type symbol classes and the inhomogeneous Sobolev spaces on graded groups established in \cite{FR}.

\subsection{Bessel potentials} \label{subs:Bessel}

Given a positive Rockland operator $\RO$ of homogeneous degree $\hdeg \in \N$, let us denote by $\RO_2$ its self-adjoint extension on $L^2(G)$ and by $\varphi(\RO_2)$ any spectral multiplier of $\RO_2$ defined via the spectral calculus for unbounded self-adjoint operators in $L^2(G)$. Since $\RO_2$ is left-invariant, the operators $\varphi(\RO_2)$ are also left-invariant, which in combination with the Schwarz kernel theorem for $\mathscr{D}'(G)$ yields the identity
\begin{align*}
	\varphi(\RO_2) f = f * \kappa_{\varphi(\RO_2)}
\end{align*}
for all $f \in \mathscr{D}(G)$ and some $\kappa_{\varphi(\RO_2)} \in \mathscr{D}'(G)$. Of particular importance for our setting are the generally unbounded self-adjoint operators $(I + \RO_2)^{\frac{s}{\hdeg}}$, $s \in \R$, called Bessel potentials.
For any integer $s = k \hdeg$, $k \in \N$, the spectral multiplier $(I + \RO_2)^{\frac{s}{\hdeg}}$ equals the left-invariant differential operator $(I + \RO)^k$ on $\mathscr{D}(G)$. For general $s \in \R$ one can use properties of the heat kernel $h_t$ on $L^2(G)$, defined by
\begin{align*}
	f * h_t := e^{-t \RO_2} f
\end{align*}
for all $f \in L^2(G)$, to show that the convolution identity
\begin{align}
	(I + \RO_2)^{-\frac{s}{\hdeg}}f = f * \mathcal{B}_s  \label{eq:Bessel_con_ker}
\end{align}
holds for all $f \in \SC(G)$ and a kernel $\mathcal{B}_s \in \TD(G)$ and that in fact $(I + \RO_2)^{\frac{s}{\hdeg}}(\SC(G)) = \SC(G)$. Moreover, for any $s > 0$ the convolution kernel can explicitly be expressed in terms of the heat kernel:
\begin{align*}
	\mathcal{B}_s(x) = \frac{1}{\Gamma(\frac{s}{\nu})} \int_0^\infty t^{\frac{s}{\nu}-1} h_t(x) \, dt.
\end{align*}
Since the group Fourier transform relates the spectral calculus of $\RO_2$ on $L^2(G)$ with that of the self-adjoint operator $\pi(\RO)$ on $\RS$ for every $\pi \in \Ghat$ (cf.~\cite[Ch.~4]{FR}), an application of the Fourier transform to \eqref{eq:Bessel_con_ker} yields the identity
\begin{align}
	\pi(I + \RO)^{-\frac{s}{\hdeg}} \widehat{f}(\pi) = \F \bigl ( f * \mathcal{B}_s \bigr )(\pi) \hspace{5pt} \mbox{ a.e. } \hspace{5pt} \pi \in \Ghat \label{eq:Bessel_symbol}
\end{align}
for the symbol of the Bessel potential. This identity will prove crucial for the rest of the paper since by \eqref{eq:tau_quant_G} we obtain, for now at least formally,
\begin{align*}
	\Op^\tau \bigl ( \pi(I + \RO)^{-\frac{s}{\hdeg}} \bigr ) = (I + \RO_2)^{-\frac{s}{\hdeg}}
\end{align*}
for all quantizing functions $\tau$.

\subsection{H\"{o}rmander-type symbol classes} \label{subs:Horm_sc}

Let $\RO$ be a positive Rockland operator of homogeneous degree $\nu \in \N$. A smooth symbol is a field of operators
\begin{align*}
	\sigma = \bigl \{ \sigma(x, \pi): \RS^\infty \to \RS \mid x \in G, \pi \in \Ghat \bigr \}
\end{align*}
for which there exist $a, b \in \R$ such that for each $x \in G$ the field 
\begin{align*}
	\{\pi(I + \RO)^{\frac{a}{\hdeg}} \sigma(x, \pi) \pi(I + \RO)^{-\frac{b}{\hdeg}} \mid \pi \in \Ghat \bigr \}
\end{align*}
is an element of $L^\infty(\Ghat)$ and such that this field is smooth as a Banach space-valued function on $G$. By the properties of $L^\infty(\Ghat)$ discussed in Subsection~\ref{subs:GFT}, the condition on $a, b \in \R$ is equivalent  to $L^2$-continuity of the convolution operators $T_xf = f * \kappa_x$ with kernels $\kappa_x \in \mathscr{S}'(G)$, $x \in G$, defined by
\begin{align*}
    \kappa_x(y) := \F^{-1} \Bigl ( \{\pi(I + \RO)^{\frac{a}{\hdeg}} \sigma(x, \pi) \pi(I + \RO)^{-\frac{b}{\hdeg}} \}_{\pi \in \Ghat} \Bigr )(y) = \bigl (\mathcal{B}_{b} * \kappa_\sigma(x, \, . \,) * \mathcal{B}_{-a} \bigr )(y).
\end{align*}

The generalization of differentiability in the dual variable $\xi \in \R^n$ of symbols $\sigma: \R^n \times \Rhat^n \to \C$ to general graded groups $G$ involves the so-called difference operators on $\Ghat$. For any $\alpha \in \N^\dimG_0$, we write $\tilde{q}_\alpha(x) := q_\alpha(x^{-1})$, $x \in G$, for the homogeneous polynomial $q_\alpha$ determined by \eqref{eq:can_hom_pol}. Then the difference operator $\Delta^\alpha$ is defined by
\begin{align}
    \bigl ( \Delta^\alpha \widehat{f} \bigr )(\pi) := \widehat{\tilde{q}_\alpha f}(\pi), \hspace{5pt} \pi \in \Ghat,
\end{align}
for any $f \in \D(G)$ for which
\begin{align*}
    \mathcal{B}_{b} * f * \mathcal{B}_{-a}, \hspace{5pt} \mathcal{B}_{b'} * (\tilde{q}_\alpha f) * \mathcal{B}_{-a'} \in \mathcal{K}(G)
\end{align*}
for some $a, a', b, b' \in \R$, i.e., for which these kernels define continuous convolution operators on $L^2(G)$. While on $G = \R^n$ this recovers the classical notion, with
\begin{align*}
    \bigl ( \Delta^\alpha \widehat{f} \bigr )(\xi) = (-i)^{|\alpha|} (\partial^\alpha_\xi \widehat{f})(\xi)
\end{align*}
due to $q_\alpha(x) = (\alpha_1 ! \cdots \alpha_n !)^{-1} x^{\alpha}$, for the Heisenberg group $\H$, for example, one obtains a combination of  Euclidean derivatives $\partial_\lambda$ and commutators with the symbols $\pi_\lambda(X_j)$, $j = 1, \ldots, 2n$~\cite[\SS~6.3.3]{FR}. (See also~\cite{Ch} for the difference operators on the Engel and Cartan groups.)

Now let $m \in \R$ and $0 \leq \delta \leq \rho \leq 1$. Following~\cite{FR}, we will denote by $S^m_{\rho,\delta}(G)$ the class of smooth symbols which for all $a, b\in \N_0$ satisfy
\begin{align}
\begin{split} \label{defn.seminorms}
	\| \sigma\|_{S^m_{\rho,\delta},a,b} :=& \sup_{\substack{(x, \pi) \in G \times \Ghat \\ [\alpha] \leq a, [\beta] \leq b}} \bigl \| \bigl ( \Delta^\alpha X_x^\beta \sigma \bigr )(x,\pi)\pi(I+\RO)^{-\frac{m-\rho[\alpha]+\delta[\beta]}{\nu}} \bigr \|_{\mathscr{L}(\mathcal{H}_\pi)} \\
	=& \sup_{\substack{x \in G \\ [\alpha] \leq a, [\beta] \leq b}} \bigl \| \bigl \{ \bigl (\Delta^\alpha X_x^\beta \sigma \bigr )(x,\pi)\pi(I+\RO)^{-\frac{m-\rho[\alpha]+\delta[\beta]}{\nu}} \mid \pi \in \Ghat \bigr \} \bigr \|_{L^\infty(\Ghat)}< \infty.
\end{split}
\end{align}
These symbols are called smooth symbols of order $m$ and type $(\rho,\delta)$.

\medskip

Note that the seminorms defined by \eqref{defn.seminorms} are precisely the seminorms $\| \cdot \|_{S^{m,R}_{\rho,\delta},a,b}$ in \cite{FR}, which for technical reasons prove to be more convenient in this paper. However,
due to \cite[Thm.~5.5.20]{FR}, they define the same topology on $S^m_{\rho,\delta}(G)$ as the original seminorms in Definition~\cite[Def.~5.2.11]{FR}. Several equivalent characterizing families of seminorms are provided by \cite[Thm.~5.5.20]{FR} and \cite[Thm.~13.16]{CR_Sub}, and in each case the topology is independent of the choice of Rockland operator $\RO$ on $G$.

\begin{example} \label{ex:Hormander_classes_Rn}
For $G = \R^n$, equipped with the canonical isotropic dilations, and $\RO = - \sum_{j = 1}^n \partial_{x_j}^2$, one recovers the seminorms
\begin{align*}
	\| \sigma\|_{S^m_{\rho,\delta},a,b} = \sup_{\substack{(x, \xi) \in \R \times \widehat{\R}^n \\ [\alpha] \leq a, [\beta] \leq b}} \bigl | \bigl ( \partial_\xi^\alpha \partial_x^\beta \sigma \bigr )(x,\xi)(I + |\xi |^2)^{-\frac{m-\rho[\alpha]+\delta[\beta]}{2}} \bigr |,
\end{align*}
which define the H\"{o}rmander symbol classes $S^m_{\rho,\delta}(\R^n)$.
\end{example}

\begin{example} \label{ex:Hormander_classes_Hn}
The idea of using operator-valued symbols was first proposed for the case of convolution operators on the Heisenberg group $\H$ in \cite[Chapter~2]{TayNCMA}. Right from the start, the author highlights the crucial observation that if one chooses to realize the elements of $\Hhat$ as Schr\"{o}dinger represenations on $L^2(\R^n)$, then the symbol of an operator $K f = f * \kappa$ can be described by
\begin{align} \label{eq:Hn_symbols_Weyl_Taylor}
    \sigma(\pi_{\pm \lambda}) = \widehat{\kappa}(\pi_{\pm \lambda}) = \Opw_{\R^n}(a_{\pm \lambda}), \quad \lambda \in \R \setminus \{ 0 \},
\end{align}
that is, by a family of Weyl-quantized operators with scalar-valued symbols
\begin{align*}
    a_{{\pm \lambda}}(u, \xi) = \iint_{\R^{2n}} \kappa(\pm \lambda, \pm \lambda^{1/2} q, \lambda^{1/2} p) \, e^{i (qx + p\xi)} \, dq \, dp, \quad \lambda \in \R \setminus \{ 0 \},
\end{align*}
that are defined on the classical phase space $\R^{2n}$.\footnote{Our notation differs slightly from \cite{TayNCMA} since the translation-invariant operators in this paper are left-invariant and not right-invariant as in \cite{TayNCMA}.}

The first calculus for $x$-dependent pseudo-differential operators that made use of an extended version \eqref{eq:Hn_symbols_Weyl_Taylor} was developed in \cite{BFG}. Similar motivations led the authors of \cite{FR} to a characterization in the spirit of \eqref{eq:Hn_symbols_Weyl_Taylor} of their H\"{o}rmander-type classes $S^m_{\rho,\delta}(\H)$ (cf~\cite[Theorem~6.5.1]{FR}).
\end{example}

\begin{remark} \label{rem:pw_product}
It was shown in \cite[Thm.~5.2.22~(ii)]{FR} that if $\sigma_1 \in S^{m_1}_{\rho, \delta}(G)$ and $\sigma_2 \in S^{m_2}_{\rho, \delta}(G)$, for $m_1, m_2 \in \R$ and $0 \leq \delta \leq \rho \leq 1$, then the smooth symbol defined by
\begin{align}
    \sigma(x, \pi) := \sigma_1(x, \pi) \sigma_2(x, \pi)
\end{align}
lies in $S^{m_1 + m_2}_{\rho, \delta}(G)$.
\end{remark}

\subsection{Associated kernels} \label{subs:ass_ker}

By~\cite[Thm.~5.2.15]{FR}, any pseudo-differential operator
\begin{align}
    \Op(\sigma)f(x) = \iint\limits_{\Ghat \times G} \Tr \bigl ( \pi(y^{-1} x) \sigma(x, \pi) f(y) \bigr ) \, dy \, d\mu(\pi) \label{eq:KN_2}
\end{align}
with $\sigma \in S^{m}_{\rho, \delta}(G)$ for some $m \in \R$ and $0 \leq \delta \leq \rho \leq 1$, is continuous on the Schwartz space $\SC(G)$. It follows from the Schwartz kernel theorem that $\Op(\sigma)$ has a distributional kernel $\mathsf{Ker}_\sigma$ in $\TD(G \times G)$. Abusing the notation of integral kernels, we will frequently express this fact by writing
\begin{align}
     \Op(\sigma)f(x) = \int_G \mathsf{Ker}_\sigma(x, y) f(y) \, dy \nonumber
\end{align}
for $f \in \SC(G)$. The so-called \textit{associated kernel} $\kappa_\sigma$ is formally defined by
\begin{align}
    \Op(\sigma)f(x) = \int_G \kappa_\sigma(x, y^{-1} x) f(y) \, dy, \label{eq:ass_ker}
\end{align}
so that, again formally, it is related to the symbol by
\begin{align}
    \sigma(x, \pi) = \F \bigl ( \kappa_\sigma(x, \, \cdot \,) \bigr )(\pi) \label{eq:symbol}
\end{align}
for $x \in G$, $\pi \in \Ghat$. The change of variables
\begin{align}
     &\mathrm{cv}: G \times G \to G \times G, \nonumber \\
    &\mathrm{cv}(x, y) := (x,y^{-1}x), \label{eq:Ker_kappa}
\end{align}
which formally relates $\mathsf{Ker}_\sigma$ and $\kappa_\sigma$ to each other, is a Lie group automorphism of the direct product group $G \times G$, and its pullback $\mathrm{CV}$ maps $\SC(G \times G)$ continuously onto itself. By the open mapping theorem for Fr\'{e}chet spaces (cf.~\cite[Thm.~17.1]{T}), it is even a Fr\'{e}chet space isomorphism on $\SC(G \times G)$. So, under the assumption that $\mathsf{Ker}_\sigma \in \TD(G \times G)$, by duality, we also have $\kappa_\sigma \in \TD(G \times G)$.

Maintaining the formal relation~\eqref{eq:ass_ker} for a given $\tau$-quantization~\eqref{eq:tau_quant_G}, we may rewrite the quantization as
\begin{align}
    \Optau(\sigma)f(x) = \int_G \kappa_\sigma(x \tau(y^{-1}x)^{-1}, y^{-1}x)f(y) \, dy = \int_G \Ker_\sigma(x , y) f(y)\, dy, \label{eq:tau_quant_ass_kernel_int_kernel}
\end{align}
for its distributional kernel $\Ker_\sigma$. Under the assumption that the quantizing function $\tau:G \to G$ is smooth and that its pullback maps the Schwartz space $\SC(G \times G)$ continuously into itself, the change of coordinates
\begin{align}
    &\cv: G \times G \to G \times G, \nonumber \\
    &\cv(x,y) := (x\tau(y^{-1}x)^{-1},y^{-1}x), \label{eq:cv}
\end{align}
which formally relates $\Ker_\sigma$ to $\kappa_\sigma$, is easily seen to be a Lie group automorphism of $G \times G$ with inverse
\begin{align}
    \cvi(x,y) = (x\tau(y),x\tau(y)y^{-1}).  \nonumber
\end{align}
So, the associated pullbacks
\begin{align}
    (\CV \circ \kappa_\sigma)(x,y) &:= \kappa(\mathrm{cv}^\tau(x,y)), \\
    (\CVI \circ \kappa_\sigma)(x,y) &:= \kappa((\mathrm{cv}^\tau)^{-1}(x,y)),
\end{align}
are in fact Fr\'{e}chet space isomorphisms on $\SC(G\times G)$ and $\Ker_\sigma \in \TD(G \times G)$.
For any reasonable pseudo-differential calculus on graded groups this is an indispensable condition, hence throughout this paper we will assume
\begin{itemize}
    \item[(CV)] $\tau$ is smooth and the pullback of $\cv$ is a Fr\'{e}chet space isomorphism on the Schwartz space $\SC(G \times G)$.
\end{itemize}

\medskip

Let us finally point out that the formal relation between the distributional $\Ker_\sigma$ kernel and the symbol $\sigma$, given by
\begin{align}\label{symbol-kernel}
\sigma(x,\pi) = \bigl ( (\mathrm{id}\otimes\F_{G})\circ \kappa \bigr )(x,\pi) = \F_{y \to \pi}(\CVI \circ \Ker_\sigma)(x,y),
\end{align}
which, for symbols $\sigma \in S^{-\infty}(G)$, is given by an absolutely convergent integral.

\subsection{Kernel estimates} \label{sub:ker_est}

The following estimates from \cite{FR} for the kernels $\kappa_\sigma$ associated to symbols $\sigma \in S^m_{\rho, \delta}(G)$, $m \in \R$, $0 \leq \delta \leq \rho \leq 1$, will be used frequently throughout the paper. The quasi-norm in question will always be the canonical supremum quasi-norm defined by \eqref{eq:inf_qn}.\footnote{This is a matter of convenience since any two quasi-norms on a given $G$ are equivalent and therefore give the same kernel estimates.}

The first result gives a general $L^2$-decay condition, which for large negative $m$ gives square integrability.

\begin{proposition} [\cite{FR}, Prop.~5.2.16] \label{prop:L2_int_ass_ker}
Let $m \in \R$, $0 \leq \delta \leq \rho \leq 1$ and let $\sigma \in S^m_{\rho, \delta}(G)$. If $m < -\frac{\hdim}{2}$, then for any $x \in G$, the tempered distribution $y \mapsto \kappa_\sigma(x, y)$ is square-integrable and
\begin{align}
\begin{split} \label{eq:L^2_bd_ass_ker}
	\|\kappa(x, \, \cdot \,)\|_{L^2(G)} &\leq C \sup_{\pi \in \Ghat} \| \pi(I + \RO)^{-\frac{m}{\hdeg}} \sigma(x, \pi) \|_{\mathcal{L}(\RS)}, \\
\|\kappa(x, \, \cdot \,)\|_{L^2(G)} &\leq C \sup_{\pi \in \Ghat} \| \sigma(x, \pi) \pi(I + \RO)^{-\frac{m}{\hdeg}} \|_{\mathcal{L}(\RS)},
\end{split}
\end{align}
with $C = C(m) \in (0, \infty)$ independent of $\sigma$ and $x$.
\end{proposition}

The second estimate ensures Schwartz decay away from the origin.

\begin{proposition} [\cite{FR}, Prop.~5.4.4] \label{prop:ker_est_at_infty}
Let $m \in \R$, $0 \leq \delta \leq \rho \leq 1$ and $\rho > 0$. Let $\sigma \in S^m_{\rho, \delta}(G)$ and let $|\, \cdot \,|$ be a quasi-norm on $G$. Then for any $M \in \R$ and any $\alpha, \beta_1, \beta_2, \beta_o \in \N^\dimG_0$ there exist some $C > 0$ and $a, b \in \N$ independent of $\sigma$ such that for all $x \in G$ and all $z \in G$ with $|z| \geq 1$, we have
\begin{align*}
	\bigl | X^{\beta_1}_{z_1 = z} X^{\beta_2}_{z_2 = z_1} X^{\beta_o}_{x_1 = x} \tilde{q}_\alpha(z) \kappa_\sigma(x_1, z_2) \bigr | \leq C \|\sigma\|_{S^m_{\rho, \delta}} |z|^{-M}.
\end{align*}
\end{proposition}

The third result, a direct consequence of Lemma~5.5.6 and Propositions~5.4.4 and 5.4.6 in \cite{FR}, provides some integrability conditions, dictated by the behavior of the kernel at the origin. In addition, the result shows that such oscillatory integrals can be bounded by suitable seminorms of the corresponding symbol, a fact that will turn out to be crucial to perform the approximation argument we use in Section \ref{sec:sym_calc}.

\begin{lemma}\label{FRcorollary}
Let $\sigma\in S^{m}_{\rho,\delta}(G)$, with $0\leq \delta\leq\rho\leq 1$ and $\rho\neq 0$, and let $\kappa_x$  be the associated kernel. Then for every $\gamma\in\mathbb{R}$, with $\gamma+Q>\max\left\{ \frac{Q+m+\delta[\beta_0]-\rho[\alpha]+[\beta_1]+[\beta_2]}{\rho},0\right\}$, there exists $C>0$ and a seminorm $\|\cdot\|_{S^{m}_{\rho,\delta}a,b}$ such that
\begin{align*}
	\int_{G} |z|^\gamma \sup_{x\in G}|X^{\beta_1}_z\tilde{X}^{\beta_2}_z (X^{\beta_0}_x\tilde{q}_\alpha(z)\kappa_x(z))|dz\leq C \|\sigma\|_{S^{m}_{\rho,\delta}a,b}.
\end{align*}
\end{lemma}

\subsection{Sobolev spaces} \label{subs:Sobolev_spaces}

The notion of Sobolev space on a graded group established in \cite{FR, FR2} is based on the use of a positive Rockland operator $\RO$ in a very canonical fashion. This makes them the ideal spaces to express and judge the mapping properties of the $\tau$-calculus developed in this paper. This class of Sobolev spaces includes, in particular, the Sobolev spaces on stratified groups by \cite{F3}, since for a given stratified $G$ equipped with the corresponding canonical homogeneous structure, one may choose $\RO = -\sL$, the negative canonical homogeneous sub-Laplacian on $G$ from \eqref{eq:gen_sL}. For the trivially stratified group $G = \R^n$ the choice $\RO = - \partial_1^2 - \ldots -\partial_n^2$ recovers the classical homogeneous and inhomogeneous Sobolev spaces $\dot{L}^p_s(\R^n)$ and $L^p_s(\R^n)$, respectively, with $p \in (1, \infty)$, $s \in \R$.

The original definition of the inhomogeneous Sobolev spaces $L^p_s(G) = L^p_{s, \RO}(G)$ in \cite{FR, FR2} relies on the properties of the heat semi-group  $f \mapsto f * h_t$, $t > 0$, on $L^p(G)$, and on the spectral theory of its infinitesimal generator $\RO_p$. Since this is outside the scope of this paper, we choose a slightly different, nevertheless equivalent definition, reminiscent of \cite[\S~5.2]{S}, on $\R^n$.\footnote{We only rely on the familiarity with the spectral theory of unbounded self-adjoint operators in Hilbert space, in this case $\RO_2$, to readily access the relevant Sections~4.1 and 4.2 in \cite{FR} on the heat kernel and the Bessel potentials when $p = 2$.}

In order to avoid any circular arguments in our approach, let us first recall that for the Bessel potential $(I + \RO_2)^{\frac{s}{\hdeg}}$ on $L^2(G)$, and $s \in \R$, the Sobolev space $L^2_s(G) = L^2_{s, \RO}(G)$ is defined as the completion of $\SC(G)$ in $\TD(G)$ with respect to the norm
\begin{align}
	\| f \|_{L^2_s(G)} := \| (I + \RO_2)^{\frac{s}{\hdeg}}f \|_{L^2(G)}.
\end{align}
To treat the case of $p \in (1, \infty)$, we use \eqref{eq:Bessel_con_ker} to deduce that the right convolution kernel of the adjoint of the operator $(I + \RO_2)^{\frac{s}{\hdeg}}: \SC(G) \to \SC(G)$ is an element of $\TD(G)$, which by an abuse of notation is also denoted by $\mathcal{B}_{-s}$.
However, the operator can also be viewed as a Fourier multiplier and its symbol $\F(\mathcal{B}_{-s}) = \pi(I + \RO)^{\frac{s}{\hdeg}}$ is clearly an element of $S^s_{1, 0}(G)$ by \cite[Prop.~5.3.4]{FR}.\footnote{This is a consequence of Hulanicki's theorem.} So, for $p \in (1, \infty)$, $s \in \R$, we can define the space $L^p_s(G) = L^p_{s, \RO}(G)$ as the space of all $f \in \TD(G)$ such that
\begin{align}
	\| f \|_{L^p_s(G)} := \| (I + \RO_2)^{\frac{s}{\hdeg}}f \|_{L^p(G)} < \infty.
\end{align}
It is easily seen that for $p = 2$ the two definitions agree, and for $s = 0$ one obviously recovers $L^p(G)$. Since the operators $(I + \RO_2)^{\frac{s}{\hdeg}}$ and $(I + \RO_p)^{\frac{s}{\hdeg}}$ coincide on $\SC(G)$, our definition gives the same spaces as \cite[Def.~4.2.2]{FR}, which was shown to be independent of the choice of $\RO$.

\begin{notation}
Throughout the rest of this paper we will no longer distinguish $\RO$ and $\RO_2$ notationally and simply write $\RO$,  since it will be clear from the context whether we mean the differential operator on its natural domain $\mathrm{dom}(\RO) \supseteq \SC(G)$ or its self-adjoint extension on $L^2(G)$.
\end{notation}

\section{Symbolic calculus} \label{sec:sym_calc}

In this section we develop a symbolic calculus for a wide range of $\tau$-quantizations on graded groups and the H\"{o}rmander symbol classes $S^m_{\rho, \delta}(G)$ introduced in~\cite{FR}.
For technical reasons that will become clear soon, we will mostly have to require that $\rho$ and $\delta$ satisfy $0\leq \delta<\frac{\rho}{v_n}\leq \rho\leq1$. Note that the fundamental case $(\rho,\delta)=(1,0)$ always satisfies this condition.

\subsection{Admissible quantizing functions} \label{subs:adm_qu_f}

In order to make the proposed $\tau$-calculus work for the symbol classes $S^m_{\rho, \delta}(G)$, we need to control oscillatory integrals such as \eqref{eq:tau_quant_ass_kernel_int_kernel}. We will therefore require that the quantizing functions $\tau: G \to G$ satisfy certain properties that are trivial for the $\tau$-quantizations~\eqref{eq:tau_quant} on $G = \R^n$. While on $\R^n$ this is owed to the trivial nature of its Lie algebra, on non-trivially graded groups it hinges on the compatibility with the given homogeneous structure. Thus, for any graded $G$ under consideration we fix, once and for all, a family of dilations that admits a Rockland operator and a Malcev basis $\{ X_1, \ldots, X_n \}$ of $\Lie{g}$ that passes through the gradation and whose elements are eigenvectors of the dilations with integer weights $1 \leq v_1 \leq \ldots \leq v_n$ (cf.~Subsections~\ref{subs:gr_strat_gr} and \ref{subs:RO}). We then define the following primary condition on quantizing functions $\tau$ on $G$:
\begin{itemize}
    \item[(HP)] The exponential coordinates of $\tau(x) = \bigl ( c^\tau_1(x), \ldots, c^\tau_n(x) \bigr ) \in G$ as functions of $x \in G$ either vanish identically, i.e., $c^\tau_j(x) = 0$ for all $x \in G$, or are \textit{homogeneous polynomials} of the form
    \begin{align} \label{eq:c_tau_j}
        c^\tau_j(x) = c^\tau_j(x_1, \ldots, x_j) = C^\tau_j x_j + d^\tau_j(x_1, \ldots, x_{j-1})
    \end{align}
    for some $C^\tau_j \neq 0$ and some homogeneous polynomial $d^\tau_j$ of degree $v_j$ which only depends on $x_k$ with $k = 1, \ldots, j-1$.
\end{itemize}

Whenever $\tau$ satisfies the symmetry condition~\eqref{eq:sym_quant} for the classes $S^m_{\rho, \delta}(G)$, $m \in \R$, $0\leq \delta<\frac{\rho}{v_n}\leq \rho\leq1$, we will call $\tau$ a \textit{symmetry function} and the corresponding $\tau$-calculus a \textit{symmetric calculus}.

\medskip

Since for such $\tau$ the change of coordinates $\cv$ defined by \eqref{eq:cv} is a Fr\'{e}chet space isomorphism on the Schwartz space $\SC(G)$ (see, e.g., \cite[Lem.~A.2.1]{CG}), the condition (HP) supersedes the condition (CV) from Subsection~\ref{subs:ass_ker}, which we have considered to be the minimal assumption on $\tau$ so far. To prove the main theorems of the $\tau$-calculus presented in this paper, we will always assume (HP). This covers all the classical $\tau$-quantizations~\eqref{eq:tau_quant} on $\R^n$ and the Kohn-Nirenberg quantization~\eqref{eq:KN_quant_G} as well as the symmetric quantizations arising from~\eqref{eq:MR_tau} and \eqref{eq:onehalf_tau} on general graded groups.

\begin{remark}
The explicit dependence of (HP) on the choice of Malcev basis may seem a bit stark at first. Admittedly, a more natural, larger class of admissible quantizing functions may emerge from relaxing the conditions on (HP) to merely depend on subspaces $\Lie{g}_i$ of the gradation rather than individual coordinates.\footnote{This very reasonable assumption was suggested to us by the referee.} While this is a very interesting aspect worth exploring in the future, we have tried not to involve even more algebraic properties of the Lie groups under consideration in our endeavor to develop a symbolic calculus for interesting families of quantizations.
\end{remark}

\begin{remark}
One could also dispense with the homogeneity of the coefficients of $\tau$ and admit quantizing functions with more general polynomial coordinates, including quadratic maps $\tau: \R^n \to \R^n$ as in \cite{CT} and the nonlinear quantizing functions $\tau$ studied in \cite{ER}. This would come at the cost of carefully keeping track of products and sums of homogeneous and isotropic orders of polynomials and vector fields throughout the already quite lengthy estimates of many oscillatory integrals, while it would not generate much additional insight. Implementing this in our framework would also require that if $k_o \in \N$ denoted the highest isotropic order of any polynomial coefficient of $\tau$, then the parameters $\rho, \delta \in [0, 1]$ would generally have to satisfy $0 \leq \delta < \frac{\rho}{k_o v_n} \leq 1$, a stricter condition than $0 \leq \delta < \frac{\rho}{v_n} \leq 1$ required currently.
\end{remark}

Let us provide a few examples of very distinct quantizing functions to illustrate that neither condition is overly restrictive but, on the contrary, rather natural.

\begin{example}
The quantizing function $\tau: x \mapsto e_G = (0, \ldots, 0)$ and the corresponding change of variables $\cv(x, y) = (x, y^{-1}x)$ clearly characterizes the Kohn-Nirenberg quantization from \cite{FR}.
\end{example}

\begin{example}
The quantizing function $\tau: x \mapsto x$ trivially satisfy (HP) and the corresponding change of variables $\cv(x, y) = (y, y^{-1}x)$ corresponds to the quantization which on $G = \R^n$ is sometimes called the right quantization.
\end{example}

\begin{example} \label{ex:fam_sym_fun_Hn}
On the Heisenberg group $\H$, equipped with the canonical homogeneous dilations, we consider the quantizing functions
\begin{align}
	\tau(x) = \Bigl ( \frac{x_1}{2}, \ldots, \frac{x_{2n}}{2}, \frac{x_{2n+1}}{2}+ \sum_{j, k = 1}^{2n} c_{j, k} \, x_j x_k \Bigr ) \label{eq:fam_sym_fun_Hn}
\end{align}
for any choice of $c_{j, k} \in \R$, $j, k = 1, \ldots, 2n$.

Since the canonical dilation weights are given by $v_1 = \ldots = v_{2n} = 1$ and $v_{2n+1} = 2$, each function $\tau$ clearly satisfies (HP). To see that it is also symmetric, by Theorem~\ref{thm:asym_exp_adjoint_tau} below, it suffices to show that $\tau(x)=\tau(x^{-1})x$ holds for all $x \in \H$. Explicitly, this means that
\begin{align}
\begin{split} \label{eq:coef_fam_sym_fun_Hn}
&\Bigl ( \frac{x_1}{2}, \ldots, \frac{x_{2n}}{2}, \frac{x_{2n+1}}{2} + \sum_{j, k = 1}^{2n} c_{j, k} \, x_j x_k \Bigr ) = \\
&\Bigl ( -\frac{x_1}{2} + x_1, \ldots, -\frac{x_{2n}}{2} + x_{2n}, -\frac{x_{2n+1}}{2} + \sum_{j, k = 1}^{2n} c_{j, k} \, x_j x_k + x_{2n+1} - \frac{1}{2} \sum_{l=1}^n \underbrace{(x_l x_{n+l} - x_{n+l} x_l )}_{=0}  \Bigr )
\end{split}
\end{align}
has to hold for all $x \in \H$, which is clearly the case.

In fact, these are the only symmetric quantizing functions $\tau: \H \to \H$ which satisfy (HP): because of the homogeneous structure of $\H$ the function $\tau$ has to be of the form
\begin{align*}
    \tau(x) = \Bigl ( a_1 x_1 + L_1(x),\ldots, a_{2n} x_{2n} + L_{2n}(x), a_{2n+1} x_{2n+1} + \sum_{j, k = 1}^{2n} c_{j, k} \, x_j x_k \Bigr ) 
\end{align*}
for some linear forms $L_l(x) = b_{l, 1} x_1 + \ldots + b_{l, 2n} x_{2n}$ with $b_{l, 1}, \ldots, b_{l, 2n} \in \R$ and some coefficients $c_{j, k} \in \R$, $j, k, l = 1, \ldots, n$. Now, by a comparison of coefficients as in \eqref{eq:coef_fam_sym_fun_Hn}, the condition $\tau(x)=\tau(x^{-1})x$ for all $x \in \H$ requires that
\begin{align*}
\left\{ \begin{array}{ll}
 a_j x_j + L_j(x) = (1-a_j)x_j - L_j(x) & \Rightarrow  a_j = \frac{1}{2}, \, L_j(x) = 0, \\
  a_{2n+1} x_{2n+1} = (1-a_{2n+1})x_{2n+1} & \Rightarrow  a_{2n+1} = \frac{1}{2}\\
\end{array}\right.
\end{align*}
for all $j = 1, \ldots, 2n$. Hence, any such $\tau$ is necessarily of the form~\eqref{eq:fam_sym_fun_Hn}.
\end{example}

\begin{example}
The quantizing function \eqref{eq:onehalf_tau}, which for $G = \R^n$ coincides with $\tau(x) = (\frac{x_1}{2}, \ldots, \frac{x_\dimG}{2})$ and thus yields the Weyl quantization and which for $G = \H$ is an instance of \eqref{eq:fam_sym_fun_Hn}, trivially satisfying (HP) on general graded $G$.
\end{example}

\begin{example}
The quantizing function \eqref{eq:MR_tau} by \cite{MR1}, which for $G = \R^n$ also coincides with $\tau(x) = (\frac{x_1}{2}, \ldots, \frac{x_\dimG}{2})$ and which for $G = \H$ is another instance of \eqref{eq:fam_sym_fun_Hn}, can also easily be seen to satisfy (HP) if one, for example, employs an upper diagonal matrix representation of the nilpotent group $G$ in question.
\end{example}

The following lemmas will be key ingredients in proving the convergence of the main oscillatory integrals in this section. The first lemma establishes a useful relation between the group law in exponential coordinates and the condition (HP). The second lemma characterizes the chain rule of differentiation for compositions involving the quantizing function $\tau$.

\begin{lemma} \label{lem:product_tau}
Let the quantizing function $\tau: G \to G$ satisfy (HP). There exist polynomials $T_1, \ldots, T_\dimG: G \times G \to \R$ such that
\begin{align}
    \tau(x) \tau(y) &= \exp \bigl( T_1(x_1, y_1) X_1 + \ldots T_\dimG(x_1, \ldots x_\dimG, y_1, \ldots y_\dimG) X_\dimG \bigr ) \tau(x y) \\
    &= \bigl( T_1(x, y), \ldots, T_\dimG(x, y) \bigr ) \tau(x y) \label{eq:product_tau}
\end{align}
 for all $x = (x_1, \ldots, x_\dimG)$, $y = (y_1, \ldots, y_\dimG) \in G$. Moreover, each $T_j(x, y)$, $j = 1, \ldots, \dimG$, equals a sum of monomials $x^\alpha y^\beta$, $\alpha, \beta \neq 0$, which only depend on the variables $x_1, \ldots x_j$ and $y_1, \ldots y_j$.
\end{lemma}

\begin{proof}
To show that our claim holds true for any value of $\dim(G) = \dimG \in \N$, we prove it by induction. The proof will make crucial use of the homogeneous polynomials $R_1, \ldots, R_\dimG$ that describe the group multiplication of $G$ in terms of the exponential coordinates. We recall that they are uniquely determined by the identity \eqref{eq:exp_coord} and the choice of Malcev basis $X_1, \ldots, X_\dimG$ of $\Lie{g}$. To show the base case, we recall that for any $x, y \in G$ we have $c^\tau_1(x) = C^\tau_1 x_1$ due to (HP) and $R_1(xy) = x_1 + y_1$ due to \eqref{eq:pol_exp_coord1}, hence
\begin{align*}
    R_1 \bigl ( \tau(x) \tau(y) \bigr ) = C^\tau_1 x_1 +  C^\tau_1 y_1 = c^\tau_1(xy).
\end{align*}
This proves the base case since $T_1 = 0$ due to the Abelian behavior of the first variable.

To show the induction step $j \to j+1$, we assume the existence of polynomials $T_2, \ldots T_j: G \times G \to \R$ such that the first $j$ coordinates of the LHS and the RHS of \eqref{eq:product_tau} coincide for all $x, y \in G$. Now, on the one hand, the $(j+1)$-th coordinate of the LHS satisfies
\begin{align}
    R_{j+1} \Bigl ( c^\tau_1(x), \ldots, &c^\tau_{j+1}(x), c^\tau_1(y), \ldots, c^\tau_{j+1}(y) \Bigr ) \nonumber \\
    &= c^\tau_{j+1}(x) + c^\tau_{j+1}(y) + \tilde{R}_{j+1} \Bigl ( c^\tau_1(x), \ldots, c^\tau_j(y) \Bigr ) \label{eq:LHS_product_tau}
    \end{align}
for the polynomials $\tilde{R}_j$ from \eqref{eq:pol_exp_coord1}. On the other hand, the $(j+1)$-th coordinate of $\tau(xy)$ can be expressed as
\begin{align}
\begin{split} \label{eq:RHS_product_tau}
    c^\tau_{j+1}(xy) &= c^\tau_{j+1} \bigl ( R_1(x_1, y_1), \ldots, R_{j+1}(x_1, \ldots, y_{j+1}) \bigr ) \\
    &= c^\tau_{j+1} \bigl ( x_1 + y_1, \ldots, x_{j+1} + y_{j+1} + \tilde{R}_{j+1}(x_1, \ldots, y_j) \bigr ) \\
    &= c^\tau_{j+1}(x) + c^\tau_{j+1}(y) + \underset{\substack{\alpha, \beta \neq 0, \\ [\alpha] + [\beta] \leq v_{j+1}}}{\overline{\sum}} x^\alpha y^\beta
   \end{split}
\end{align}

Since the polynomial $T_{j+1}$ not only accounts for the difference of the terms \eqref{eq:LHS_product_tau} and \eqref{eq:RHS_product_tau} but also compensates the contribution of $\tilde{R}_{j+1}$, necessitated by representing this difference as the $(j+1)$-th coordinate of the product element given by RHS(\eqref{eq:product_tau}), it is given by
\begin{align*}
    T_{j+1}(x, y) &= \tilde{R}_{j+1} \bigl ( c^\tau_1(x), \ldots, c^\tau_{j+1}(y) \bigr ) \\ &- \underset{\substack{\alpha, \beta \neq 0, \\ [\alpha] + [\beta] \leq v_{j+1}}}{\overline{\sum}} x^\alpha y^\beta - \tilde{R}_{j+1} \bigl ( T_1(x, y), \ldots, T_j(x, y), c^\tau_1(xy), \ldots, c^\tau_j(xy) \bigr ).
\end{align*}
By the induction hypothesis, the polynomials $T_1, \ldots, T_j$ satisfy the identity \eqref{eq:product_tau} in the coordinates $1, \ldots, j$, respectively, hence the polynomial $T_{j+1}$ satisfies the identity \eqref{eq:product_tau} in the $(j+1)$-th coordinate. Hence, by a combination of \eqref{eq:pol_exp_coord1}, (HP) and the induction hypothesis, $T_{j+1}(x,y)$ is necessarily a sum of monomials $x^\alpha y^\beta$, $\alpha, \beta \neq 0$, which only depend on $x_1, \ldots x_{j+1}$ and $y_1, \ldots y_{j+1}$. This proves the induction step and thus the lemma.
\end{proof}

\begin{lemma} \label{lem:pvf}
Let $E$ be a Banach space and let $f: G \to E$. Moreover, let $p: G \times G \to G$ be a smooth function defined by
\begin{align}
    p(y,z)=(p_{1}(y,z),\ldots,p_{n}(y,z)),
\end{align}
where, for every $j=1,\ldots, n$, the $j$-th coefficient $p_{j}(y,z)$ is homogeneous of degree $v_j$ in $y$ and $z$, that is,
\begin{align}
    p_j(y,z)=\sum_{\substack{[\eta_1]+[\eta_2]=
v_j}} C_{\eta_1,\eta_2} y^{\eta_1} z^{\eta_2}.
\end{align}
Then, for every $\alpha \in \N^n_0$ there exist coefficients $c_{\alpha_1, \alpha_2}, c'_{\alpha_1, \alpha_2} \in \R$ such that
\begin{align}
    X^\alpha_{y_1 = y} f \bigl (p(y_1, z) \bigr ) &= \sum\limits_{\substack{|\alpha_3| \leq |\alpha|, \\ [\alpha_3] \geq [\alpha]}} \sum_{[\alpha_1]+[\alpha_2]=[\alpha_3]-[\alpha]} c_{\alpha_1,\alpha_2}q_{\alpha_1}(y)q_{\alpha_2}(z)( X^{\alpha_3} f \bigr )\bigl ( p(y, z) \bigr ), \label{eq:der_p1_p2} \\
    \tilde{X}^\alpha_{y_2 = y} f \bigl (p(y_2, z) \bigr ) &= \sum\limits_{\substack{|\alpha_3| \leq |\alpha|, \\ [\alpha_3] \geq [\alpha]}} \sum_{[\alpha_1]+[\alpha_2]=[\alpha_3]-[\alpha]} c'_{\alpha_1,\alpha_2}q_{\alpha_1}(y)q_{\alpha_2}(z)( \tilde{X}^{\alpha_3} f \bigr )\bigl ( p(y, z) \bigr ), \label{eq:der_p1_p2_2}
\end{align}
where the homogeneous polynomials $q_{\alpha_2}$ are determined by \eqref{eq:can_hom_pol}. Moreover, the same conclusion holds true if we replace $X^\alpha_{y_1=y}$ by $X^\alpha_{z_1=z}$  in \eqref{eq:der_p1_p2}, and $\tilde{X}^\alpha_{y_2=y}$ by $\tilde{X}^\alpha_{z_2=z}$ in \eqref{eq:der_p1_p2_2}.
\end{lemma}

\begin{proof}
Recall that, due to the use of exponential coordinates, every canonical vector field $X_j$, homogeneous of degree $v_j$, is of the form
\begin{align*}
    X_j = \partial_j+\sum_{k=j+1}^n P_{k,j} \, \partial_k,
\end{align*}
with $P_{k,j}$ being a polynomial of homogeneous degree $v_k-v_j$. In particular, due to homogeneity, the Euclidean derivative $\partial_j$, for every $j=1,\ldots, n$, is a first order homogeneous differential operator of homogeneous degree $v_j$, which implies that $\partial^\gamma:=\partial_1^{\gamma_1}\cdots \partial_n^{\gamma_n}$ is a homogeneous differential operator of degree $[\gamma]$ for every multi-index $\gamma\in \mathbb{N}^n_0$.

Now, from the chain rule for multi-variate (Banach space-valued) smooth functions and the homogeneity properties of Euclidean derivatives, we have that for all $\beta \in \N^n_0$ there exist polynomials $C_{\beta, \gamma}: G \times G \to \R$, $|\gamma| \leq |\beta|$, of homogeneous degree equal to $[\gamma]-[\beta]$, such that
\begin{align*}
    \partial^\beta_{y_1 = y} f \bigl (p(y_1, z) \bigr ) = \sum_{\substack{|\gamma| \leq |\beta|\\ [\gamma]\geq[\beta]}} C_{\beta, \gamma}(y, z) \bigl ( \partial^\gamma f \bigr )(p(y,z)).
\end{align*}
By switching back and forth between left-invariant derivatives and Euclidean ones (cf.~\eqref{RelationsVFs1}), we can thus write
\begin{align*}
    X^\alpha_{y_1 = y} f \bigl (p(y_1, z) \bigr ) &= \sum\limits_{\substack{|\beta| \leq |\alpha|, \\ [\beta] \geq [\alpha]}} R_{\alpha, \beta}(y) \partial^\beta_{y_1 = y} f \bigl (p(y_1, z) \bigr ) \\
    &= \sum\limits_{\substack{|\beta| \leq |\alpha|, \\ [\beta] \geq [\alpha]}} R_{\alpha, \beta}(y) \sum_{\substack{|\gamma| \leq |\beta|\\ [\gamma]\geq[\beta]}} C_{\beta, \gamma}(y, z) \bigl ( \partial^\gamma f \bigr )\bigl ( p(y, z) \bigr ) \\
    &= \sum\limits_{\substack{|\beta| \leq |\alpha|, \\ [\beta] \geq [\alpha]}} R_{\alpha, \beta}(y) \sum_{\substack{|\gamma| \leq |\beta|\\ [\gamma]\geq[\beta]}} C_{\beta, \gamma}(y, z) \sum\limits_{\substack{|\alpha_3| \leq |\gamma|, \\ [\alpha_3] \geq [\gamma]}} S_{\gamma,\alpha_3}\bigl ( p(y, z) \bigr ) \bigl ( X^\delta f \bigr )\bigl ( p(y, z) \bigr )
\end{align*}
for some polynomials $R_{\alpha, \beta}$ and $S_{\gamma, \alpha_3}$ of homogeneous degree $[\beta] - [\alpha]$ and $[\alpha_3] - [\gamma]$, respectively. 
Since
\begin{align}
  C_{\beta, \gamma}(y, z)&=\sum_{[\alpha_1]+[\alpha_2]=[\gamma]-[\beta]}c_{\alpha_1,\alpha_2}q_{\alpha_1}(y)q_{\alpha_2}(z) \\
 S_{\alpha_3, \delta}\bigl ( p(y, z) \bigr )&=\sum_{[\beta_1]+[\beta_2]=[\alpha_3]-[\gamma]}c_{\beta_1,\beta_2}q_{\beta_1}(y)q_{\beta_2}(z)\\
 R_{\alpha, \beta}(y)
q_{\alpha_1}(y)q_{\alpha_2}(z)
   q_{\beta_1}(y)q_{\beta_2}(z)&=q_{\beta-\alpha+\alpha_1+\beta_1}(y)q_{\alpha_2+\beta_2}(z),
\end{align}
with $[\beta-\alpha+\alpha_1+\beta_1]+[\alpha_2+\beta_2]=[\alpha_3]-[\alpha]$, we get
\begin{align}
   X^\alpha_{y_1 = y} f \bigl (p(y_1, z) \bigr ) =& \sum\limits_{\substack{|\beta| \leq |\alpha|, \\ [\beta] \geq [\alpha]}} 
   \sum_{\substack{|\gamma| \leq |\beta|\\ [\gamma]\geq[\beta]}}\sum\limits_{\substack{|\alpha_3| \leq |\gamma|, \\ [\alpha_3] \geq [\gamma]}}
   \sum_{[\alpha_1]+[\alpha_2]=[\gamma]-[\beta]}
    \sum_{[\beta_1]+[\beta_2]=[\alpha_3]-[\gamma]}c_{\alpha_1,\alpha_2}c_{\beta_1,\beta_2}\\
   &  R_{\alpha, \beta}(y)
q_{\alpha_1}(y)q_{\alpha_2}(z)
   q_{\beta_1}(y)q_{\beta_2}(z) \bigl ( X^{\alpha_3} f \bigr )\bigl ( p(y, z) \bigr )\\
   =& \sum\limits_{\substack{|\alpha_3| \leq |\alpha|, \\ [\alpha_3] \geq [\alpha]}} \sum_{[\alpha_1]+[\alpha_2]=[\alpha_3]-[\alpha]} \tilde{c}_{\alpha_1,\alpha_2}\, q_{\alpha_1}(y)q_{\alpha_2}(z)( X^{\alpha_3} f \bigr )\bigl ( p(y, z) \bigr ),
\end{align}
which concludes the proof of \eqref{eq:der_p1_p2}.

The identity \eqref{eq:der_p1_p2_2} follows by the same argument if one additionally employs the relations between right-invariant and Euclidean derivatives in \eqref{RelationsVFs1}.
\end{proof}

\subsection{Change of quantization} \label{sub:cont_ChQ}

In this subsection we show that one can switch between any $\tau$-quantization and the Kohn-Nirenberg quantization, and therefore between any two $\tau$-quantizations, without losing the essential properties of the Kohn-Nirenberg calculus for the range of symbol classes under consideration.

First, however, we need an auxiliary proposition on the continuity on the Schwartz space $\SC(G)$ of $\tau$-quantized operators. This result extends \cite[Thm.~5.2.15]{FR} to $\tau$-quantized operators with symbols in H\"{o}rmander classes for $0 < \rho$, $\delta < \frac{1}{v_n}$. Note that these restrictions can be omitted for $\tau = e_G$, i.e., in the setting of \cite{FR}. While at least $\delta < 1$ is necessary\footnote{Confer, e.g., \cite[Thm.~2.21]{F1} for $\delta < 1$ for the Weyl quantization on $\R^n$, in which case $v_1 = \ldots = v_n = 1$.} whenever $\tau \neq e_G$, the condition $0 < \rho$ serves us to control the derivatives of $\tau$ at infinity (cf.~\cite[Thm.~5.4.1]{FR}). 
Explicit computations in the case of $G = \H$ equipped with the canonical dilations show that $\delta < \frac{1}{v_n} = \frac{1}{2}$ can be relaxed to $\delta < 1$. Whether this can be improved for general graded groups remains an open question.

\begin{proposition} \label{prop:cont_on_SC}
Let $m \in \R$, $0 \leq \delta \leq \rho \leq 1$, $0 < \rho$, $\delta < \frac{1}{v_n}$, and let $\sigma \in S^m_{\rho, \delta}(G)$. Then for every $f \in \SC(G)$ the quantization formula~\eqref{eq:tau_quant_G} defines an element in $\SC(G)$ given by the absolutely convergent integrals
\begin{align}
    \Optau(\sigma)f(x) &= \int_{\Ghat} \Tr \Bigl ( \int_G \pi(y^{-1} x) \sigma(x\tau(y^{-1} x)^{-1}, \pi) f(y) \, dy \Bigr ) \, d\mu(\pi) \label{eq:tau_quant_G_SC_1} \\
    &= \int_G \kappa_\sigma\bigl ( x \tau(y^{-1} x)^{-1}, y^{-1} x \bigr ) f(y) \, dy. \label{eq:tau_quant_G_SC_2}
\end{align}
Moreover, the linear operator $\Optau(\sigma): \SC(G) \to \SC(G)$ is continuous:
for each seminorm $\| \, . \, \|_{\SC, N_1}$, $N_1 \in \N$, there exist a constant $C > 0$, a seminorm $\| \, . \, \|_{\SC, N_2}$,$N_2 \in \N$, and a seminorm $\|\, . \, \|_{S^{m}_{\rho,\delta}a,b}$, $a, b \in \N$, such that
\begin{align} \label{eq:tau_quant_G_on_SC}
	\| \Optau(\sigma)f \|_{\SC, N_1} \leq C \| \sigma \|_{S^{m}_{\rho,\delta}a,b} \, \| f \|_{\SC, N_2}.
\end{align}
\end{proposition}

\begin{proof}
Thus, let $m \in \R$, $0 \leq \delta \leq \rho \leq 1$, $\delta < \frac{1}{v_\dimG}$, and let $\sigma \in S^m_{\rho, \delta}(G)$. By Proposition~\ref{prop:L2_int_ass_ker}, we have that
\begin{align*}
	y \mapsto \tilde{\kappa}(x, y) := ( (I+\widetilde{\RO})^{-N} \kappa_\sigma(x, \, . \,) \bigr )(y) = \F^{-1} \Bigl ( \bigl \{ \sigma(x, \pi) \pi(I + \RO)^{-N} \bigr \}_{\pi \in \Ghat} \Bigr )(y),
\end{align*}
is in $L^2(G)$ for all $x \in G$ whenever $N \in \N$ satisfies $m + \hdim/2 < N \hdeg$, and $\| \tilde{\kappa}(x, \, . \,) \|_{L^2(G)}$ is uniformly bounded in $x$ due to \eqref{eq:L^2_bd_ass_ker} with
\begin{align}
	\sup_{x \in G} \| \tilde{\kappa}(x, \, . \,) \|_{L^2(G)} \leq C_1 \sup_{(x,\pi) \in G \times \Ghat} \| \sigma(x, \pi) \pi(I + \RO)^{-N} \|_{\RS} \leq C_2 \| \sigma \|_{S^m_{\rho, \delta}, a, b} \label{eq:L2_modif_ass_ker_1_cont_on_SC}
\end{align}
for some $a, b \in \N_0$ and constants $C_1, C_2 > 0$ independent of $\sigma$ and $x$. Since $\kappa_\sigma, \Ker_\sigma \in \TD(G \times G)$ whenever $\sigma \in S^m_{\rho, \delta}(G)$ (cf.~Subsection~\ref{subs:ass_ker}), for any $f \in \SC(G)$ the integral \eqref{eq:tau_quant_G_SC_2} defines a tempered distribution $g := \Optau(\sigma)f$. To see that $g \in \SC(G)$, we fix some $N_0 \in \N$ and use integration by parts (in the distributional sense) to rewrite \eqref{eq:tau_quant_G_SC_2} as
\begin{align}
	g(x) &= \int_G f(y) \bigl [ (I+\widetilde{\RO})^{N_0} (I+\widetilde{\RO})^{-N_0} \bigr ]_{y_0 = y} \kappa_\sigma \bigl ( x \tau(y^{-1} x)^{-1}, y_0^{-1} x \bigr ) \, dy \nonumber \\
	&= \int_G \underset{\substack{[\alpha] \leq \hdeg N_0, \\ [\beta] \geq [\alpha], \\ |\beta| \leq |\alpha|}}{\overline{\sum}} Q_{\alpha, \beta}(y) \underset{[\beta_1] + [\beta_2] = [\beta]}{\overline{\sum}} X^{\beta_1}_{y_1 = y} f(y_1) \\
    &\hspace{130pt} \times X^{\beta_2}_{y_2 = y} (I+\widetilde{\RO})^{-N_0}_{y_0 = y} \kappa_\sigma\bigl ( x \tau(y_2^{-1} x)^{-1}, y_0^{-1} x \bigr ) \, dy,
\end{align}
where the homogeneous polynomials $Q_{\alpha, \beta}$ are determined by \eqref{RelationsVFs1}. Before we estimate $g(x)$ by an application of the Cauchy-Schwarz inequality for $L^2(G, dy)$, we observe that, due to (HP), we can use Lemma \ref{lem:pvf} and write
\begin{align}
    &X^{\beta_2}_{y_2 = y}(I+\widetilde{\RO})^{-N_0}_{y_0 = y} \kappa_\sigma\bigl ( x \tau(y_2^{-1} x)^{-1}, y_0^{-1} x )\\
    &=\underset{\substack{[\beta_{2,3}]\geq [\beta_2]\\ |\beta_{2,3}|\leq |\beta_{2}|\\
    [\beta_{2,1}]+[\beta_{2,2}]=[\beta_{2,3}]-[\beta_2]}}{\overline{\sum}}q_{\beta_{2,1}}(y)q_{\beta_{2,2}}(x)X^{\beta_{2,3}}_{x_1=x\tau(y^{-1}x)^{-1}}(I+\widetilde{\RO})^{-N_0}_{y_0 = y} \kappa_\sigma\bigl ( x_1), y_0^{-1} x \bigr )\\
    &=\underset{\substack{[\beta_{2,3}]\geq [\beta_2]\\ |\beta_{2,3}|\leq |\beta_{2}|\\
    [\beta_{2,1}]+[\beta_{2,2}]=[\beta_{2,3}]-[\beta_2]}}{\overline{\sum}}
    \underset{\substack{[\beta_{2,2}']+[\beta_{2,2}'']=[\beta_{2,2}]}}{\overline{\sum}}
    q_{\beta_{2,1}}(y)q_{\beta_{2,2}'}(y)q_{\beta_{2,2}''}(y^{-1}x)\\
    &\hspace{190pt} \times X^{\beta_{2,3}}_{x_1=x\tau(y^{-1}x)^{-1}}(I+\widetilde{\RO})^{-N_0}_{y_0 = y} \kappa_\sigma\bigl ( x_1, y_0^{-1} x \bigr ),
\end{align}
so that the polynomials $q_{\beta_{2,1}+\beta_{2,2}'}(y)$ in these summands can be grouped with the Schwartz function $X^{\beta_1}f$ and the polynomials $q_{\beta_{2,2}''}(y^{-1}x)$ with the associated kernel. 
Since $\rho > 0$, the kernel $\kappa_\sigma$ decays like a Schwartz function away from the origin, due to Proposition~\ref{prop:ker_est_at_infty}, so the factors $q_{\beta_{2,2}''}(y^{-1}x)$ do not affect our eventual estimate, while close to the origin they at most improve the $L^2$-convergence of
\begin{align}
	y \mapsto X^{\beta_{2,3}}_{x_1=x \tau(y_2^{-1} x)^{-1}} (I+\widetilde{\RO})^{-N_0}_{y_0 = y} \kappa_\sigma \bigl ( x_1, y_0^{-1} x \bigr ). \label{eq:L2_modif_ass_ker_2_cont_on_SC}
\end{align}
So we can ignore these polynomial factors without loss of generality and focus on the fact that, by \eqref{eq:L2_modif_ass_ker_1_cont_on_SC}, the $L^2$-norm of \eqref{eq:L2_modif_ass_ker_2_cont_on_SC} is uniformly bounded in $x \in G$ whenever we choose $N_0 \in N$ such that $m + \delta [\beta_{2,3}] + \hdim/2 \leq m + \delta \hdeg N_0 v_\dimG + \hdim/2  < N_0 \hdeg$. Such a choice is in fact possible due to $\delta < \frac{1}{v_n}$. We can now estimate
\begin{align*}
	|g(x)| &\leq \underset{\substack{[\alpha] \leq \hdeg N_0, \\ [\beta] \geq [\alpha], \\ |\beta| \leq |\alpha|}}{\overline{\sum}} \, 
 \underset{\substack{[\beta_{2,3}]\geq [\beta_2]\\ |\beta_{2,3}|\leq |\beta_{2}|\\
    [\beta_{2,1}]+[\beta_{2,2}]=[\beta_{2,3}]-[\beta_2]}}{\overline{\sum}}
    \underset{\substack{[\beta_{2,2}']+[\beta_{2,2}'']=[\beta_{2,2}]}}{\overline{\sum}}
 \int_G |q_{\beta_{2,1}+\beta'_{2,2}}(y)||Q_{\alpha, \beta}(y)| \\
	&\hspace{60pt} \times |(X^{\beta_1}f)(y)| |q_{\beta_{2,2}''}(y^{-1}x)| \sup_{x_1 \in G} |X^{\beta_{2,3}}_{x_1} (I+\widetilde{\RO})^{-N_0}_{y_0 = y} \kappa_\sigma\bigl ( x_1, y_0^{-1} x \bigr )| \, dy \\
	&\leq C \| f \|_{\SC, N_0} \, \| \sigma \|_{S^{m}_{\rho,\delta}a,b},
\end{align*}
for some sufficiently large $N_0 \in \N$, where for each of the integrals in the sum we have used \eqref{eq:L2_modif_ass_ker_1_cont_on_SC} to estimate $\sup_{x_1 \in G} |X^{\beta_{2,3}}_{x_1} \tilde{\kappa}(x_1, y)|$ and the Cauchy-Schwarz inequality for $L^2(G, dy)$. This in turn gives \eqref{eq:tau_quant_G_on_SC} for $N_1 = 0$.

To show that $x^{\alpha_o} (X^{\beta_o}g)(x)$ is defined for all $x \in G$ and that $x \mapsto x^{\alpha_o} (X^{\beta_o}g)(x) \in L^2(G)$ for all $\alpha_o, \beta_o \in \N^\dimG_0$ together with an estimate of the type \eqref{eq:tau_quant_G_on_SC}, we use dominated convergence to differentiate under the integral sign. Next, we split $x^{\alpha_o}$ as 
\begin{align}
    x^{\alpha_o}=\underset{[\alpha_{o,1}]+[\alpha_{o,2}]=[\alpha_{o}]}{\overline{\sum}}q_{\alpha_{o,1}}(y)q_{\alpha_{o,2}}(y^{-1}x)
\end{align}
in order to apply essentially the same argument as above, with $N_0 \in N$ chosen sufficiently large in order to compensate additional polynomial factors in the integral when $\alpha_o$ or $\beta_o$ are different from zero. This completes the proof.
\end{proof}

\begin{assumptions}
In what follows we will often assume $0 \leq \delta < \frac{\rho}{v_n} 
\leq \frac{1}{v_n}$, which in turn implies the conditions $0 < \rho$ and $\delta < \frac{1}{v_n}$ in Propostion~\ref{prop:cont_on_SC}. With the exception of $G = \R^n$ equipped with isotropic dilations, for which $v_n = 1$, this condition is stricter than $0 \leq \delta < \rho \leq 1$. The restriction $\delta < \rho$ is always assumed because of the authors' personal interest in asymptotic extensions for the symbols of adjoints, composite operators, etc, and it cannot be relaxed even in the case of $G = \R^n$, while $\delta < \frac{1}{v_n}$ and, in particular, $\delta < \frac{\rho}{v_n}$ are needed to make the proofs work for general graded $G$. Note that for the Kohn-Nirenberg quantization it suffices to assume $0 \leq \delta < \rho \leq 1$ as in \cite{FR} since the interaction of the quantizing function $\tau = e_G$ with the calculus developed here becomes trivial. The proofs given in this paper simplify in this case and recover the well-known results due to \cite{FR}.
\end{assumptions}

\begin{theorem} \label{thm:asym_exp_ch_qu}
Let $T$ be a continuous linear operator from $\SC(G)$ to $\SC(G)$ and let $m \in \R$ and $0 \leq \delta < \min\{\rho,\frac{1}{v_n}\} \leq 1$.  Let $\tau: G \to G$ be a quantizing function different from the constant function $\tau = e_G$ which satisfies (HP). 
Then, if $\sigma$ and $\sigma_\tau$ are two symbols such that $T=\Optau(\sigma_\tau)=\Op(\sigma)$, $\sigma\in S^{m}_{\rho,\delta}(G)$ if and only if $\sigma_\tau\in S^{m}_{\rho,\delta}(G)$.

The map $\sigma \mapsto \sigma_\tau$ is a Fr\'{e}chet space isomorphism from $S^m_{\rho, \delta}(G)$ onto itself, and the symbols are related to one another by the asymptotic expansions
\begin{align}
	\sigma_\tau &\sim \sum_{j=0}^\infty\Bigl(\sum_{[\alpha]= j} \sum_{[\alpha'] = [\alpha]} c^{\tau, \mathrm{KN}}_{\alpha', \alpha} \, \Delta^{\alpha'} X^\alpha \sigma\Bigr), \label{eq:asym_exp_tau_KN} \\
	\sigma &\sim \sum_{j=0}^\infty \Bigl(\sum_{[\beta] =j } \sum_{[\beta'] = [\beta]} c^{\mathrm{KN}, \tau}_{\beta', \beta} \, \Delta^{\beta'} X^\beta \sigma_\tau\Bigr), \label{eq:asym_exp_KN_tau}
\end{align}
in the sense that for given $M, N \in \mathbb{N}_0$
\begin{align*}
	R^{\tau, \mathrm{KN}}_M &:= \sigma_\tau - \sum_{[\alpha]\leq M} \sum_{[\alpha'] = [\alpha]} c^{\tau, \mathrm{KN}}_{\alpha', \alpha} \, \Delta^{\alpha'} X^\alpha \sigma\in S^{m-(\rho-\delta)(M+1)}_{\rho,\delta}(G), \\
	R^{\mathrm{KN}, \tau}_N &:= \sigma - \sum_{[\beta] \leq N } \sum_{[\beta'] = [\beta]} c^{\mathrm{KN}, \tau}_{\beta', \beta} \, \Delta^{\beta'} X^\beta \sigma_\tau \in S^{m-(\rho-\delta)(N+1)}_{\rho,\delta}(G).
\end{align*}
Moreover, the coefficients $c^{\tau, \mathrm{KN}}_{\alpha', \alpha}, c^{\mathrm{KN}, \tau}_{\beta', \beta} \in \R$
are uniquely determined by the equations
\begin{align}
	 q_\alpha \bigl ( \tau(y) \bigr ) &= \sum_{[\alpha'] = [\alpha]} c^{\tau, \mathrm{KN}}_{\alpha', \alpha} \, \tilde{q}_{\alpha'}(y), \label{eq:coeff_W_KN} \\
	 q_\beta \bigl ( \tau(y)^{-1} \bigr ) &= \sum_{[\beta'] = [\beta]} c^{\mathrm{KN}, \tau}_{\beta', \beta} \, \tilde{q}_{\beta'}(y).
\end{align}
\end{theorem}

\begin{proof}
Our proof crucially relies on the relation between the associated kernels. We will see that it suffices to prove our statement for the dense subset of $\sigma \in S^{-\infty}(G)$ with $\kappa_\sigma \in \SC(G\times G)$, and conversely for $\sigma_\tau \in S^{-\infty}(G)$ with $\kappa_{\sigma_\tau} \in \SC(G\times G)$, and extend the obtained seminorm estimates to $\sigma \in S^m_{\rho, \delta}(G)$, and to $\sigma_\tau \in S^m_{\rho, \delta}(G)$, respectively.
Thus, let $m \in \R$, $0 \leq \delta < \min\{\frac{1}{v_n},\rho \}$ and let $\tau: G \to G$ be a quantizing function which satisfies (HP). To show the necessary condition and the asymptotic expansion~\eqref{eq:asym_exp_tau_KN}, let us therefore suppose that $\sigma \in S^{-\infty}(G)$ with associated kernel satisfying $\kappa_\sigma\in \SC(G\times G)$. Then also $\kappa_{\sigma_\tau} \in \SC(G\times G)$ due to (HP) and the relation
\begin{align} \label{eq:rel_ass_ker_tau_KN}
	\kappa_{\sigma_\tau}(x, y) = \kappa_\sigma \bigl( x\tau(y), y \bigr),
\end{align}
which we deduce from the identity
\begin{align*}
	(Tf)(x) &= (\Op(\sigma)f)(x)=\iint\limits_{\widehat{G} \times G} \Tr \bigl ( \pi(y^{-1} x) \sigma(x, \pi) f(y) \bigr ) \, dy \, d\mu(\pi) \\
	&= (\Optau(\sigma_\tau)f)(x)= \iint\limits_{\widehat{G} \times G} \Tr \bigl ( \pi(y^{-1} x) \sigma_\tau(x\tau(y^{-1}x)^{-1}, \pi) f(y) \bigr ) \, dy \, d\mu(\pi).
\end{align*}
and an application of the pullback $(CV^\tau)^{-1}$ defined by \eqref{eq:Ker_kappa} to the associated kernels $\kappa_\sigma$ and $\kappa_{\sigma_\tau}$. Moreover, the corresponding symbol
\begin{align}
    \sigma_\tau(x, \pi) = \F_G \bigl ( \kappa_{\sigma_\tau}(x, \, \cdot \,) \bigr )(\pi) \label{eq:tau_symbol}
\end{align}
is an element of $S^{-\infty}(G)$ by~\cite[Lem.~5.5.20]{FR}.
As in the proofs of Theorems \ref{thm:asym_exp_adjoint_tau} and \ref{thm:asym_exp_comp_tau}, we employ a Taylor expansion to reveal the asymptotic expansion of $\sigma_\tau$ in terms of $\sigma$. To do so, we observe that the condition~(HP) implies that for any $\alpha \in \N^n_0$ there exist coeffcients $c_{\alpha', \alpha} \in \R$, $\alpha' \in \N^n_0$, $[\alpha'] = [\alpha]$, such that
\begin{align*}
	 (q_\alpha \circ \tau)(y) = \sum_{[\alpha'] = [\alpha]} c^{\tau, \mathrm{KN}}_{\alpha', \alpha} \, \tilde{q}_{\alpha'}(y),
\end{align*}
for the homogeneous polynomials $q_\alpha$ and $\tilde{q}_{\alpha'} = x \mapsto q_{\alpha'}(x^{-1})$ determined by \eqref{eq:can_hom_pol} for all $\alpha, \alpha' \in \N^n_0$. So, we may expand
\begin{align} \label{eq:Taylor_exp_ch_qu}
\begin{split}
	\kappa_{\sigma_\tau}(x, y) &= \kappa_\sigma(x\tau(y),y) \underbrace{\sum_{[\alpha] \leq M} \sum_{[\alpha'] = [\alpha]} c^{\tau, \mathrm{KN}}_{\alpha', \alpha} \, \tilde{q}_{\alpha'}(y) X^\alpha_{x_1 = x} \kappa_\sigma(x_1, y)}_{:= \kappa_{T_0}(x, y)} \\
	&\hspace{30pt} + \sum_{M+1 \leq [\alpha] \leq M_0} \sum_{[\alpha'] = [\alpha]} c^{\tau, \mathrm{KN}}_{\alpha', \alpha} \, \tilde{q}_{\alpha'}(y) X^\alpha_{x_1 = x} \kappa_\sigma(x_1, y) + R^{\kappa_\sigma(\, . \,, y)}_{x, M_0} \bigl (\tau(y) \bigr )
\end{split}
\end{align}
for some $M_0 \geq M + 1$, where the Taylor remainder is controlled by \eqref{eq:est_Taylor_rem}. Since
\begin{align}
	\Delta^{\alpha'} X^\alpha \sigma \in S^{m-(\rho - \delta)[\alpha]}_{\rho, \delta}(G), \label{eq:T_0_T_1}
\end{align}
we immediately obtain that $T_0 := \F_G (\kappa_{T_0})$ yields the asymptotic expansion \eqref{eq:asym_exp_tau_KN} provided that
\begin{align*}
	T_1(x, \pi) := \int_G R^{\kappa_\sigma(\, . \,, y)}_{x, M_0} \bigl (\tau(y) \bigr ) \pi(y)^* \, dy
\end{align*}
belongs to the symbol class $S^{m-(\rho - \delta)(M+1)}_{\rho, \delta}(G)$.
To prove this, we will show that
\begin{align}
	\sup_{(x, \pi) \in G \times \Ghat} \, \Bigl \| (\Delta^{\alpha_0} X^{\beta_0} T_1)(x, \pi) (I + \pi(\RO))^{-\frac{m - (\rho - \delta)(M+1)-\rho [\alpha_0] + \delta [\beta_0]}{\hdeg}} \Bigr \|_{\mathcal{L}(\RS)} < \infty \label{eq:sn_T1_tau_KN_1}
\end{align}
for arbitrary, but fixed $\alpha_0, \beta_0 \in \N^n_0$. Thus, let $\alpha_0, \beta_0 \in \N^n_0$ and let $M_1$ be the least non-negative integer such that $- m + (\rho - \delta)(M+1) + \rho [\alpha_0] - \delta [\beta_0] - \hdeg M_1 \leq 0$. Since
\begin{align*}
	&\sup_{\pi \in \Ghat} \, \Bigl \| (I + \pi(\RO))^{-\frac{m - (\rho - \delta)(M+1)-\rho [\alpha_0] + \delta [\beta_0]}{\hdeg} - M_1} \Bigr \|_{\mathcal{L}(\RS)} \\
	&\hspace{100pt} = \Bigl \| (I + \RO)^{-\frac{m - (\rho - \delta)(M+1)-\rho \alpha_0 + \delta \beta_0}{\hdeg} - M_1} \Bigr \|_{L^2 \to L^2} =: C < \infty,
\end{align*}
we can now bound \eqref{eq:sn_T1_tau_KN_1} by estimating
\begin{align*}
	\Bigl \| (&\Delta^{\alpha_0} X^{\beta_0} T_1)(x, \pi) (I + \pi(\RO))^{-\frac{m - (\rho - \delta)(M+1)-\rho \alpha_0 + \delta \beta_0}{\hdeg}} \Bigr \|_{\mathcal{L}(\RS)} \\
	&\leq C \Bigl \| (\Delta^{\alpha_0} X^{\beta_0} T_1)(x, \pi) (I + \pi(\RO))^{M_1} \Bigr \|_{\mathcal{L}(\RS)} \\
	&= C \Bigl \| \int_G X^{\beta_0}_{x_1 = x} R^{\kappa_\sigma(\, . \,, y)}_{x_1, M_0} \bigl (\tau(y) \bigr ) \tilde{q}_{\alpha_0}(y) \pi(y)^* (I + \pi(\RO))^{M_1} \, dy \Bigr \|_{\mathcal{L}(\RS)} \\
	&= C \Bigl \| \underset{[\beta_1] + [\beta_2] + [\beta_3] \leq \hdeg M_1}{\overline{\sum}} \int_G \tilde{X}^{\beta_1}_{y_1 = y} \tilde{X}^{\beta_2}_{y_2 = y} X^{\beta_0}_{x_1 = x} R^{\kappa_\sigma(\, . \,, y_1)}_{x_1, M_0} \bigl (\tau(y_2) \bigr ) (\tilde{X}^{\beta_3}\tilde{q}_{\alpha_0})(y) \pi(y)^* \, dy \Bigr \|_{\mathcal{L}(\RS)} \\
	&= C \Bigl \| \underset{[\beta_1] + [\beta_2] + [\beta_3] \leq \hdeg M_1}{\overline{\sum}} \int_G \tilde{X}^{\beta_2}_{y_2 = y} R^{X^{\beta_0}_{x_1} \tilde{X}^{\beta_1}_{y_1 = y} \kappa_\sigma(x_1 , y_1)}_{0, M_0} \bigl (\tau(y_2) \bigr ) (\tilde{X}^{\beta_3}\tilde{q}_{\alpha_0})(y) \pi(y)^* \, dy \Bigr \|_{\mathcal{L}(\RS)}.
\end{align*}
Because of $\| \pi(y)^* \|_{\mathcal{L}(\RS)} = 1$, so we can rewrite
\begin{align*}
	\tilde{X}^{\beta_2}_{y_2 = y} &R^{X^{\beta_0}_{x_1} \tilde{X}^{\beta_1}_{y_1 = y} \kappa_\sigma(x_1 , y_1)}_{0, M} \bigl (\tau(y_2) \bigr ) \\
	&\hspace{30pt} = \underset{\substack{|\beta'_2| \leq |\beta_2| \\ [\beta'_2] \geq [\beta_2]}}{\sum} Q_{\beta'_2, \beta_2}(y) X^{\beta'_2}_{y_2 = y} R^{X^{\beta_0}_{x_1} \tilde{X}^{\beta_1}_{y_1 = y} \kappa_\sigma(x_1 , y_1)}_{0, M_0} \bigl (\tau(y_2) \bigr ) \\
	&\hspace{30pt} = \underset{\substack{|\beta'_2| \leq |\beta_2| \\ [\beta'_2] \geq [\beta_2]}}{\sum} Q_{\beta'_2, \beta_2}(y) \underset{\substack{[\beta'_2] \leq [\beta'''_2] \leq v_\dimG |\beta'_2|\\ [\beta_2''']-[\beta_2'']=[\beta_2']}}{\overline{\sum}} q_{\beta''_2}(y) X^{\beta'''_2}_{y_2 = \tau(y)} R^{X^{\beta_0}_{x_1} \tilde{X}^{\beta_1}_{y_1 = y} \kappa_\sigma(x_1 , y_1)}_{0, M_0} (y_2) \\
	&\hspace{30pt} = \underset{[\beta'''_2] \leq v_\dimG [\beta_2]}{\overline{\sum}} q_{\beta'''_2-\beta_2}(y) X^{\beta'''_2}_{y_2 = \tau(y)} R^{X^{\beta_0}_{x_1} \tilde{X}^{\beta_1}_{y_1 = y} \kappa_\sigma(x_1 , y_1)}_{0, M_0} (y_2),
\end{align*}
where in the second step we have used Lemma~\ref{lem:pvf}. We further estimate
\begin{align*}
	\Bigl \| &(\Delta^{\alpha_0} X^{\beta_0} T_1)(x, \pi) (I + \pi(\RO))^{-\frac{m - (\rho - \delta)(M+1)-\rho \alpha_0 + \delta \beta_0}{\hdeg}} \Bigr \|_{\mathcal{L}(\RS)} \\
	&\leq C \underset{\substack{[\beta_1] + [\beta_2] + [\beta_3] \leq \hdeg M_1 \\ [\beta'''_2] \leq v_\dimG [\beta_2]}}{\overline{\sum}}
	\int_G |(\tilde{X}^{\beta_3}\tilde{q}_{\alpha_0})(y)| |q_{\beta'''_2-\beta_2}(y)| | X^{\beta'''_2}_{y_2 = \tau(y)} R^{X^{\beta_0}_{x_1} \tilde{X}^{\beta_1}_{y_1 = y} \kappa_\sigma(x_1 , y_1)}_{0, M_0} (y_2)| \, dy \\
	&\leq C \underset{\substack{[\beta_1] + [\beta_2] + [\beta_3] \leq \hdeg M_1 \\ [\beta'''_2] \leq v_\dimG [\beta_2]}}{\overline{\sum}} \int_G |q_{\beta'''_2-\beta_2}(y)| \underset{\substack{|\gamma| \leq \lceil (M_0 - [\beta'''_2]) \rfloor_+ + 1 \\ [\gamma] \geq \lceil (M_0 - [\beta'''_2]) \rfloor_+}}{\sum} |\tau(y)|^{[\gamma]} \\
 &\hspace{210pt} \times \sup_{x_1 \in G} |X^\gamma_{x_1} X^{\beta'''_2}_{x_2 = x_1} \tilde{X}^{\beta_1}_{y_1 = y} \kappa_\sigma(x_2, y_1)| \, dy.
\end{align*}
Away from $y = 0$, this integral converges absolutely since $X^\gamma_{x_1} X^{\beta'''_2}_{x_2 = x_1} \tilde{X}^{\beta_1}_{y_1 = y} \kappa_\sigma$ has Schwartz decay and $|\tau(y)|$ is of polynomial growth due to (HP). To see that it also converges in a neighborhood of the origin, we notice that there exist $a, b \in \N_0$ and a constant $C_0 > 0$ such that
\begin{align*}
	\sup_{x_1 \in G} |X^\gamma_{x_1} X^{\beta'''_2}_{x_2 = x_1} \tilde{X}^{\beta_1}_{y_1 = y} \kappa_\sigma( x_2, y_1)| \leq C_0 \| \sigma \|_{S^m_{\rho, \delta}, a, b} |y|^{-\frac{m + [\beta_1] + \delta([\gamma] + [\beta'''_2])}{\rho}}.
\end{align*}
Hence, the integral converges in this neighborhood, since we can bound $|\tau(y)| \lesssim |y|$ and pick $M_0 \in \N$, $M_0>\nu M_1\geq [\beta_2''']$, large enough so that
\begin{align*}
	\rho \bigl ( Q + [\gamma] \bigr ) \geq \rho \bigl ( Q + M_0 \bigr )> Q + m + \hdeg M_1 (1 + v_\dimG) + \delta M_0>Q + m + \hdeg M_1 (1 + v_\dimG) + \delta [\gamma],
\end{align*}
because the latter is greater or equal $Q + m + [\beta_1] + \delta([\gamma] + [\beta'''_2])$. Note that such a choice of $M_0$ is possible due to $\frac{\delta}{\rho} < 1$. So, we obtain
\begin{align}
    \sup_{(x, \pi) \in G \times \Ghat} \, \Bigl \| (\Delta^{\alpha_0} X^{\beta_0} T_1)(x, \pi) (I + \pi(\RO))^{-\frac{m - (\rho - \delta)(M+1)-\rho \alpha_0 + \delta \beta_0}{\hdeg}} \Bigr \|_{\mathcal{L}(\RS)} \lesssim  \| \sigma \|_{S^m_{\rho, \delta}, a, b}, \label{eq:sn_T1_tau_KN_2}
\end{align}
and, in view of \eqref{eq:T_0_T_1}, also
\begin{align}
    \| \sigma_\tau \|_{S^m_{\rho, \delta}, [\alpha_0], [\beta_0]} \lesssim  \| \sigma \|_{S^m_{\rho, \delta}, a, b}. \label{eq:sn_T1_tau_KN_3}
\end{align}
This completes the proof of \eqref{eq:asym_exp_tau_KN}.

To prove the sufficient condition and the asymptotic expansion~\eqref{eq:asym_exp_KN_tau}, our line of arguments applies essentially verbatim. From the identity~\eqref{eq:rel_ass_ker_tau_KN} we immediately obtain the converse identity
\begin{align}
	\kappa_\sigma \bigl(x, y) = \kappa_{\sigma_\tau}( x\tau(y)^{-1}, y \bigr). \label{eq:rel_ass_ker_KN_tau}
\end{align}
Since the condition (HP) implies the condition (CV), it follows that $\kappa_\sigma \in \SC(G \times G)$ whenever $\kappa_{\sigma_\tau} \in \SC(G \times G)$, and we may follow precisely the same steps as above to show \eqref{eq:asym_exp_KN_tau} and 
\begin{align}
    \| \sigma \|_{S^m_{\rho, \delta}, [\alpha_0], [\beta_0]} \lesssim  \| \sigma_\tau \|_{S^m_{\rho, \delta}, a, b}\label{eq:sn_T1_KN_tau_3}.
\end{align}
This proves the converse direction.

To pass to general symbols $\sigma \in S^m_{\rho, \delta}(G)$, we employ an argument based on the approximation of $\sigma$ by a net of smoothing symbols. By \cite[Lem.~5.4.11]{FR} and the equivalence of seminorms \cite[Thm.~5.5.20]{FR}, there exists a net $\{ \sigma_{\varepsilon} \}_{\varepsilon \in (0, 1)} \subseteq S^{-\infty}(G)$, with $\kappa_{\sigma_{\varepsilon}}\in \SC(G\times G)$, such that for every $\theta > 0$ and every seminorm $\| \, \cdot \, \|_{S^m_{\rho, \delta}, a, b}$ there exist a constant $C = C(a, b, c, m, \theta, \rho, \delta) > 0$ and $a', b' \in \N_0$ such that
\begin{align}
    \| \sigma - \sigma_{\varepsilon} \|_{S^{m + \theta}_{\rho, \delta}, a, b} &\leq C \| \sigma \|_{S^m_{\rho, \delta}, a', b'} \, \varepsilon^{\frac{\theta}{\hdeg}}, \label{eq:smoothing_app_adj} \\
    \|\sigma_{\varepsilon} \|_{S^m_{\rho, \delta}, a, b} &\leq C \| \sigma \|_{S^m_{\rho, \delta}, a', b'} \, \varepsilon^{\frac{\theta}{\hdeg}}. \label{eq:smoothing_bound_adj}
\end{align}
Now, let $a, b \in \N_0$ and $\eta \geq 0$. Since $\sigma \mapsto \sigma_\tau$ is a linear map from $S^{-\infty}(G)$ into itself with the property that $\kappa_{\sigma_\tau} \in \SC(G \times G)$ precisely when $\kappa_\sigma \in \SC(G \times G)$, the estimates \eqref{eq:sn_T1_tau_KN_3} and \eqref{eq:smoothing_bound_adj} yield the existence of integers $a', b', a'', b'' \in \N$ and a constant $C > 0$ such that
\begin{align*}
    \| (\sigma_{\varepsilon})_\tau &- (\sigma_{\varepsilon'})_\tau  \|_{S^{m + \theta}_{\rho, \delta}, a, b} 
    \leq \| \sigma_{\varepsilon} - \sigma_{\varepsilon'} \|_{S^{m + \theta}_{\rho, \delta}, a', b'} \leq C \| \sigma \|_{S^m_{\rho, \delta}, a'', b''} \, ( \varepsilon^{\frac{\theta}{\hdeg}} + {\varepsilon'}^{\frac{\theta}{\hdeg}} ) \leq \eta,
\end{align*}
for all $\varepsilon, \varepsilon' \in (0, \varepsilon_0)$ and a sufficiently small $\varepsilon_0 \in (0, 1)$. Let us denote the by $(\sigma_\tau)_\varepsilon$ the sequence $(\sigma_{\varepsilon})_\tau$.  
Since $a, b \in \N_0$ and $\eta > 0$ were arbitrary, the net $\{ (\sigma_{\tau})_\varepsilon)\}_{\varepsilon \in (0, 1)} \subseteq S^m_{\rho, \delta}(G)$ is Cauchy in the coarser topology of $S^{m + \theta}_{\rho, \delta}(G) \supseteq S^{m_1 + m_2}_{\rho, \delta}(G)$, but its limit $\lim_{\varepsilon} (\sigma_\tau)_\varepsilon =: \sigma'_\tau$ nevertheless lies in a closed bounded subset of $S^m_{\rho, \delta}$, due to \eqref{eq:sn_T1_tau_KN_3}. On the one hand, by \cite[Lem.~5.4.11~(3)]{FR} and Proposition~\ref{prop:cont_on_SC}, this implies that $\Optau\bigl ( (\sigma_\tau)_{\varepsilon}\bigr )$ converges strongly on $\SC(G)$ to $\Optau(\sigma'_\tau)$ with $\sigma'_\tau \in S^m_{\rho, \delta}(G)$.\footnote{To employ Proposition~\ref{prop:cont_on_SC}, we require $\delta < \frac{1}{v_n}$, unlike in the preceding part of the proof.} On the other hand, since $\Optau\bigl ( (\sigma_\tau)_{\varepsilon}\bigr ) = \Optau\bigl ( (\sigma_{\varepsilon})_\tau \bigr ) = \Op(\sigma_{ \varepsilon})$ strongly converges on $\SC(G)$ to $\Op(\sigma)$ and $\sigma \in S^m_{\rho, \delta}(G)$, this gives
\begin{align*}
    \Optau(\sigma'_\tau) f = \lim_{\varepsilon} \Optau\bigl ( (\sigma_\tau)_{\varepsilon}\bigr )f = \lim_{\varepsilon} \Op(\sigma_{\varepsilon})f = \Op(\sigma)f
\end{align*}
for all $f \in \SC(G)$. By the Schwartz kernel theorem, the kernels, and hence the associated kernels, in $\TD(G \times G)$ of these two continuous operators coincide. The extension of \eqref{eq:rel_ass_ker_tau_KN} to $\TD(G \times G)$ by duality thus yields $\sigma'_\tau = \sigma_\tau \in S^m_{\rho, \delta}(G)$. This proves the continuity of $\sigma \mapsto \sigma_\tau: S^m_{\rho, \delta}(G) \to S^m_{\rho, \delta}(G)$.

Due to \eqref{eq:sn_T1_KN_tau_3}, we can apply these arguments verbatim to show that also $\sigma_\tau \mapsto \sigma: S^m_{\rho, \delta}(G) \to S^m_{\rho, \delta}(G)$ is continuous.

Clearly, the identities \eqref{eq:rel_ass_ker_KN_tau} and \eqref{eq:rel_ass_ker_tau_KN} establish a one-to-one correspondence between the associated kernels and by applying the Fourier transform in $y$ also between the symbols (cf.~\eqref{eq:symbol} and \eqref{eq:tau_symbol}). Together with the continuity of the map, this yields a vector space isomorphism $\sigma \mapsto \sigma_\tau$ from $S^{m-(\rho - \delta)(M+1)}_{\rho, \delta}(G)$ onto itself, which completes the proof.
\end{proof}

\begin{example}
For the Weyl quantization on $G = \R^n$, the asymptotic expansion~\eqref{eq:asym_exp_tau_KN} recovers the well-known relation
\begin{align}
    \sigma_{\mathrm{w}} \sim \sum_{|\alpha| \leq M} \frac{2^{|\alpha|}}{\alpha!} (i \partial_{\xi})^\alpha \partial_x^\alpha \sigma_{\mathrm{KN}},
\end{align}
since the coefficients of both series are immediately seen to coincide due to the Abelian structure of $G$ and \eqref{eq:coeff_W_KN}. (Cf., e.g., \cite[Thm.~2.41]{F1} or \cite[Lem.~4.1.5]{L}.)
\end{example}

\subsection{The symbol of the adjoint} \label{subs:adjoint}

In this subsection we show that the symbol of the formal adjoint of an operator $\Optau(\sigma)$ with $\sigma \in S^m_{\rho, \delta}(G)$, $m \in \R$, $0 \leq \delta < \min\{\rho,\frac{1}{v_n} \}$, is in the same symbol class, and we provide an explicit asymptotic expansion for it. Below we shall denote by $\sigma^{(*)}_\tau$ the symbol of the adjoint of $\Op^\tau(\sigma)$, and by $\sigma^*$ the pointwise adjoint of the symbol $\sigma$. When $\tau=e_G$, we shall simply write $\sigma^{(*)}$ for $\sigma^{(*)}_{\tau=e_G}$.

\begin{theorem} \label{thm:asym_exp_adjoint_tau}
Let $m \in \R$, $0 \leq \delta < \min\{\rho,\frac{1}{v_n} \} \leq 1$, and let $\tau: G \to G$ be a quantizing function which satisfies (HP). Let $T$ be a continuous linear operator from $\SC(G)$ to $\TD(G)$ given by $T = \Optau(\sigma)$ for some $\sigma \in S^m_{\rho, \delta}(G)$. Then its formal adjoint $T^*$, defined with respect to the $\mathscr{S}(G)$-$\mathscr{S}'(G)$-duality, is given by $T^* = \Optau(\sigma^{(*)}_\tau)$ for a uniquely determined symbol $\sigma^{(*)}_\tau \in S^m_{\rho, \delta}(G)$.

The map $\sigma \mapsto \sigma^{(*)}_\tau$ is a Fr\'{e}chet space isomorphism from $S^m_{\rho, \delta}(G)$ onto itself, and the symbols are related to one another by the asymptotic expansion
\begin{align}
	\sigma^{(*)}_\tau &\sim \sum_{j=0}^\infty\Bigr(\sum_{[\alpha] = j} \sum_{[\alpha'] = [\alpha]} c^{(*), \tau}_{\alpha', \alpha} \, \Delta^{\alpha'} X^\alpha \sigma^*\Bigr), \label{eq:asym_exp_adjoint_tau}
\end{align}
in the sense that for given $M \in \mathbb{N}_0$
\begin{align*}
	R^{\sigma^{(*)}_\tau}_M &:= \sigma_\tau^{(*)} -  \sum_{[\alpha] \leq M} \sum_{[\alpha'] = [\alpha]} c^{(*), \tau}_{\alpha', \alpha} \, \Delta^{\alpha'} X^\alpha \sigma^* \in S^{m-(\rho-\delta)(M+1)}_{\rho,\delta}(G).
\end{align*}
The coefficients $c^{(*), \tau}_{\alpha', \alpha} \in \R$ are uniquely determined by the equations
\begin{align*}
    q_{\alpha} \bigl ( \tau(y) y^{-1} \tau(y^{-1})^{-1} \bigr ) = \sum_{[\alpha'] = [\alpha]} c^{(*), \tau}_{\alpha', \alpha} \, \tilde{q}_{\alpha'}(y).
\end{align*}

For the quantizing function $\tau = e_G$ the expansion \eqref{eq:asym_exp_adjoint_tau} recovers the asymptotic expansion
\begin{align}
	\sigma^{(*)} \sim \sum_{j=0}^\infty\bigl( \sum_{[\alpha] = j} \Delta^\alpha X^\alpha \sigma^* \bigr)\label{eq:asym_exp_adjoint_KN_quant_G}
\end{align}
 in the Kohn-Nirenberg calculus, which was established in \cite[Cor.~5.5.17]{FR}.

Moreover, the quantizing function $\tau$ is symmetric, i.e., it satisfies
\begin{align}
	T^* = \Optau(\sigma)^*=\Optau(\sigma^*) \label{eq:sym_quant_G}
\end{align}
and \eqref{eq:asym_exp_adjoint_tau} collapses to $\sigma^{(*)}_\tau = \sigma^*$, if and only if $\tau$ satisfies
\begin{align}
	\tau(x)= \tau(x^{-1})x \label{eq:char_sym_fun_G}
\end{align}
for all $x \in G$.

In particular, if $T = \Optau(\sigma)$ for some $\sigma \in S^0_{\rho, \delta}(G)$ and a symmetry function $\tau$, then $T = T^* \in \mathcal{L}(L^2(G))$ if and only if $\sigma^{(*)}_\tau = \sigma^* = \sigma$.
\end{theorem}

\begin{proof}
As in the proof of Theorem~\ref{thm:asym_exp_ch_qu} our proof crucially relies on the relation between the associated kernels. Note that it suffices to prove our statement for the dense subset of symbols $S^{-\infty}(G)$ whose associated kernels lie in $\SC(G \times G)$ since the extension to symbols in $S^m_{\rho, \delta}(G)$ is identical to the one in the proof of Theorem~\ref{thm:asym_exp_ch_qu}.

Thus, let $m \in \R$, $0 \leq \delta < \min\{\rho,\frac{1}{v_n}\}$, and let $\tau: G \to G$ be a quantizing function which satisfies (HP). If $\sigma \in S^{-\infty}(G)$ with $\kappa_\sigma \in \SC(G \times G)$, then also the integral kernel $\Ker_\sigma$ of $T = \Optau(\sigma)$ is a member of $\SC(G \times G)$. The integral kernel of its formal adjoint $T^*$, which we denote by $\Ker_{\sigma^{(*)}} \in \TD(G \times G)$, is explicitly related to $\Ker_\sigma$ by $\Ker_{\sigma^{(*)}}(x, y) = \overline{\Ker_\sigma}(y, x)$, essentially due to the Schwartz kernel theorem. It immediately follows that $\Ker_{\sigma^{(*)}} \in \SC(G \times G)$ and that we can use $\Ker_\sigma = \kappa_\sigma \circ \cv$ to express the associated kernel by
\begin{align}
    \kappa_{\sigma^{(*)}_\tau}(x, y) &= \bigl ( \Ker_{\sigma^{(*)}_\tau} \circ \cvi \bigr ) (x, y) \nonumber \\
    &= \Ker_{\sigma^{(*)}_\tau} \bigl ( x \tau(y), x \tau(y) y^{-1} \bigr ) \nonumber \\
    &= \overline{\Ker_\sigma}( x \tau(y) y^{-1}, x \tau(y) \bigr ) \\
    &= \overline{\kappa_\sigma} \Bigl ( x \tau(y) y^{-1} \tau \bigl ( \tau(y)^{-1} x^{-1} x \tau(y) y^{-1} \bigr )^{-1}, \tau(y)^{-1} x^{-1} x \tau(y) y^{-1} \Bigr ) \nonumber \\
    &= \overline{\kappa_\sigma} \bigl ( x \tau(y) y^{-1} \tau(y^{-1})^{-1}, y^{-1} \bigr ) \nonumber \\
    &= \kappa_{\sigma^*} \bigl ( x \tau(y) y^{-1} \tau(y^{-1})^{-1}, y \bigr ), \label{eq:ass_ker_adjoint_tau}
\end{align}
where we recall from \cite[Thm.~5.2.22]{FR} that the kernel associated to $\sigma^{*}$ is given $\kappa_{\sigma^*}(x, y) = \overline{\kappa_\sigma}(x, y^{-1})$. We immediately notice that if $\tau$ satisfies \eqref{eq:char_sym_fun_G}, then
\begin{align}
    \tau(y) y^{-1} \tau(y^{-1})^{-1} = e_G \label{eq:trivial_adj_cond}
\end{align}
and
\eqref{eq:ass_ker_adjoint_tau} yields
\begin{align} \label{eq:char_sym_fun_G_equ_cond}
    \kappa_{\sigma^{(*)}_\tau}(x, y) = \kappa_{\sigma^*}(x, y) \hspace{10pt} \Longleftrightarrow \hspace{10pt} \sigma^{(*)}_\tau = \sigma^*,
\end{align}
hence $\tau$ is symmetric. Conversely, if $\tau$ is symmetric, that is, if the condition \eqref{eq:char_sym_fun_G_equ_cond} holds, then by the above computation the Schwartz functions $\Ker_{\sigma^{(*)}}$ and $\Ker_{\sigma^*}$ coincide only if \eqref{eq:char_sym_fun_G} holds true. By duality, this argument extends to distributions in $\TD(G \times G)$. Now, since any operator $T = \Optau(\sigma)$ with $\sigma \in S^m_{\rho, \delta}(G)$ is continuous from $\SC(G)$ into $\SC(G)$ by Proposition~\ref{prop:cont_on_SC}, the associated kernel lies in $\TD(G \times G)$ and the isomorphism property follows essentially verbatim as the last part of the proof of Theorem~\ref{thm:asym_exp_ch_qu}. To finish the part of the proof concerning symmetric $\tau$, note that $T = \Optau(\sigma)$ with $\sigma \in S^0_{\rho, \delta}(G)$ is continuous from $L^2(G)$ into itself (cf.~Theorem~\ref{thm:cont_on_Sobolev}~(i) below). Hence, if $T$ is self-adjoint, then the kernels, and hence the associated kernels and also the symbols, of $T$ and $T^*$ coincide.
\medskip

Let us continue the proof for a $\tau$ which is not necessarily symmetric. In order to derive an asymptotic expansion of $\sigma^{(*)}_\tau$ in terms of $\sigma^*$, we will employ a Taylor expansion of its associated kernel based on \eqref{eq:ass_ker_adjoint_tau}. Since we need not consider the trivial case \eqref{eq:trivial_adj_cond}, we may exclude symmetry functions. For any non-symmetric quantizing function $\tau$, the condition (HP) and the gradation of $\mathfrak{g}$ ensure that for at least one $j = 1, \ldots, \dimG$, the $j$-th coordinate of $\tau(y) y^{-1} \tau(y^{-1})^{-1}$ is given by
\begin{align} \label{eq:hom_str_tyt}
    (2C^\tau_j - 1) y_j + p_j(y_1, \ldots, y_{j-1})
\end{align}
for some $C^\tau_j \neq \frac{1}{2}$ and some homogeneous polynomial $p_j$ of degree $v_j$, which only depends on the variables $y_1, \ldots, y_{j-1}$, while the other coordinates may vanish. This yields a non-trivial Taylor expansion
\begin{align*}
    \kappa_{\sigma^{(*)}_\tau}(x, y) &= \sum_{[\alpha] \leq M} \sum_{[\alpha'] = [\alpha]} c^{(*), \tau}_{\alpha', \alpha} \, \tilde{q}_{\alpha'}(y) X^\alpha_{x_1 = x} \kappa_{\sigma^*}(x_1, y) \\
    &\hspace{-15pt} +\sum_{M+1 \leq [\alpha] \leq M_0} \sum_{[\alpha'] = [\alpha]} c^{(*), \tau}_{\alpha', \alpha} \, \tilde{q}_{\alpha'}(y) X^\alpha_{x_1 = x} \kappa_{\sigma^*}(x_1, y) + R^{\kappa_{\sigma^*}(\, . \,, y)}_{x, M_0} \bigl ( \tau(y) y^{-1} \tau(y^{-1})^{-1} \bigr )
\end{align*}
for some $M_0 \leq M+1$. Since $|\tau(y) y^{-1} \tau(y^{-1})^{-1}| \lesssim |y|$ due to \eqref{eq:hom_str_tyt}, the rest of the proof now follows verbatim the last part of the proof of Theorem~\ref{thm:asym_exp_ch_qu} from \eqref{eq:Taylor_exp_ch_qu} on.

Finally, note that for $\tau = e_G$ the identity \eqref{eq:ass_ker_adjoint_tau} reduces to $\kappa_{\sigma^{(*)}_\tau}(x, y) = \kappa_{\sigma^*}(xy, y)$. Consequently, the Taylor coefficients satisfy 
\begin{align*}
    c^{(*), \tau}_{\alpha', \alpha} = \delta_{\alpha', \alpha} =
    \left \{ \begin{array}{ll}
        1 & \mbox{ if } \alpha' = \alpha, \\
        0 & \mbox{ otherwise,}
    \end{array} \right.
\end{align*}
which yields the asymptotic expansion \eqref{eq:asym_exp_adjoint_KN_quant_G} as expected. This completes the proof.
\end{proof}

\begin{remark}
We wish to remark that the result of Theorem \ref{thm:asym_exp_adjoint_tau} can also be proved by using an alternative strategy. For instance, one could use the change of quantization from Theorem \ref{thm:asym_exp_ch_qu} in combination with the asymptotic formula for symbol of the adjoint operator in the Kohn-Nirenberg quantization in \cite{FR}, very similarly to how this is done in \cite[Thm.~2.52]{F1} for the Euclidean Weyl and Kohn-Nirenberg quantizations.
Although this indirect proof would shorten the current work substantially, it might oblige its readers to do substantial reading elsewhere. Moreover, we consider it preferable to give a direct independent proof to keep our work much more self-contained.
    
Note that the same consideration applies to Theorem \ref{thm:asym_exp_comp_tau} in the following subsection.
\end{remark}

\begin{example} \label{ex:adj_WQ}
Since for $G = \R^n$ the quantizing function $\tau(x) = \frac{x}{2}$ clearly satisfies \eqref{eq:char_sym_fun_G}, the identity~\eqref{eq:sym_quant_G} simply recovers the property \eqref{eq:sym_WQ} of the Weyl quantization. Note that Theorem~\ref{thm:asym_exp_adjoint_tau} also recovers the classical fact that among the $\tau$-quantizations~\eqref{eq:tau_quant}, defined by the quantizing functions $\tau(x) = \tau x$, $\tau \in [0, 1]$, the Weyl quantization $\Opw = \Op^{\frac{1}{2}}$ is the only one that satisfies \eqref{eq:sym_quant_G}.
\end{example}

\begin{example}
On $G = \H$, Theorem~\ref{thm:asym_exp_adjoint_tau} holds true for all members of the family of symmetry functions \eqref{eq:fam_tau_Hn}, which by Example~\ref{ex:fam_sym_fun_Hn} is precisely the set of all symmetry functions $\tau: \H \to \H$ that satisfy (HP) with respect to the canonical homogeneous structure.
\end{example}

For the sake of completeness, let us recall the inner product formula for $\tau$-quantized Hilbert-Schmidt operators on $L^2(G)$ established in~\cite[\SS~3.2]{MR1} in the context of Theorem~\ref{thm:asym_exp_adjoint_tau}. If $\sigma_1, \sigma_2$ are symbols on $G \times \Ghat$ given by measurable fields of operators whose associated kernels $\kappa_{\sigma_1}, \kappa_{\sigma_2}$ lie in $L^2(G \times G)$, then the operators $\Optau(\sigma_1), \Optau(\sigma_2)$ are Hilbert-Schmidt operators with $\sigma_j(x, \pi) \in \mathtt(HS)(\RS)$ for $j = 1,2$ and a.e. $x \in G$, $\pi \in \Ghat$. The Hilbert-Schmidt inner product of such operators can be expressed as \noeqref{eq:HS_inprod_tau}
\begin{align}
    \hspace{-25pt} \langle \Optau(\sigma_1), \Optau(\sigma_2) \rangle_{HS}
    = \iint\limits_{\Ghat \times G} \Tr \bigl ( \sigma_1(x, \pi) \sigma_2^*(x, \pi) \bigr ) \, dx \, d\mu(\pi) = \langle\sigma_1, \sigma_2 \rangle_{L^2(G\times\widehat{G})} \label{eq:HS_inprod_tau}
\end{align}
for any measurable $\tau: G \to G$. Hence, this formula holds in particular for any $\tau$ admissible in sense of Subsection~\ref{subs:adm_qu_f} and any two Hilbert-Schmidt operators $\Optau(\sigma_1), \Optau(\sigma_2)$ with symbols $\sigma_1, \sigma_2 \in S^\infty(G)$.

\subsection{The composition of symbols} \label{subs:comp}

 In this subsection we prove the main result of this paper. It states that if the quantizing function $\tau$ satisfies the condition (HP), then the associated composition of symbols $\circ_\tau$ is continuous on the H\"{o}rmander symbol classes $S^m_{\rho, \delta}(G)$, $m \in \R$, $0 \leq \delta < \frac{\rho}{v_n} \leq \frac{1}{v_n} \leq 1$, and the composite symbol is approximately given by an asymptotic expansion. If $\tau$ is symmetric and the group $G$ is stratified, the leading orders of this expansion take a specific asymmetric shape, which for $G = \R^n$ and $\tau(x) = \frac{x}{2}$ recovers the asymptotic expansion for the Weyl product. We will discuss this in more detail for symmetric quantizations on stratified groups. We stress once again that the fundamental case $(\rho,\delta)=(1,0)$ is covered by our result.

\begin{lemma}\label{lem:comp_ker}
Let $G$ be a graded group and let $\sigma_1, \sigma_2$ be smooth symbols whose associated kernels $\k, \kk$ are Schwartz functions. If the quantizing function $\tau$ satisfies (CV), then
\begin{align}
    \Optau(\sigma_1) \Optau(\sigma_2) = \Optau(\sigma) \nonumber
\end{align}
for a smooth symbol $\sigma$. The integral kernel and the associated kernel of $\Optau(\sigma)$ are Schwartz functions on $G \times G$, given by the absolutely convergent integrals
\begin{align}
\begin{split} \label{eq:Ker_comp}
    \Ker_{\sigma}(x, y) &= \kappa_{\sigma}\bigl ( x \tau(y^{-1}x)^{-1}, y^{-1}x \bigr ) \\
    &= \int_{G} \kappa_{\sigma_1}\bigl ( x\tau(z^{-1}x)^{-1},z^{-1}x \bigr ) \kappa_{\sigma_2}\bigl ( z\tau(y^{-1}z)^{-1},y^{-1}z \bigr ) \, dz,
\end{split}
\end{align}
and
\begin{align}
    \kappa_{\sigma}(x, y) = \int_{G} \k\bigl ( x \tau(y) \tau({z}^{-1}y)^{-1}, {z}^{-1}y \bigr ) \kk\bigl ( x \tau(y) y^{-1}z \tau(z)^{-1}, z \bigr ) \, dz, \label{eq:ass_ker_comp}
\end{align}
respectively. Explicitly, we have
\begin{align}
	\sigma(x, \pi) = \Bigl ( \F \bigl ( \kappa_{\sigma}(x, \, . \,)) \bigr )\Bigr )(\pi) \label{eq:comp_symb_Schw_ker}
\end{align}
for all $x \in G$ and almost all $\pi \in \Ghat$.

In particular, for the symmetry function $\tau(x) = \exp(\frac{1}{2} \log(x))$ the identity~\eqref{eq:ass_ker_comp} simplifies to
\begin{align}
    \kappa_{\sigma}(x, y) = \int_{G} \k\bigl ( x \tau(y) \tau(y^{-1}z), {z}^{-1}y \bigr ) \kk\bigl ( x \tau(y)^{-1} \tau(z), z \bigr ) \, dz, \label{eq:ass_ker_comp_onehalf_tau}
\end{align}
while for the constant quantizing function $\tau = e_G$ it recovers the composite kernel in the Kohn-Nirenberg quantization:
\begin{align}
    \kappa_{\sigma}(x, y) = \int_{G} \k(x , {z}^{-1}y) \kk(x y^{-1}z, z) \, dz = \kappa_{\sigma_1 \circ_{\mathrm{KN}} \sigma_2}(x, y). \label{eq:ass_ker_comp_KN}
\end{align}
\end{lemma}

\begin{proof}
A formal proof of the identity \eqref{eq:Ker_comp} follows immediately from writing out the two-fold application of the integral representation~\eqref{eq:tau_quant_ass_kernel_int_kernel}. By applying the measure-preserving change of variables $z' := y^{-1}z$ to \eqref{eq:Ker_comp}, we then obtain
\begin{align*}
    \kappa_{\sigma}\bigl ( x \tau(y^{-1}x)^{-1}, y^{-1}x \bigr ) = \int_G \k \bigl ( x \tau({z'}^{-1} y^{-1}x)^{-1}, {z'}^{-1}y^{-1}x, \bigr ) \kk \bigl ( yz'\tau(z')^{-1}, z' \bigr ) \, dz'.
\end{align*}
Adopting the notation 
\begin{equation*}
    \begin{array}{rrccl}
        (x' , y') &:=& \cv(x, y) &=& (x \tau(y^{-1}x)^{-1}, y^{-1}x), \\
        (x, y) &=& \cvi(x', y') &=& (x' \tau(y'), x' \tau(y') {y'}^{-1}),
    \end{array}
\end{equation*}
we can rewrite this identity as
\begin{align*}
     \kappa_{\sigma}(x', y') = \int_{G} \k\bigl ( x' \tau(y') \tau({y'}^{-1}z'), {z'}^{-1}y' \bigr ) \kk\bigl ( x'\tau(y')y'^{-1}z' \tau(z')^{-1}, z' \bigr ) \, dz',
\end{align*}
which gives \eqref{eq:ass_ker_comp}. The identity \eqref{eq:ass_ker_comp_onehalf_tau} now follows immediately from the observation that $\tau(x) = \exp(\frac{1}{2} \log(x))$ satisfies $\tau(x^{-1}) = \tau(x)^{-1}$ and $\tau(x)^{-1}x =x \tau(x)^{-1} = \tau(x)$ for all $x \in G$, while \eqref{eq:ass_ker_comp_KN} follows trivially for $\tau = e_G$.

To show the absolute convergence of the integral \eqref{eq:ass_ker_comp}, we recall that the condition (CV) ensures that $\Ker_{\sigma_1}, \Ker_{\sigma_2} \in \SC(G \times G)$ whenever $\kappa_{\sigma_1}, \kappa_{\sigma_2}\in \SC(G \times G)$. It follows that the absolutely convergent integral
\begin{align}
    (x, y) \mapsto \Ker_{\sigma} = \int_G\Ker_{\sigma_1}(x, z) \Ker_{\sigma_1}(z, y) \, dz \in \SC(G \times G) \label{eq:conv_comp_ker}
\end{align}
defines a Schwartz function on $G \times G$, and an application of $\cv$ to both $\Ker_{\sigma_1}$ and $\Ker_{\sigma_2}$ does not affect the convergence of \eqref{eq:conv_comp_ker}. Consequently, the integral
\begin{align}
    \sigma(x, \pi) := \int_G \kappa_{\sigma}(x, y) \pi(y)^* \, dy, \hspace{5pt} x \in G,
\end{align}
defines a smooth symbol, which even lies in $S^{-\infty}(G)$ by \cite[Lem.5.2.21]{FR}. This completes the proof.
\end{proof}

\medskip

The following result is the composition theorem for $\tau$-quantized operators. To prove it, we will use Lemma~\ref{lem:comp_ker} and estimates for the kernels $\kappa_\sigma$ associated with symbols $\sigma \in S^m_{\rho, \delta}(G)$.
Moreover, the functions $p_1, p_2: G \times G \to G$,
\begin{align}
\begin{split} \label{eq:p1_p2}
    p_1(y, z) &:=\tau(y) \tau({z}^{-1}y)^{-1}, \\
    p_2(y, z) &:= \tau(y) y^{-1}z \tau(z)^{-1},
\end{split}
\end{align}
will play a prominent role.

\begin{theorem} \label{thm:asym_exp_comp_tau}
Let $m_1, m_2 \in \R$, $0 \leq \delta < \frac{\rho}{v_n} \leq 1$, and let the quantizing function $\tau$ satisfy (HP). Let $T_1$ and $T_2$ be continuous linear operators from $\SC(G)$ to $\SC(G)$ given by $T_1 = \Optau(\sigma_1)$ and $T_2 = \Optau(\sigma_2)$ for some $\sigma_1 \in S^{m_1}_{\rho,\delta}(G)$ and $\sigma_2\in S^{m_2}_{\rho,\delta}(G)$, respectively. Then there exists a uniquely determined symbol $\sigma \in S_{\rho,\delta}^{m_1+m_2}(G)$ such that $\Optau(\sigma) = T_1 T_2$. If $\k, \kk \in \SC(G \times G)$, then the symbol $\sigma$ is given by the absolutely convergent integral \eqref{eq:comp_symb_Schw_ker}.

The $\tau$-composition of symbols
\begin{align}
\begin{split} \label{eq:comp_tau}
    \circ_\tau: S_{\rho,\delta}^{m_1}(G) \times S_{\rho,\delta}^{m_2}(G) &\to S_{\rho,\delta}^{m_1+m_2}(G), \\
    (\sigma_1, \sigma_2) &\mapsto \sigma  =: \sigma_1 \circ_\tau \sigma_2
\end{split}
\end{align}
is a bilinear and continuous map, and the symbol $\sigma$ is asymptotically given by
\begin{align} \label{eq:asym_exp_comp_tau}
    \hspace{-20pt} \sigma \sim \sum_{i,j=1}^\infty\left(\underset{\substack{[\alpha] =i, \\ [\beta] =j}}{\sum} \hspace{5pt} \underset{\substack{[\alpha_1] + [\alpha_2] = [\alpha], \\ [\beta_1] + [\beta_2] = [\beta]}}{\sum} c_{\alpha_1, \alpha_2} c_{\beta_1, \beta_2} \Bigl ( \Delta^{\alpha_2} \Delta^{\beta_1} X^\alpha \sigma_1 \Bigr ) \, \Bigl ( \Delta^{\beta_2} \Delta^{\alpha_1} X^\beta \sigma_2 \Bigr ) \right),
\end{align}
in the sense that for given $M, N \in \mathbb{N}_0$ and $L := \min\{M, N\}$
\begin{align}
    &R^{\sigma}_L := \sigma -\!\!\!\underset{\substack{[\alpha] \leq M, \\ [\beta] \leq N}}{\sum} \underset{\substack{[\alpha_1] + [\alpha_2] = [\alpha], \\ [\beta_1] + [\beta_2] = [\beta]}}{\sum} \hspace{5pt}\!\!\!\!\! c_{\alpha_1, \alpha_2} \, c_{\beta_1, \beta_2} \Bigl ( \Delta^{\alpha_2} \Delta^{\beta_1} X^\alpha \sigma_1 \Bigr ) \, \Bigl ( \Delta^{\beta_2} \Delta^{\alpha_1} X^\beta \sigma_2 \Bigr ) \label{eq:asymptotic_kernel_remainder_composition} \\
    &\hspace{250pt} \in S^{m_1 + m_2-(\rho-\delta)(L+1)}_{\rho,\delta}(G).
\end{align}
Moreover, the coefficients  $c_{\alpha_1, \alpha_2}, c_{\beta_1, \beta_2} \in \R$ are uniquely determined by the equations
\begin{align}
    q_\alpha \bigl( p_1(y, z) \bigr ) = \underset{[\alpha_1] + [\alpha_2] = [\alpha]}{\sum} c_{\alpha_1, \alpha_2} \tilde{q}_{\alpha_1}(z) \tilde{q}_{\alpha_2}(z^{-1}y),\label{eqn1811-1510} \\
    q_\beta \bigl ( p_2(y, z) \bigr ) = \underset{[\beta_1] + [\beta_2] = [\beta]}{\sum} c_{\beta_1, \beta_2} \tilde{q}_{\beta_1}(z^{-1}y) \tilde{q}_{\beta_2}(z). \label{eqn1811-1511}
\end{align}
\end{theorem}

\begin{remark}
We would like to emphasize that our asymptotic formula depends on two delimiting parameters, $M$ and $N$, in contrast to the asymptotic expansions for the Kohn-Nirenberg quantization in \cite{FR} and the Weyl quantization on $\R^n$. The absence of a second parameter in the case of the Kohn-Nirnberg quantization is a direct consequence of \eqref{eq:ass_ker_comp_KN}. In the case of the Euclidean Weyl quantization, the asymptotic series of the composite symbol can be written as a power series of Poisson brackets, due to the symplectic nature of phase space, which allows the two parameters $M$ and $N$ to be merged into one.

This difference motivates the arguments in the proof of Part II of Theorem \ref{thm:asym_exp_comp_tau} and in the proof of Step 2 of Lemma \ref{lemma1}, Lemma \ref{lemma2} and Lemma \ref{lemma3}, where we prove that the remainder terms belong to symbol classes whose order depends on both $M$ and $N$, since $L=L(M,N):=\min\{M,N\}$, and not just on one parameter $L$ independent of $M$ and $N$. Hence, to make the proof work one has to carefully take into account both the parameters.
\end{remark}

\medskip

\begin{proof}[Proof of Theorem~\ref{thm:asym_exp_comp_tau}, Part I]
To prove the theorem, we will make crucial use of the relation~\eqref{symbol-kernel} between the symbol $\sigma$ of a given $\tau$-quantized operator $\Optau(\sigma)$ and its associated kernel $\kappa_\sigma$ in combination with the identity~\eqref{eq:ass_ker_comp} for the composite associated kernel. Since for $\k, \kk \in \SC(G \times G)$ the integrals in \eqref{symbol-kernel} and \eqref{eq:ass_ker_comp} converge absolutely and we may differentiate under the integral sign, we will first prove the asymptotic expansion and accompanied estimates for the dense subset of symbols $\sigma_1, \sigma_2 \in S^{-\infty}(G)$ with associated kernels $\k, \kk \in \SC(G \times G)$.\footnote{This set is dense in each $S^m_{\rho,\delta}(G)$ with $m \in \R$, $0 \leq \delta < \rho \leq 1$, due to Lemma~\cite[Lem.~5.4.11]{FR}.} In particular, we will show that for such $\sigma_1, \sigma_2$ the $S^{m_1+m_2-(\rho-\delta)(\min\{M,N\}+1)}_{\rho,\delta}$ seminorms of the remainder term $R^\sigma_L$ are bounded by products of suitable  $S^{m_1}_{\rho,\delta}$ and $S^{m_2}_{\rho,\delta}$ seminorms of $\sigma_1$ and $\sigma_2$, respectively, thus showing that $R^\sigma_L \in S^{m_1+m_2-(\rho-\delta)(\min\{M,N\}+1)}_{\rho,\delta}$. This will turn out to be crucial for the final part of the proof, an approximation argument that ensures that the continuity of the composition $\circ_\tau$ and the asymptotic expansion~\eqref{eq:asym_exp_comp_tau} hold on the whole space $S^{m_1}_{\rho, \delta}(G) \times S^{m_2}_{\rho, \delta}(G)$.

\medskip

Thus, let $\sigma_1, \sigma_2 \in S^{-\infty}(G)$ with associated kernels $\k, \kk \in \SC(G \times G)$. In order to employ the homogeneous Taylor expansions of $\k$ and $\kk$ to derive \eqref{eq:asym_exp_comp_tau},
we observe that by (HP) the functions $p_1, p_2: G \times G \to G$ defined by \eqref{eq:p1_p2}
are smooth and such that, for every $j=1,\ldots, n$, their $j$-th coordinate is a homogeneous polynomial of degree $v_j$ in two variables.
So, for arbitrary but fixed orders $M, N \in \N$, we can employ \eqref{eq:Taylor_expansion} to obtain the homogeneous Taylor expansions
\begin{align}
    \k \bigl ( x p_1(y, z), {z}^{-1}y \bigr ) &= \sum_{[\alpha] \leq M} q_\alpha \bigl ( p_1(y, z) \bigr ) X^\alpha_{x_1 = x} \k(x_1, z^{-1}y) + R^{\k(\, \cdot \,, z^{-1}y)}_{x, M}\bigl ( p_1(y, z) \bigr ), \nonumber \\
    \kk \bigl ( x p_2(y, z), z \bigr ) &= \sum_{[\beta] \leq N} q_\beta \bigl( p_2(y, z) \bigr ) X^\beta_{x_2 = x} \kk(x_2, z) + R^{\kk(\, \cdot \,, z)}_{x, N}\bigl ( p_2(y, z) \bigr ), \nonumber
\end{align}
where $ R^{\k}_{x, M}$ and $R^{\kk}_{x, N}$ satisfy \eqref{eq:est_Taylor_rem}. Since we know that $q_\alpha \circ p_1, q_\beta \circ p_2: G \times G \to \R$ are homogeneous polynomials of degree $[\alpha]$ and $[\beta]$, respectively, we may rewrite\footnote{Note that we choose the homogeneous polynomials $\tilde{q}_\gamma(z) = q_\gamma(z^{-1}) = (-1)^{[\gamma]} q_\gamma(z)$, $\gamma = \alpha_1, \alpha_2, \beta_1, \beta_2$, over the polynomials $q_\gamma$ in the expansions because they ``quantize'' difference operators. Alternatively, one can work with the $q_\gamma$, replace them afterwards by the $\tilde{q}_\gamma$ and absorb the additional factors $(-1)^{[\gamma]}$, a posteriori, via the constants $c_{\alpha_1, \alpha_2}, c_{\beta_1, \beta_2}$.} $q_\alpha \circ p_1$ as in \eqref{eqn1811-1510} and $q_\beta \circ p_2$ as in \eqref{eqn1811-1511}.
With this at hand, we can now expand the composite kernel~\eqref{eq:ass_ker_comp} as
\begin{align*}
    \kappa_\sigma(x, y) =& \underset{\substack{[\alpha] \leq M, \\ [\beta] \leq N}}{\sum} \hspace{5pt} \underset{\substack{[\alpha_1] + [\alpha_2] = [\alpha], \\ [\beta_1] + [\beta_2] = [\beta]}}{\sum} \hspace{5pt} c_{\alpha_1, \alpha_2} c_{\beta_1, \beta_2} \int_G \tilde{q}_{\alpha_2}(z^{-1}y) \tilde{q}_{\beta_1}(z^{-1}y) X^\alpha_{x_1 = x} \k(x_1, z^{-1}y) \\
    &\hspace{195pt} \times \tilde{q}_{\beta_2}(z) \tilde{q}_{\alpha_1}(z) X^\beta_{x_2 = x} \kk(x_2, z) \, dz \\
    &\hspace{-50pt} + \underset{[\beta] \leq N}{\sum} \hspace{5pt} \underset{[\beta_1] + [\beta_2] = [\beta]}{\sum} \hspace{5pt} c_{\beta_1, \beta_2} \int_G \tilde{q}_{\beta_1}(z^{-1}y) \tilde{q}_{\beta_2}(z) X^\beta_{x_2 = x} \kk(x_2, z) R^{\k(\, \cdot \,, z^{-1}y)}_{x, M} \bigl ( p_1(y, z) \bigr ) \, dz \\
    &\hspace{-50pt} + \underset{[\alpha] \leq N}{\sum} \hspace{5pt} \underset{[\alpha_1] + [\alpha_2] = [\alpha]}{\sum} c_{\alpha_1, \alpha_2} \int_G \tilde{q}_{\alpha_1}(z) \tilde{q}_{\alpha_2}(z^{-1}y) X^\alpha_{x_1 = x} \k(x_1, z^{-1}y) R^{\kk(\, \cdot\,, z)}_{x, N} \bigl( p_2(y, z) \bigr ) \, dz \\
    &\hspace{-50pt} + \int_G R^{\k(\, \cdot \,, z^{-1}y)}_{x, M} \bigl ( p_1(y, z) \bigr ) R^{\kk(\, \cdot\,, z)}_{x, N} \bigl ( p_2(y, z) \bigr ) \, dz.
\end{align*}
These very distinct types of singular integrals suggest to split the associated kernel into four terms and to treat each of them as an associated kernel of a symbol of specific type. To see that the three latter contributions together form the remainder $R^{\sigma}_L(x,\pi)$ defined by \eqref{eq:asymptotic_kernel_remainder_composition}, we observe that
\begin{align*}
    \int_G \tilde{q}_{\alpha_2}(z^{-1}y) \tilde{q}_{\beta_1}(z^{-1}y) X^\alpha_x \k(x, z^{-1}y) \pi(z^{-1}y)^* \, dy &= \Bigl ( \Delta^{\alpha_1} \Delta^{\beta_1} X^\alpha \sigma_1 \Bigr )(x, \pi),
\end{align*}
and that
\begin{align*}
    \int_G \tilde{q}_{\beta_2}(z) \tilde{q}_{\alpha_1}(z) \tilde{q}_{\beta_2}(z) X^\beta_x \kk(x, z) \pi(z)^* \, dz = \Bigl ( \Delta^{\alpha_2} \Delta^{\beta_2} X^\beta \sigma_2 \Bigr )(x, \pi),
\end{align*}
since this implies that the Fourier transform of the associated kernel
\begin{align}
\begin{split}
    \kappa_{T_0} (x,y):=& \underset{\substack{[\alpha] \leq M, \\ [\beta] \leq N}}{\sum} \hspace{5pt} \underset{\substack{[\alpha_1] + [\alpha_2] = [\alpha], \\ [\beta_1] + [\beta_2] = [\beta]}}{\sum} \hspace{5pt} c_{\alpha_1, \alpha_2} c_{\beta_1, \beta_2} \int_G \tilde{q}_{\alpha_2}(z^{-1}y) \tilde{q}_{\beta_1}(z^{-1}y) X^\alpha_{x_1 = x} \k(x_1, z^{-1}y) \\
    &\hspace{160pt} \times \tilde{q}_{\beta_2}(z) \tilde{q}_{\alpha_1}(z) \tilde{q}_{\beta_2}(z) X^\beta_{x_2 = x} \kk(x_2, z) \, dz
\end{split}
\end{align}
equals
\begin{align*}
    T_0(x, \pi) = \int_G \kappa_{T_0}(x, y) \pi(y)^* \, dy = \iint\limits_{G \times G} \kappa_{T_0}(x, y) \pi(z^{-1}y)^* \, dy \, \pi(z)^* \, dz 
    = (\sigma- R_L^\sigma)(x,\pi),
\end{align*}
as desired. Since this is a finite sum of (products of) symbols of decreasing order (cf.~Remark~\ref{rem:pw_product} for the product of symbols), nothing more has to be said about the symbol $T_0$ until the final part of the proof, where we justify the extension from $\sigma_1, \sigma_2 \in S^{-\infty}(G)$ with $\k, \kk \in \SC(G \times G)$ to
$\sigma_1 \in S^{m_1}_{\rho,\delta}(G)$, $\sigma_2\in S^{m_1}_{\rho,\delta}(G)$.
The seminorm estimates for the remaining symbols
\begin{align}
    T_1(x, \pi) :=& \underset{\substack{[\beta] \leq N, \\ [\beta_1] + [\beta_2] = [\beta]}}{\sum} c_{\beta_1, \beta_2} \iint\limits_{G \times G} \tilde{q}_{\beta_1}(z^{-1}y) \tilde{q}_{\beta_2}(z) X^\beta_{x_2 = x} \kk(x_2, z)  \label{eq:def_T1} \\
    &\hspace{150pt} \times R^{\k(\, \cdot \,, z^{-1}y)}_{x, M}\bigl ( p_1(y, z) \bigr ) \, dz \, \pi(y)^* dy, \\
    T_2(x, \pi) :=& \underset{\substack{[\alpha] \leq M, \\ [\alpha_1] + [\alpha_2] = [\alpha]}}{\sum}  c_{\alpha_1, \alpha_2} \iint\limits_{G \times G} \tilde{q}_{\alpha_1}(z) \tilde{q}_{\alpha_2}(z^{-1}y) X^\alpha_{x_1 = x} \k(x_1, z^{-1}y) \label{eq:def_T2} \\
    &\hspace{150pt} \times R^{\kk(\, \cdot \,, z)}_{x, N}\bigl ( p_2(y, z) \bigr ) \, dz \, \pi(y)^* dy, \\
    T_3(x, \pi) :=& \iint\limits_{G \times G} R^{\k(\, \cdot \,, z^{-1}y)}_{x, M}\bigl ( p_1(y, z) \bigr ) R^{\kk(\, \cdot \,, z)}_{x, N}\bigl ( p_2(y, z) \bigr ) \, dz \pi(y)^* \, dy, \label{eq:def_T3}
\end{align}
which will justify the validity of \eqref{eq:asym_exp_comp_tau}, will constitute Part II of the proof.
\end{proof}

\medskip

\begin{proof}[Proof of Theorem~\ref{thm:asym_exp_comp_tau}, Part II]
In this part of the proof we will show that each of the \\ $S^{m_1+m_2-(\rho-\delta)(L+1)}(G)$ seminorms of the symbols $T_j$, $j=1,2,3$, defined by \eqref{eq:def_T1} -- \eqref{eq:def_T3}, are finite and bounded by the product $\|\sigma_1\|_{S^{m_1}_{\rho,\delta,a_1,b_1}}$ $\|\sigma_2\|_{S^{m_2}_{\rho,\delta,a_2,b_2}}$ for suitable $a_1,a_2,b_1,b_2\in \mathbb{R}$. As a byproduct, this will give  $R^\sigma_L\in S^{m_1+m_2-(\rho-\delta)(L+1)}(G)$ when $\sigma_1,\sigma_2$ are smooth symbols with Schwartz kernels.
Moreover,  the estimates we obtain here will be fundamental in the final part of the proof, where we deal with symbols $\sigma_j\in S^{m_j}_{\rho,\delta}(G)$, $j=1,2$.

We provide the necessary estimates of the remainder terms $T_j$, $j=1,2,3$, in the asymptotic expansion~\eqref{eq:asym_exp_comp_tau} in a series of lemmas. After an auxiliary result, the actual estimates are given in Lemmas~\ref{lemma1}, \ref{lemma2} and \ref{lemma3}, which are deeply based on kernel estimates. In the proofs of these lemmas we will also make frequent use of Lemma \ref{FRcorollary}.

\medskip

Let us begin with an auxiliary lemma.

\begin{lemma} \label{prop-p1}
For any quantizing function $\tau: G \to G$ that satisfies (HP), the polynomial functions $p_1(y, z) = \tau(y) \tau({z}^{-1}y)^{-1}$ and $p_2(y, z) = \tau(y) y^{-1}z \tau(z)^{-1}$ satisfy the following pointwise estimates for all $y, z \in G$:
\begin{align}
    |p_1(y,z)| &\lesssim |z| + \sum_{j=1}^n\sum_{\substack{[\alpha]+[\beta]=v_j\geq2\\ \alpha,\beta\in\mathbb{N}^n
    }}|z|^{\frac{[\alpha]}{v_j}}|z^{-1} y|^{\frac{[\beta]}{v_j}}, \label{eq:est_p1} \\
    |p_2(y,z)| &\lesssim |z^{-1}y| + \sum_{j=1}^n\sum_{\substack{[\alpha]+[\beta]=v_j\geq2\\ \alpha,\beta\in\mathbb{N}^n
    }}|z|^{\frac{[\alpha]}{v_j}}|z^{-1} y|^{\frac{[\beta]}{v_j}}. \label{eq:est_p2}
\end{align}
\end{lemma}

\begin{proof}
For the sake of convenience, we will use the quasi-norm $|\cdot|_\infty$ defined by \eqref{eq:inf_qn} to prove the lemma. Although we normally suppress the specific choice of equivalent quasi-norm notationally, we will highlight it in this proof in order to distinguish it from the absolute values of the components $z_j$ and $(z^{-1}y)_k$, $j, k = 1, \ldots, \dimG$.

Let us write the functions $p_1$ and $ p_2$ as the $\dimG$-tuples
\begin{align}
    p_1(y, z) &= \bigl ( p_1^1(y, z), \ldots, p_1^\dimG(y, z) \bigr ), \nonumber \\
    p_2(y, z) &= \bigl ( p_2^1(y, z), \ldots, p_2^\dimG(y, z) \bigr ). \nonumber
\end{align}
Since the functions $p_1$ and $ p_2$ are defined pointwise as the products of functions, whose components are homogeneous polynomials, each coefficient function $p_k^j: G \times G \to \R$, $k = 1, 2$, $j = 1, \ldots, n$ is again a homogeneous polynomial of degree $v_j$.

Let us first prove our claim for $p_1$. We recall that we are working in exponential coordinates. By our choice of basis (cf.~Subsection~\ref{subs:hom_gr}), the group law is given by
\begin{align}
	x y &= \exp(x_1 X_1 + \ldots + x_{\dimG} X_{\dimG}) \exp(y_1 X_1 + \ldots + y_{\dimG} X_{\dimG}) \nonumber \\
	&= \exp(R_1(x, y) X_1 + \ldots + R_{\dimG}(x, y) X_{\dimG}),
\end{align}
for some homogeneous polynomials $R_j(x, y)$, $j = 1, \ldots, \dimG$ of degree $v_j$. By the Baker-Campbell-Hausdorff formula, these polynomials satisfy $R_j(x, x^{-1}) = 0$ for all $j = 1, \ldots, \dimG$, $x \in G$.
Thanks to these properties of the polynomials $p_k^j$ and $R_j$, we can rewrite each $p_1^j$ as
\begin{align}
    \hspace{-25pt} p_1^j(y,z) &= R_j \bigl (\tau(y), \tau(z^{-1}y)^{-1} \bigr ) \nonumber \\
    &= \underset{[\alpha] + [\beta] = v_j}{\sum} c_{j, \alpha, \beta} \, \bigl ( \tau(y) \bigr )^\alpha \bigl ( \tau(z^{-1} y)^{-1} \bigr )^\beta \nonumber \\
    &= \underbrace{R_j \bigl( \tau(y), \tau(y)^{-1} \bigr )}_{=0} + \!\!\underset{\substack{[\alpha] + [\beta] = v_j, \\
    [\beta_1] + [\beta_2] + [\beta_3] = \beta, \\ [\beta_2] + [\beta_3] \neq 0}}{\overline{\sum}} \!\!\!\!\!\!\!\!\!\!\!\bigl ( \tau(y) \bigr )^\alpha \bigl ( \tau(y)^{-1} \bigr )^{\beta_1} ( \tau(z^{-1})^{-1} \bigr )^{\beta_2} \bigl (x^{-1} \bigr )^{\beta_3} \label{eq:purely_mixed_summand}
\end{align}
where the letter $x$ simply abbreviates the element $(T_1(z^{-1}, y), \ldots, T_\dimG(z^{-1}, y)) \in G$ from Lemma~\ref{lem:product_tau} which satisfies
\begin{align*}
    \tau(z^{-1}y) = \bigl ( T_1(z^{-1}, y), \ldots, T_\dimG(z^{-1}, y) \bigr ) \tau(z^{-1}) \tau(y).
\end{align*}
We recall that the coordinates $T_j(z^{-1}, y)$ are either trivial or sums of monomials $z^\alpha y^\beta$, $\alpha, \beta \neq 0$, which only depend on $z_1, \ldots z_j$ and $y_1, \ldots y_j$. Hence, the non-trivial sum in \eqref{eq:purely_mixed_summand} can be written as a sum of monomials in $z$ or of the type $z^\alpha y^\beta$, $\alpha, \beta \neq 0$. Moreover, the condition (HP) and the nilpotent group law  imply that $p_j^1$ only depends on the coefficients $y_1, \ldots, y_{j-1}$ and $z_1, \ldots, z_{j-1}$ as well as a linear term in $z_j$. Excluding the trivial case of (HP), the coefficient $p_1^j$ can therefore be expressed as
\begin{align*}
    p_1^j(y,z) = C^\tau_j z_j + \underset{\substack{[\alpha(j)] + [\beta(j)] = v_j\geq 2, \\ \alpha(j), \beta(j) \neq 0, \\ \alpha(j)_j, \ldots, \alpha(j)_\dimG = 0, \\ \beta(j)_j, \ldots, \beta(j)_\dimG = 0}}{\overline{\sum}} z^{\alpha(j)} y^{\beta(j)}
\end{align*}
for some $C^\tau_j \neq 0$, with $p_1^j(y,z)=C_j^\tau z_j$ when $v_j=1$.
In order to estimate $|p_1(y, z)|_\infty = \max_{j = 1, \ldots, \dimG} |p^j_1(y, z)|^{1/v_j}$, we use the basic estimate
\begin{align*}
    |z^{\alpha(j)}|^{\frac{1}{v_j}}&=|z_1^{\alpha(j)_1}\ldots z_n^{\alpha(j)_n}|^{\frac{1}{v_j}}=|z_1|^{\frac{1}{v_1} \frac{\alpha(j)_1 v_1}{v_j}}\ldots |z_n|^{\frac{1}{v_n} \frac{\alpha(j)_n v_n}{v_j}}\lesssim |z|_{\infty}^{\frac{[\alpha(j)]}{v_j}},
\end{align*}
which, for any $j = 1, \ldots, \dimG$, yields
\begin{align*}
   |p_1^j(y,z)|^{\frac{1}{v_j}}= \bigl | C^\tau_j z_j + \!\!\!\!\!\!\!\!\!\ \underset{\substack{[\alpha(j)] + [\beta(j)] = v_j\geq 2, \\ \alpha(j),\beta(j) \neq 0, \\ \alpha(j)_j, \ldots, \alpha(j)_\dimG = 0, \\ \beta(j)_j, \ldots, \beta(j)_\dimG = 0}}{\overline{\sum}} z^{\alpha(j)} y^{\beta(j)} \bigr |^{\frac{1}{v_j}} \leq& |C^\tau_j z_j|^{\frac{1}{v_j}} + \!\!\!\!\!\!\!\!\!\! \underset{\substack{[\alpha(j)] + [\beta(j)] = v_j\geq 2, \\ \alpha(j),\beta(j) \neq 0, \\ \alpha(j)_j, \ldots, \alpha(j)_\dimG = 0, \\ \beta(j)_j, \ldots, \beta(j)_\dimG = 0}}{\overline{\sum}} \!\!\!\!\!\!\!\!\!\!\!\!\! |z^{\alpha(j)}|^{\frac{1}{v_j}} |y^{\beta(j)}|^{\frac{1}{v_j}}\\
    \lesssim&  |z|_\infty + \underset{\substack{[\alpha] + [\beta] = v_j\geq 2, \\ \alpha,\beta \neq 0, }}{\overline{\sum}} |z|_\infty^{\frac{[\alpha]}{v_j}} |y|_\infty^{\frac{[\beta]}{v_j}}.
\end{align*}
Observe that all the sums above are taken over a finite number of multi-indices, or, equivalently, they are sums where only a finite number of coefficients are nonzero.
Finally, taking the maximum over $j$ in the previous inequality, and using that
$$\max_{j=1,\ldots,\dimG}\,\underset{\substack{[\alpha] + [\beta] = v_j\geq 2, \\ \alpha,\beta \neq 0, }}{\overline{\sum}} |z|_\infty^{\frac{[\alpha]}{v_j}} |y|_\infty^{\frac{[\beta]}{v_j}}\leq \sum_{j=1}^\dimG
\,\underset{\substack{[\alpha] + [\beta] = v_j\geq 2, \\ \alpha,\beta \neq 0, }}{\overline{\sum} }|z|_\infty^{\frac{[\alpha]}{v_j}} |y|_\infty^{\frac{[\beta]}{v_j}},$$
we conclude the estimate \eqref{eq:est_p1} for $|p_1(y,z)|_\infty$.

As for the estimate \eqref{eq:est_p2} for $|p_2(y,z)|_\infty$, it can be concluded from the proof of \eqref{eq:est_p1} by considering the function $\tilde{\tau}: G \to G$, $\tilde{\tau}(x) :=\tau(x) x^{-1}$. Unless $\tau(x) = x$ or $\tau(x) = e_G$ for all $x \in G$, in which case nothing remains to be proved, the function $\tilde{\tau}$ satisfies (HP) non-trivially, a property it clearly inherits from $\tau$. Hence, for some $j \in \{ 1, \ldots, \dimG \}$, $\tilde{\tau}$ has non-trivial coordinate functions, that is,
\begin{align*}
    \tilde{\tau}_j(x) = \tilde{C}^\tau_j x_j + \tilde{d}^\tau_j(x_1, \ldots, x_{j-1}),
\end{align*}
with $\tilde{C}^\tau_j \neq 0$. Now, by Lemma~\ref{lem:product_tau}  applied to $\tilde{\tau}$, there exist $\dimG$ polynomials $\tilde{T}_1, \ldots, \tilde{T}_\dimG: G \times G \to \R$ of the form
\begin{align*}
    \tilde{T}_j(y, z) = \underset{\substack{[\alpha(j)] + [\beta(j)] = v_j, \\ \alpha(j), \beta(j) \neq 0, \\ \alpha(j)_j, \ldots, \alpha(j)_\dimG = 0, \\ \beta(j)_j, \ldots, \beta(j)_\dimG = 0}}{\overline{\sum}} z^{\alpha(j)} y^{\beta(j)},\quad \forall j=1,\ldots,\dimG,
\end{align*}
such that, for all $y, z \in G$,
\begin{align*}
    \tilde{\tau}(yz) = \bigl ( \tilde{T}_1(y,z), \ldots, \tilde{T}_\dimG( y,z) \bigr )^{-1} \tilde{\tau}(y) \tilde{\tau}(z).
\end{align*}
The latter identity then gives
\begin{align*}
   p_2(y,z)= \tilde{\tau}(y)\tilde{\tau}(z)^{-1}&=\tilde{\tau}(yz^{-1} z) \tilde{\tau}(z)^{-1}\\
   &=\bigl (\tilde{T}_1(y z^{-1} ,z), \ldots, \tilde{T}_\dimG(y z^{-1}, z) \bigr )^{-1} \tilde{\tau}(y z^{-1}) 
   \underbrace{\tilde{\tau}(z)\, \tilde{\tau}(z)^{-1}}_{=e_G}
\end{align*}
which, as in the proof of \eqref{eq:est_p1}, implies 
\begin{align*}
    p_2^j(y,z) = C^{\tilde{\tau}}_j (y^{-1}z)_j + \underset{\substack{[\alpha(j)] + [\beta(j)] = v_j\geq 2, \\  \alpha(j),\beta(j)\neq 0, \\ \alpha(j)_j, \ldots, \alpha(j)_\dimG = 0, \\ \beta(j)_j, \ldots, \beta(j)_\dimG = 0}}{\overline{\sum}} (y^{-1}z)^{\alpha(j)} y^{\beta(j)}
\end{align*}
 for some $C^{\tilde{\tau}}_j \neq 0$, for all $j=1,\ldots, n$.
By repeating the last steps in the proof of \eqref{eq:est_p1}, we obtain the desired \eqref{eq:est_p2}.

We conclude the proof with the observation that since all homogeneous quasi-norms on $G$ are equivalent, the estimates hold true for any homogeneous quasi-norm on $G$, up to an implicit positive multiplicative constant in \eqref{eq:est_p1} and \eqref{eq:est_p2}.
\end{proof}

\begin{remark}
Note that \eqref{eq:est_p1} implies the useful estimate
\begin{align}
    |p_1(y,z)|\lesssim (|z|+|z|^{\frac{1}{v_n}}) (1+ |y^{-1}z|+|y^{-1}z|^{\frac{1}{v_n}}), \label{eq.estp1.old}
\end{align}
which we will use in the proof of Lemma \ref{lemma1}.
To see this, it suffices to observe that
\begin{align}
    |z| \lesssim (|z| + |z|^{\frac{1}{v_j}}) \quad \forall z\in G,
\end{align}
and that, for $\alpha,\beta\neq 0$ satisfying $[\alpha]+[\beta]=v_j\geq 2$, we have
   \begin{align}
       |z|^{\frac{[\alpha]}{v_j}} |y|^{\frac{[\beta]}{v_j}}\lesssim (|z| + |z|^{\frac{1}{v_j}})(|y^{-1}z| +|y^{-1}z|^{\frac{1}{v_j}}) \quad \forall y,z\in G.
   \end{align}
Hence, using the latter inequalities in \eqref{eq:est_p1}, we get \eqref{eq.estp1.old}. 
   
An analogous estimate, based on \eqref{eq:est_p2}, holds for $p_2$ with the roles of $z$ and $y^{-1}z$ exchanged in \eqref{eq.estp1.old}.
\end{remark}

\begin{remark}
Note that when $G = \R^n$ and $v_n=1$, then $1 \leq v_1 \leq \ldots \leq v_n = 1$. The condition (HP) then implies that $\tau$ is linear, hence $|p_1(y,z)|\lesssim |z|$ and $|p_2(y,z)|\lesssim |y-z|$.
\end{remark}

\medskip

We can now start proving the estimates for the remainder terms $T_1, T_2, T_3$ defined by \eqref{eq:def_T1}, \eqref{eq:def_T2} and \eqref{eq:def_T3}, respectively.

\noindent \underline{The symbol $T_1$:} We begin with the symbol $T_1$ defined by \eqref{eq:def_T1}. For technical reasons that will become obvious soon, we employ the measure-preserving change of variables $(y, z) =: (y, ys^{-1})$ to rewrite $T_1(x,\pi)$ as
\begin{align} \label{eq:chv_T_1}
\begin{split}
    T_1(x, \pi) =& \underset{\substack{[\beta] \leq N, \\ [\beta_1] + [\beta_2] = [\beta]}}{\sum} c_{\beta_1, \beta_2} \iint\limits_{G \times G} \tilde{q}_{\beta_1}(s) \tilde{q}_{\beta_2}(ys^{-1}) X^\beta_{x_2 = x} \kk(x_2, ys^{-1}) \\
    &\hspace{130pt} \times R^{\k(\, . \,, s)}_{x, M}\bigl ( p_1(y, ys^{-1}) \bigr ) \, ds \, \pi(y)^* dy.
\end{split}
\end{align}
For $T_1$ written in this form, the desired estimate is now provided by Lemma \ref{lemma1} below. Notice that the absence of the integrands $\tilde{q}_{\beta_2}$ is justified by the fact that while homogeneous polynomials in the variables of integration may improve the oscillatory behavior close to the origin, they never worsen it. To apply the same argument to the integrands $\tilde{q}_{\beta_1}$, we simply rewrite $\tilde{q}_{\beta_1}(s) = \overline{\sum}_{[\beta'_1] + [\beta''_1] = [\beta'_1]} \tilde{q}_{\beta'_1}(y)\tilde{q}_{\beta''_1}(y^{-1}s)$, due to \eqref{HomPol1}.

\begin{lemma}\label{lemma1}
Let $\sigma_1, \sigma_2$ be smooth symbols with $\k, \kk \in \SC(G \times G)$, and let $m_1, m_2\in\mathbb{R}$. For every arbitrary but fixed $M_0\in\mathbb{N}$, let $\tau_{M_0}$ be the symbol defined by
\begin{align*}
    \tau_{M_0}(x,\pi):=\iint_{G \times G}R_{x,M_0}^{\k(\cdot,s)}(p_1(y,ys^{-1})) \kappa_{\sigma_2}(x,ys^{-1})\pi(y)^*dy \,ds.
\end{align*}

Then for all $\alpha_0,\beta_0\in \mathbb{N}^{2n+1}_0$ and for all $\rho,\delta\in [0,1]$ such that $0\leq\delta<\min\{\rho,\frac{1}{v_n}\}$ there exist a constant $C>0$ and two seminorms $\|\cdot\|_{S^{m_1}_{\rho,\delta},a_1,b_1}$, $\|\cdot\|_{S^{m_2}_{\rho,\delta},a_2,b_2}$ such that
\begin{align*}
    \sup_{(x, \pi) \in G \times \Ghat}\| \bigl(X_x^{\beta_0} \Delta^{\alpha_0}\tau_{M_0}\bigr)(x,\pi) \pi(I+\mathcal{R})&^{-\frac{m_1+m_2-(\rho-\delta)(M_0+1) -\rho[\alpha_0]+\delta[\beta_0]}{\nu}}\|_{\mathscr{L}(\mathcal{H}_\pi)} \\
    &\leq C \|\sigma_1\|_{S^{m_1}_{\rho,\delta},a_1,b_1}\|\sigma_2\|_{S^{m_2}_{\rho,\delta},a_2,b_2}.
\end{align*}
\end{lemma}
\begin{proof}
We give the proof of Lemma~\ref{lemma1} in two steps. Below we will assume that $\alpha_0=0$ because the presence of a (non-vanishing) difference operator $\Delta^{\alpha_0}$, i.e., of a homogeneous polynomial $\tilde{q}_{\alpha_0}$ in the oscillatory integral~\eqref{eq2.lem.1}, renders the integral less singular close to the origin.

\medskip

\noindent\textit{Step~1.} We prove that for $m_1,m_2\in \mathbb{R}$, $\beta_0\in\mathbb{N}^{2n+1}$ and $M_0\in\mathbb{N}$ there exist $C>0$, $M=M(\beta_0) \in \N_{>M_0}$ and $a_1, b_1, a_2, b_2 \in \mathbb{N}_0$ such that
\begin{equation}\label{eq1.lem.1}
    \| X_x^{\beta_0}\tau_M(x,\pi) \pi(I+\mathcal{R})^{-\frac{m_1+m_2-(\rho-\delta)(M_0+1) +\delta[\beta_0]}{\nu}}\|_{\mathscr{L}(\mathcal{H}_\pi)}\leq C \|\sigma_1\|_{S^{m_1}_{\rho,\delta},a_1,b_1}\|\sigma_2\|_{S^{m_2}_{\rho,\delta},a_2,b_2}.
\end{equation}
\textit{Step~2.} We show that the estimate in Step~1 also holds true for $M=M_0$ and arbitrary $\beta_0$. 

\vspace{0.5cm}

\noindent\textit{Proof of Step~1.} 
We denote by $m:=(m_1+m_2-(\rho-\delta)(M_0+1)+\delta[\beta_0])/\nu$, and take $M_1\in\mathbb{N}$ as the smallest nonnegative integer such that $M_1\geq-m $. With this choice of $M_1=M_1(\beta_0, M_0)$ the operator $(I+\mathcal{R})^{-m-M_1}$ is bounded, and we can reduce the proof of \eqref{eq1.lem.1} to that of
\begin{equation*}\label{eq1.lem.1.1}
    \| X_x^{\beta_0}\tau_M(x,\pi) \pi(I+\mathcal{R})^{M_1}\|_{\mathscr{L}(\mathcal{H}_\pi)}\leq C \|\sigma_1\|_{S^{m_1}_{\rho,\delta},a_1,b_1}\|\sigma_2\|_{S^{m_2}_{\rho,\delta},a_2,b_2}.
\end{equation*}
Using integration by parts, Lemma~\ref{lem:pvf} and \eqref{Prop.Rem-vf}, we get
\begin{align}
    &\tau_M(x,\pi)\pi(I+\mathcal{R})^{M_1}\nonumber\\
    &=\iint_{G \times G}\kk(x,ys^{-1}) R_{x,M}^{\k(\cdot,s)}(p_1(y,ys^{-1})) \overline{\sum}_{[\beta]\leq \nu M_1}(-1)^{|\beta|}(\tilde{X}^\beta \pi^*)(y)ds\, dy\nonumber\\
   &= \!\!\!\!\underset{[\beta]\leq \nu M_1}{\overline{\sum}} 
    \,\,\underset{[\beta_1]+[\beta_2]=[\beta]}{\overline{\sum}}
   \iint_{G \times G} \!\!\!\! \tilde{X}_{y_2=y}^{\beta_1}\kk(x,y_2s^{-1})\tilde{X}_{y_1=y}^{\beta_2}R_{x,M}^{\k(\cdot,s)}(p_1(y_1,y_1s^{-1}))\pi^*(y)ds\, dy \nonumber\\
    &= \underset{[\beta]\leq \nu M_1}{\overline{\sum}} 
    \,\,\underset{[\beta_1]+[\beta_2]=[\beta]}{\overline{\sum}}
    \,\, \underset{\substack{[\beta_2']+[\beta_2'']=[\beta_2^0]-[\beta_2]\\ [\beta_2]\leq |\beta_2^0|\leq v_\dimG|\beta_2|}}{\overline{\sum}}
   \iint_{G \times G}
   \tilde{X}_{y_2=ys^{-1}}^{\beta_1}\kk(x,y_2)
   \tilde{q}_{\beta_2'}(s)\tilde{q}_{\beta_2''}(ys^{-1})\nonumber\\
   & \hspace{7,5cm} \times \tilde{X}_{y_1 = p_1(y,ys^{-1})}^{\beta_2^0}R_{x,M}^{\k(\cdot,s) }(y_1)\pi^*(y)ds\, dy\nonumber\\
   &= \underset{[\beta]\leq \nu M_1}{\overline{\sum}} 
    \,\,\underset{[\beta_1]+[\beta_2]=[\beta]}{\overline{\sum}}
     \,\,\underset{\substack{[\beta_2']+[\beta_2'']=[\beta_2^0]-[\beta_2]\\ [\beta_2]\leq |\beta_2^0|\leq v_\dimG|\beta_2|}}{\overline{\sum}}
   \iint_{G \times G}
   \tilde{X}_{y_2=ys^{-1}}^{\beta_1}\kk(x,y_2)\tilde{q}_{\beta_2'}(s)\tilde{q}_{\beta_2''}(ys^{-1})\nonumber\\
   & \hspace{7,5cm} \times R_{0,M-[\beta_2^0]}^{X_x^{\beta_2^0}\k(x\cdot,s) }(p_1(y,ys^{-1}))\pi^*(y)ds\, dy.\label{eq2.lem.1}
\end{align} 
Note that we applied Lemma \ref{lem:pvf} in \eqref{eq2.lem.1} to rewrite the differentiation in the variable $s$ as a differentiation in the variable $y_1 = p_1(y,ys^{-1})$. This was possible since $p_1(y,ys^{-1})$ is again a function from $G\times G$ to $G$ with coefficients that are homogeneous polynomials in $y$ and $s$.

Next we apply $X^{\beta_0}_x$ to $\tau_M(x,\pi)\pi(I+\mathcal{R})^{M_1}$, that is, to \eqref{eq2.lem.1}. Since $\alpha_0=0$ and $\beta_0 \in \N^\dimG$, integration by parts and \eqref{Prop.Rem-vf} yield
\begin{align}
    X^{\beta_0}_x \tau_M(x,\pi)\pi(I+\mathcal{R})^{M_1}=&
    \underset{[\beta]\leq \nu M_1}{\overline{\sum}} 
     \,\,\underset{[\beta_1]+[\beta_2]=[\beta]}{\overline{\sum}}
     \,\,\underset{\substack{[\beta_2']+[\beta_2'']=[\beta_2^0]-[\beta_2]\\ [\beta_2]\leq |\beta_2^0|\leq v_\dimG|\beta_2|}}{\overline{\sum}}
    \,\, \underset{[\beta_{0,1}]+[\beta_{0,2}]=[\beta_0]}{\overline{\sum}}\nonumber\\
   \iint_{G \times G} &
   X^{\beta_{0,1}}_{x_2=x}\tilde{X}_{y_2=ys^{-1}}^{\beta_1}\kk(x_2,y_2)\nonumber
   \tilde{q}_{\beta_2'}(s)\tilde{q}_{\beta_2''}(ys^{-1})\\
   \times\, & R_{0,M-[\beta_2^0]}^{X^{\beta_{0,2}}_{x}X_x^{\beta_2^0}\k(x\cdot,s) }(p_1(y,ys^{-1}))\pi^*(y)ds\, dy \label{eq2.1.lem.1}.
\end{align}
Note that the growth of $|p_1|$ at the origin, brought into play via the Taylor estimate, will help us control the behavior of $\kk$ in \eqref{eq2.1.lem.1}. In order to control the behavior of $\k$, after the Taylor estimate, we change the order of its symbol. We do so by using the identity
\begin{align*}
    f=(I+\mathcal{R})^{M_2}(f\ast \mathcal{B}_{\nu M_2})=\underset{[\eta]\leq \nu M_2}{\overline{\sum}}X^\eta (f\ast \mathcal{B}_{\nu M_2}), \hspace{5pt} f \in \SC(G),
\end{align*}
where $\mathcal{B}_{\nu M_2}$ is the convolution kernel of the Bessel potential introduced in Subsection~\ref{subs:Bessel} and $M_2$ a positive integer to be chosen later. The trade-off for introducing the derivative $X^\eta$ into \eqref{eq2.1.lem.1}, however, namely $x$-derivatives of $\k$, will need to be carefully counterbalanced by the order $M$ of the Taylor expansion and the order $M_2$ of the Bessel potential. So, by employing the above identity, we rewrite
\begin{align*}
   R_{0,M-[\beta_2^0]}^{X^{\beta_{0,2}}_{x}X_x^{\beta_2^0}\k(x\cdot,s) }(p_1(y,ys^{-1}))=
   \underset{[\eta]\leq \nu M_2}{\overline{\sum}}
   X^\eta_{s_1=s} R_{0,M-[\beta_2^0]}^{(X^{\beta_{0,2}}_{x}X_x^{\beta_2^0}\k(x\cdot,\cdot)\ast \mathcal{B}_{\nu M_2}(\cdot))(s_1) }(p_1(y,ys^{-1}))
\end{align*}
in \eqref{eq2.1.lem.1}, and integrating by parts, we obtain
\begin{align}
    &\hspace{-80pt} \eqref{eq2.lem.1}=
    \underset{[\beta]\leq \nu M_1}{\overline{\sum}} 
     \,\,\underset{[\beta_1]+[\beta_2]=[\beta]}{\overline{\sum}}
     \,\,\underset{\substack{[\beta_2']+[\beta_2'']=[\beta_2^0]-[\beta_2]\\ [\beta_2]\leq |\beta_2^0|\leq v_\dimG|\beta_2|}}{\overline{\sum}}
    \,\, \underset{[\beta_{0,1}]+[\beta_{0,2}]=[\beta_0]}{\overline{\sum}}\nonumber\\
    \underset{[\eta_1]+\cdots+[\eta_4]\leq \nu M_2}{\overline{\sum}}\iint_{G \times G} &
    X^{\eta_1}_{s_1=s}\tilde{q}_{\beta_2'}(s_1)X^{\eta_2}_{s_2=s}\tilde{q}_{\beta_2''}(ys_2^{-1})
   X^{\eta_3}_{s_3=s}X^{\beta_{0,1}}_{x_2=x}\tilde{X}_{y_2=ys_3^{-1}}^{\beta_1}\kk(x_2,y_2)\nonumber
   \\
    &\hspace{-30pt} \times\, X^{\eta_4}_{s_4=s}R_{0,M-[\beta_2^0]}^{(X^{\beta_{0,2}}_{x}X_x^{\beta_2^0}\k(x\cdot,\cdot)\ast \mathcal{B}_{\nu M_2}(\cdot))(s) }(p_1(y,ys_4^{-1}))
   \pi^*(y)ds\, dy \label{eq2.2.lem.1}.
\end{align}
We will now rewrite the derivatives in $s$ in a convenient way. Since
\begin{align*}
    X_j|_s f(ys^{-1}) &= \partial_{t = 0} f \bigl ( y (s \exp(t X_j))^{-1} \bigr ) =  \partial_{t = 0} f \bigl ( y  \exp(-t X_j)s^{-1} \bigr )\\
    =&-\tilde{X}_j|_{s^{-1}}f(ys^{-1}), 
\end{align*}
the relation between left and right-invariant vector fields yields
\begin{align*}
X_j|_s f(ys^{-1})=-\tilde{X}_j|_{s^{-1}}f(ys^{-1})=-\sum_{\substack{|\beta| > 1 \\ [\beta] \geq v_j}}Q_{v_j, \beta}(s^{-1}) X^\beta_{s^{-1}} f(ys^{-1}),
\end{align*}
with $Q_{v_j, \beta}$ being a homogeneous polynomial of homogeneous degree $[\beta]-v_j$.
Inductively, for every multi-index $\alpha\in \mathbb{N}^{n}_0$, we obtain
\begin{equation}\label{formula.inv.diff.}
    X^\alpha_{s^{-1}}f(ys^{-1})=(-1)^{|\alpha|}\sum_{\substack{|\beta|\leq |\alpha|\\ [\beta]\geq[\alpha]}}Q_{\alpha,\beta}(s^{-1}) X^\beta_{s^{-1}} f(ys^{-1}),
\end{equation}
with $Q_{\alpha,\beta}$ being a homogeneous polynomial of homogeneous degree $[\beta]-[\alpha]$. For the sake of convenience, we will write $Q_{\alpha,\beta}(s)$ instead of $Q_{\alpha,\beta}(s^{-1})$.

We now use \eqref{formula.inv.diff.} to write
\begin{align}
   X^{\eta_3}_{s_3=s}X^{\beta_{0,1}}_{x_2=x}\tilde{X}_{y_2=ys_3^{-1}}^{\beta_1}\kk(x_2,y_2)=&X^{\beta_{0,1}}_{x_2=x}X^{\eta_3}_{s_3=s}\tilde{X}_{y_2=ys_3^{-1}}^{\beta_1}\kk(x_2,y_2) \nonumber\\
   &\hspace{-40pt} =(-1)^{|\eta_3|}\sum_{\substack{|\eta_3'|\leq |\eta_3|\\ [\eta_3']\ge[\eta_3]}}Q_{\eta_3,\eta_3'}(s) X^{\beta_{0,1}}_{x_2=x}X^{\eta_3'}_{s^{-1}}\tilde{X}_{y_2=ys^{-1}}^{\beta_1}\kk(x_2,y_2), \nonumber
   \nonumber\\
   &\hspace{-40pt} =(-1)^{|\eta_3|}\sum_{\substack{|\eta_3'|\leq |\eta_3|\\ [\eta_3']\ge[\eta_3]}}Q_{\eta_3,\eta_3'}(s) X^{\beta_{0,1}}_{x_2=x}X^{\eta_3'}_{y_2=ys^{-1}}\tilde{X}_{y_2}^{\beta_1}\kk(x_2,y_2), \nonumber
\end{align}
and
\begin{align*}
   X^{\eta_2}_{s_2=s} \tilde{q}_{\beta_2''}(ys_2^{-1})=(-1)^{|\eta_2|}\sum_{\substack{|\eta_2'|\leq |\eta_2|\\ [\eta_2']\ge[\eta_2]}}Q_{\eta_2,\eta_2'}(s)(X^{\eta_2'}\tilde{q}_{\beta_2''})(ys^{-1}),
\end{align*}
which, by substitution in \eqref{eq2.2.lem.1}, give
\begin{align}
    \eqref{eq2.lem.1}=&
    \underset{[\beta]\leq \nu M_1}{\overline{\sum}} 
     \,\,\underset{[\beta_1]+[\beta_2]=[\beta]}{\overline{\sum}}
     \,\,\underset{\substack{[\beta_2']+[\beta_2'']=[\beta_2^0]-[\beta_2]\\ [\beta_2]\leq |\beta_2^0|\leq v_\dimG|\beta_2|}}{\overline{\sum}}
    \,\, \underset{[\beta_{0,1}]+[\beta_{0,2}]=[\beta_0]}{\overline{\sum}}
    \,\,\underset{\substack{[\eta_1]+\cdots+[\eta_4] \\ \leq \nu M_2}}{\overline{\sum}}
     \,\,\underset{\substack{|\eta_2'|\leq |\eta_2|\\ [\eta_2']\geq[\eta_2]}}{\overline{\sum}} \,\,\underset{\substack{|\eta_3'|\leq |\eta_3|\\ [\eta_3']\geq[\eta_3]}}{\overline{\sum}}
    \nonumber\\
    &\hspace{-40pt} \iint_{G \times G} \!\!
    Q_{\eta_2,\eta_2'}(s)Q_{\eta_3,\eta_3'}(s)(X^{\eta_1}\tilde{q}_{\beta_2'})(s)(X^{\eta_2'}\tilde{q}_{\beta_2''})(ys^{-1})
    X^{\beta_{0,1}}_{x_2=x}X^{\eta_3'}_{y_2=ys^{-1}}\tilde{X}_{y_2}^{\beta_1}\kk(x_2,y_2)\nonumber
   \\
    &\hspace{2.5cm} \times\, X^{\eta_4}_{s_4=s}R_{0,M-[\beta_2^0]}^{(X^{\beta_{0,2}}_{x}X_x^{\beta_2^0}\k(x\cdot,\cdot)\ast \mathcal{B}_{\nu M_2}(\cdot))(s) }(p_1(y,ys_4^{-1}))
   \pi^*(y)ds\, dy. \nonumber
\end{align}
We will now focus on the last factor. Using Lemma~\ref{lem:pvf} in combination with \eqref{Prop.Rem-vf} in the last term of \eqref{eq2.lem.1}, gives
\begin{align}
   & \eqref{eq2.lem.1}=
    \underset{[\beta]\leq \nu M_1}{\overline{\sum}} 
     \,\,\underset{[\beta_1]+[\beta_2]=[\beta]}{\overline{\sum}}
     \,\,\underset{\substack{[\beta_2']+[\beta_2'']=[\beta_2^0]-[\beta_2]\\ [\beta_2]\leq |\beta_2^0|\leq v_\dimG|\beta_2|}}{\overline{\sum}}
    \,\, \underset{[\beta_{0,1}]+[\beta_{0,2}]=[\beta_0]}{\overline{\sum}}
    \,\,\underset{\substack{[\eta_1]+\cdots+[\eta_4] \\ \leq \nu M_2}}{\overline{\sum}}
     \,\,\underset{\substack{|\eta_2'|\leq |\eta_2|\\ [\eta_2']\ge[\eta_2]}}{\overline{\sum}}
     \,\,\underset{\substack{|\eta_3'|\leq |\eta_3|\\ [\eta_3']\ge[\eta_3]}}{\overline{\sum}}
    \nonumber\\
    &\underset{\substack{[\eta_4']+[\eta_4''] = [\eta_4^0] - [\eta_4] \\ [\eta_4] \leq |\eta_4^0| \leq v_\dimG |\eta_4|}}{\overline{\sum}}
    \iint_{G \times G} 
    \tilde{q}_{\eta_4'}(s)\tilde{q}_{\eta_4''}(ys^{-1})Q_{\eta_2,\eta_2'}(s)Q_{\eta_3,\eta_3'}(s)(X^{\eta_1}\tilde{q}_{\beta_2'})(s)(X^{\eta_2'}\tilde{q}_{\beta_2''})(ys^{-1})\nonumber
   \\
   &\times\,  X^{\beta_{0,1}}_{x_2=x}X^{\eta_3'}_{y_2=ys^{-1}}\tilde{X}_{y_2}^{\beta_1}\kk(x_2,y_2)\,R_{0,M-[\beta_2^0]-[\eta_4^0]}^{(X^{\eta_4^0}_{z=0}X^{\beta_{0,2}}_{x}X_x^{\beta_2^0}\k(xz,\cdot)\ast \mathcal{B}_{\nu M_2}(\cdot))(s) }(p_1(y,ys^{-1})) \\
   &\hspace{250pt} \times \pi^*(y)ds\, dy. \label{lem1.eqn.3}
\end{align}
So, taking the operator norm on $\RS$, we bound this integral by
\begin{align}
    &\| X^{\beta_0}_x \tau_M(x,\pi)\pi(I+\mathcal{R})^{M_1}\|_{\mathscr{L}(\mathcal{H}_\pi)} \nonumber\\
     &\lesssim\underset{[\beta]\leq \nu M_1}{\overline{\sum}} 
     \,\,\underset{[\beta_1]+[\beta_2]=[\beta]}{\overline{\sum}}
     \,\,\underset{\substack{[\beta_2']+[\beta_2'']=[\beta_2^0]-[\beta_2]\\ [\beta_2]\leq |\beta_2^0|\leq v_\dimG|\beta_2|}}{\overline{\sum}}
    \,\, \underset{[\beta_{0,1}]+[\beta_{0,2}]=[\beta_0]}{\overline{\sum}}
    \,\,\underset{[\eta_1]+\cdots+[\eta_4]\leq \nu M_2}{\overline{\sum}}
     \,\,\underset{\substack{|\eta_2'|\leq |\eta_2|\\ [\eta_2']\ge[\eta_2]}}{\overline{\sum}}
     \,\,\underset{\substack{|\eta_3'|\leq |\eta_3|\\ [\eta_3']\ge[\eta_3]}}{\overline{\sum}}
    \nonumber\\
    &\underset{\substack{[\eta_4']+[\eta_4''] = [\eta_4^0] - [\eta_4] \\ [\eta_4] \leq |\eta_4^0| \leq v_\dimG |\eta_4|}}{\overline{\sum}}
    \iint_{G \times G}
     |ys^{-1}|^{[\beta_2'']+[\eta_4'']-[\eta_2']}|X^{\beta_{0,1}}_{x_2=x}X^{\eta_3'}_{y_2=ys^{-1}}\tilde{X}_{y_2}^{\beta_1}\kk(x_2,y_2)| \label{lem1.eqn.rem1}\\
    &\times |s|^{[\beta_2']+[\eta_2'] + [\eta_3'] + [\eta_4']-[\eta_1]-[\eta_2]-[\eta_3]}
   \, |R_{0,M-[\beta_2^0]-[\eta_4^0]}^{(X^{\eta_4^0}_{z=0}X^{\beta_{0,2}}_{x}X_x^{\beta_2^0}\k(xz,\cdot)\ast \mathcal{B}_{\nu M_2}(\cdot))(s) }(p_1(y,s))|
   ds\, dy. \nonumber
\end{align}
The estimates for the remainder of Taylor's expansion now gives
\begin{align}
    &|R_{0,M-[\beta_2^0]-[\eta_4^0]}^{X^{\eta_4^0}_{s_2}(X^{\beta_{0,2}}_{x}X_x^{\beta_2^0}\k(x\cdot,\cdot)\ast \mathcal{B}_{\nu M_2}(\cdot))(s_2) }(p_1(y,ys^{-1}))| \lesssim \!\!\!\!\!\!\!\!\!\!\!\! \underset{\substack{((M-[\beta_0^2])_++v_\dimG>[\gamma]\geq (M-[\beta_0^2])_+}}{\overline{\sum}} \!\!\!\!\!\!|p_1(y,ys^{-1})|^{[\gamma]} \nonumber\\
   &\hspace{40pt} \times \sup_{|z|\leq \eta ^{\lceil (M-[\beta_2^0]-[\eta_4^0])_+\rfloor+1}|p_1(y,ys^{-1})|} \!\!\!\!\!\! |X^\gamma_{z}(X^{\eta_4^0}_{z_1=z}X^{\beta_{0,2}}_{x}X_x^{\beta_2^0}\k(xz_1,\cdot)\ast \mathcal{B}_{\nu M_2}(\cdot))(s )|\nonumber\\
   &\hspace{20pt} \lesssim \underset{\substack{((M-[\beta_2^0]-[\eta_4^0])_++v_\dimG>[\gamma]\geq (M-[\beta_2^0]-[\eta_4^0])_+\\
    }}{\overline{\sum}} \underset{\substack{|\gamma_0|\leq |\beta_{0,2}|+|\beta_{2}^0|\\
    [\gamma_0]= [\beta_{0,2}]+[\beta_{2}^0]}}{\overline{\sum}} \underset{\substack{|\gamma_0'|\leq |\gamma_0|\\
    [\gamma_0']\geq [\gamma_0]}}{\overline{\sum}} |p_1(y,ys^{-1})|^{[\gamma]} \\
    &\hspace{40pt} \times \sup_{|z|\leq \eta ^{\lceil (M-[\beta_2^0]-[\eta_4^0])_+\rfloor+1}|p_1(y,ys^{-1})|} \!\!\!\!\!\! |Q_{\gamma_0,\gamma_0'}(z)(X^\gamma_z X^{\eta_4^0}_z X_z^{\gamma_0'}\k(xz,\cdot)\ast \mathcal{B}_{\nu M_2}(\cdot))(s )|, \label{lem1.eqn2}
\end{align}
where in the last inequality we have used \eqref{l.i.prod} to rewrite $X^{\beta_{2,0}}_x X^{\beta_2^0}_x$ as a sum of left invariant vector fields, and \eqref{RelationsVFs2} to rewrite the latter in terms of left-invariant vector fields in the variable $z$.

Note that the pointwise estimate \eqref{eq.estp1.old} of $|p_1|$, based on \eqref{eq:est_p1}, allows us to bound each term in the sum above by
\begin{align}
 & |p_1(y,ys^{-1})|^{[\gamma]} \!\!\!\!\!\! \sup_{|z|\leq \eta ^{\lceil (M-[\beta_2^0]-[\eta_4^0])_+\rfloor+1}p_1(y,ys^{-1})} \!\!\!\!\!\! |Q_{\gamma_0,\gamma_0'}(z) X^\gamma_z X^{\eta_4^0}_z X_z^{\gamma_0'}\k(xz,\cdot)\ast \mathcal{B}_{\nu M_2}(\cdot))(s )| \\
 &\lesssim (|ys^{-1}|^\frac{1}{v_\dimG}+|ys^{-1}|)^{[\gamma]+[\gamma_0']-[\gamma_0]}(1+|s|^\frac{1}{v_\dimG}+|s|)^{[\gamma]+[\gamma_0']-[\gamma_0]}\\
 &\quad \times \sup_{x\in G}|(X^\gamma_x X^{\eta_4^0}_x X_x^{\gamma_0'}\k(x,\cdot)\ast \mathcal{B}_{\nu M_2}(\cdot))(s )|. \,\label{lem1.eqn4}
\end{align}
Therefore, using \eqref{lem1.eqn4} in \eqref{lem1.eqn2}, and replacing the resulting estimate in \eqref{lem1.eqn.rem1}, we have
\begin{align}
 &\| X^{\beta_0}_x \tau_M(x,\pi)\pi(I+\mathcal{R})^{M_1}\|_{\mathscr{L}(\mathcal{H}_\pi)} \nonumber\\
     &\lesssim\underset{[\beta]\leq \nu M_1}{\overline{\sum}} 
     \,\,\underset{[\beta_1]+[\beta_2]=[\beta]}{\overline{\sum}}
     \,\,\underset{\substack{[\beta_2']+[\beta_2''] = [\beta_2^0] -  [\beta_2] \\ [\beta_2] \leq |\beta_2^0|\leq v_\dimG[\beta_2]}}{\overline{\sum}}
    \,\, \underset{[\beta_{0,1}]+[\beta_{0,2}]=[\beta_0]}{\overline{\sum}}
    \,\,\underset{[\eta_1]+\cdots+[\eta_4]\leq \nu M_2}{\overline{\sum}}
     \,\,\underset{\substack{|\eta_2'|\leq |\eta_2|\\ [\eta_2']\ge[\eta_2]}}{\overline{\sum}}
     \,\,\underset{\substack{|\eta_3'|\leq |\eta_3|\\ [\eta_3']\ge[\eta_3]}}{\overline{\sum}}
    \nonumber\\
    &\underset{\substack{[\eta_4']+[\eta_4''] = [\eta_4^0] - [\eta_4] \\ [\eta_4] \leq |\eta_4^0| \leq v_\dimG |\eta_4|}}{\overline{\sum}}\,\,
    \underset{\substack{((M-[\beta_2^0]-[\eta_4^0])_++v_\dimG>[\gamma]\geq ((M-[\beta_2^0]-[\eta_4^0])_+\\
    |\gamma|< \lceil (M-[\beta_2^0]-[\eta_4^0]})_+\rfloor+1}{\overline{\sum}} \,\,
    \underset{\substack{|\gamma_0|\leq |\beta_{0,2}|+|\beta_{2}^0|\\
    [\gamma_0]= [\beta_{0,2}]+[\beta_{2}^0]}}{\overline{\sum}}\,\,
    \underset{\substack{|\gamma_0'|\leq |\gamma_0|\\
    [\gamma_0']\geq [\gamma_0]}}{\overline{\sum}} \nonumber \\
    &\iint_{G \times G}
     |ys^{-1}|^{[\beta_2'']+[\eta_4'']-[\eta_2']}\sup_{x\in G}|X^{\beta_{0,1}}_{x_2=x}X^{\eta_3'}_{y_2=ys^{-1}}\tilde{X}_{y_2}^{\beta_1}\kk(x_2,y_2)| \, |s|^{[\beta_2']+[\eta_2']+[\eta_3']}\nonumber\\
    &\times |s|^{[\eta_4']-[\eta_1]-[\eta_2]-[\eta_3]}
    (|ys^{-1}|^\frac{1}{v_\dimG}+|ys^{-1}|)^{[\gamma]+[\gamma_0']-[\gamma_0]}(1+|s|^\frac{1}{v_\dimG}+|s|)^{[\gamma]+[\gamma_0']-[\gamma_0]}\nonumber\\
 &\quad \times \sup_{x\in G}|(X^\gamma_x X^{\eta_4^0}_x X_x^{\gamma_0'}\k(x,\cdot)\ast \mathcal{B}_{\nu M_2}(\cdot))(s )|\, ds\,dy.\label{eqn 1711-1104}
\end{align}
Now because of
$$(\k(x,\cdot)\ast \mathcal{B}_{\nu M_2}(\cdot))(s)=\kappa_{\pi(I+\mathcal{R})^{-M_2}\sigma_1}(x,s)$$
and the left invariance of the Haar measure, we can rewrite the integral in \eqref{eqn 1711-1104} as the product of 
\begin{align}
    \int_{G}
     |y|^{[\beta_2'']+[\eta_4'']-[\eta_2']}(|y|^\frac{1}{v_\dimG}+|y|)^{[\gamma]+[\gamma_0']-[\gamma_0]}\sup_{x\in G}|X^{\beta_{0,1}}_{x_2=x}X^{\eta_3'}_{y}\tilde{X}_{y}^{\beta_1}\kk(x_2,y)|  \,dy, \label{eqn1711-1125}
     \end{align}
     and 
     \begin{align}
     &\int_{G}
     |s|^{[\beta_2']+[\eta_2']+[\eta_3']+[\eta_4']-[\eta_1]-[\eta_2]-[\eta_3]} (1+|s|^\frac{1}{v_\dimG}+|s|)^{[\gamma]+[\gamma_0']-[\gamma_0]}\nonumber\\
     &\hspace{4cm}\times \sup_{x\in G}|X^\gamma_x X^{\eta_4^0}_x X_x^{\gamma_0'}\kappa_{\pi(I+\mathcal{R})^{-M_2}\sigma_1}(x,s)|\, ds,\label{eqn1711-1132}
\end{align}
and estimate each of the two integrals via suitable $S^{m_1}_{\rho,\delta}$- and $S^{m_2}_{\rho,\delta}$-seminorms of $\sigma_1$ and $\sigma_2$, respectively. 
To obtain the desired estimates, we will make use of Lemma \ref{FRcorollary} in a suitable way.

First, we split the integrals in \eqref{eqn1711-1125} and \eqref{eqn1711-1132} into the sum of integrals over the regions $\{x\in G, |x|<1\}$ and $\{x\in G, |x|\geq1\}$, with $x=y$ in \eqref{eqn1711-1125} and $x=s$ in \eqref{eqn1711-1132}.
By using the Schwartz decay of kernels of pseudo-differential operators away from the origin together with the inequalities
$$|x|^\frac{1}{v_\dimG}+|x|\leq 2|x|^\frac{1}{v_\dimG},\quad1+|x|^\frac{1}{v_\dimG}+|x|\leq 3,\quad  \text{when}\quad |x|<1,$$
we have
\begin{align}
    \eqref{eqn1711-1125}\lesssim \|\sigma_2\|_{S^{m_2}_{\rho,\delta, a_2,b_2}}+
     \int_{G}
     |y|^{[\beta_2'']+[\eta_4'']-[\eta_2']}|y|^\frac{[\gamma]+[\gamma_0']-[\gamma_0]}{v_\dimG}\sup_{x\in G}|X^{\beta_{0,1}}_{x_2=x}X^{\eta_3'}_{y}\tilde{X}_{y}^{\beta_1}\kk(x_2,y)|dy, \quad \quad \label{eqn1711-1803}
\end{align}
and, since $-M_2<0$,
\begin{align}
  \eqref{eqn1711-1132}  \lesssim \|\sigma_1\|_{S^{m_1}_{\rho,\delta, a_1,b_1}}+
     \int_{G}&|s|^{[\beta_2']+[\eta_2']+[\eta_3']+[\eta_4']-[\eta_1]-[\eta_2]-[\eta_3]} \\
     & \times \sup_{x\in G}|X^\gamma_x X^{\eta_4^0}_x X_x^{\gamma_0'}\kappa_{\pi(I+\mathcal{R})^{-M_2}\sigma_1}(x,s)|\, ds.
\label{eqn1711-1816}
\end{align}

A crucial aspect to stress is that we have nonnegative exponents in the integrals above, that is, $[\beta_2'']+[\eta_4'']-[\eta_2']\geq 0$ and  $[\beta_2']+[\eta_2']+[\eta_3']+[\eta_4']-[\eta_1]-[\eta_2]-[\eta_3]\geq 0$, since they come from differentiations of homogeneous polynomials of nonnegative degree. 
By the previous considerations, the desired estimate for \eqref{eqn1711-1125} and \eqref{eqn1711-1132} can be reached by estimating \eqref{eqn1711-1803} and \eqref{eqn1711-1816} in terms of suitable seminorms, that is, by using Lemma \ref{FRcorollary}.
In order to apply the lemma to the integrals in \eqref{eqn1711-1803} and in \eqref{eqn1711-1816}, the following conditions have to be satisfied simultaneously:
\begin{align}
    [\beta_2'']+[\eta_4'']-[\eta_2'] +\frac{[\gamma]+[\gamma_0']-[\gamma_0]}{v_\dimG}+Q& > \max \left \{ \frac{Q+m_2+\delta[\beta_{0,1}]+[\beta_1]+[\eta_3']}{\rho}, 0 \right \}, \label{lem.1.cond-2_1} \\
  [\beta_2']+[\eta_2']+[\eta_3']+[\eta_4']-[\eta_1]-[\eta_2]-[\eta_3] &+Q \\
    & \hspace{-30pt} > \max \left \{ \frac{Q+m_1-\hdeg M_2+\delta([\gamma_0']+[\gamma] + [\eta_4^0])}{\rho}, 0 \right \}, \label{lem.1.cond-2_2}
\end{align}
where $M_2 \in \N_0$ is still to be chosen.
To ensure this, we have a closer look at the range of the summation index $\gamma \in N_0^\dimG$, given by
\begin{align}
    ((M-[\beta_2^0]-[\eta_4^0])_++v_\dimG >[\gamma] &\geq ((M-[\beta_2^0]-[\eta_4^0])_+, \label{eq:range_gamma_T1} \\
    |\gamma| &< \lceil (M-[\beta_2^0]-[\eta_4^0])_+\rfloor+1,
\end{align}
and show that \eqref{lem.1.cond-2_1} and \eqref{lem.1.cond-2_2} are satisfied for suitable choices of $M, M_2 \in \N_0$ for the two cases when $M-[\beta_2^0]-[\eta_4^0]$ is either $\leq 0$ or $> 0$.

\medskip

\noindent \underline{$M-[\beta_2^0]-[\eta_4^0] \leq 0$:} 
Since the left-hand sides in the inequalities above are always positive, here we show that the right-hand sides of \eqref{lem.1.cond-2_1} and \eqref{lem.1.cond-2_2} are zero for suitable choices of $M$ and $M_2$, which will trivially prove the validity of these conditions.
For \eqref{lem.1.cond-2_1}, we first observe that the lower bound
\begin{align}
     0 \geq M-[\beta_2^0]-[\eta_4^0] \geq M - v_\dimG [\beta_2] -v_\dimG [\eta_4^0] = M -v_\dimG (\hdeg M_1 - [\beta_1]) - v_\dimG \hdeg M_2
\end{align}
provides the useful upper bound
\begin{align}
     [\beta_1] \leq -\frac{M - v_\dimG \hdeg M_1 - v_\dimG \hdeg M_2}{v_\dimG}.
\end{align}
In combination with $[\eta'_3] \leq v_\dimG [\eta_3] \leq v_\dimG \hdeg M_2$, it gives
\begin{align}
   \hdim + m_2 + \delta [\beta_{0, 1}] &+ [\beta_1] + [\eta'_3]
   \leq \hdim + m_2 + \delta [\beta_0] - \frac{M - v_\dimG \hdeg M_1 - v_\dimG \hdeg M_2}{v_\dimG} + v_\dimG \hdeg M_2 \\
   &= \hdim + m_2 + \delta [\beta_0] + \frac{-M + v_\dimG \hdeg M_1 + v_\dimG (1 + v_\dimG)M_2}{v_\dimG}. \label{eqn1811-1419_1_1}
\end{align}

On the other hand, we observe that $\rho$ times the right-hand side of \eqref{lem.1.cond-2_2} can be bounded from above by
\begin{align}
    \hdim+m_1-\hdeg M_2 +\delta([\gamma_0']&+[\gamma] + [\eta_4^0])
    \leq \hdim+m_1-\hdeg M_2+\delta(v_n[\beta_{0, 2}] + v_n [\beta^0_2] + [\gamma] + [\eta_4^0]) \\
    &< \hdim +m_1-\hdeg M_2 + \delta (v_n[\beta_{0, 2}] + v_n[\beta^0_2] + v_\dimG + [\eta_4^0]) \\
    &\leq \hdim + m_1 - \hdeg M_2 + \delta (v_n[\beta_0] + v_n^2 \hdeg M_1 + v_\dimG + v_\dimG \hdeg M_2) \\
    &= \hdim +m_1- ((1 - \delta v_\dimG) \hdeg M_2) + \delta (v_n [\beta_0] + v_\dimG + v_n^2 \hdeg M_1) \label{eqn1811-1419_1_2},
\end{align}
where we have used $[\beta_2^0] \leq v_n |\beta_2^0| \leq v_n [\beta_2^0] \leq v_n [\beta_2] \leq v_n [\beta] \leq v_n \hdeg M_1$ and $[\gamma] \leq v_n |\gamma| \leq v_n$. Since the left-hand sides of \eqref{lem.1.cond-2_1} and \eqref{lem.1.cond-2_2} are bounded from below by $\hdim$ because of 
\begin{align}
    [\beta_2'']+[\eta_4'']-[\eta_2']+\frac{[\gamma]+[\gamma_0']-[\gamma_0]}{v_\dimG} &\geq 0, \\
    [\beta_2']+[\eta_2']+[\eta_3']+[\eta_4']-[\eta_1]-[\eta_2]-[\eta_3] &\geq 0,
\end{align}
the condition $v_\dimG \delta < 1$ allows us to first pick $M_2 = M_2(\beta_0, M_1) \in \N_0$ large enough so that \eqref{lem.1.cond-2_2} is less than zero, and then pick $M = M(\beta_0, M_1, M_2) \in \N_0$\footnote{All the other parameters are fixed.} large enough to ensure that \eqref{eqn1811-1419_1_1} is less than zero. This implies that the right-hand sides of \eqref{lem.1.cond-2_1} both less than zero and, a fortiori, less than $\hdim$, hence that for suitable choices of $M, M_2 \in \N_0$ both \eqref{lem.1.cond-2_1} and \eqref{lem.1.cond-2_2} are satisfied in the first case.

\medskip

\noindent \underline{$M-[\beta_2^0]-[\eta_4^0] > 0$:}
Here we begin by choosing $M_2$ large enough to ensure that \eqref{eqn1811-1419_1_2} is negative, so we just have to treat \eqref{lem.1.cond-2_1}, \eqref{lem.1.cond-2_2} being  trivially satisfied.
Since $[\gamma] \geq M-[\beta_0^2]-[\eta_4^0]$, by \eqref{eq:range_gamma_T1}, and $[\beta_0^2] \leq v_\dimG \hdeg M_1$, we deduce that $\rho$ times the left-hand side of \eqref{lem.1.cond-2_1} can be bounded from below by
\begin{align}
    \rho \frac{[\gamma] + [\gamma'_0] - [\gamma_0]}{v_\dimG} + \rho \hdim
    &\geq \rho \frac{[\gamma]}{v_\dimG} + \rho \hdim\\
    & \geq \rho \frac{[\gamma]}{v_\dimG}-\rho\frac{[\beta_2^0]}{v_n}-\rho\frac{[\eta_4^0]}{v_n} + \rho \hdim\\
    &\geq \rho \frac{M}{v_\dimG} - \rho \frac{v_\dimG \hdeg M_1}{v_\dimG} - \rho \frac{ v_\dimG^2[\eta_4]}{v_\dimG} + \rho \hdim \\
    &\geq \rho \frac{M}{v_\dimG} - \rho \hdeg M_1 - v_\dimG [\eta_4] + \rho \hdim, \label{eqn1811-1419_2_1_a}
\end{align}
where we have used $[\beta_2^0] \leq v_n \hdeg M_1$ as in the case above
and $[\eta_4^0]\leq v_\dimG|\eta_4^0|\leq v_\dimG^2|\eta_4|\leq v_n^2[\eta_4]$
, while $\rho$ times the right-hand side of \eqref{lem.1.cond-2_1}
can be bounded from above by
\begin{align}
   \hspace{-20pt}\hdim + m_2 + \delta [\beta_{0, 1}] + [\beta_1] + [\eta'_3]
   &\leq \hdim + m_2 + \delta [\beta_0] + [\beta_1] + v_\dimG [\eta_3] \\
   &\leq \hdim + m_2 + \delta [\beta_0] + [\beta_1] + v_\dimG ([\eta_1] + [\eta_1] + [\eta_3]). \label{eqn1811-1419_2_1_b}
\end{align}
So, for \eqref{lem.1.cond-2_1} to be satisfied, it suffices to show the stricter inequality
\begin{align}
    \rho \frac{M}{v_\dimG} - \rho \hdeg M_1 + \rho \hdim > \hdim + m_2 + \delta [\beta_0] + [\beta_1] + v_\dimG \hdeg M_2, \label{eqn1811-1419_2_1_c}
\end{align}
 which is obtained by adding $v_\dimG [\eta_4]$ to both \eqref{eqn1811-1419_2_1_a} and \eqref{eqn1811-1419_2_1_b}.

On the other hand, because of $M-[\beta_0^2]-[\eta_4^0] +v_\dimG > [\gamma]$, we can bound the right-hand side of \eqref{lem.1.cond-2_2} from above by
\begin{align}
    \frac{Q+m_1-\hdeg M_2+\delta([\gamma_0']+[\gamma] + [\eta_4^0])}{\rho}
    &= \frac{Q+m_1-\hdeg M_2+\delta([\beta_{0, 2}] + [\beta^0_2] + [\gamma] + [\eta_4^0])}{\rho} \\
    &< \frac{Q+m_1-\hdeg M_2 + \delta ( [\beta_0] + M + v_\dimG)}{\rho} \label{eqn1811-1419_2_2}.
\end{align}
Now we may choose $M_2 = M_2(\beta_0, M_1)$, $M = M(\beta_0, M_1, M_2) \in \N_0$ large enough so that \eqref{eqn1811-1419_2_1_c} is satisfied and so that \eqref{eqn1811-1419_1_2} is less than zero. The latter implies that the right-hand side of \eqref{lem.1.cond-2_2} $\leq 0 < \hdim$, hence that \eqref{lem.1.cond-2_1} is satisfied. Moreover, since \eqref{eqn1811-1419_2_1_c} is satisfied for this choice of $M_2, M \in \N_0$, we also see that, by subtracting $v_\dimG [\eta_4]$ from both sides again, the so-adjusted right-hand side is positive, hence that \eqref{lem.1.cond-2_2} is satisfied for a positive maximum on the right-hand side. In summary, both \eqref{lem.1.cond-2_1} and \eqref{lem.1.cond-2_2} are satisfied also in the second case for suitably chosen of $M_2, M \in \N_0$.

Since we can pick $M_2$ larger than the maximum of the respective choices of $M_2$ for the two cases studied above, both \eqref{lem.1.cond-2_1} and \eqref{lem.1.cond-2_2} can be satisfied simultaneously in both cases. Hence, by Lemma~\ref{FRcorollary}, this finally gives
\begin{align}
    \eqref{eqn1711-1803}&\lesssim \|\sigma_2\|_{S^{m_2}_{\rho,\delta},a_2,b_2}, \\
     \eqref{eqn1711-1816}&\lesssim\|\sigma_1\|_{S^{m_1}_{\rho,\delta},a_1,b_1},
\end{align}
possibly with a new choice of $a_i,b_i \in \N$, $i=1,2$.\footnote{We may simply pick the largest seminorm.} Plugging the previous estimates for \eqref{eqn1711-1803} and \eqref{eqn1711-1816} back into \eqref{eqn 1711-1104}, we have concluded the proof of Step~1.
\medskip

 \noindent\textit{Proof of Step~2.} We now use Step~1 to prove \eqref{eq1.lem.1} for $M=M_0$.
     
Notice that, given arbitrary $\beta_0 \in \N^\dimG$ and $M_0\in\mathbb{N}$, and choosing $M(\beta_0)$ as in Step~1 and additionally larger than $M_0$,
we can write
\begin{align*}
    \tau_{M_0}(x,\pi) &= \tau_{M(\beta_0)}(x,\pi)+\!\!\!\sum_{M_0< [\alpha]\leq M(\beta_0)} \hspace{2pt} \iint_{G \times G} q_\alpha(p_1(y,ys^{-1}))X^\alpha_{x_1=x}\k(x_1,s) \\
    & \hspace{5cm} \times \kk(x,ys^{-1})\pi(y)^*ds\, dy,\\
    &= \tau_{M(\beta_0)}(x,\pi)+\!\!\!\!\!\!\!\!\!\sum_{M_0<[\alpha]\leq M(\beta_0)}
    \underset{[\alpha_1] + [\alpha_2] = [\alpha]}{\overline{\sum}} \hspace{3pt} \iint\limits_{G \times G}\!\! \tilde{q}_{\alpha_1}(s) 
    X^\alpha_{x_1=x}\k(x_1,s)\\
    &\hspace{5cm}\times \tilde{q}_{\alpha_2}(ys^{-1})\kk(x,ys^{-1})\pi(s)^* \pi( ys^{-1})^* dy \, ds\nonumber\\
     &= \tau_{M(\beta_0)}(x,\pi)+\!\!\!\!\!\!\!\!\!\sum_{M_0< [\alpha]\leq M(\beta_0)}
    \underset{[\alpha_1] + [\alpha_2] =[\alpha]}{\overline{\sum}} (\Delta^{\alpha_1} X_x^\alpha\s)(x,\pi)(\Delta^{\alpha_2}\ss)(x,\pi),
    \end{align*}
    where in the second line we have applied \eqref{eqn1811-1510}.
Here, once again, we consider only the case $\alpha_0=0$. By using the estimate for $\tau_{M(\beta_0)}$ proved in Step~1 of the proof, we get
\begin{align}
    \| (X^{\beta_0}\tau_{M_0})(x,\pi)\pi(I + \RO)^{-m}\|_{\mathscr{L}(\RS)}\lesssim& 
    \| (X^{\beta_0}\tau_{M(\beta_0)})(x,\pi)\pi(I + \RO)^{-m}\|_{\mathscr{L}(\RS)}\nonumber\\
    &\hspace{-6cm}+\sum_{M_0<[\alpha]\leq M(\beta_0)}\underset{[\alpha_1] + [\alpha_2] =[\alpha]}{\overline{\sum}}
    \| X^{\beta_0}_x(\Delta^{\alpha_1}X^\alpha\s)(x,\pi)(\Delta^{\alpha_2})\ss(x,\pi)\pi(I + \RO)^{-m}\|_{\mathscr{L}(\mathcal{H}_\pi)}\nonumber\\
     & \hspace{-6cm} \lesssim \|\s\|_{S^{m_1}_{\rho,\delta},a'_1,b'_1}\|\ss\|_{S^{m_2}_{\rho,\delta},a'_2,b'_2}
+\sum_{M_0< [\alpha]\leq M(\beta_0)}\underset{[\alpha_1] + [\alpha_2] =[\alpha]}{\overline{\sum}} \nonumber\\
&\hspace{-5cm} \underset{[\beta_{0,1}]+[\beta_{0,2}]=[\beta_0]}{\overline{\sum}}\, \| (X^{\beta_{0,1}}\Delta^{\alpha_1}X^\alpha\s)(x,\pi)(X^{\beta_{0,2}}\Delta^{\alpha_2}\ss)(x,\pi)\pi(I + \RO)^{-m}\|_{\mathscr{L}(\RS)}.\label{lem.1.final}
\end{align}
Since  $m=(m_1+m_2-(\rho-\delta)(M_0+1)+\delta[\beta_0])/\nu$, we can estimate
\begin{align}
    &\| (X^{\beta_{0,1}} \Delta^{\alpha_1} X^\alpha\s)(x,\pi) (X^{\beta_{0,2}} \Delta^{\alpha_2}\ss)(x,\pi) \pi(I + \RO)^{-m}\|_{\mathscr{L}(\mathcal{H}_\pi)}\nonumber\\
  &\leq \| (X^{\beta_{0,1}} \Delta^{\alpha_1} X^\alpha\s)(x,\pi)\pi(I + \RO)^{-\frac{m_1+\delta[\alpha]+\delta[\beta_{0,1}]-\rho[\alpha_1]}{\nu}}\|_{\mathscr{L}(\mathcal{H}_\pi)} \, \times \nonumber\\
    & \|\pi(I + \RO)^{\frac{m_1+\delta[\alpha]+\delta[\beta_{0,1}]-\rho[\alpha_1]}{\nu}}
    (X^{\beta_{0,2}} \Delta^{\alpha_2}\ss)(x,\pi)\pi(I + \RO)^{-\frac{m_1+m_2-(\rho-\delta)(M_0+1)+\delta[\beta_0]}{\nu}}\|_{\mathscr{L}(\mathcal{H}_\pi)}\nonumber\\
    &\lesssim \|\s\|_{S^{m_1}_{\rho,\delta},M(\beta_0),[\beta_0]+M(\beta_0)} \\
    & \hspace{1cm} \times \|\pi(I + \RO)^{\frac{m_1+\delta[\alpha]+\delta[\beta_{0,1}]-\rho[\alpha_1]}{\nu}}
    (X^{\beta_{0,2}} \Delta^{\alpha_2}\ss)(x,\pi) \pi(I + \RO)^{-\frac{m_2+\delta[\beta_{0,2}]-\rho[\alpha_2] }{\nu}} \\
    & \hspace{2cm} \pi(I + \RO)^{-\frac{m_1+\delta[\alpha]+\delta[\beta_{0,1}]-\rho[\alpha_1]}{\nu}} \pi(I + \RO)^{-\frac{\rho[\alpha_1]+\rho[\alpha_2]-\delta[\alpha]-(\rho-\delta)(M_0+1)}{\nu}}
    \|_{\mathscr{L}(\RS)}\nonumber\\
    &\lesssim \|\s\|_{S^{m_1}_{\rho,\delta},M(\beta_0),[\beta_0]+M(\beta_0)} \, 
    \|\pi(I + \RO)^{\frac{m_1+\delta[\alpha]+\delta[\beta_{0,1}]-\rho[\alpha_1]}{\nu}}
    (X^{\beta_{0,2}} \Delta^{\alpha_2}\ss)(x,\pi) \nonumber\\
    & \quad \pi(I + \RO)^{-\frac{m_2+\delta[\beta_{0,2}]-\rho[\alpha_2] }{\nu}} \pi(I + \RO)^{-\frac{m_1+\delta[\alpha]+\delta[\beta_{0,1}]-\rho[\alpha_1]}{\nu}}
    \|_{\mathscr{L}(\RS)},\label{lem.1.finalest}
\end{align}
where in the the last inequality we have used that 
$$ \|\pi(I + \RO)^{-\frac{\rho[\alpha_1]+\rho[\alpha_2]-\delta[\alpha]-(\rho-\delta)(M_0+1)}{\nu}}
    \|_{\mathscr{L}(\mathcal{H}_\pi)}\lesssim C,$$
since $\delta[\alpha]-\rho([\alpha_1]+[\alpha_2])(\rho-\delta)(M_0+1)\leq 0$ for all the parameters in the ranges appearing in the sums in \eqref{lem.1.final}.
For the sake of convenience, we will use the notation
\begin{align}
    \sigma_3(x,\pi):=  (X^{\beta_{0,2}}_x\Delta^{\alpha_2}\ss)(x,\pi)\pi(I + \RO)^{-\frac{m_2+\delta[\beta_{0,2}]-\rho[\alpha_2] }{\nu}}
\end{align}
in the next few lines.
Since taking the pointwise adjoint of the symbol is a continuous map in the $S^m_{\rho,\delta}(G)$-topology,
(see \cite[Theorem 5.5.12]{FR}), we have
\begin{align}
    &\|\pi(I + \RO)^{\frac{m_1+\delta[\alpha]+\delta[\beta_{0,1}]-\rho[\alpha_1]}{\nu}}
    \sigma_3(x,\pi)\pi(I + \RO)^{-\frac{m_1+\delta[\alpha]+\delta[\beta_{0,1}]-\rho[\alpha_1]}{\nu}}
    \|_{\mathscr{L}(\mathcal{H}_\pi)}\nonumber\\
    &\lesssim \|\pi(I + \RO)^{\frac{m_1+\delta[\alpha]+\delta[\beta_{0,1}]-\rho[\alpha_1]}{\nu}}
    \sigma_3
    \|_{S^{m_1+\delta[\alpha]+\delta[\beta_{0,1}]-\rho[\alpha_1]}_{\rho,\delta},0,0}\nonumber\\
    &\lesssim \|
    \sigma_3^*\, \pi(I + \RO)^{\frac{m_1+\delta[\alpha]+\delta[\beta_{0,1}]-\rho[\alpha_1]}{\nu}}
    \|_{S^{m_1+\delta[\alpha]+\delta[\beta_{0,1}]-\rho[\alpha_1]}_{\rho,\delta},a_1,b_1}\nonumber\\
    &\lesssim \|\sigma_3^* \|_{S^0_{\rho,\delta}, a_2,b_2}\lesssim \|\sigma_3 \|_{S^0_{\rho,\delta},a_3,b_3}
    \lesssim \|\ss\|_{S^{m_2}, a_4,b_4}.\nonumber
\end{align}
The latter inequality yields
\begin{align}
    \eqref{lem.1.finalest}\lesssim\|\s\|_{S^{m_1}, a_1,b_1} \|\ss\|_{S^{m_2}, a_2,b_2},
\end{align}
for some suitably chosen $a'_1, a'_2, b'_1, b'_2\in\mathbb{N}_0$ and for all the parameters in the ranges in the sums in \eqref{lem.1.final}. 
This finally shows \eqref{eq1.lem.1} for $M=M_0$ and thus concludes the proof.
\end{proof}

\medskip

\noindent \underline{The symbol $T_2$:}  The estimate of the symbol $T_2$ defined by \eqref{eq:def_T2} is completely analogous to the case of $T_1$ and based on Lemma~\ref{lemma2} below, whose proof follows the lines of the proof of Lemma \ref{lemma1}.

\begin{lemma}\label{lemma2}
Let $\sigma_1, \sigma_2$ be smooth symbols with $\k, \kk \in \SC(G \times G)$, and let $m_1, m_2\in\mathbb{R}$. For every arbitrary but fixed $M_0\in\mathbb{N}$, let $\tau_{M_0}$ be the symbol defined by
$$\tau_{M_0}(x,\pi):=\iint_{G\times G}R_{x,M_0}^{\kk(\cdot,s)}(p_2(y,s)) \kappa_{\sigma_1}(x,s^{-1}y)\pi(y)^*dy \,ds.$$

Then for all $\alpha_0,\beta_0\in \mathbb{N}^{\dimG}_0$ and $\rho,\delta\in [0,1]$ such that $0\leq\delta<\min\{\rho,\frac{1}{v_n}\} \leq 1$ there exist a constant $C>0$ and two seminorms $\|\cdot\|_{S^{m_1}_{\rho,\delta},a_1,b_1},\|\cdot\|_{S^{m_2}_{\rho,\delta},a_2,b_2}$ such that
\begin{align*}
\| (X^{\beta_0}\Delta^{\alpha_0}\tau_{M_0})(x,\pi) &\pi(I+\mathcal{R})^{-\frac{m_1+m_2-(\rho-\delta)(M_0+1) -\rho[\alpha_0]+\delta[\beta_0]}{\nu}}\|_{\mathscr{L}(\mathcal{H}_\pi)} \\
&\leq C \|\sigma_1\|_{S^{m_1}_{\rho,\delta},a_1,b_1}\|\sigma_2\|_{S^{m_2}_{\rho,\delta},a_2,b_2}.
\end{align*}
\end{lemma}

\medskip

\noindent \underline{The symbol $T_3$:} To estimate the seminorms of the remaining symbol $T_3$, defined by \eqref{eq:def_T3}, we need to employ a very different strategy, whose main argument is captured by the following lemma.

\begin{lemma}\label{lemma3}
Let $\sigma_1, \sigma_2$ be smooth symbols with $\k, \kk \in \SC(G \times G)$, and let $m_1, m_2\in\mathbb{R}$ and $M_0,N_0\in\mathbb{N}$. Let also $r_{M_0,N_0}$ be the symbol defined by
$$r_{M_0,N_0}(x,\pi):=\int_{G \times G} R_{x,M_0}^{\kappa_{\sigma_1}(\cdot,s^{-1}y)}(p_1(y,s))R_{x,N_0}^{\kappa_{\sigma_2}(\cdot,s)}(p_2(y,s)) \pi(y)^* ds\, dy.$$
Then for all $\alpha_0,\beta_0\in \mathbb{N}^{\dimG}_0$ and for all $\rho,\delta \in [0,1]$ such that $0\leq\delta<\frac{\rho}{v_n} \leq 1$, there exist a constant $C>0$ and two seminorms $\|\cdot\|_{S^{m_1}_{\rho,\delta},a_1,b_1}$, $\|\cdot\|_{S^{m_2}_{\rho,\delta},a_2,b_2}$ such that
\begin{align*}
    \| (X^{\beta_0} \Delta^{\alpha_0} r_{M_0,N_0})(x,\pi) \pi(I+\mathcal{R}&)^{-\frac{m_1+m_2-(\rho-\delta)(\min\{M_0,N_0\}+1) - \rho[\alpha_0]+\delta[\beta_0]}{\nu}}\|_{\mathscr{L}(\mathcal{H}_\pi)} \\
    &\leq C \|\sigma_1\|_{S^{m_1}_{\rho,\delta},a_1,b_1}\|\sigma_2\|_{S^{m_2}_{\rho,\delta},a_2,b_2}.\end{align*}
\end{lemma}

\begin{proof}
    \noindent\textit{Proof of Step~1.}
Let $L_1$ be the smallest nonnegative integer such that
\begin{align}
L_1\geq -\frac{m_1+m_2-(\rho -\delta)L_0+\delta[\beta_0]}{\nu},\label{CondL_1-new}
\end{align}
since this yields
\begin{align}
     \| X_x^{\beta_0} r_{M,N}(x,\pi)\pi(I+&\mathcal{R})^{-\frac{m_1+m_2-(\rho -\delta)L_0-\rho[\alpha_0]+\delta[\beta_0]}{\nu}}\|_{\mathscr{L}(\mathcal{H}_\pi)}\nonumber\\
    &\lesssim  \|  X^{\beta_0}_x r_{M,N}(x,\pi)\pi(I+\mathcal{R})^{L_1}\|_{\mathscr{L}(\mathcal{H}_\pi)} \label{Est.T3.1-new}
\end{align}
and it therefore suffices to provide an estimate for $\|  X^{\beta_0}_x r_{M,N}(x,\pi)\pi(I+\mathcal{R})^{L_1}\|_{\mathscr{L}(\mathcal{H}_\pi)}$ in order to conclude the result. As before, the first step towards the desired estimate is integration by parts using the relation
\begin{align*}
   \pi(y)^* \pi(I+\mathcal{R})^{L_1}=\underset{[\eta]\leq \nu L_1}{\overline{\sum}} \pi(y)^*\pi(X)^\eta= \underset{[\eta]\leq \nu L_1}{\overline{\sum}}(-1)^{|\eta|}(\Tilde{X}^\eta \pi^*)(y).
\end{align*}
Thus, combining the above identity with integration by parts and the Leibniz rule, we can bound \eqref{Est.T3.1-new} by
\begin{align}
    \|  (X^{\beta_0} r_{M,N})&(x,\pi) \pi(I+\mathcal{R})^{L_1}\|_{\mathscr{L}(\mathcal{H}_\pi)} \\
    &\lesssim 
    \underset{\substack{[\eta_0]+[\eta_1]+[\eta_2]\leq \nu L_1, \\ [\beta_{01}]+[\beta_{02}]\leq [\beta_0] }}{\overline{\sum}}  \iint_{G \times G} 
    |\underbrace{\tilde{X}^{\eta_0}_{y_0=y}\tilde{X}^{\eta_1}_{y_1=y}X^{\beta_{01}}_{x_1=x}R_{x_1, M}^{\kappa_{\sigma_1}(\cdot, s^{-1}y_0)}(p_1(y_1, s))}_{ =: I_1(y,s)}| \\
    &\hspace{3cm} \times |\underbrace{\tilde{X}^{\eta_2}_{y}X^{\beta_{02}}_{x_2=x} R_{x_2,N}^{\kappa_{\sigma_2}(\cdot, s)}(p_2(y,s))}_{ =: I_2(y,s)}| \, ds \, dy, \label{Est.T3.2-new}
    \end{align}
where we have pulled the $\mathscr{L}(\mathcal{H}_\pi)$-norm into the integral and used the trivial bound $\|\pi(y)^*\|_{\mathscr{L}(\mathcal{H}_\pi)} = 1$, $y \in G$. In what follows we estimate $I_1$ and $I_2$ separately.

Let us begin with $I_1$. Using \eqref{RelationsVFs1} to rewrite the derivatives in $y_0$ and the property
\begin{align}
    X^{\beta_{01}}_{x_1=x}R_{x_1, M}^{\kappa_{\sigma_1}(\, \cdot \,, s^{-1}y_0)}(p_1(y_1, s))
    = R_{0, M}^{X^{\beta_{01}}_{x_1=x}\kappa_{\sigma_1}( x_1\cdot \,, s^{-1}y_0)}(p_1(y_1, s))
\end{align}
of the Taylor remainder, we have
\begin{align}
	\tilde{X}^{\eta_0}_{y_0=y} \tilde{X}^{\eta_1}_{y_1=y} &X^{\beta_{01}}_{x_1=x}R_{x_1, M}^{\kappa_{\sigma_1}(\, . \,, s^{-1}y_0)}(p_1(y_1, s)) \\
	&= \underset{\substack{|\eta'_0| \leq |\eta_0|, \\ [\eta'_0] \geq [\eta_0]}}{\overline{\sum}} Q_{\eta_0, \eta'_0}(s) \, \tilde{X}^{\eta_1}_{y_1=y} R_{0, M}^{X^{\eta'_0}_{y_0 = y} X^{\beta_{01}}_{x_1=x}\kappa_{\sigma_1}(x_1\cdot, s^{-1}y_0)}(p_1(y_1, s)). \label{eq:T3_I1_1}
\end{align}
Now, using \eqref{Prop.Rem-vf} for the Taylor remainder in combination with Lemma~\ref{lem:pvf} to express the differentiation in $y_1$ into one in $p_1$, we can get
\begin{align}
	\eqref{eq:T3_I1_1} &= \underset{\substack{|\eta'_0| \leq |\eta_0|, \\ [\eta'_0] \geq [\eta_0]}}{\overline{\sum}} \, \underset{\substack{[\zeta_3] -([\zeta_1]+[\zeta_2])= [\eta_1], \\\ [\eta_1]\leq [\zeta_3] \leq v_n|\eta_1|}}{\overline{\sum}} \, \underset{\substack{[\zeta_2'] + [\zeta_2''] =[\zeta_2]}}{\overline{\sum}}  Q_{\eta_0, \eta'_0}(s) \, \tilde{q}_{\zeta_1+\zeta_2'}(s) \tilde{q}_{\zeta_2''}(s^{-1}y) \\
    &\hspace{60pt} \times 
	R_{0, M - [\zeta_3]}^{X^{\zeta_3}_{x_2=x} X^{\eta'_0}_{y_0 = y}  X^{\beta_{01}}_{x_1=x_2}\kappa_{\sigma_1}(x_1 \cdot, s^{-1}y_0)}(p_1(y_1, s)),
\end{align}
where we have rewritten $\tilde{q}_{\zeta_2}(y)$ to accommodate the change of variables. Next we rewrite
\begin{align}
    X^{\zeta_3}_{x_2=x} X^{\eta'_0}_{y_0 = y}  X^{\beta_{01}}_{x_1=x_2}&\kappa_{\sigma_1}(x_1\cdot , s^{-1}y_0)=\underset{\substack{[\zeta]=[\zeta_3]+[\beta_{01}]}}{\overline{\sum}}
     X^{\eta'_0}_{y_0 = y}X^{\zeta}_{x_1 = x}
     \kappa_{\sigma_1}(x \cdot, s^{-1}y_0)\\
     &= \underset{\substack{[\zeta]=[\zeta_3]+[\beta_{01}]}}{\overline{\sum}}
     \underset{\substack{|\zeta'|\leq|\zeta|\\ [\zeta']\geq[\zeta]}}{\overline{\sum}}
     Q_{\zeta,\zeta'}(z_1)X^{\eta'_0}_{y_0 = y}X^{\zeta'}_{z_1=z}
     \kappa_{\sigma_1}(x z_1 , s^{-1}y_0),
\end{align}
so that, by the estimate for the Taylor remainder~\eqref{eq:est_Taylor_rem}, $I_1$ can be bounded pointwise by
\begin{align}
	I_1(y,s) &\lesssim \underset{\substack{|\eta'_0| \leq |\eta_0|, \\ [\eta'_0] \geq [\eta_0]}}{\overline{\sum}} \, \underset{\substack{[\zeta_3] -([\zeta_1]+[\zeta_2])= [\eta_1], \\\ [\eta_1]\leq [\zeta_3] \leq v_n|\eta_1|}}{\overline{\sum}} \, \underset{\substack{[\zeta_2'] + [\zeta_2''] =[\zeta_2]}}{\overline{\sum}}  
 \underset{\substack{[\zeta]=[\zeta_3]+[\beta_{01}]}}{\overline{\sum}}
     \underset{\substack{|\zeta'|\leq|\zeta|\\ [\zeta']\geq[\zeta]}}{\overline{\sum}}
 \sum_{\substack{M - [\zeta_3] < [\gamma_1] \\ \leq M - [\zeta_3] + v_\dimG }}
 |s|^{[\eta_0']-[\eta_0]} \\
&\hspace{-10pt} \times |s|^{ [\zeta_1] + [\zeta_2']} |s^{-1}y|^{\zeta_2''} |p_1(y, s)|^{[\gamma_1]} \!\!\!\!\!\!\!  \sup_{|z|\leq \eta^{\lceil M \rfloor}|p_1(y,s)|} \!\!\!\!\!\!\! \bigl | X^{\gamma_1}_{z_1=z}Q_{\zeta,\zeta'}(z_1) X^{\eta'_0}_{y_0 = y} X^{\zeta'}_{z_1}  \kappa_{\sigma_1}(x z_1, s^{-1}y_0) \bigr | \\
&\lesssim \underset{\substack{|\eta'_0| \leq |\eta_0|, \\ [\eta'_0] \geq [\eta_0]}}{\overline{\sum}} \, \underset{\substack{[\zeta_3] -([\zeta_1]+[\zeta_2])= [\eta_1], \\\ [\eta_1]\leq [\zeta_3] \leq v_n|\eta_1|}}{\overline{\sum}} \, \underset{\substack{[\zeta_2'] + [\zeta_2''] =[\zeta_2]}}{\overline{\sum}}  
 \underset{\substack{[\zeta]=[\zeta_3]+[\beta_{01}]}}{\overline{\sum}}
     \underset{\substack{|\zeta'|\leq|\zeta|\\ [\zeta']\geq[\zeta]}}{\overline{\sum}}
 \sum_{\substack{M - [\zeta_3] < [\gamma_1] \\ \leq M - [\zeta_3] + v_\dimG }} 
 \underset{[\gamma_1']+[\gamma_1'']=[\gamma_1]}{\overline{\sum}} \\
& |s|^{[\eta_0']-[\eta_0] + [\zeta_1] + [\zeta_2']} |s^{-1}y|^{\zeta_2''} \Big(|s|^{[\gamma_1]+[\zeta']-[\zeta]-[\gamma_1']}+\sum_{j=1}^n\sum_{\substack{[\mu'_1]+[\mu'_2]=v_j\\ \mu'_1,\mu'_2\neq 0}}\!\!\!\! |s|^{\frac{([\gamma_1]+[\zeta']-[\zeta]-[\gamma_1'])[\mu'_1]}{v_j}} \\
&\times |s^{-1}y|^{\frac{([\gamma_1]+[\zeta']-[\zeta]-[\gamma_1'])[\mu'_2]}{v_j}}\Big) \sup_{x\in G}|   X^{\eta'_0}_{y_0 = y} X^{\gamma_1''}_{x_1=x}X^{\zeta'}_{x_1}\kappa_{\sigma_1}(x_1, s^{-1}y_0) \bigr |,
\end{align}
where, for the second inequality, we have applied the Leibniz rule to $X^{\gamma_1}_{z_1=z}$. Since  the derivatives $X^{\gamma_1'}Q_{\zeta,\zeta'}$ vanish for $[\gamma_1']\geq [\zeta']-[\zeta]$, we may assume $[\zeta'] - [\zeta] \geq [\gamma_1']$.

Analogously, the integrand $I_2$ can be expressed as
\begin{align}
	&\tilde{X}^{\eta_2}_{y} X^{\beta_{02}}_{x_2=x} R_{x_2,N}^{\kappa_{\sigma_2}(\, . \,, s)}(p_2(y,s)) \\
	&\hspace{10pt}= \underset{\substack{[\theta_3]-([\theta_1]+[\theta_2])=[\eta_2]\\
  [\eta_2]\leq [\theta_3]\leq v_n|\eta_2|}}{\overline{\sum}}
 \, \underset{\substack{[\theta_2'] + [\theta_2''] =[\theta_2]}}{\overline{\sum}}  \, \tilde{q}_{\theta_1+\theta_2'}(s) \tilde{q}_{\theta_2''}(s^{-1}y) 
	R_{0, N - [\theta_3]}^{X^{\theta_3}_{s_1=s}X^{\beta_{02}}_{x_2=x} \kappa_{\sigma_2}(x_2\, \cdot, s_1)}(p_2(y_1, s)),
\end{align}
where
\begin{align}
    X^{\theta_3}_{s_1=s}X^{\beta_{02}}_{x_2=x} \kappa_{\sigma_2}(x_2 z, s_1)=&
    \underset{\substack{|\beta_{02}'|\leq |\beta_{02}|\\ [\beta_{02}']\geq [\beta_{02}]}}{\overline{\sum}} 
    Q_{\beta_{02},\beta_{02}'}(z) X^{\theta_3}_{s_1=s}X^{\beta_{02}'}_{x_2=x} \kappa_{\sigma_2}(x_2z, s_1).
\end{align}
Therefore, arguing as for $I_1$, we find 
\begin{align}
I_2(y,s)\lesssim&
 \underset{\substack{[\theta_3]-([\theta_1]+[\theta_2])=[\eta_2]\\
  [\eta_2]\leq [\theta_3]\leq v_n|\eta_2|}}{\overline{\sum}}
   \underset{\substack{[\theta_2'] + [\theta_2''] =[\theta_2]}}{\overline{\sum}} 
 \underset{\substack{|\beta_{02}'|\leq |\beta_{02}|\\ [\beta_{02}']\geq [\beta_{02}]}}{\overline{\sum}}
 \sum_{\substack{N - [\theta_3] < [\gamma_2] \\ \leq N - [\theta_3] + v_\dimG }} |s|^{\theta_1+\theta_2'} |s^{-1}y|^ {\theta_2''}
 \\
 &\hspace{10pt} \times |p_2(y, s)|^{[\gamma_2]} \,  \sup_{|z|\leq \eta^{\lceil M\rfloor}|p_2(y,s)|}\bigl |  X^{\gamma_2}_{z_2 = z} Q_{\beta_{02},\beta_{02}'}(z_2)  X^{\theta_3}_{s} X^{\beta_{02'}}_{z_2} \kappa_{\sigma_2}(xz_2, s) \bigr | \\
\lesssim &
\underset{\substack{[\theta_3]-([\theta_1]+[\theta_2])=[\eta_2]\\
  [\eta_2]\leq [\theta_3]\leq v_n|\eta_2|}}{\overline{\sum}}
   \underset{\substack{[\theta_2'] + [\theta_2''] =[\theta_2]}}{\overline{\sum}} 
 \underset{\substack{|\beta_{02}'|\leq |\beta_{02}|\\ [\beta_{02}']\geq [\beta_{02}]}}{\overline{\sum}}
\sum_{\substack{N - [\theta_3] < [\gamma_2] \\ \leq N - [\theta_3] + v_\dimG }} 
\underset{[\gamma_2']+[\gamma_2'']=[\gamma_2]}{\overline{\sum}} |s|^{\theta_1+\theta_2'} |s^{-1}y|^ {\theta_2''} \\
&\hspace{10pt} \times \Big(|s^{-1}y|^{[\gamma_2]+[\beta_{02}']-[\beta_{02}]-[\gamma_2']} + \sum_{k=1}^n\sum_{\substack{[\mu''_1]+[\mu''_2]=v_k\\ \mu''_1,\mu''_2\neq 0}}\!\!\!\! |s|^{\frac{([\gamma_2]+[\beta_{02}']-[\beta_{02}]-[\gamma_2'])[\mu''_1]}{v_k}} \\
&\hspace{10pt} \times |s^{-1}y|^{\frac{([\gamma_2]+[\beta_{02}']-[\beta_{02}]-[\gamma_2'])[\mu''_2]}{v_k}}\Big) \sup_{x\in G}\bigl | X^{\theta_3}_{s} X^{\gamma_{2}''}_{x_1=x}X^{\beta_{02}'}_{x_1} \kappa_{\sigma_2}(x, s) \bigr |,
\end{align}
and because of $X^{\gamma_2} Q_{\beta_{02},\beta_{02}'} = 0$ for $[\gamma_2'] \geq [\beta_{02}']-[\beta_{02}]$ we may assume $[\beta_{02}']-[\beta_{02}] \geq [\gamma_2']$.

Combining the estimates for $I_1$ and $I_2$, we obtain
\begin{align}
  \eqref{Est.T3.2-new}&\lesssim
  \underset{\substack{[\eta_0]+[\eta_1]+[\eta_2]\leq \nu L_1, \\ [\beta_{01}]+[\beta_{02}]\leq [\beta_0] }}{\overline{\sum}}
\underset{\substack{|\eta'_0| \leq |\eta_0|, \\ [\eta'_0] \geq [\eta_0]}}{\overline{\sum}} \, \underset{\substack{[\zeta_3] -([\zeta_1]+[\zeta_2])= [\eta_1], \\\ [\eta_1]\leq [\zeta_3] \leq v_n|\eta_1|}}{\overline{\sum}} \, \underset{\substack{[\zeta_2'] + [\zeta_2''] =[\zeta_2]}}{\overline{\sum}} 
 \underset{\substack{[\zeta]=[\zeta_3]+[\beta_{01}]}}{\overline{\sum}}\\
 &
     \underset{\substack{|\zeta'|\leq|\zeta|\\ [\zeta']\geq[\zeta]}}{\overline{\sum}}
\underset{\substack{[\theta_3]-([\theta_1]+[\theta_2])=[\eta_2]\\
  [\eta_2]\leq [\theta_3]\leq v_n|\eta_2|}}{\overline{\sum}}
   \underset{\substack{[\theta_2'] + [\theta_2''] =[\theta_2]}}{\overline{\sum}} 
 \underset{\substack{|\beta_{02}'|\leq |\beta_{02}|\\ [\beta_{02}']\geq [\beta_{02}]}}{\overline{\sum}}
 \sum_{\substack{M - [\zeta_3] < [\gamma_1] \\ \leq M - [\zeta_3] + v_\dimG }}
  \sum_{\substack{N - [\theta_3] < [\gamma_2] \\ \leq N - [\theta_3] + v_\dimG }}
  \underset{\substack{[\gamma_1']+[\gamma_1'']=[\gamma_1]\\ [\gamma_2']+[\gamma_2'']=[\gamma_2] }}{\overline{\sum}}\\
&  \iint_{G\times G} |s|^{[\eta_0']-[\eta_0] + [\zeta_1] + [\zeta'_2] + [\theta_1] + [\theta'_2]} |s^{-1}y|^{[\zeta''_2] + [\theta''_2]} \\
&\hspace{-20pt} \times \Big(|s|^{[\gamma_1]+[\zeta']-[\zeta]-[\gamma_1']}+\sum_{j=1}^n\sum_{\substack{[\mu'_1]+[\mu'_2]=v_j\\ \mu'_1,\mu'_2\neq 0}}\!\!\!\! |s|^{\frac{([\gamma_1]+[\zeta']-[\zeta]-[\gamma_1'])[\mu'_1]}{v_j}} |s^{-1}y|^{\frac{([\gamma_1]+[\zeta']-[\zeta]-[\gamma_1'])[\mu'_2]}{v_j}}\Big) \\
&\hspace{-20pt} \times \Big(|s^{-1}y|^{[\gamma_2]+[\beta_{02}']-[\beta_{02}]-[\gamma_2']}
+\sum_{k=1}^n\sum_{\substack{[\mu''_1]+[\mu''_2]=v_k\\ \mu''_1,\mu''_2\neq 0}}\!\!\!\! |s|^{\frac{([\gamma_2]+[\beta_{02}']-[\beta_{02}]-[\gamma_2'])[\mu''_1]}{v_k}} \Big)\\
&\hspace{-20pt} \times |s^{-1}y|^{\frac{([\gamma_2]+[\beta_{02}']-[\beta_{02}]-[\gamma_2'])[\mu''_2]}{v_k}} \sup_{x\in G}|  X^{\eta'_0}_{y_0 = y} X^{\gamma_1''}_{x_1=x}X^{\zeta'}_{x_1}\kappa_{\sigma_1}(x_1, s^{-1}y_0) \bigr | \\
&\hspace{-20pt} \times \sup_{x\in G}\bigl | X^{\theta_3}_{s} X^{\gamma_2''}_{x_2=x}X^{\beta_{02}'}_{x_2} \kappa_{\sigma_2}(x_2, s) \bigr |
ds dy.\label{pf.T3.1}
\end{align}

To prove the convergence of the integral~\eqref{pf.T3.1}, we apply the change of variables $y\mapsto s t$ and write the integral in \eqref{pf.T3.1} as the sum of the following four integrals:
\begin{align*}
    J_1=&
    \int_{G}|s|^{[\eta_0']-[\eta_0] + [\zeta_1] + [\zeta'_2] + [\theta_1] + [\theta'_2] +[\gamma_1]+[\zeta']-[\zeta]-[\gamma_1']}
    \sup_{x\in G}\bigl | X^{\theta_3}_{s} X^{\gamma_2''}_{x_2=x}\tilde{X}^{\beta_{02}'}_{x_2} \kappa_{\sigma_2}(x_2, s) \bigr | ds\,\,\\
    &\hspace{2cm}\times \int_G|t|^{[\zeta''_2] + [\theta''_2] + [\gamma_2]+[\beta_{02}']-[\beta_{02}]-[\gamma_2']} \sup_{x\in G}|  X^{\eta'_0}_{t} X^{\gamma_1''}_{x_1=x}\tilde{X}^{\zeta'}_{x_1}\kappa_{\sigma_1}(x_1, t) \bigr |dt,\\
    J_2=&\sum_{j=1}^n\sum_{\substack{[\mu'_1]+[\mu'_2]=v_j\\ \mu'_1,\mu'_2\neq 0}}
    \int_{G}|s|^{[\eta_0']-[\eta_0] + [\zeta_1] + [\zeta'_2] + [\theta_1] + [\theta'_2] +\frac{([\gamma_1]+[\zeta']-[\zeta]-[\gamma_1'])[\mu'_1]}{v_j}} \\
    &\hspace{2cm} \times \sup_{x\in G}\bigl | X^{\theta_3}_{s} X^{\gamma_2''}_{x_2=x}\tilde{X}^{\beta_{02}'}_{x_2} \kappa_{\sigma_2}(x_2, s) \bigr | ds\, \int_G|t|^{[\zeta''_2] + [\theta''_2] + [\gamma_2]+[\beta_{02}']-[\beta_{02}]-[\gamma_2']} \\
    &\hspace{2cm}\times |t|^{\frac{([\gamma_1]+[\zeta']-[\zeta]-[\gamma_1'])[\mu'_2]}{v_j}} \sup_{x\in G}|  X^{\eta'_0}_{t} X^{\gamma_1''}_{x_1=x}\tilde{X}^{\zeta'}_{x_1}\kappa_{\sigma_1}(x_1, t) \bigr | dt,\\
    J_3=&\sum_{k=1}^n\sum_{\substack{[\mu''_1]+[\mu''_2]=v_k\\ \mu''_1,\mu''_2\neq 0}}
    \int_{G}|s|^{[\eta_0']-[\eta_0]+ [\zeta_1] + [\zeta'_2] + [\theta_1] + [\theta'_2] +[\gamma_1]+[\zeta']-[\zeta]-[\gamma_1']}\\
    &\hspace{2cm}\times |s|^{\frac{([\gamma_2]+[\beta_{02}']-[\beta_{02}]-[\gamma_2'])[\mu''_1]}{v_k}} \sup_{x\in G}\bigl | X^{\theta_3}_{s} X^{\gamma_2''}_{x_2=x}\tilde{X}^{\beta_{02}'}_{x_2} \kappa_{\sigma_2}(x_2, s) \bigr | ds\,\,\\
    &\hspace{2cm}\times \int_G|t|^{[\zeta''_2] + [\theta''_2] + \frac{([\gamma_2]+[\beta_{02}']-[\beta_{02}]-[\gamma_2'])[\mu''_2]}{v_k}}\sup_{x\in G}|  X^{\eta'_0}_{t} X^{\gamma_1''}_{x_1=x}\tilde{X}^{\zeta'}_{x_1}\kappa_{\sigma_1}(x_1, t) \bigr | dt,\\
    J_4=&\sum_{j,k=1}^n \sum_{\substack{[\mu'_1]+[\mu'_2]=v_j\\ \mu'_1,\mu'_2\neq 0}}\sum_{\substack{[\mu''_1]+[\mu''_2]=v_k\\ \mu''_1,\mu''_2\neq 0}}
    \int_{G}|s|^{[\eta_0']-[\eta_0]+ [\zeta_1] + [\zeta'_2] + [\theta_1] + [\theta'_2] +\frac{([\gamma_1]+[\zeta']-[\zeta]-[\gamma_1'])[\mu'_1]}{v_j}} \\
    &
    \hspace{2cm}\times |s|^{\frac{([\gamma_2]+[\beta_{02}']-[\beta_{02}]-[\gamma_2'])[\mu''_1]}{v_k}} \sup_{x\in G}\bigl | X^{\theta_3}_{s} X^{\gamma_2''}_{x_2=x}\tilde{X}^{\beta_{02}'}_{x_2} \kappa_{\sigma_2}(x_2, s) \bigr |  ds\,\,\\
    &\hspace{2cm}\times \int_G|t|^{[\zeta''_2] + [\theta''_2] + \frac{([\gamma_1]+[\zeta']-[\zeta]-[\gamma_1'])[\mu'_2]}{v_j}+
    \frac{([\gamma_2]+[\beta_{02}']-[\beta_{02}]-[\gamma_2'])[\mu''_2]}{v_k}} \\
    &\hspace{2cm}\times \sup_{x\in G}|  X^{\eta'_0}_{t} X^{\gamma_1''}_{x_1=x}\tilde{X}^{\zeta'}_{x_1}\kappa_{\sigma_1}(x_1, t) \bigr | dt.
\end{align*}
To prove the convergence of $J_1, \ldots, J_4$ for all the indices in the ranges appearing in \eqref{pf.T3.1}, we will use  Lemma~\ref{FRcorollary}. Moreover, this will allow us to show that the terms above are bounded by the product of suitable seminorms of $\sigma_1$ and $\sigma_2$, as desired.

We start by considering term $J_1$.
By Lemma~\ref{FRcorollary}, $J_1$ converges and is bounded by the product of suitable seminorms of $\sigma_1$ and $\sigma_2$ if the following conditions are satisfied simultaneously:
\begin{align}
\begin{split}\label{pf.T3.2.0}
    [\eta_0']-[\eta_0]+ [\zeta_1] + [\zeta'_2] + [\theta_1] + [\theta'_2] +[\gamma_1]+[\zeta']-[\zeta]-[\gamma_1']+Q& \\
    &\hspace{-105pt}>\max\left\{ \frac{Q+m_2+\delta([\gamma_2'']+[\beta_{02}'])+[\theta_3]}{\rho},0\right\},\\
   [\zeta''_2] + [\theta''_2] + [\gamma_2]+[\beta_{02}']-[\beta_{02}]-[\gamma_2']+Q& \\
   &\hspace{-105pt}>\max\left\{\frac{Q+m_1+\delta([\gamma_1'']+[\zeta'])+[\eta_0']}{\rho},0 \right\}.
\end{split}
\end{align}
 Note that, since $[\eta_0']-[\eta_0]\geq 0$, $[\zeta_1] + [\zeta'_2] + [\theta_1] + [\theta'_2] \geq 0$, $[\zeta''_2] + [\theta''_2] \geq 0$, $[\zeta']-[\zeta] \geq [\gamma_1']$, $[\beta_{02}']-[\beta_{02}] \geq [\gamma_2']$, $[\gamma_2]\geq [\gamma_2'']$ and $[\gamma_1]\geq [\gamma_1'']$, for \eqref{pf.T3.2.0} to be satisfied, it suffices that the following two inequalities hold:
 \begin{align}
 \begin{split}\label{pf.T3.2}
    [\gamma_1]+Q&>\max\left\{ \frac{Q+m_2+\delta([\gamma_2]+[\beta_{02}'])+[\theta_3]}{\rho},0\right\},\\
   [\gamma_2]+Q&>\max\left\{\frac{Q+m_1+\delta([\gamma_1]+[\zeta'])+[\eta_0']}{\rho},0 \right\}.
\end{split}
\end{align}
Moreover, since we may always choose $\gamma_1,\gamma_2$ sufficiently large, that is, by choosing $M>C_1(\beta_0)$, $N>C_2(\beta_0)$ for some sufficiently large constants $C_1(\beta_0), C_2(\beta_0) > 0$ and such that $M-[\zeta_3]$, $N-[\theta_3]$ and
the right-hand sides of \eqref{pf.T3.2} are always positive, we can consider the latter case directly.
Under this assumption and considering the ranges of the indices, we can reduce 
\eqref{pf.T3.2} to
\begin{align}
\begin{split} \label{pf.T3.3}
   [\gamma_1]&> \frac{\delta[\gamma_2]}{\rho}+c_1(\beta_0,n,\delta,\rho,m_2,\nu,L_1),\\
    [\gamma_2]&> \frac{\delta[\gamma_1]}{\rho}+c_2(\beta_0,n,\delta,\rho,m_1,\nu,L_1).
\end{split}
\end{align}
We observe that $0 < M-[\zeta_3]<[\gamma_1]\leq M-[\zeta_3]+v_n$, $0 < N-[\theta_3]<[\gamma_1]\leq N-[\theta_3]+v_n$, and that $[\zeta_3],[\theta_3]$ are bounded by constants depending on $\beta_0,\nu,L_1,v_n$, so we conclude that \eqref{pf.T3.3} is satisfied if $M,N \in \N_0$ are chosen such that
\begin{align}
\begin{split} \label{pf.T3.4}
   M&> \frac{\delta}{\rho}N+c_1(\beta_0,n,\delta,\rho,m_2,\nu,L_1),\\
    N&> \frac{\delta}{\rho}M+c_2(\beta_0,n,\delta,\rho,m_1,\nu,L_1)\underset{\text{if}\,\delta\neq 0}{\Longleftrightarrow} M<\frac{\rho}{\delta}N-\frac{\rho}{\delta}c_2(\beta_0,n,\delta,\rho,m_1,\nu,L_1).
\end{split}
\end{align}
Note that if $\delta=0$ the system of inequalities in \eqref{pf.T3.4} is trivially solvable. In particular, the set of solutions in this case is the set
\begin{align}
H_1(\beta_0):=&\{(M,N)\in\mathbb{N}^2; M> c_1(\beta_0,n,\delta,\rho,m_2,\nu,L_1), N>c_2(\beta_0,n,\delta,\rho,m_2,\nu,L_1)\}\\
&\bigcap \{(M,N)\in\mathbb{N}^2; M> C_1(\beta_0), N>C_2(\beta_0\},
\end{align}
which depends on the multi-index $\beta_0$, fixed at the beginning. On the other hand, if $\delta\neq 0<\rho$ then $\rho/\delta>1> \delta/\rho$, implying that the set of solutions of \eqref{pf.T3.4} is the set
\begin{align}
    \Gamma_1(\beta_0):=&\{ (M,N)\in\mathbb{N}^2; \frac{\delta}{\rho}N+c_1(\beta_0)<M<\frac{\rho}{\delta}N-\frac{\rho}{\delta}c_2(\beta_0)\}\\
    &\bigcap \{(M,N)\in\mathbb{N}^2; M> C_1(\beta_0), N>C_2(\beta_0\}.
\end{align}
This in turn implies that \eqref{pf.T3.3} is solvable for any $(M,N)=(M(\beta_0),N(\beta_0))\in \Gamma_1(\beta_0)$ if $\delta\in (0,\rho)$ and for any $(M,N)=(M(\beta_0),N(\beta_0))\in H_1(\beta_0)$ if $\delta=0$, and that we can choose $(M(\beta_0),N(\beta_0))$ as large as we want.
Therefore, by Lemma \ref{FRcorollary}, $J_1$ is convergent and satisfies $J_1\lesssim \|\sigma_1\|_{S^{m_1}_{\rho,\delta},a_1,b_1} \|\sigma_2\|_{S^{m_2}_{\rho,\delta},a_2,b_2}$ for suitably chosen $a_1,a_2,b_1,b_2 \in \N_0$.

To treat the integral $J_2$, we use the same line of arguments as for $J_1$. Hence, once again by Lemma~\ref{FRcorollary}, $J_2$ is convergent if
\begin{align}
    [\eta_0']-[\eta_0] + [\zeta_1] + [\zeta'_2] + [\theta_1] + [\theta'_2] +&\frac{[\mu'_1]([\gamma_1]+[\zeta']-[\zeta]-[\gamma_1'])}{v_j}+Q \\
    &>\max\left\{ \frac{Q+m_2+\delta([\gamma_2'']+[\beta_{02}'])+[\theta_3]}{\rho},0\right\}\\
   [\zeta''_2] + [\theta''_2] +  [\gamma_2]+[\beta_{02}']-[\beta_{02}]-[\gamma_2']+&\frac{[\mu'_2]([\gamma_1]+[\zeta']-[\zeta]-[\gamma_1'])}{v_j}+Q \\
   &>\max\left\{\frac{Q+m_1+\delta([\gamma_1'']+[\zeta'])+[\eta_0']}{\rho},0 \right\},\label{pf.T3.5}
\end{align}
are satisfied simultaneously. Estimating the parameters and arguing as for $J_1$, it suffices to require
\begin{align}
\begin{split} \label{pf.T3.6}
   \frac{[\mu_1']}{ v_j}M &>\frac{\delta}{\rho} N+c_1(\beta_0,n,\delta,\rho,m_2,\nu,L_1),\\
    N+\frac{[\mu_2']}{v_j}M &> \frac{\delta}{\rho}M+c_2(\beta_0,n,\delta,\rho,m_1,\nu,L_1),
\end{split}
\end{align}
 to have the desired convergence and estimate. Now for all $j=1,\ldots,n$, and for all $M,N \in \N_0$, with
\begin{align}
  \frac{[\mu_1']}{ v_j}M &> \frac{1}{ v_n}M, \\
  N+\frac{[\mu_2']}{v_j}M &>\frac{1}{ v_n}M,
\end{align}
the inequalities in \eqref{pf.T3.6} are satisfied if we pick $M,N \in \N_0$ sufficiently large and such that
\begin{align}
\begin{split} \label{pf.T3.6.1}
  &\frac{1}{ v_n}M>\frac{\delta}{\rho} N+c_1(\beta_0,n,\delta,\rho,m_2,\nu,L_1), \\
  &\frac{1}{ v_n}M>\frac{\delta}{\rho}M+c_2(\beta_0,n,\delta,\rho,m_1,\nu,L_1)\Longleftrightarrow\Big(\frac{1}{ v_n}-\frac{\delta}{\rho}\Big)M>c_2(\beta_0,n,\delta,\rho,m_1,\nu,L_1).
\end{split}
\end{align}
The system above is always solvable provided that $\frac{1}{ v_n}-\frac{\delta}{\rho}>0$, that is, if $\delta<\frac{\rho}{v_n}$. Thus, 
denoting $c_1(\beta_0):=c_1(\beta_0,n,\delta,\rho,m_1,\nu,L_1)$ and $ c_2(\beta_0):=c_2(\beta_0,n,\delta,\rho,m_1,\nu,L_1)$ for the sake of simplicity, the inequalities in \eqref{pf.T3.6.1} are simultaneously satisfied in the open cone 
\begin{align}
    \Gamma_2(\beta_0):=&\{(M, N)\in \mathbb{N}^2;   M>\frac{\delta v_n }{\rho} N+v_nc_1(\beta_0), M>\frac{v_n\rho}{\rho-v_n\delta} c_2(\beta_0)\}\\
    &\bigcap \{(M,N)\in\mathbb{N}^2; M> C_1(\beta_0), N>C_2(\beta_0\} \quad \text{if}\,\, 0<\delta<\frac{\rho}{v_n},
    \end{align}
and in the open quadrant
\begin{align}
H_{2}(\beta_0):=&\{(M, N)\in \mathbb{N}^2;  M>\max\{v_n c_1(\beta_0), v_n c_2(\beta_0)\}\}\\
&\bigcap \{(M,N)\in\mathbb{N}^2; M> C_1(\beta_0), N>C_2(\beta_0\} \quad \text{if}\,\, \delta=0.\end{align}
Hence the set of conditions that guarantee the convergence of $J_2$ is satisfied for any $(M,N)=$ $(M(\beta_0),N(\beta_0)) \in \Gamma_2(\beta_0)$ when $0<\delta<\frac{\rho}{v_n}$, and for any $(M,N)=(M(\beta_0),N(\beta_0))\in H_1(\beta_0)$  when $\delta=0$. Therefore, once again by Lemma \ref{FRcorollary}, $J_2\lesssim \|\sigma_1\|_{S^{m_1}_{\rho,\delta},a_1,b_1}\|\sigma_2\|_{S^{m_2}_{\rho,\delta},a_2,b_2}$ for suitably chosen $a_1,a_2,b_1,b_2 \in \N_0$.

Of course, in order to have the convergence of $J_1, J_2$ simultaneously, it is sufficient to choose $M(\beta_0), N(\beta_0)$ large enough, namely in the intersection $\Gamma_1\cap\Gamma_2$ when $0\neq\delta<\frac{\rho}{v_n}$, or in $H_1\cap H_2$ when $\delta=0$, which is always possible. Finally note that the validity of \eqref{pf.T3.6.1} does not depend on the index $j$ under consideration, so it holds for any $j=1,\ldots, n$.
 
 We are now left with the proof of the convergence of $J_3$ and $J_4$. Since the strategy should be clear now, we will immediately state the inequalities which are sufficient to conclude the convergence in the two cases.
 
 For $J_3$ to be convergent, we need that, for every $k=1,\ldots, n$,
\begin{align}
\begin{split} \label{pf.T3.7}
   M+\frac{[\mu_1'']}{v_k}N &> \frac{\delta}{\rho} N+c_1(\beta_0,n,\delta,\rho,m_2,\nu,L_1)\\
    \frac{[\mu_2'']}{ v_k}N &>\frac{\delta}{\rho}M+c_2(\beta_0,n,\delta,\rho,m_1,\nu,L_1).
\end{split}
\end{align}
Repeating the same reasoning as in the analysis of $J_2$, the system \eqref{pf.T3.7} is solvable if
\begin{align}
\begin{split} \label{pf.T3.7.2}
   &\frac{1}{v_n}N> \frac{\delta}{\rho} N+c_1(\beta_0,n,\delta,\rho,m_2,\nu,L_1)\\
    &\frac{1}{ v_n}N>\frac{\delta}{\rho}M+c_2(\beta_0,n,\delta,\rho,m_1,\nu,L_1),
\end{split}
\end{align}
meaning that a set of solutions of \eqref{pf.T3.7} is given by
\begin{align}
    \Gamma_3(\beta_0):=&\left\{(M,N)\in\mathbb{N}^2|\,   M<\frac{\rho}{\delta v_n}N-\frac{\rho}{\delta}c_2(\beta_0), N>\frac{\rho v_n}{\rho-\delta v_n} c_1(\beta_0)\right\}\\
    &\bigcap \{(M,N)\in\mathbb{N}^2; M> C_1(\beta_0), N>C_2(\beta_0\} \quad \text{if} \,\, 0<\delta<\frac{\rho}{v_n},
\end{align}
and by
\begin{align}
H_3(\beta_0):=&\left\{(M,N)\in\mathbb{N}^2| N>\max\{v_n c_1(\beta_0),v_n c_2(\beta_0)\}\right\}\\
&\bigcap \{(M,N)\in\mathbb{N}^2; M> C_1(\beta_0), N>C_2(\beta_0\} \quad \text{if} \,\, \delta=0.
\end{align}
Hence we have a set of solutions of \eqref{pf.T3.7} and the convergence of $J_3$ for suitably chosen positive $M(\beta_0)$ and $N(\beta_0)$ in the set of solutions. Moreover, by Lemma \ref{FRcorollary}, we have $J_3\lesssim \|\sigma_1\|_{S^{m_1}_{\rho,\delta},a_1,b_1}\|\sigma_2\|_{S^{m_2}_{\rho,\delta},a_2,b_2}$  for suitably chosen $a_1,a_2,b_1,b_2 \in \N_0$.

To have the convergence of $J_4$ (and the desired upper bound) for all $j,k=1,\ldots \dimG$, we need the parameters to satisfy
\begin{align}
   &\frac{[\mu_1']}{v_j}M+\frac{[\mu_1'']}{v_k}N> \frac{\delta}{\rho}N+c_1(\beta_0,n,\delta,\rho,m_2,\nu,L_1), \\
    &\frac{[\mu_2']}{v_j}M+\frac{[\mu_2'']}{v_k}N> \frac{\delta}{\rho}M+c_2(\beta_0,n,\delta,\rho,m_1,\nu,L_1).
    \label{pf.T3.8}
\end{align}
Now, since
\begin{align}
  &\frac{[\mu_1']}{v_j}M+\frac{[\mu_1'']}{v_k}N\geq \frac{1}{v_n}M+\frac{1}{v_n}N>\frac{1}{v_n}M,\\  
  &\frac{[\mu_2']}{v_j}M+\frac{[\mu_2'']}{v_k}N\geq \frac{1}{v_n}M+\frac{1}{v_n}N>\frac{1}{v_n}N,
\end{align}
for all $j,k=1,\ldots,\dimG$, we have \eqref{pf.T3.8} if the stricter conditions
\begin{align}
\frac{1}{v_n}M&>\frac{\delta}{\rho}N+c_1(\beta_0,n,\delta,\rho,m_2,\nu,L_1),\\
\frac{1}{v_n}N&> \frac{\delta}{\rho}M+c_2(\beta_0,n,\delta,\rho,m_1,\nu,L_1). 
\end{align}
 hold.
Note these conditions are satisfied for $\Gamma_2(\beta_0) \cap \Gamma_3(\beta_0)$ if $\delta\in (0,\frac{\rho}{v_n})$ and in
$H_2(\beta_0) \cap H_3(\beta_0)$ if $\delta=0$.
So for all solutions chosen from the according set, the conditions are satisfied and we have
$J_4\lesssim \|\sigma_1\|_{S^{m_1}_{\rho,\delta},a_1,b_1}\|\sigma_2\|_{S^{m_2}_{\rho,\delta},a_2,b_2}$ for suitably chosen $a_1,a_2,b_1,b_2 \in \N_0$.

Now to conclude the proof, it suffices to show that we can choose $(M,N)=(M(\beta_0), N(\beta_0))\in \Gamma :=\bigcap_{i=1}^3 \Gamma_i(\beta_0)$ if $\delta\in(0,\frac{\rho}{v_n})$ and $(M,N)=(M(\beta_0),N(\beta_0))\in H:=\bigcap_{i=1}^3 H_i$ if $\delta=0$.
Since the case when $\delta=0$ is trivial, we focus on the case when $\delta\in(0,\frac{\rho}{v_n})$. Recall that
\begin{align}
    \Gamma_1(\beta_0):=&
    \{ (M,N)\in\mathbb{N}^2; \frac{\delta}{\rho}N+c_1(\beta_0)<M<\frac{\rho}{\delta}N-\frac{\rho}{\delta}c_2(\beta_0)\}\\
    &\bigcap \{(M,N)\in\mathbb{N}^2; M> C_1(\beta_0), N>C_2(\beta_0\},
\end{align}
and that
\begin{align}
    &\Gamma_2(\beta_0)\bigcap\Gamma_3(\beta_0)=\{(M, N)\in \mathbb{N}^2;   \frac{\delta v_n }{\rho} N+v_nc_1(\beta_0)<M< \frac{\rho}{\delta v_n}N-\frac{\rho}{\delta}c_2(\beta_0),\\ 
    &N>\frac{\rho v_n}{\rho-\delta v_n} c_1(\beta_0), M>\frac{v_n\rho}{\rho-v_n\delta} c_2(\beta_0)\}\bigcap \{(M,N)\in\mathbb{N}^2; M> C_1(\beta_0), N>C_2(\beta_0\}.
\end{align}
Therefore, since 
\begin{align}
    \frac{\delta}{\rho}<\frac{\delta v_n}{\rho}<\frac{\rho}{v_n \delta}<\frac{\rho}{\delta}
\end{align}
for $\delta\in(0,\frac{\rho}{v_n})$, we have
\begin{align}
    \Gamma:=\bigcap_{i=1}^3 \Gamma_i\neq\emptyset
\end{align}
and we can choose $(M(\beta_0),N(\beta_0))\in \Gamma$ as large as necessary and such that $J_i$ is convergent and satisfies $J_i\lesssim \|\sigma_1\|_{S^{m_1}_{\rho,\delta},a_1,b_1}\|\sigma_2\|_{S^{m_2}_{\rho,\delta},a_2,b_2}$ for all $i=1,\ldots,4$ and suitably chosen $a_1,a_2,b_1,b_2 \in \N_0$.
Then, by choosing $M(\beta_0)\geq L_0, N(\beta_0)\geq L_0$\footnote{We recall that $L_0\in\mathbb{N}$ is fixed at the beginning of the proof of Step~1.}, and such that $(M(\beta_0), N(\beta_0)) \in \Gamma$ if $\delta\in (0,\frac{\rho}{v_n})$ and such that $(M(\beta_0), N(\beta_0)) \in H$ if $\delta=0$, we conclude that
 \begin{align}
 \| X^{\beta_0}_x r_{M(\beta_0),N(\beta_0)}(x,\pi)\pi(I+\mathcal{R})^{L_1}\|_{\mathscr{L}(\mathcal{H}_\pi)}
 \lesssim \|\sigma_1\|_{S^{m_1}_{\rho,\delta},a_1,b_1}\|\sigma_2\|_{S^{m_2}_{\rho,\delta},a_2,b_2}
\end{align}
 for suitably chosen $a_1,a_2,b_1,b_2 \in \N_0$. This completes the proof of Step~1.

\medskip
 
\noindent{\it Proof of Step~2.} 
Now we use Step~1 to prove that \eqref{Est.T3.1-new} holds true for $M_0,N_0\in\mathbb{N}$ with $L_0=\min\{M_0,N_0\}$.
 
Recall that for any fixed $\beta_0\in\mathbb{N}^n_0$ Step~1 holds true for some $M(\beta_0),N(\beta_0)\geq L_0$. However, it is not restrictive to choose $M(\beta_0), N(\beta_0)>\max\{M_0,N_0\}>L_0$ (for details we refer to the proof of Step~1), so we will make such a choice for these parameters in the rest of the proof. Now, for such $M(\beta_0),N(\beta_0)$ we have the following identity
\begin{align}
    r_{M_0,N_0}(x,\pi)&= 
    \iint_{G \times G} \Big(R_{x
    ,M(\beta_0)}^{\kappa_{\sigma_1}(\cdot, s^{-1}y)}(p_1(y,s)) + \!\!\!\!\underset{\substack{M_0<[\alpha]\leq M(\beta_0)} }{\overline{\sum}} q_{\alpha}(p_1(y,s))X^{\alpha}_x \kappa^\tau_{\sigma_1}(x,s^{-1}y) \Big) \nonumber
\\
    &\hspace{-10pt} \times \Big(R_{x,N(\beta_0)}^{\kappa_{\sigma_2}(\cdot, s)}(p_2(y,s))  +\underset{\substack{N_0<[\beta]\leq N(\beta_0)} }{\overline{\sum}} q_{\beta}(p_2(y,s))X^{\beta}_x \kappa_{\sigma_2}(x,s) \Big)  \hspace{2pt} \pi(y)^*  ds\,dy \nonumber\\
    &=\iint_{G\times G}R_{x
    ,M(\beta_0)}^{\kappa_{\sigma_1}(\cdot,  s^{-1}y)}(p_1(y,s))R_{x,N(\beta_0)}^{\kappa_{\sigma_2}(\cdot, s)}(p_2(y,s))\hspace{2pt} \pi(y)^* \,ds\,dy \label{T3.1}\\
    &\hspace{-45pt}+ \underset{\substack{N_0<[\beta]\leq N(\beta_0)} }{\overline{\sum}}\iint_{G\times G}R_{x
    ,M(\beta_0)}^{\kappa_{\sigma_1}(\cdot, s^{-1}y)}(p_1(y,s))q_{\beta}(p_2(y,s))X^{\beta}_x \kappa_{\sigma_2}(x,s) \hspace{2pt} \pi(y)^*  ds\,dy \label{T3.2}\\
    &\hspace{-45pt}+\underset{\substack{M_0<[\alpha]\leq M(\beta_0)} }{\overline{\sum}} \iint_{G\times G} R_{x,N(\beta_0)}^{\kappa_{\sigma_2}(\cdot,  s)}(s^{-1})q_{\alpha}(p_1(y,s))X^{\alpha}_x \kappa_{\sigma_1}(x,s^{-1}y)   \hspace{2pt} \pi(y)^*  ds\,dy \label{T3.3} \\
    &\hspace{-45pt}+\underset{\substack{M_0<[\alpha]\leq M(\beta_0)\\ N_0<[\beta]\leq N(\beta_0)} }{\overline{\sum}}\iint_{G\times G}
    q_{\alpha}(p_1(y,s))X^{\alpha}_{x_1=x} \kappa_{\sigma_1}(x_1,s^{-1}y)\\
    &\hspace{-10pt} \times q_{\beta}(p_2(y,s))X^{\beta}_x \kappa_{\sigma_2}(x,s) \hspace{2pt} \pi(y)^*  ds\,dy,\label{T3.4}
\end{align}
with 
\begin{equation}\label{T3.1.1}
  \| X^{\beta_0}_x\eqref{T3.1}(I+\mathcal{R})^{-\frac{m_1+m_2+\delta[\beta_0]-(\rho-\delta)(L+1)}{\nu}}\|_{\mathscr{L}(\mathcal{H}_\pi)}\lesssim \|\s\|_{S^{m_1}_{\rho,\delta}a_1,b_1}\|\ss\|_{S^{m_2}_{\rho,\delta}a_2,b_2},
\end{equation}
\begin{equation}\label{T3.1.2}
  \| X^{\beta_0}_x\eqref{T3.2}(I+\mathcal{R})^{-\frac{m_1+m_2+\delta[\beta_0]-(\rho-\delta)(L+1)}{\nu}}\|_{\mathscr{L}(\mathcal{H}_\pi)}\lesssim \|\s\|_{S^{m_1}_{\rho,\delta}a_1,b_1}\|\ss\|_{S^{m_2}_{\rho,\delta}a_2,b_2},
\end{equation}
\begin{equation}\label{T3.1.3}
  \| X^{\beta_0}_x\eqref{T3.3}(I+\mathcal{R})^{-\frac{m_1+m_2+\delta[\beta_0]-(\rho-\delta)(L+1)}{\nu}}\|_{\mathscr{L}(\mathcal{H}_\pi)}\lesssim \|\s\|_{S^{m_1}_{\rho,\delta}a_1,b_1}\|\ss\|_{S^{m_2}_{\rho,\delta}a_2,b_2},
\end{equation}
\begin{equation}\label{T3.1.4}
  \| X^{\beta_0}_x\eqref{T3.4}(I+\mathcal{R})^{-\frac{m_1+m_2+\delta[\beta_0]-(\rho-\delta)(L+1)}{\nu}}\|_{\mathscr{L}(\mathcal{H}_\pi)}\lesssim \|\s\|_{S^{m_1}_{\rho,\delta}a_1,b_1}\|\ss\|_{S^{m_2}_{\rho,\delta}a_2,b_2},
\end{equation}
for suitably chosen $a_1,a_2,b_1,b_2\in\mathbb{N}_0$.
More precisely, due to the similarity of \eqref{T3.1}, \eqref{T3.2},\eqref{T3.3} and \eqref{T3.4} with $T_3, T_1, T_2$ and $T_0$ respectively, we have that \eqref{T3.1.1} follows from Step~1 of the proof, while
\eqref{T3.1.2} and \eqref{T3.1.3} follow from Lemma~\ref{lemma1} and 
Lemma~\ref{lemma2}, respectively. We remark that in these terms the presence of $q_\alpha(p_1(y,s))$ and $q_\beta(p_2(y,s))$  are not an obstruction to the convergence of the integrals. Indeed, they can be rewritten as combinations of homogeneous polynomials according to formulas \eqref{eqn1811-1510} and \eqref{eqn1811-1511}, so their presence actually improves the convergence and therefore yields the same estimates as in the Lemmas~\ref{lemma1} and \ref{lemma2}. Moreover note that $X^\beta_x \kk(x,s)$ is the associated kernel of the symbol $X_x^\beta\sigma_2(x,\pi)$, whose seminorm is bounded by a seminorm of $\sigma_2$. Hence, an application of Lemma~\ref{lemma1} yields exactly \eqref{T3.1.2}. For the same reason \eqref{T3.1.3} is given by an application of Lemma~\ref{lemma2}.

Finally, in order to prove \eqref{T3.1.4}, we first note that a combination of \eqref{eqn1811-1510}, \eqref{eqn1811-1511} and a standard change of variables gives
\begin{align}
   & \iint_{G\times G}
    q_{\alpha}(p_1(y,s))X^{\alpha}_{x_1=x} \kappa_{\sigma_1}(x_1,s^{-1}y)
    q_{\beta}(p_2(y,s))X^{\beta}_x \kappa_{\sigma_2}(x,s)  \hspace{2pt} \pi(y)^*  ds\,dy\\
    &=\underset{\substack{[\alpha_1]+[\alpha_2]=[\alpha]\\ [\beta_1]+[\beta_2]=[\beta]}}{\overline{\sum}}\iint_{G\times G}
    \tilde{q}_{\alpha_2}(s^{-1}y)\tilde{q}_{\beta_2}(s^{-1}y)X^{\alpha}_{x_1=x}\kappa_{\sigma_1}(x_1,s^{-1}y)\pi(s^{-1}y)^*\\
    &\quad \times  
    \tilde{q}_{\alpha_1}(s)\tilde{q}_{\beta_1}(s)X^{\beta}_x \kappa_{\sigma_2}(x,s) \pi(s)^* ds\,dy\\
    &=\underset{\substack{[\alpha_1]+[\alpha_2]=[\alpha]\\ [\beta_1]+[\beta_2]=[\beta]}}{\overline{\sum}} (\Delta^{\alpha_2}\Delta^{\beta_2}X^\alpha \sigma_1 (x,\pi)) (\Delta^{\alpha_1}\Delta^{\beta_1}X^\beta \sigma_2 (x,\pi))\label{pf.T3.12}.
\end{align}
So, by \eqref{pf.T3.12} and an application of the Leibniz rule, we obtain
\begin{align}
    \eqref{T3.1.4}\leq&\underset{\substack{M_0<[\alpha]\leq M(\beta_0)\\ N_0<[\beta]\leq N(\beta_0)} }{\overline{\sum}}\underset{\substack{[\alpha_1]+[\alpha_2]=[\alpha]\\ [\beta_1]+[\beta_2]=[\beta]}}{\overline{\sum}}\underset{\substack{[\beta_{01}]+[\beta_{02}]=[\beta_0]}}{\overline{\sum}}\\
   &\hspace{20pt} \|(\Delta^{\alpha_2}\Delta^{\beta_2}X^{\beta_{01}}X^\alpha \sigma_1 (x,\pi)) (\Delta^{\alpha_1}\Delta^{\beta_1}X^{\beta_{02}}X^\beta \sigma_2 (x,\pi))\\
    & \times 
    (I+\mathcal{R})^{-\frac{m_1+m_2+\delta[\beta_0]-(\rho-\delta)(L+1)}{\nu}}
    \|_{\mathscr{L}(\mathcal{H}_\pi)} \lesssim 
    \|\s\|_{S^{m_1}_{\rho,\delta}a_1,b_1}\|\ss\|_{S^{m_2}_{\rho,\delta}a_2,b_2}
\end{align}
for suitably chosen $a_1,a_2,b_1,b_2\in\mathbb{N}_0$, which proves \eqref{T3.1.4}. Let us emphasize that we may choose the same $a_1,a_2,b_1, b_2$ for each of the estimates \eqref{T3.1.1}-\eqref{T3.1.4}, by choosing them sufficiently large.

Hence, we have proved that, for all $\beta_0\in\mathbb{N}_0^n$,
\begin{align}
    \| X^{\beta_0}_x r_{M_0,N_0}(x,\pi)(I+\mathcal{R})^{-\frac{m_1+m_2+\delta[\beta_0]-(\rho-\delta)(L_0+1)}{\nu}}\|_{\mathscr{L}(\mathcal{H}_\pi)}\lesssim \|\s\|_{S^{m_1}_{\rho,\delta}a_1,b_1}\|\ss\|_{S^{m_2}_{\rho,\delta}a_2,b_2},
\end{align}
 which concludes the proof of Step~2 and thus of the lemma.
\end{proof}

\medskip

To summarize, we have proved the following so far: if $\sigma_1, \sigma_2\in S^{-\infty}(G)$ and such that $\k,\kk\in\mathscr{S}(G\times G)$, then for every fixed $M,N\in \mathbb{N}$, $L:=\min\{M,N\}$, and for all $a,b\in \mathbb{N}_0$, there exist $a_1,a_2,b_1,b_2\in\mathbb{N}_0$ such that 
\begin{align}
\|R_L^\sigma \|_{S^{m_1+m_2-(\rho-\delta)(L+1)}_{\rho,\delta,a,b}}=\|\sum_{j=1}^3T_j\|_{S^{m_1+m_2-(\rho-\delta)(L+1)}_{\rho,\delta,a,b}}\lesssim \|\sigma_1\|_{S^{m_1}_{\rho,\delta,a_1,b_1}}\|\sigma_2\|_{S^{m_2}_{\rho,\delta,a_2,b_2}},
\end{align}
where
\begin{align}
R^{\sigma}_L := \sigma -\!\!\!\underset{\substack{[\alpha] \leq M, \\ [\beta] \leq N}}{\sum} \underset{\substack{[\alpha_1] + [\alpha_2] = [\alpha], \\ [\beta_1] + [\beta_2] = [\beta]}}{\sum} \hspace{5pt}\!\!\!\!\! c_{\alpha_1, \alpha_2} \, c_{\beta_1, \beta_2} \Bigl ( \Delta^{\alpha_2} \Delta^{\beta_1} X^\alpha \sigma_1 \Bigr ) \, \Bigl ( \Delta^{\beta_2} \Delta^{\alpha_1} X^\beta \sigma_2 \Bigr ).
\end{align}
This, in particular, implies that $R_L^\sigma \in S^{m_1+m_2-(\rho-\delta)(L+1)}_{\rho,\delta}(G)$, which is the desired conclusion of Part~II.

Our final goal now is the extension of this result to the general case when $\sigma_1\in S^{m_1}_{\rho, \delta}(G)$, $\sigma_2 \in S^{m_2}_{\rho, \delta}(G)$. This will be done in the third and final part of the proof.
\end{proof}

\medskip

\begin{proof}[Proof of Theorem~\ref{thm:asym_exp_comp_tau}, Part III]
To pass from $\sigma_1, \sigma_2 \in S^{-\infty}(G)$ with $\kappa_{\sigma_1}, \kappa_{\sigma_2} \in \SC(G \times G)$ to general $\sigma_1\in S^{m_1}_{\rho, \delta}(G)$, $\sigma_2 \in S^{m_2}_{\rho, \delta}(G)$, we employ an argument based on the approximation of $\sigma_1$ and $\sigma_2$ by nets of smoothing symbols as in the proof of Theorem~\ref{thm:asym_exp_ch_qu}.
By \cite[Lem.~5.4.11]{FR} and the equivalence of seminorms \cite[Thm.~5.5.20]{FR}, there exist nets
\begin{align*}
    \{ \sigma_{j, \varepsilon} \}_{\varepsilon \in (0, 1)} \subseteq S^{-\infty}(G), \hspace{5pt} j = 1, 2,
\end{align*}
with $\kappa_{j,\varepsilon}\in \TD(G\times G)$, such that for every $\theta \geq 0$ and every seminorm $\| \, . \, \|_{S^{m_j}_{\rho, \delta}, a, b}$ there exist a constant $C_j = C_j(a, b, m_j, \theta, \rho, \delta) > 0$ and $a', b' \in \N_0$ such that
\begin{align}
    \| \sigma_j - \sigma_{j, \varepsilon} \|_{S^{m_j + \theta}_{\rho, \delta}, a, b} &\leq C_j \| \sigma_1 \|_{S^{m_j}_{\rho, \delta}, a', b'} \, \varepsilon^{\frac{\theta}{\hdeg}}, \label{eq:smoothing_app} \\
    \|\sigma_{j, \varepsilon} \|_{S^{m_j}_{\rho, \delta}, a, b} &\leq C_j \| \sigma_j \|_{S^{m_j}_{\rho, \delta}, a', b'} \, \varepsilon^{\frac{\theta}{\hdeg}} \label{eq:smoothing_bound}
\end{align}
for $j = 1, 2$.

Let now $\sigma_\varepsilon$ be the symbol such that $\Optau(\sigma_{1,\varepsilon})\Optau(\sigma_{2,\varepsilon})=\Optau(\sigma_{\varepsilon})$. Then, by the main part of the proof and the fact that
\begin{align}
    \sigma_\varepsilon= \underset{\substack{[\alpha] \leq M, \\ [\beta] \leq N}}{\sum} \underset{\substack{[\alpha_1] + [\alpha_2] = [\alpha], \\ [\beta_1] + [\beta_2] = [\beta]}}{\sum} \hspace{5pt}\!\!\!\!\! c_{\alpha_1, \alpha_2} \, c_{\beta_1, \beta_2} \Bigl ( \Delta^{\alpha_2} \Delta^{\beta_1} X^\alpha \sigma_{1,\varepsilon} \Bigr ) \, \Bigl ( \Delta^{\beta_2} \Delta^{\alpha_1} X^\beta \sigma_{2,\varepsilon} \Bigr )+R_{L}^{\sigma_\varepsilon},
\end{align}
we have, for every $a,b\in\mathbb{N}_{0}, \theta >0$, $M,N\in \mathbb{N}$, $L:=\min\{M,N\}$,   
\begin{align}
    \| \sigma_{ \varepsilon}  \|_{S^{m_1 + m_2 + \theta}_{\rho, \delta}, a, b}\lesssim \|\sigma_{1,\varepsilon}\|_{S^{m_1}_{\rho, \delta}, a_1, b_1} \|\sigma_{2,\varepsilon}\|_{S^{m_2}_{\rho, \delta}, a_2, b_2}
\end{align}
for some $a_1,a_2,b_1,b_2\in\mathbb{N}_0$.
Thus, let $a, b \in \N_0$ and $\theta >0$. Since for all $\varepsilon, \varepsilon'\in(0,1)$
\begin{align}
\Optau(\sigma_{\varepsilon})-\Optau(\sigma_{\varepsilon'})&= \Optau(\sigma_{1,\varepsilon}-\sigma_{1,\varepsilon'}) \Optau(\sigma_{2,\varepsilon})+ \Optau(\sigma_{2,\varepsilon}-\sigma_{2,\varepsilon'}) \Optau(\sigma_{1,\varepsilon'})
\end{align}
and 
\begin{align}
    \| \sigma_{ j,\varepsilon}  - \sigma_{j, \varepsilon'}  \|_{S^{m_j}_{\rho, \delta}, a_j, b_j}\lesssim \| \sigma_{ j,\varepsilon}  - \sigma_{j, \varepsilon'}  \|_{S^{m_j+ \theta}_{\rho, \delta}, a_j, b_j},
\end{align}
we obtain the bilinear estimate
\begin{align}
     \| \sigma_{ \varepsilon}  - \sigma_{ \varepsilon'}  &\|_{S^{m_1 + m_2 + \theta}_{\rho, \delta}, a, b} \\
    &\hspace{-25pt}\lesssim
    \| \sigma_{ 1,\varepsilon}  - \sigma_{1, \varepsilon'}  \|_{S^{m_1+ \theta}_{\rho, \delta}, a_1, b_1} \| \sigma_{ 2,\varepsilon} \|_{S^{m_2}_{\rho, \delta}, a_2, b_2}+ 
    \| \sigma_{ 2,\varepsilon}  - \sigma_{2, \varepsilon'}  \|_{S^{m_2+ \theta}_{\rho, \delta}, a_2, b_2} \| \sigma_{ 1,\varepsilon'} \|_{S^{m_1}_{\rho, \delta}, a_1, b_1},
\end{align}
for some $a_1, a_2, b_1, b_2\in \N_0$.

Now, by \eqref{eq:smoothing_app}, for all $\eta>0$ and all
\begin{align}
    \varepsilon, \varepsilon'\in \left( 0, \min\left\{ \left(\frac{\eta}{2 C_1 C_2 \|\sigma_1\|_{S^{m_1}_{\rho,\delta,a_1',b_1'}}\|\sigma_2\|_{S^{m_2}_{\rho,\delta,a_2',b_2'}} } \right)^{\nu/\theta}, 1 \right\}\right),
\end{align}
 we have
\begin{align*}
    \| \sigma_{1, \varepsilon} - \sigma_{1, \varepsilon'} \|_{S^{m_1 + \theta}_{\rho, \delta}, a_1, b_1} &\leq \frac{\eta}{C_2 \| \sigma_2 \|_{S^{m_2}_{\rho, \delta}, a'_2, b'_2}}, \\
    \| \sigma_{2, \varepsilon} - \sigma_{2, \varepsilon'} \|_{S^{m_2 + \theta}_{\rho, \delta}, a_2, b_2} &\leq \frac{\eta}{ C_1 \| \sigma_1 \|_{S^{m_1}_{\rho, \delta}, a'_1, b'_1}},
\end{align*}
which allows us to conclude from \eqref{eq:smoothing_bound} that
\begin{align}
    \| \sigma_{ \varepsilon}  - \sigma_{ \varepsilon'} \|_{S^{m_1 + m_2 + \theta}_{\rho, \delta}, a, b} \leq \eta.
\end{align}
Since $a, b \in \N_0$ and $\eta >0$ were arbitrary, the net $\{\sigma_\varepsilon\}_{\varepsilon \in (0, 1)}\subseteq S^{m_1 + m_2}_{\rho, \delta}$ is Cauchy in the coarser topology of $S^{m_1 + m_2 + \theta}_{\rho, \delta}(G) \supseteq S^{m_1 + m_2}_{\rho, \delta}(G)$, but its limit $\lim\limits_{\varepsilon\rightarrow 0} \sigma_{\varepsilon}  =: \sigma'$ nevertheless lies in a closed bounded subset of $S^{m_1 + m_2}_{\rho, \delta}$ because of \eqref{eq:smoothing_bound}. By \cite[Lem.~5.4.11~(3)]{FR} and Proposition~\ref{prop:cont_on_SC}, this implies that $\Optau\bigl ( \sigma_{\varepsilon} \bigr )$ converges strongly on $\SC(G)$ to $\Optau(\sigma')$ with $\sigma' \in S^{m_1 + m_2}_{\rho, \delta}(G)$. Since each $\Optau(\sigma_{j, \varepsilon})$, $j = 1, 2$, strongly converges on $\SC(G)$ to $\Optau(\sigma_j)$, this gives
\begin{align*}
    \Optau(\sigma') f = \lim_{\varepsilon} \Optau \bigl ( \sigma_{\varepsilon}  \bigr )f = \lim_{\varepsilon} \Optau(\sigma_{1, \varepsilon}) \Optau (\sigma_{2, \varepsilon})f = \Optau(\sigma_1) \Optau (\sigma_2)f =: \Optau(\sigma)f
\end{align*}
for all $f \in \SC(G)$. By the Schwartz kernel theorem, the kernels in $\TD(G \times G)$ of these two continuous operators coincide, so $\sigma= \sigma' \in S^{m_1 + m_2}_{\rho, \delta}(G)$.
Hence, since $\sigma_\varepsilon\underset{\varepsilon\rightarrow 0}{\longrightarrow}\sigma$ in $S^{m_1+m_2}_{\rho,\delta}(G)$, and $\sigma_{j,\varepsilon}\underset{\varepsilon\rightarrow 0}{\longrightarrow}\sigma_j$ in $S^{m_j}_{\rho,\delta}(G)$ for $j=1,2$, we can finally conclude that, for all fixed $M,N\in\mathbb{N}$, and for $L:=\min\{M,N\}$,
\begin{align}
    R_L^{\sigma_\varepsilon}\underset{\varepsilon\rightarrow 0}{\longrightarrow} R_L^\sigma \quad  \text{in}\,\,S^{m_1+m_2-(\rho-\delta)(L+1)}_{\rho,\delta}(G).
\end{align}
 Note that the net converges in the space $S^{m_1+m_2-(\rho-\delta)(L+1)}_{\rho,\delta}(G)$, due to \eqref{eq:smoothing_app}, but not necessarily in $S^{m_1+m_2-(\rho-\delta)L}_{\rho,\delta}(G)$.
This completes the proof of the theorem.
\end{proof}

\medskip

In order to recover in an easy way the well-known classical asymptotic expansions for the composite symbol in, e.g., the Weyl calculus on $\mathbb{R}^n$ or the Kohn-Nirenberg calculus on $\H$, it is convenient to rewrite \eqref{eq:asym_exp_comp_tau} as
\begin{align}
    \sigma \sim \sum_{j=0}^\infty\left(\underset{\substack{[\alpha] +[\beta] =j}}{\sum} \hspace{5pt} \underset{\substack{[\alpha_1] + [\alpha_2] = [\alpha], \\ [\beta_1] + [\beta_2] = [\beta]}}{\sum} \hspace{5pt} c_{\alpha_1, \alpha_2} \, c_{\beta_1, \beta_2} \Bigl ( \Delta^{\alpha_2} \Delta^{\beta_1} X^\alpha \sigma_1 \Bigr ) \, \Bigl ( \Delta^{\beta_2} \Delta^{\alpha_1} X^\beta \sigma_2 \Bigr ) \right).\label{eqn.new.asympt.formula}
\end{align}
Introducing the notations $\omega_j := \omega_j(\sigma_1,\sigma_2)$ and
\begin{align}
    \hspace{-20pt} \omega_j(\sigma_1,\sigma_2):= \!\!\underset{\substack{[\alpha] +[\beta] =j}}{\sum} \hspace{5pt} \underset{\substack{[\alpha_1] + [\alpha_2] = [\alpha], \\ [\beta_1] + [\beta_2] = [\beta]}}{\sum}  c_{\alpha_1, \alpha_2} \, c_{\beta_1, \beta_2} \Bigl ( \Delta^{\alpha_2} \Delta^{\beta_1} X^\alpha \sigma_1 \Bigr ) \, \Bigl ( \Delta^{\beta_2} \Delta^{\alpha_1} X^\beta \sigma_2\Bigr)
    \label{eqn.Tj}
\end{align}
and
\begin{align}
    R_M=R_M(\sigma_1,\sigma_2):=\sigma-\sum_{j=0}^M\, \omega_j(\sigma_1,\sigma_2) \label{eqn.Rj}
\end{align}
for
$\omega_j(\sigma_1,\sigma_2)\in S^{m_1+m_2-(\rho-\delta)j}_{\rho,\delta}(G)$ and $R_M(\sigma_1,\sigma_2)\in S^{m_1+m_2-(\rho-\delta)(M+1)}_{\rho,\delta}(G)$ by Theorem \ref{thm:asym_exp_comp_tau}, we may abbreviate the asympotitic expansion as
\begin{align}
    \sigma\sim \sum_{j=0}^\infty \, \omega_j.
\end{align}
Note that each symbol appearing in this expression belongs to $S^{m_1+m_2-(\rho-\delta)j}_{\rho,\delta}(G)$. 

\medskip

In what follows we shall use the symbolic asymptotic composition formula as written in \eqref{eqn.new.asympt.formula} to recover some well-known classical results.

\begin{example} \label{ex:asym_exp_comp_KN_G}
Let $G$ be an arbitrary graded group. Choosing the constant quantizing function $\tau = e_G$, we get $p_1(y, z) = e_G$ and $p_2(y, z) = y^{-1}z$. It follows that $c_{\alpha_1, \alpha_2} = \delta_{\alpha_1, 0} \, \delta_{\alpha_2, 0}$ for all $\alpha \in \N_0$, and that $c_{\beta_1, \beta_2} = \delta_{\beta_2, 0}$ for all $\beta \in \N_0$, thus \eqref{eqn.new.asympt.formula}
 recovers the asymptotic expansion
\begin{align}
    \sigma_1 \circ_{\mathrm{KN}} \sigma_2 \sim \sum_{j= 0}^\infty\sum_{[\beta] = j} \bigl ( \Delta^{\beta} \sigma_1 \bigr ) \, \bigl ( X^\beta \sigma_2 \bigr )
\end{align}
in the \textit{Kohn-Nirenberg quantization} due to \cite[Cor.~5.58]{FR}. For the special case $G = \R^n$, this gives the well-known expansion
\begin{align}
    \sigma_1 \circ_{\mathrm{KN}} \sigma_2 \sim \sum_{j= 0}^\infty\sum_{|\beta| = j}^\infty \bigl ( (i^{-1} \partial_\xi)^{\beta} \sigma_1 \bigr ) \, \bigl ( \partial_x^\beta \sigma_2 \bigr )
\end{align}
for all symbols $\sigma_1 \in S^{m_1}_{\rho,\delta}(\R^n)$, $\sigma_2\in S^{m_1}_{\rho,\delta}(\R^n)$ with $m_1, m_2 \in \R$, $0 \leq \delta < \rho \leq 1$.
\end{example}

\begin{example} \label{ex:asym_exp_comp_Weyl}
When $G = \R^n$ and the symmetry function $\tau(x) = \frac{x}{2}$, i.e., the group is Abelian and $\tau$ symmetric and linear, we have
\begin{align}
\begin{array}{rcccccccl}
    p_1(y, z) &=& \tau(y) \tau({z}^{-1}y)^{-1} &=& \frac{y}{2} - \frac{y - z}{2} &=& \frac{z}{2} &=& \tau(z), \\
    p_2(y, z) &=& \tau(y) y^{-1}z \tau(z)^{-1} &=& \frac{y}{2} - y + z - \frac{z}{2} &=& \frac{z - y}{2} &=& \tau(z^{-1}y)^{-1}.
\end{array}
\end{align}
It follows that
\begin{align}
\begin{array}{rcccl}
     q_\alpha \bigl( p_1(y, z) \bigr ) &=& q_\alpha \bigl( \frac{z}{2} \bigr ) &=& \frac{z^{\alpha}}{2^{|\alpha|}\alpha !}, \\
    q_\beta \bigl ( p_2(y, z) \bigr ) &=& q_\beta \bigl( \frac{z - y}{2} \bigr ) &=& (-1)^{|\beta|} \frac{(y - z)^{\beta}}{2^{|\beta|}\beta !},
\end{array}
\end{align}
so the expansion~\eqref{eqn.new.asympt.formula} recovers the classical asymptotic expansion
\begin{align}
    \sigma_1 \circ_\mathrm{w} \sigma_2 \sim \sum_{j = 0}^\infty \underset{\substack{|\alpha|+|\beta| = j}}{\sum} \frac{(-1)^{|\beta|}}{\alpha ! \beta ! 2^{|\alpha| + |\beta|}} \Bigl ( (i^{-1} \partial_\xi)^\beta \partial_x^\alpha \sigma_1 \Bigr ) \Bigl ( (i^{-1} \partial_\xi)^\alpha \partial_x^\beta \sigma_2 \Bigr ) \label{eq:asym_exp_comp_Weyl}
\end{align}
in the \textit{Weyl quantization}~\eqref{eq:Weyl_quant} for all symbols $\sigma_1 \in S^{m_1}_{\rho,\delta}(\R^n)$, $\sigma_2\in S^{m_1}_{\rho,\delta}(\R^n)$ with $m_1, m_2 \in \R$, $0 \leq \delta < \rho \leq 1$.
\end{example}

\begin{example} \label{ex:asym_exp_comp_KN_Weyl}
A curious example on $G = \R^2$ arises from $\tau(x_1, x_2) = (0, \frac{x_2}{2})$, which defines a sort of \textit{hybrid Kohn-Nirenberg-Weyl quantization}, which is not included in the classical $\tau$-quantizations \eqref{eq:tau_quant}. By the observations made in Examples~\ref{ex:asym_exp_comp_KN_G} and \ref{ex:asym_exp_comp_Weyl}, the corresponding asymptotic expansion \eqref{eq:asym_exp_comp_tau} is given by
\begin{align*}
    \sigma_1 \circ_\tau \sigma_2 \sim 
    \underset{\substack{ j= 0}}{\sum^\infty}
\,\underset{\substack{l_1+l_2+k_2 = j \\  }}{\sum} \frac{(-1)^{l_2}}{k_2 ! l_2 ! 2^{k_2 + l_2}} 
    \Bigl ( (i^{-1} \partial_{\xi_1})^{l_1}(i^{-1} \partial_{\xi_2})^{l_2} \, \partial_{x_2}^{k_2} \sigma_1 \Bigr ) \Bigl ( (i^{-1} \partial_{\xi_2})^{k_2} \, \partial_{x_1}^{l_1} \partial_{x_2}^{l_2} \sigma_2 \Bigr ).
\end{align*}
\end{example}

\begin{example} \label{ex:asym_exp_comp_onehalf_tau_Hn}
Let $G$ be the Heisenberg group $\H$, equipped with the canonical homogeneous dilations, and let $\tau$ be the symmetry function~\eqref{eq:onehalf_tau_Hn}, which in exponential coordinates is given by $\tau(x) = \big ( \frac{x_1}{2}, \ldots, \frac{x_{2n+1}}{2} \bigr )$. In the following, we will write out the asymptotic expansion for the homogeneous orders greater or equal $m_1 + m_2 - 3(\rho - \delta)$ and compare the resulting version of \eqref{eq:asym_exp_comp_tau} with its ($(2n+1)$-dimensional) Euclidean counterpart from Example~\ref{ex:asym_exp_comp_Weyl}. To determine the desired coefficients, we have to solve
\begin{align*}
    q_\alpha \bigl( p_1(y, z) \bigr ) &= \underset{[\alpha_1] + [\alpha_2] = [\alpha]}{\sum} c_{\alpha_1, \alpha_2} \tilde{q}_{\alpha_1}(z) \tilde{q}_{\alpha_2}(z^{-1}y), \\
    q_\beta \bigl ( p_2(y, z) \bigr ) &= \underset{[\beta_1] + [\beta_2] = [\beta]}{\sum} c_{\beta_1, \beta_2} \tilde{q}_{\beta_1}(z^{-1}y) \tilde{q}_{\beta_2}(z)
\end{align*}
for $[\alpha], [\beta] = 1, 2$. To do so, we observe that 
\begin{align*}
    &p_1(y, z) = \tau(y) \tau({y}^{-1}z) \\
    &= \Bigl ( \frac{z_1}{2}, \ldots, \frac{z_{2n}}{2}, \frac{z_{2n+1}}{2} - \frac{1}{8} \sum_{j=1}^n \bigl( z_j (y_{n+j} - z_{n+j}) -  z_{n+j} (y_j - z_j) \bigr ) \Bigr ), \\
    &p_2(y, z) =\tau({y}^{-1}) \tau(z) \\
    &= \Bigl ( \frac{z_1 - y_1}{2}, \ldots, \frac{z_{2n} - y_{2n}}{2}, \frac{z_{2n+1} - y_{2n+1}}{2} + \frac{1}{8} \sum_{j=1}^n \bigl( (y_{n+j} - z_{n+j}) z_j - (y_j - z_j) z_{n+j} \bigr ) \Bigr ),
\end{align*}
and we recall from \cite[Ex.~5.2.4]{FR} that the homogeneous polynomials $q_{\gamma}$ of degree $[\gamma] = 1$ are precisely the monomials
\begin{align*}
    q_{e_j}(x) = x_j, \hspace{5pt} j = 1, \ldots, 2n,
\end{align*}
and that of degree $[\gamma] = 2$ the polynomials
\begin{align*}
\begin{array}{rcll}
    q_{e_j + e_k}(x) &=& x_j x_k &,\hspace{5pt} j, k = 1, \ldots, 2n, \\
    q_{e_{2n+1}}(x) &=& x_{2n+1} - \frac{1}{2} \sum_{j = 1}^n x_j x_{n+j}&.
\end{array}
\end{align*}
It follows that
\begin{align*}
\begin{array}{rcccll}
    q_{e_j}\bigl ( p_1(y, z) \bigr ) &=& \frac{z_j}{2} &=& &-\frac{1}{2} \tilde{q}_{e_j}(z), \\
    q_{e_k}\bigl ( p_2(y, z) \bigr ) &=& \frac{y_k - z_k}{2} &=& &+\frac{1}{2} \tilde{q}_{e_k}(z^{-1}y),
\end{array}
\hspace{5pt} j, k = 1, \ldots, 2n,
\end{align*}
and that
\begin{align*}
    &q_{e_{2n+1}} \bigl ( p_1(y, z) \bigr ) = \frac{z_{2n+1}}{2} - \frac{1}{8} \sum_{j=1}^n \bigl( z_j (y_{n+j} - z_{n+j}) -  z_{n+j} (y_j - z_j) \bigr ) - \frac{1}{8} \sum_{j=1}^n z_j z_{j+n} \\
    &\hspace{0pt} = -\frac{1}{2}  \tilde{q}_{e_{2n+1}}(z)  - \frac{1}{8} \sum_{j=1}^n \Bigl( \tilde{q}_{e_j}(z) \tilde{q}_{e_{n+j}}(z^{-1}y) -  \tilde{q}_{e_{n+j}}(z) \tilde{q}_{e_j}(z^{-1}y) \Bigr )- \frac{3}{8} \sum_{j=1}^n (\tilde{q}_{e_j} \tilde{q}_{e_{n+j}})(z), \\
    &q_{e_{2n+1}} \bigl ( p_2(y, z) \bigr ) =  \frac{z_{2n+1} - y_{2n+1}}{2} -\frac{1}{2}\sum_{j=1}^n \frac{z_j-y_j}{2} \frac{z_{n+j} -y_{n+j}}{2}\\
    &  \hspace{40pt}+\frac{1}{8} \sum_{j=1}^n \Bigl( (y_{n+j} - z_{n+j}) z_j - (y_j - z_j) z_{n+j}  - (z_j - y_j) (z_{j+n} - y_{n+j}) \Bigr ) \\
    &\hspace{0pt} = \frac{1}{2} \tilde{q}_{e_{2n+1}}(z^{-1}y) + \frac{1}{8} \sum_{j=1}^n \Bigl ( \tilde{q}_{e_j}(z^{-1}y)\tilde{q}_{e_{n+j}}(z) - \tilde{q}_{e_{n+j}}(z^{-1}y)\tilde{q}_{e_j}(z) + (\tilde{q}_{e_j} \tilde{q}_{e_{n+j}})(z^{-1}y) \Bigr ).
\end{align*}
Collecting all the contributions, we can now approximate \eqref{eq:asym_exp_comp_tau} by
\begin{align*}
    \sigma \sim \sigma_1 \sigma_2 & + \omega_{1} +\omega_{2} \mod S^{m_1 + m_2 - 3(\rho - \delta)}(\H),
\end{align*}
with $\omega_{1} \in S^{m_1 + m_2 - (\rho - \delta)}(\H)$ and $\omega_{2} \in S^{m_1 + m_2 - 2(\rho - \delta)}(\H)$, given by\footnote{The multi-index $\alpha = e_j \in \N_0^\dimG$ denotes the Euclidean coordinate vector whose $j$-th coordinate equals $1$, while all others vanish. In accordance with this notation, the left-invariant derivative $X{e_j + e_k}$ denotes $X_j X_k$ when $j \leq k$ and $X_k X_j$ when $j \geq k$}
\begin{align}
	\omega_{1} = \frac{1}{2} \sum_{j=1}^{2n} \Bigl ( \bigl ( X^{e_j} \sigma_1 \bigr ) \bigl ( \Delta^{e_j} \sigma_2) - \bigl ( \Delta^{e_j} \sigma_1 \bigr ) \bigl ( X^{e_j} \sigma_2 \bigr ) \Bigr ) \label{eq:hom_Poisson_bracket_Hn}
\end{align}
and
\begin{align}
	\omega_{2} &= \frac{1}{4} \sum_{j, k = 1}^{2n} \Bigl ( \bigl( X^{e_j + e_k} \sigma_1 \bigr ) \bigl ( \Delta^{e_j + e_k} \sigma_2 \bigr ) - \bigl( \Delta^{e_j} X^{e_k} \sigma_1 \bigr ) \bigl ( \Delta^{e_k} X^{e_j} \sigma_2 \bigr ) \label{eq:asym_exp_comp_onehalf_tau_Hn_standard_terms_hom1} \\
	&+ \bigl( \Delta^{e_j + e_k} \sigma_1 \bigr ) \bigl ( X^{e_j + e_k} \sigma_2 \bigr ) \Bigr ) - \frac{3}{8}  \sum_{j=1}^n \bigl (X^{e_j + e_k} \sigma_1 \bigr ) \bigl ( \Delta^{e_j + e_{n+j}} \sigma_2 \bigr ) \nonumber \\
	&- \frac{1}{8} \sum_{j=1}^n \Bigl ( \bigl ( \Delta^{e_{n+j}} X^{e_j + e_{n+j}} \sigma_1 \bigr ) \bigl ( \Delta^{e_j} \sigma_2 \bigr ) - \bigl ( \Delta^{e_j} X^{e_j + e_{n+j}} \sigma_1 \bigr ) \bigl ( \Delta^{e_{n+j}} \sigma_2 \bigr ) \nonumber \\
	&- \bigl( \Delta^{e_j} \sigma_1 \bigr ) \bigl ( \Delta^{e_{n+j}} X^{e_j + e_{n+j}} \sigma_2 \bigr ) + \bigl ( \Delta^{e_{n+j}} \sigma_1 \bigr ) \bigl ( \Delta^{e_j} X^{e_j + e_{n+j}} \sigma_2 \bigr ) \nonumber \\
	&- \bigl ( \Delta^{e_j + e_{n+j}} \sigma_1 \bigr ) \bigl ( X^{e_j + e_{n+j}} \sigma_2 \bigr ) \Bigr ) \nonumber \\
	&+ \frac{1}{2} \Bigl ( \bigl ( X^{e_{2n+1}} \sigma_1 \bigr ) \bigl ( \Delta^{e_{2n+1}} \sigma_2) - \bigl ( \Delta^{e_{2n+1}} \sigma_1 \bigr ) \bigl ( X^{e_{2n+1}} \sigma_2 \bigr ) \Bigr ). \label{eq:asym_exp_comp_onehalf_tau_Hn_Poisson_hom2}
\end{align}
\end{example}

\subsection{The \texorpdfstring{$G$}{G}-Poisson bracket} \label{subs:Poisson}
Despite the striking differences between the asymptotic expansions in the Examples~\ref{ex:asym_exp_comp_Weyl} and \ref{ex:asym_exp_comp_onehalf_tau_Hn}, we observe that the term $\omega_1$ can in both cases be interpreted as a kind of Poisson bracket (up to a multiplicative constant) defined by the first strata of the respective groups. 
 As we will see in the following, this property is actually shared by all the symmetric quantizations on stratified groups.

\begin{definition}\label{defn.Hom.Poisson.B}
Let $G$ be a stratified group and let it be equipped with the canonical dilations, i.e., $v_j = j$ for all $j = 1, \ldots, \dimG$. Let $m_1, m_2 \in \R$ and $0 \leq \delta \leq \rho \leq 1$. Then for any two symbols $\sigma_1 \in S^{m_1}_{\rho,\delta}(G)$ and $\sigma_2\in S^{m_2}_{\rho,\delta}(G)$ we define the \textit{$G$-Poisson bracket} of $\sigma_1$ and $\sigma_2$ to be the symbol
\begin{align}
    \{ \sigma_1, \sigma_2 \}_{G} := \sum_{\substack{[\alpha]=1}}\Bigl ( \bigl( X^\alpha \sigma_1 \bigr ) \bigl ( \Delta^\alpha \sigma_2 \bigr ) - \bigl ( \Delta^\alpha \sigma_1 \bigr ) \bigl ( X^\alpha \sigma_2 \bigr ) \Bigr ) \in S^{m_1 + m_2 + (\rho - \delta)}_{\rho,\delta}(G),
\end{align}
\end{definition}

\begin{remark}
The fact that pointwise products of symbols belong to the expected symbol classes (even for $\delta = \rho$) was proved in \cite[Thm.~5.2.22~(ii)]{FR} (see also Remark~\ref{rem:pw_product}).
\end{remark}

\begin{remark}
    Due to the general non commutative nature of our setting, the $G$-Poisson bracket does not satisfy the anticommutativity, the Leibniz rule and the Jacobi identity which characterize Poisson brackets.  However, when $G$ is commutative, hence $G=\mathbb{R}^n$, our definition agrees with the standard one and we have a Poisson bracket in the usual sense (see Remark \ref{rmk.eucl.poss}). The compatibility with the Euclidean case and the possible applications in Hamiltonian mechanics in the group setting lead us to call this object the $G$-Poisson bracket.
\end{remark}

\begin{remark}\label{rmk.eucl.poss}
On $G = \R^n$, where the classical Poisson bracket is defined by
\begin{align*}
    \{ \sigma_1, \sigma_2 \} := \sum_{j=1}^n\Bigl ( \bigl( \partial_{x_j} \sigma_1 \bigr ) \bigl (  -i\partial_{\xi_j} \sigma_2 \bigr ) - \bigl ( -i\partial_{\xi_j} \sigma_1 \bigr ) \bigl ( \partial_{x_j} \sigma_2 \bigr ) \Bigr ),
\end{align*}
we have
\begin{align}
    \omega_1=\frac{1}{2}\{ \sigma_1, \sigma_2 \}=\frac{1}{2}\{ \sigma_1, \sigma_2 \}_{G}.
\end{align}
Since the first stratum of the Abelian Lie algebra $\R^n$ (equipped with the trivial Lie bracket) coincides with the whole Lie algebra, $\omega_1$ can be fully expressed by the $G$-Poisson bracket. This is no coincidence as we will shortly show in Proposition~\ref{prop.poiss.bra}.
\end{remark}

Before we prove the proposition, we briefly make an interesting observation about the symbols $\omega_j$ in \eqref{eqn.Tj}, in particular about $\omega_1$, when $\sigma_1 \in S^{m_1}_{1, 0}(G)$, $\sigma_2 \in S^{m_2}_{1, 0}$ are additionally homogeneous in the sense of \cite[Def.~4.1]{FF_DM}. We recall that any symbol $\sigma = \bigl \{ \sigma(x, \pi): \RS^\infty \to \RS \mid x \in G, \pi \in \Ghat \bigr \}$ on a graded group is \textit{homogeneous} of degree $m \in \mathbb{R}$ (or $m$-homogeneous) if
\begin{align}
    \sigma (x,r\cdot \pi)= r^m \sigma (x,\pi)
\end{align}
for  all $x\in G$, a.e. $\pi \in \Ghat$ and a.e. $r \in \R^+$, where
\begin{align} \label{eq:dilation_symbol}
    \sigma (x,r\cdot \pi) := r^{-\hdim} \F_{y \mapsto \pi} \bigl ( \kappa_\sigma(x, \, D_{r^{-1}}(y)) \bigr ).
\end{align}
An $m$-homogeneous symbol is called regular if it is smooth in the $x$-variable and satisfies the semi-norm estimates \eqref{defn.seminorms}, and the space of all $m$-homogeneous symbols is denoted by $\dot{S}^m$. Clearly, $\dot{S}^m \subseteq S^m_{1, 0}(G)$.

Now observe that, given two symbols $\sigma_1 \in \dot{S}^{m_1},\sigma_2 \in \dot{S}^{m_2}$, then for all $\alpha,\beta\in\mathbb{N}_0^n$,
\begin{align}
    \Delta^\alpha\sigma_1( x,r\cdot \pi)= r^{m-[\alpha]} (\Delta^\alpha \sigma_1)(x,\pi)
\end{align}
and
\begin{align} \label{eq:diff_op_hom_symbols}
    \bigl ( \Delta^\alpha\sigma_1(x,r \cdot \pi) \bigr ) \bigl ( \Delta^\beta \sigma_2(x, r \cdot \pi) \bigr )= r^{m_1+m_2-[\alpha]-[\beta]}(\Delta^\alpha\sigma_1)(x,\pi)(\Delta^\beta\sigma_2)(x,\pi),\quad \quad 
\end{align}
that is, the two symbols above are homogeneous of degree $m_1-[\alpha]$ and $m_1+m_2-[\alpha]-[\beta]$, respectively.
This immediately implies that for any graded group $G$, any $\tau$ satisfying (HP) and any two symbols $\sigma_1 \in \dot{S}^{m_1}$ and $\sigma_2 \in \dot{S}^{m_2}$, the symbol $\omega_j$ defined by \eqref{eqn.Tj} belongs to $ \dot{S}^{m_1+m_2-j}$. These arguments can now easily be used to generalize well-known facts about classical (or polyhomogeneous) pseudo-differential symbols on $\R^n$ to the setting of general graded groups.

\begin{proposition}\label{prop.poiss.bra}
Let $G$ be a stratified group, equipped with the canonical dilations, and let $\tau$ be a symmetric quantizing function that satisfies  (HP). Then, for any $\sigma_1 \in S^{m_1}_{\rho,\delta}(G)$ and $\sigma_2\in S^{m_1}_{\rho,\delta}(G)$, with $m_1, m_2 \in \R$ and $0 \leq \delta < \frac{\rho}{v_\dimG} \leq 1$, the uniquely determined summand $\omega_1$ of order $m_1 + m_2 - (\rho - \delta)$ in the asymptotic expansion \eqref{eqn.new.asympt.formula} of their composite symbol is given by $\frac{1}{2}\{ \sigma_1, \sigma_2 \}_{G}$.

In particular, this holds true for the quantizing functions \eqref{eq:MR_tau} and \eqref{eq:onehalf_tau} on general stratified groups and the family \eqref{eq:fam_tau_Hn} on $G = \H$.

In the special case of $\sigma_1 \in  \dot{S}^{m_1}$ and $\sigma_2\in \dot{S}^{m_2}$ we even have $\{\sigma_1,\sigma_2\}_G\in \dot{S}^{m_1+m_2-1}$.
\end{proposition}

\begin{proof}
    Denote by $\Lie{g} = \bigoplus_{i = 1}^k \Lie{g}_i$ the stratification of $\Lie{g}$ and by $\{ X_{1_i}, \ldots, X_{\dimG_i}\}$ the arbitrary but fixed bases of its strata $\Lie{g}_i$, $i = 1, \dots, k$, chosen to begin with (cf.~Subsection~\ref{subs:gr_strat_gr}). Since $\tau$ satisfies $(HP)$ with respect to the basis $\{ X_{1_1}, \ldots, X_{\dimG_{1}} \}$ by assumption, we have $c^\tau_{j_1}(x) = C^\tau_{j_1} x_{j_1}$, with $C^\tau_{j_1} \in \R$, for all coordinates $j_1 = 1, \ldots, \dimG_1$ of the first stratum. These are precisely the components $c^\tau_j$ of $\tau$ which are homogeneous of order $v_1 = 1$. Moreover, by Theorem~\ref{thm:asym_exp_adjoint_tau}, the quantizing function $\tau$ is symmetric if and only if $\tau(x) = \tau(x^{-1}) x$ for all $x \in G$. In combination with the previous observation, this formula immediately implies that $C^\tau_{j_1} = \frac{1}{2}$ for all $j_1 = 1, \ldots, \dimG_1$ (cf.~Example~\ref{ex:fam_sym_fun_Hn}). Now, it is easily seen that because of the stratification of $\Lie{g}$, the monomials $x_{j_1}$ are precisely the ones satisfying the defining condition \eqref{eq:can_hom_pol} for the chosen basis $\{ X_{1_1}, \ldots, X_{\dimG_1}\}$ of $\Lie{g}_1$. Hence, they are precisely the homogeneous polynomials which determine the difference operators of homogeneous order $1$ in the $\tau$-caclulus under consideration. A calculation identical to the one in Example~\ref{ex:asym_exp_comp_onehalf_tau_Hn} finally shows that the summand $\omega_1$ of  order $m_1 + m_2 - (\rho - \delta)$. The property for homogeneous symbols now follows directly from \eqref{eq:diff_op_hom_symbols}. This completes the proof.
\end{proof}

\section{Invariance properties of symmetric quantizations} \label{sec:symp_inv_G}

In this section we will present an appropriately restricted version of the symplectic invariance of the Weyl quantization on $\R^n$ that generalizes to all graded groups and show that
\begin{itemize}
    \item[--] it is always satisfied by the quantization arising from
\begin{align*}
    \tau(x) = \exp \Bigl ( \frac{1}{2} \log(x) \Bigr ) = \Bigl (\frac{x_1}{2}, \ldots, \frac{x_\dimG}{2} \Bigr );
\end{align*}
    \item[--] it singles out the latter among all symmetric quantizations at least on $\R^n$ and the Heisenberg group $\H$.
\end{itemize}

\subsection{The Euclidean case} \label{subs:symp_inv_Rn}

To start with, we give a quick review of symplectic invariance on $\R^n$, following the presentation in \cite[\S.~2.1]{F1}. We recall that the symplectic group $\Sp$ is the group of all $2n \times 2n$ real matrices which preserve the standard symplectic form on $\R^{2n}$, given by
\begin{align*}
    \omega(v, w) := \sum_{j=1}^n (v_j w_{n+j} - v_{n+j} w_j),
\end{align*}
for $v = (v_1, \ldots, v_\dimG)$, $w = (w_1, \ldots, w_n) \in \R^n$. The group $\Sp$ is generated by either of the unions of subgroups $\mathrm{D} \cup \mathrm{N} \cup \{ J \}$ and $\mathrm{D} \cup \overline{\mathrm{N}} \cup \{ J \}$, where
\begin{align*}
    \mathrm{N} :=& \left \{ \left (
    \begin{array}{cc}
        I & C \\
        0 & I
    \end{array}
    \right ) \mid C = C^* \in \mathrm{M}(n \times n, \mathbb{R}) \right \}, \\ \overline{\mathrm{N}} :=& \left \{ \left (
    \begin{array}{cc}
        I & 0 \\
        C & I
    \end{array}
    \right ) \mid C = C^* \in \mathrm{M}(n \times n, \mathbb{R}) \right \}, \\
    \mathrm{D} :=& \left \{ \left (
    \begin{array}{cc}
        A & 0 \\
        0 & {A^*}^{-1}
    \end{array}
    \right ) \mid  A \in \mathrm{GL}(n, \mathbb{R}) \right \}, \hspace{15pt} J := \left (
    \begin{array}{cc}
        0 & I \\
        -I & 0
    \end{array}
    \right ).
\end{align*}
The symplectic invariance~\eqref{eq:symp_inv_WQ} now arises from the metaplectic group $\mathrm{Mp}(2n, \mathbb{R})$, a double cover of $\Sp$ which is characterized by the exact sequence
\begin{align*}
    0 \longrightarrow \mathbb{Z}_2 \longrightarrow \mathrm{Mp}(2n, \mathbb{R}) \longrightarrow \Sp \longrightarrow 0.
\end{align*}
The metaplectic group can be represented as a group of unitary operators on $L^2(\R^n)$, which are intimately related to the Schr\"{o}dinger representations $\pi_\lambda \in \Hhat$, $\lambda \in \R \setminus \{ 0 \}$, whose definition was recalled in~\eqref{eq:Schr_rep}.
Namely, for any arbitrary but fixed $\lambda \in \mathbb{R} \setminus \{ 0 \}$ and any $S \in \Sp$ there exists an operator $\eta_\lambda(S) \in \mathcal{U}(L^2(\mathbb{R}^n))$, uniquely determined up to a factor $\pm 1$, such that for the block matrix
\begin{align}
    \tilde{S} :=
    \left (\begin{array}{cc}
        S & 0 \\
        0 & 1
    \end{array}
    \right ) \in \mathop{GL}(2n+1; \R) \label{extSymp}
\end{align}
the identity
\begin{align}
    \pi_\lambda(\tilde{S}x) = \eta_\lambda(S) \hspace{2pt} \pi_\lambda(x) \hspace{2pt} \eta_\lambda(S)^{-1} \label{eq:symp_inv_Schr_rep}
\end{align}
holds true for all $x \in \H$. Moreover, for any two $S_1, S_2 \in \Sp$,
\begin{align*}
    \eta_\lambda(S_1) \hspace{2pt} \eta_\lambda(S_2) = \pm \eta_\lambda(S_1 S_2).
\end{align*}
For the realization of $\pi_\lambda$ as in \eqref{eq:Schr_rep}, the action of the metaplectic representation on  $h \in L^2(\mathbb{R}^n)$ can be characterized by
\begin{align}
    \left(\eta_\lambda
    \left (\begin{array}{cc}
        A & 0 \\
        0 & {A^*}^{-1}
    \end{array}
    \right )
    h \right)(u) &= \sqrt{\lambda} \det(A)^{-1/2} \hspace{2pt} h( \sqrt{\lambda} A^{-1} u), \label{MP1} \\
    \left(\eta_\lambda
    \left (\begin{array}{cc}
        I & 0 \\
        C & I
    \end{array}
    \right )
    h \right)(u) &= e^{-\frac{i \lambda}{2} \langle u, C u \rangle} \hspace{2pt} h(u), \\
    \left( \eta_\lambda
    \left (\begin{array}{cc}
        0 & I \\
        -I & 0
    \end{array}
    \right)
    h \right)(\xi) &= (2 \pi)^{-\frac{n}{2}} \bigl ( \F h \bigr )(-\sqrt{\lambda} \xi).
\end{align}
In the special case when $\lambda = 1$, the invariance property \eqref{eq:symp_inv_Schr_rep} is easily shown to yield the symplectic invariance of the Weyl quantization~\eqref{eq:symp_inv_WQ}:
\begin{align*}
    \Opw(\sigma \circ S) = U^{-1}_S \Opw(\sigma) U_S
\end{align*}
for all $S \in \Sp$ and the unitary operators $U_S = \eta_1(S^*)^{-1}$. This invariance a priori holds for all symbols $\sigma \in \SC(\R^n \times \Rhat^n)$, but extends to all $\sigma \in \TD(\R^n \times \Rhat^n)$ since the operators $\eta_\lambda(S)$ preserve $\SC(\R^n)$ and $\TD(\R^n)$.

\subsection{The general graded case} \label{subs:symp_inv_G}

For an arbitrary graded group $G$ we cannot in general expect the existence of continuous maps from $G \times \Ghat$ into itself that generalize in a meaningful way the action of the symplectomorphism $S \in \mathrm{N} \cup \{ J \}$ since the spaces $G$ and $\Ghat$ coincide if and only if $G = \R^n$ for some $n \in \N$. However, a straight-forward generalization is possible for the action of the symplectomorphisms $S \in \mathrm{D}$.

To show this, we recall that $\mathrm{GL}(n, \mathbb{R})$ is precisely the group of automorphisms of the Lie group $(\R^n, +)$. Now, it is easy to check that for any $\aut \in \mathrm{Aut}(G)$ and any $\pi \in [\pi] \in \Ghat$ of a given graded group $G$, the composite map $\pi_\aut := \pi \circ \aut$ defines a unitary irreducible representation of $G$. It follows that $\pi_\aut$ is an element of some equivalence class $[\pi'] \in \Ghat$, which, depending on the type of automorphism $\aut$, may or may not coincide with $[\pi]$.

Let us illustrate this dichotomy for two particular types of subgroups of $\mathrm{Aut}(G)$:
\begin{itemize}
    \item[(i)]
the normal subgroup of inner automorphisms
\begin{align*}
    \mathrm{conj}_y := x \mapsto y x y^{-1}, \hspace{5pt} y \in G;
\end{align*}
    \item[(ii)] any group of homogeneous dilations $\{ D_r \}_{r > 0}$ on $G$, for which the direct summands of the given gradation $\mathfrak{g}=\bigoplus_{i =1}^{\infty} \mathfrak{g}_i$ form eigenspaces of the matrix $\log(D_1)$, with eigenvalues $v_1, \ldots, v_\dimG \in \N$ (cf.~Subsection~\ref{subs:gr_strat_gr}).
\end{itemize}
On the one hand, if $\aut = \mathrm{conj}_y$ for some $y \in G$, then $[\pi_{\mathrm{conj}_y}] = [\pi]$ for any $\pi \in \Ghat$ since the representations are intertwined by
\begin{align}
   \pi_{\mathrm{conj}_y}(x) = \pi(y) \pi(x) \pi(y)^{-1} \label{eq:conj_rep}
\end{align}
for all $x \in G$. On the other hand, if $\aut = D_r$ for some $r \in \R^+ \setminus \{ 1 \}$, then $[\pi_\aut] \neq [\pi]$ unless $[\pi] = [1] \in \Ghat$. This can be seen by realizing $\pi$ as an induced representation $\pi_l$, which by Kirillov's orbit method is uniquely determined by the (any) representative $l \in \Lie{g}^*$ of the corresponding co-adjoint orbit $\mathcal{O}_\pi = \mathrm{Ad}^*(G)l \subseteq \Lie{g}^*$: by a routine computation, one can show that $\pi_l \circ D_r$ coincides with the induced representation $\pi_{D_r^*(l)}$, where $D_r^*: \Lie{g}^* \to \Lie{g}^*$ denotes the adjoint of $D_r: \Lie{g} \to \Lie{g}$, $r > 0$. Since the representatives $l$ and $D_r^*(l)$ necessarily determine disjoint co-adjoint orbits, it follows that $[\pi] = [\pi_l] \neq [\pi_{D_r^*(l)}] = [\pi_{D_r}]$. (Compare with the closely related \eqref{eq:dilation_symbol}.)

In order to show now that any $\aut \in \mathrm{Aut}(G)$ yields an invariance relation of the type \eqref{eq:symp_inv_WQ} for the symmetric quantization $\Optau$ with $\tau(x) = \exp(\frac{1}{2}\log(x))$, let us settle the necessary notation. Thus, let us denote by $\aut^*$ the induced dual automorphism $\pi \mapsto \pi_\aut$ of $\Ghat$ and by $\mathcal{S} = \mathcal{S}(\aut)$ the automorphism
\begin{align*}
     \mathcal{S}: G \times \Ghat &\to G \times \Ghat, \\
     (x, \pi) &\mapsto (\aut x, {\aut^*}^{-1} \pi).
\end{align*}
With this at hand, we can show the following statement.

\begin{theorem} \label{thm:symp_inv_G}
Let $G$ be a graded group and let $\tau$ be the symmetry function $\tau(x) = \exp(\frac{1}{2}\log(x))$. Then for any $\aut \in \mathrm{Aut}(G)$ there exists a unitary map $U_\mathcal{S}$ such that
\begin{align}
    \Optau(\sigma \circ \mathcal{S})f = U_\mathcal{S}^{-1} \Optau(\sigma) U_\mathcal{S}f \label{eq:symp_inv_onehalf_tau}
\end{align}
for all $\sigma \in S^{0}_{\rho,\delta}(G)$, with $0 \leq \delta < \frac{\rho}{v_\dimG} \leq 1$, and all $f \in L^2(G)$.
\end{theorem}

\begin{proof}
Let $\aut \in \mathrm{Aut}(G)$ and choose an arbitrary $\sigma \in S^{0}_{\rho,\delta}(G)$, with $0 \leq \delta < \frac{\rho}{v_\dimG} \leq 1$, and $f \in L^2(G)$. For the sake of explicit computations involving absolutely convergent integrals, we will assume that $f$ lies in the dense subspace $\SC(G)$. This is possible without loss of generality since by Corollary~\ref{thm:cont_on_Sobolev} both $\sigma \circ S$ and $\sigma$ quantize continuous operators on $L^2(G)$, therefore the computations  extend to all of $L^2(G)$.

To begin with, observe that
\begin{align*}
    \Optau(\sigma \circ S)f(x) &= \iint\limits_{\widehat{G} \times G} \Tr \biggl ( \pi(y^{-1} x) \sigma \Bigl ( \aut \bigl ( x\tau(y^{-1}x)^{-1} \bigr ), \pi_{\aut^{-1} } \Bigr ) f(y) \biggr ) \, dy \, d\mu(\pi) \\
    &= \iint\limits_{\widehat{G} \times G} \Tr \biggl ( \pi_{\aut}(y^{-1} x) \sigma \Bigl ( \bigl ( \aut(x) \aut(\tau(y^{-1}x))^{-1} \bigr ), \pi \Bigr ) f(y) \biggr ) \, dy \, d\mu(\pi_{\aut})
\end{align*}
for all $x \in G$, but in order to continue our computation, we need to show that
\begin{align}
    \aut \bigl ( \tau(y^{-1}x) \bigr ) = \tau \bigl ( \aut(y)^{-1} \aut(x) \bigr ) \label{eq:aut_preserves_onehalf_tau}
\end{align}
holds for all $x, y \in G$ for our choice of $\tau$. However, since the Lie algebra automorphism $\aut' := d\aut(e_G)$ not only preserves the Lie bracket on $\Lie{g}$ but also the Baker-Campbell-Hausdorff product
\begin{align*}
    X \star Y :=& \log \bigl ( \exp(X) \exp(Y) \bigr ) \\
    =& X + Y + \frac{1}{2} [X, Y] + \frac{1}{12} [X, [X, Y]] - \frac{1}{12} [Y, [X, Y]] + \ldots,
\end{align*}
the identity~\eqref{eq:aut_preserves_onehalf_tau} follows from
\begin{align*}
    \log \bigl ( \aut \bigl ( \tau(y^{-1}x) \bigr ) &= \aut' \log \bigl ( \tau(y^{-1}x) \bigr ) = \aut' \Bigl ( \frac{1}{2} \bigl ( (-Y) \star X \bigr ) \Bigr ) \\
    &= \frac{1}{2} \bigl ( (-\aut'Y) \star \aut' X \bigr ) = \log \Bigl ( \tau \bigl ( (\aut y)^{-1} \aut(x) \bigr ) \Bigr ),
\end{align*}
where $X := \log(x)$ and $Y := \log(y)$. Using the change of variables $y' := \aut (y)$, we then get
\begin{align*}
    &\Optau(\sigma \circ S)f(x) = \! \int\limits_{\widehat{G}} \Tr \biggl ( \pi_{\aut}(x) \int\limits_G \! \sigma \Bigl ( \bigl ( \aut(x) \tau \bigl ( \aut(y)^{-1} \aut(x) \bigr )^{-1} \bigr ) \pi \Bigr )  \pi_{\aut}(y)^* f(y) dy \biggr ) d\mu(\pi_{\aut}) \\
    &= \frac{1}{|\det(\aut')|} \int\limits_{\widehat{G}} \Tr \biggl ( \pi(\aut (x)) \!\! \int\limits_G \! \sigma \Bigl ( \bigl ( \aut(x) \tau \bigl ( y'^{-1} \aut(x) \bigr )^{-1} \bigr ), \pi \Bigr ) \pi(y')^* f(\aut^{-1} (y')) dy' \biggr ) d\mu(\pi).
\end{align*}
Since the identity $d\mu(\pi_{\aut}) = |\det(\aut')| \, d\mu(\pi)$ is an immediate consequence of the Plancherel formula~\ref{eq:Plancherel}, the above identity can be rewritten as
\begin{align*}
    \Optau(\sigma \circ S)f = \bigl ( \Optau(\sigma) (f \circ \aut^{-1} ) \bigr ) \circ \aut = U_\mathcal{S}^{-1} \Optau(\sigma) U_\mathcal{S}f
\end{align*}
for the unitary operator
\begin{align*}
    U_\mathcal{S}: L^2(G) &\to L^2(G), \\
    U_\mathcal{S}f &= |\det(\aut')|^{-\frac{1}{2}} \, f \circ \aut^{-1}.
\end{align*}
Since $\sigma \in S^{0}_{\rho,\delta}(G)$ and $f \in \SC(G)$ were arbitrary, this completes the proof.
\end{proof}

At this point we ought to make a few relevant observations. The proof of Theorem~\ref{thm:symp_inv_G} not only works for $\tau(x) = \exp(\frac{1}{2}\log(x))$, but in fact any admissible $\tau$ that commutes with the group automorphisms $\aut \in \mathrm{Aut}(G)$. On general graded $G$ this is obviously the case for $\tau = e_G$, the quantizing function of the Kohn-Nirenberg quantization, while on $G = \R^n$ all linear quantizing functions $\tau (x) = \tau x$, $\tau \in [0, 1]$, commute with all $A \in \mathrm{GL}(n, \mathbb{R}) = \mathrm{Aut}(\R^n)$.

So, unlike the full symplectic invariance~\eqref{eq:symp_inv_WQ} on $\R^n$, on a generic graded group the invariance~\eqref{eq:symp_inv_onehalf_tau} under group automorphisms does not necessarily single out any specific $\tau$-quantization. In particular, the homogeneous dilations $\{ D_r \}_{r > 0}$ clearly commute with all $\tau$ that satisfy (HP). However, a combination of \eqref{eq:symp_inv_onehalf_tau} and the preservation of involution~\eqref{eq:sym_quant} suffice already on $\R^n$: even if we admit all symmetry functions $\tau:\R^n \to \R^n$ which satisfy (HP) for any given admissible homogeneous structure with weights $v_1, \ldots, v_\dimG \in \N$, only the linear symmetry function $\tau(x) = \frac{1}{2} x$ commutes with all $A \in \mathrm{GL}(n, \mathbb{R})$.

Admittedly, the groups of automorphisms can vary wildly among general graded groups, and canonical subgroups like the inner automorphism may not suffice to single out specific symmetry functions, as we will see in the case of $G = \H$ in Subsection~\ref{subs:symp_inv_Hn}, but a relatively explicit description of $\mathrm{Aut}(G)$ combined with the use of a symmetric function $\tau$ may nevertheless do the job, as this is the case on both $\R^n$ and $\H$.

Let us also point out that a related, yet different type of invariance~\cite{FFF} has recently been shown in the context of a novel semi-classical calculus on graded groups (cf.~\cite{FF_SC, FF_QE, F_SC}).

\subsection{The Heisenberg case} \label{subs:symp_inv_Hn}

In this subsection we illustrate the general arguments of Subsection~\ref{subs:symp_inv_G} for the special case of the Heisenberg group $\H$ and prove the following sharper version of Theorem~\ref{thm:symp_inv_G}:
\begin{theorem} \label{thm:Weyl_quant_Hn}
Let $\H$ be the Heisenberg group, equipped with the canonical dilations. Then among all quantizing functions $\tau: \H \to \H$ which satisfy (HP) only $$\tau(x) = \exp(\frac{1}{2}\log(x))$$ satisfies the following two conditions for all $\sigma \in S^{0}_{\rho,\delta}(G)$, $0 \leq \delta < \frac{\rho}{v_{2n+1}} = \frac{\rho}{2} \leq 1$, and all $f \in L^2(\H)$:
\begin{itemize}
    \item[-] Preservation of involution:
    \begin{align*}
        \Optau(\sigma)^*f = \Optau(\sigma^*)f;
    \end{align*}
    \item[-] Automorphic invariance:\\
    for each $\aut \in \mathrm{Aut}(\H)$ there exists an operator $U_\mathcal{S} \in \mathcal{U}(L^2(\H))$ such that
    \begin{align*}
    \Optau(\sigma \circ \mathcal{S})f = U_\mathcal{S}^{-1} \Optau(\sigma) U_\mathcal{S}f.
    \end{align*}
\end{itemize}
That is, among all admissible symmetric quantizations on $\H$, only the one defined by $\tau(x) = \exp(\frac{1}{2}\log(x))$ is invariant under all automorphic changes of variables.
\end{theorem}

\medskip

To set the stage for the proof of Theorem~\ref{thm:Weyl_quant_Hn}, we will recall the classification of $\mathrm{Aut}(\H)$, following the exposition in \cite[\S~1.2]{F1}. For each type of $\aut \in \mathrm{Aut}(\H)$ we will additionally provide an explicit description of the map $\pi \mapsto \pi_\aut$, which we used in Subsection~\ref{subs:symp_inv_G}.

Thus, let us recall that each $\aut \in \mathrm{Aut}(\H)$, when viewed as a map on the underlying vector space $\mathbb{R}^{2n+1}$, is a linear isomorphism which can be uniquely written as\footnote{The fourth map in \cite[\S~1.2]{F1} is $\iota = \tilde{J} \circ \Theta$, whereas one may equivalently use $\Theta$, which proved very convenient in \cite[Ch.~6]{FR}.}
\begin{align*}
    \aut = \tilde{S} \circ \mathrm{conj}_y \circ D_r \circ \Theta
\end{align*}
for some $\tilde{S}$ as in \eqref{extSymp}, some $y \in \H$ and some $r > 0$, where
\begin{align*}
    \mathrm{conj}_y(x) =& \Bigl( x_1, \ldots, x_{2n}, x_{2n+1} + \sum_{j=1}(y_j x_{n+j} - y_{n+j} x_j) \Bigr), \\
    D_r(x) =& (r x_1, \ldots, r x_{2n}, r^2 x_{2n+1}), \\
    \Theta (x_1, \ldots, x_{2n}, x_{2n+1}) :=& (x_1, \ldots, x_n, -x_{n+1}, \ldots, -x_{2n}, -x_{2n+1}).
\end{align*}
In the computation of the map $\aut^* = \pi \mapsto \pi_\aut = \pi \circ \aut$ for each type of automorphism, we restrict ourselves to the set of Schr\"{o}dinger representations since the remaining one-dimensional representations have Plancherel measure zero.

(i) For any symplectic map $S \in \Sp$ the invariance property~\eqref{eq:symp_inv_Schr_rep} implies that the dual automorphism $(\tilde{S})^*$ is given by
\begin{align*}
    \pi_\lambda \mapsto \eta_\lambda(S) \pi_\lambda \eta_\lambda(S^{-1}) \in [\pi_\lambda]
\end{align*}
for all $\pi_\lambda \in \Hhat$, and $d\mu(\pi_{\tilde{S}}) = d\mu(\pi)$ since all symplectic matrices are of determinant $1$.

(ii) A general formula for the conjugations is given by \eqref{eq:conj_rep}, and by the bi-invariance of the Haar measure on nilpotent Lie groups we have $d\mu(\pi_{\mathrm{conj}_y}) = d\mu(\pi)$.

(iii) For the homogeneous dilations $\{ D_r \}_{r >0}$ we compute
\begin{align*}
    \bigl ( \F(f \circ D_r) \bigr )(\pi_\lambda) &= \int_{\H} f\bigl( D_r(x) \bigr) \hspace{2pt} \pi_\lambda(x)^* \,dx \\
    &= r^{-\hdim} \hspace{2pt} \int_{\H} f(x) \hspace{2pt} \pi_\lambda \Bigl ( \bigl( D_{r^{-1}}(x) \bigr) \Bigr )^* \,dx \\
    &=  r^{-\hdim} \hspace{2pt} \int_{\H} f(x) \hspace{2pt} \pi_{r^{-2} \lambda}(x)^* \,dx \\
    &= r^{-\hdim} \hspace{2pt} \widehat{f}(\pi_{r^{-2}\lambda})
\end{align*}
for any $f \in \SC(\H)$ and the homogeneous dimension  $\hdim = 2n+2$. The Plancherel formula~\eqref{eq:Plancherel} then implies that the dual automorphism $(D_r)^*$ is given by
\begin{align*}
    \pi_\lambda \mapsto \pi_{r^2\lambda}
\end{align*}
for all $\pi_\lambda \in \Hhat$, and that $d\mu(\pi_{D_r}) = r^{\hdim} \, d\mu(\pi)$.

(iv) For the Haar measure-preserving, self-inverse map $\Theta = \Theta^{-1}$ one immediately gets $\pi_\lambda(\Theta x) = \pi_{-\lambda}(x)$ for all $x \in \H$. So, the dual automorphism $\Theta^*$ is given
\begin{align*}
    \pi_\lambda \mapsto \pi_{-\lambda}
\end{align*}
for all $\pi_\lambda \in \Hhat$, and $d\mu(\pi_{\Theta}) =  d\mu(\pi)$.

\medskip

\begin{proof}[Proof of Theorem~\ref{thm:Weyl_quant_Hn}]
Suppose that the quantizing function $\tau$ satisfies the condition (HP). Then by Theorem~\ref{thm:asym_exp_adjoint_tau}, the associated $\tau$-quantization preserves the involution if and only if $\tau$ is symmetric. So, for the rest of the proof we may limit ourselves to symmetric quantizations.

Among the latter, the quantization defined by $\tau(x) = \exp(\frac{1}{2}\log(x))$ is invariant under automorphic changes of variables by Theorem~\ref{thm:symp_inv_G}. To complete the proof, it therefore suffices to show that no other symmetry function $\tau$ that satisfies (HP) commutes with all $\aut \in \mathrm{Aut}(\H)$. By the definition of (HP), all such $\tau$ commute with the homogeneous dilations. To check the remaining types of automorphisms, we recall from Example~\ref{ex:fam_sym_fun_Hn} that the symmetry functions which satisfy (HP) for the homogeneous dilation structure on $\H$ are precisely the functions of the form
\begin{align*}
	\tau(x) = \Bigl ( \frac{x_1}{2}, \ldots, \frac{x_{2n}}{2}, \frac{x_{2n+1}}{2}+ \sum_{j, k = 1}^{2n} c_{j, k} \, x_j x_k \Bigr )
\end{align*}
for any choice of $c_{j, k} \in \R$, $j, k = 1, \ldots, 2n$. One easily checks that all of these functions also commute with  the inner autmorphisms $\mathrm{conj}_y$, $y \in \H$,
so we may pass to checking the commutation with the automorphism $\Theta$.

Note that due  to its form $\tau$ commutes with $\Theta$ if and only if $\tau(\Theta(x))_{2n+1}=\Theta(\tau(x))_{2n+1}$, i.e. if the matrix  $C:=\{c_{j,k}\}_{j,k=1}^{2n}$ of the quadratic form $q(x)=\sum_{j, k = 1}^{2n} c_{j, k} \, x_j x_k$  satisfies certain properties:
Since the antisymmetric part of the quadratic form
\begin{align*}
    q:\,(x_1, \ldots, x_{2n}) \mapsto \sum_{j, k = 1}^{2n} c_{j, k} \, x_j x_k
\end{align*}
always vanishes, it is not restrictive to assume for the rest of the proof that $C$ is symmetric. Now it is easy to check that $\tau(\Theta(x))_{2n+1}=\Theta(\tau(x))_{2n+1}$ if and only if $C$ is of the form
\begin{align*}
    C :=
    \left (\begin{array}{cc}
        0 & C_1 \\
        C_2 & 0
    \end{array}
    \right ),
\end{align*}
where $C_1, C_2$ are two symmetric $n \times n$-matrices. 

With this information about the structure of $C$ at hand, we can focus on the maps $\tilde{S}$ with $S \in \Sp$.
Note that $\tau$ commutes with all $\tilde{S}$, $S \in \Sp$, precisely when the invariance $S^* C S = C$, where, recall, $C$ is symmetric and has a block anti-diagonal form. If we now pick $S = J$, then we must also have $J^*CJ=-C^*$, since the latter relation holds in general for all $C\in \mathrm{M}_{2n}(\mathbb{R})$.\footnote{See, e.g., \cite[Prop.~4.1~(f)]{F1}} The two conditions above and the symmetry of $C$ then yield 
\begin{align}
    C = S^* C S = J^* C J = -C^*=-C,
\end{align}
which implies that $C$ is the null matrix.
This leaves only $\tau(x) = \exp(\frac{1}{2}\log(x))$, and the proof is complete. 
\end{proof}

\section{Applications} \label{sec:applications}

We conclude this paper with a section that collects several useful applications of the $\tau$-calculus developed here. We show that $\tau$-quantized pseudo-differential operators with symbols in H\"{o}rmander classes enjoy the same continuity properties on Lebesgue spaces, inhomogeneous Sobolev spaces, etc. as the Kohn-Nirenberg-quantized operators, and we extend the parametrix construction and the G\r{a}rding inequality for elliptic pseudo-differential to general $\tau$-quantizations. In fact, we will see that these results hold for the Kohn-Nirenberg quantization precisely when they hold for each admissible $\tau$-quantization, provided $0 \leq \delta \leq \frac{\rho}{v_\dimG} \leq 1$. This is in fact a direct consequence of changing quantization according to Theorem~\ref{thm:asym_exp_ch_qu}. Future results based on any of these properties may consequently be proved in the $\tau$-quantization that seems most convenient for the task at hand.

\subsection{Continuity on function spaces} \label{subs:cont_FS}

In addition to the continuity of $\tau$-quantized operators on $\SC(G)$, and consequently on $\TD(G)$, due to Proposition~\ref{prop:cont_on_SC}, we can immediately extend several crucial continuity results of the Kohn-Nirenberg calculus (cf.~\cite[Cor.~5.7.2]{FR}, \cite[Cor.~5.7.4]{FR}, \cite[Thm.~1.2]{CDR_Lp} and \cite[Thm.~4.18]{CDR_Lp}).

For the definitions and properties of the Hardy space $H^1(G)$ and its dual space $\mathrm{BMO}(G)$ we refer to \cite{FS}; for the definitions and embedding properties of the inhomogeneous Besov spaces $B^s_{p, q}(G)$, $p \in (1, \infty)$, $q \in [1, \infty]$, $s \in \R$ we refer to \cite{CR_Be}.

\begin{theorem} \label{thm:cont_on_Sobolev}
Let $m \in \R$, $0 \leq \delta < \min\{\rho,\frac{1}{v_n} \} \leq 1$,\footnote{The assumption $\delta < \rho$ can be weakened to $\delta \leq \rho$, $\delta \neq 1$, for the Kohn-Nirenberg calculus in several cases.} and let $\tau$ satisfy (HP). Then for all $s \in \R$, $p  \in (1, \infty)$, $q \in (0, \infty]$, and $m_p := \hdim (1 - \rho) \bigl | \frac{1}{2} - \frac{1}{p} \bigr |$ the operator $T = \Optau(\sigma)$ lies in
\begin{itemize}
	\item[(i)] $\mathcal{L}(L^2_{s}(G), L^2_{s-m}(G))$ for each $\sigma \in S^m_{\rho, \delta}(G)$;
	\item[(ii)] $\mathcal{L}(L^p_{s}(G), L^p_{s-m}(G))$ for each $\sigma \in S^m_{1, 0}(G)$;
	\item[(iii)] $\mathcal{L}(H^1(G), L^1(G))$ for each $\sigma \in S^{-m}_{\rho, \delta}(G)$ with $m = \frac{\hdim (1 - \rho)}{2}$;
    \item[(iv)] $\mathcal{L}(\mathrm{BMO}(G), L^\infty(G))$ for each $\sigma \in S^{-m}_{\rho, \delta}(G)$ with $m = \frac{\hdim (1 - \rho)}{2}$;
    \item[(v)] $\mathcal{L}(L^p(G))$ for each $\sigma \in S^{-m}_{\rho, \delta}(G)$ with $m \geq m_p$;
    \item[(vi)] $\mathcal{L}(B^s_{p, q}(G))$ for each $\sigma \in S^{-m}_{\rho, \delta}(G)$ with $m \geq m_p$. \\ In particular, this holds for $L^p_{s}(G) = B^s_{p, 2}(G)$.
\end{itemize}
\end{theorem}

\begin{proof}
Since by Theorem~\ref{thm:asym_exp_ch_qu}, every $T$ in (i) and (ii) can be uniquely written as an operator $T = \Op(\sigma')$ with $\sigma'$ in $S^m_{\rho, \delta}(G)$ and $S^m_{1, 0}(G)$, respectively, the statements (i), (ii), (iii) - (v), and (iv) follow directly from  \cite[Cor.~5.7.2]{FR}, \cite[Cor.~5.7.4]{FR}, \cite[Thm.~1.2]{CDR_Lp}, and \cite[Thm.~4.18]{CDR_Lp}, respectively.
\end{proof}

\subsection{Ellipticity, parametrices and the G\r{a}rding inequality} \label{subs:Garding}

This subsection extends the parametrix construction and the G\r{a}rding inequality for elliptic pseudo-differential operators in the Kohn-Nirenberg quantization on graded groups to the setting of $\tau$-quantizations.

Let $\RO$ be a positive essentially self-adjoint Rockland operator of degree $\hdeg$. If $E_{(a, b)}(\RO)$ denotes its spectral projection on $L^2(G)$ corresponding to the interval $(a, b)$, with $0 \leq a \leq b \leq +\infty$, then for any $\pi \in \Ghat$ the corresponding spectral projection $E_{(a, b)}(\pi(\RO))$ on $\RS$ of the essentially self-adjoint operator $\pi(\RO)$ corresponds to the Fourier transform of $E_{(a, b)}(\RO)$ in the representation $\pi$.
 
We recall from \cite[\SS~5.8]{FR} that a smooth symbol $\sigma$ for which
\begin{align*}
	\{\pi(I + \RO)^{\frac{a}{\hdeg}} \sigma(x, \pi) \pi(I + \RO)^{-\frac{b}{\hdeg}} \mid \pi \in \Ghat \bigr \}
\end{align*}
is an element of $L^\infty(\Ghat)$ for some $a, b \in \R$ and all $x \in G$, is said to be \textit{elliptic} with respect to $\RO$ of \textit{elliptic order} $m_o \in \R$ if there exists a $\Lambda \in \R$ such that for any $\gamma \in \R$, $x \in G$, almost any $\pi \in \Ghat$ and any $v \in \mathcal{H}^\infty_{\pi, \Lambda}:= E_{(\Lambda, +\infty)}(\RO)\RS^\infty$ the lower bound
\begin{align*}
    \bigl \| \pi(I + \RO)^{\frac{\gamma}{\hdeg}} \sigma(x, \pi)v \bigr \|_{\RS} \geq C_\gamma \bigl \| \pi(I + \RO)^{\frac{\gamma}{\hdeg}} \pi(I + \RO)^{m_o{\hdeg}} v \bigr \|_{\RS}
\end{align*}
is satisfied for some $C_\gamma = C_\gamma(\sigma, \RO, m_o, \Lambda)$.
In accordance with \cite[Def.~2.8]{CDR_Ga}, such a symbol $\sigma$ and, correspondingly, the Kohn-Nirenberg-quantized operator $\Op(\sigma)$, are called globally elliptic if $\Lambda \leq 0$.

The following result is an extension of the left parametrix construction for elliptic differential operators in the Kohn-Nirenberg quantization, given by \cite[Thm.~5.8.7]{FR}. We recall that these parametrices on graded groups are one-sided as were the ones on homogeneous groups in \cite{CGGP}.

\begin{theorem}
Let $m \in \R$, $0 \leq \delta < \min\{\rho,\frac{1}{v_n}\} \leq 1 $, and let $\tau$ satisfy (HP). Then for any $\sigma \in S^m_{\rho, \delta}(G)$ which is elliptic of elliptic order $m$ there exists a left parametrix $B \in \Optau \bigl ( S^{-m}_{\rho, \delta}(G) \bigr )$ for the operator $A = \Optau(\sigma)$, in the sense that
\begin{align}
    B A - I \in \Optau \bigl ( S^{-\infty}(G) \bigr ).
\end{align}
\end{theorem}

\begin{proof}
Let $\psi \in C^\infty(\R)$ satisfy $\psi|_{(-\infty, \Lambda_1]} = 0$ and $\psi|_{[\Lambda_2, +\infty)} = 1$ for $\Lambda_2 > \Lambda_1 > \Lambda$. For such $\psi$ the proof of \cite[Thm.~5.8.7]{FR} shows that
\begin{align*}
    \Op \bigl ( \psi(\pi(\RO)) \sigma(x, \pi)^{-1} \bigr ) \Op(\sigma) = \psi(\RO) \mod \Op \bigl ( S^{-(\rho - \delta)}_{\rho, \delta}(G) \bigr ).
\end{align*}
Since the $x$-independence of the symbol $\psi(\pi(\RO)) \in S^{-\infty}(G)$ implies
\begin{align*}
    \psi(\RO) = \Op \bigl ( \psi(\pi(\RO)) \bigr ) = \Optau \bigl ( \psi(\pi(\RO)) \bigr ),
\end{align*}
by Theorem~\ref{thm:asym_exp_ch_qu} we also have
\begin{align*}
    \Optau \bigl ( \psi(\pi(\RO)) \sigma(x, \pi)^{-1} \bigr ) A = \psi(\RO) \mod \Optau \bigl ( S^{-(\rho - \delta)}_{\rho, \delta}(G) \bigr ).
\end{align*}
The rest of the proof is identical to the proof of \cite[Thm.~5.8.7]{FR}: if we write $\psi(\RO)=I-(1-\psi(\RO))$ and denote the correction term by $U:=1-\psi(\RO) \in \Optau \bigl ( S^{-(\rho - \delta)}_{\rho, \delta}(G) \bigr )$, then there exists some $T \in \Optau \bigl ( S^0_{\rho, \delta}(G) \bigr )$ with
\begin{align*}
    T \sim I + U + U^2 + \ldots
\end{align*}
such that
\begin{align*}
    B := T \Optau \bigl ( \psi(\pi(\RO)) \sigma(x, \pi)^{-1} \bigr ) \in \Optau \bigl ( S^{-m}_{\rho, \delta}(G) \bigr )
\end{align*}
is the desired parametrix.
\end{proof}

\medskip

The following theorem is an extension of the G\r{a}rding inequality for elliptic pseudo-differential operators in the Kohn-Nirenberg quantization, established by Cardona and the third author in~\cite{CDR_Ga}. The G\r{a}rding inequality is one of the fundamental lower bounds allowing one to control relevant terms in energy estimates for solutions of evolution equations. As a consequence, one can derive the corresponding well-posedness results for evolution equations with $\tau$-quantizations, in analogy to those obtained in \cite{CDR_Ga} and \cite{CDR_Ga2}. We omit the details of such applications to avoid repetitions.

\begin{theorem} \label{thm:Garding}
Let $m \in \R$, $0 \leq \delta < \min\{\rho,\frac{1}{v_n}\} \leq 1$, $\sigma \in S^m_{\rho, \delta}(G)$, and let $\tau$ satisfy (HP). Suppose that the real part of the symbol,
\begin{align*}
    \Re ( \sigma ) := \frac{1}{2} \bigl ( \sigma + \sigma^* \bigr ),
\end{align*}
is elliptic and that the (possibly unbounded) operator $\Re ( \sigma )(x, \pi): \RS \to \RS$ satisfies
\begin{align*}
    \Re ( \sigma )(x, \pi) \geq C_0 \pi(I + \RO)^m
\end{align*}
for almost all $x \in G, \pi \in \Ghat$. Then there exist constants $C_1, C_2 > 0$ such that
\begin{align}
    \Re \bigl ( \langle \Optau(\sigma) f, f \rangle_{L^2} \bigr ) \geq C_1 \| f \|_{L^2_{\frac{m}{2}}} - C_2 \| f \|_{L^2} \label{eq:Garding}
\end{align}
holds true for all $f \in \SC(G)$.
\end{theorem}

\begin{proof}
We recall that Theorem~\ref{thm:asym_exp_ch_qu} established the relation between the symbol $\sigma$ and the uniquely determined symbol $\sigma_{\mathrm{KN}} \in S^m_{\rho, \delta}(G)$ which satisfies $\Optau(\sigma) = \Op(\sigma_{\mathrm{KN}})$. Moreover, by the asymptotic expansion \eqref{eq:asym_exp_tau_KN} the two symbols coincide up to terms of order $m - (\rho - \delta)$, hence $\sigma_{\mathrm{KN}}$ satisfies the conditions in the statement of this theorem up to a correction term of lower order. By Theorem~\cite[Thm.~4.5]{CDR_Ga}, it follows that $\Op(\sigma_{\mathrm{KN}})$ satisfies the G\r{a}rding inequality \eqref{eq:Garding}, hence also $\Optau(\sigma) = \Op(\sigma_{\mathrm{KN}})$ up to a lower order correction. The latter, however, can be absorbed by $C_2 > 0$ since by Theorem~\ref{thm:cont_on_Sobolev}~(i) and the Sobolev duality pairing and embedding properties~\cite[Thm.~4.4.28]{FR}
\begin{align*}
    \bigl | \langle\Op(\sigma - \sigma_{\mathrm{KN}})f, f \rangle_{L^2} \bigr | \leq \| f \|_{L^2_{- m + (\rho - \delta)}}\| f \|_{L^2_{m - (\rho - \delta)}} \leq C_2 \| f \|_{L^2}.
\end{align*}
This completes the proof.
\end{proof}

\section*{Acknowledgments}

The authors would like to express their utmost gratitude to the anonymous referee for so meticulously proofreading their paper and for all the good suggestions they got.

\medskip

The authors were supported by the FWO Odysseus 1 grant G.0H94.18N: Analysis and Partial Differential Equations.

Serena Federico was also supported by the European Union's Horizon 2020 research and innovation programme under the Marie Sk\l odowska-Curie actions H2020-MSCA-IF-2018, grant No. 838661.

David Rottensteiner was also supported the FWO Senior Research Grant G022821N: Niet-commu\-ta\-tieve wavelet analyse.

Michael Ruzhansky was also supported by the Methusalem programme of the Ghent University Special Research Fund (BOF) (Grant number 01M01021) and by the EPSRC grants EP/R003025/2 and EP/V005529/1.

\bibliographystyle{abbrvurl}
\bibliography{Bib_SQ2}

\end{document}